\documentclass[reqno,11pt]{amsart}

\usepackage{amssymb}
\usepackage{amsmath}
\usepackage{amsthm}
\usepackage[usenames]{color}
\usepackage{graphicx}
\usepackage{cite}

\usepackage[colorlinks,linkcolor=blue,anchorcolor=red,citecolor=blue]{hyperref}
\usepackage{marginnote}

\newcommand{\R}{\mathbb{R}}

\newtheorem{theorem}{Theorem}[section]

\newtheorem{lemma}[theorem]{Lemma}

\newtheorem{remark}[theorem]{Remark}
\newtheorem{definition}[theorem]{Definition}

\numberwithin{equation}{section}

\newcommand{\dis}{\displaystyle}

\allowdisplaybreaks
\usepackage{color}

\usepackage[justification=centering]{caption}
\usepackage{tikz}
\usetikzlibrary{positioning}

\makeatletter
\@namedef{subjclassname@2020}{%
  \textup{2020} Mathematics Subject Classification}
\makeatother

\begin{document}

\title[Global quasineutral Euler limit for VPL system]{Global quasineutral Euler limit for the Vlasov-Poisson-Landau system with rarefaction waves}

\author[R.-J. Duan]{Renjun Duan}
\address[R.-J. Duan]{Department of Mathematics, The Chinese University of Hong Kong, Shatin, Hong Kong,
        People's Republic of China}
\email{rjduan@math.cuhk.edu.hk}

\author[D.-C. Yang]{Dongcheng Yang}
\address[D.-C. Yang]{Department of Mathematics, The Chinese University of Hong Kong, Shatin, Hong Kong,
	People's Republic of China}
\email{dcyang@math.cuhk.edu.hk}

\author[H.-J. Yu]{Hongjun Yu}
\address[H.-J. Yu]{School of Mathematical Sciences, South China Normal University, Guangzhou 510631, People's Republic of China}
\email{yuhj2002@sina.com}


\begin{abstract}
In the paper, we consider the Cauchy problem on the spatially one-dimensional Vlasov-Poisson-Landau system modelling the motion of ions under a generalized Boltzmann relation. Let the Knudsen number and the Debye length be given as $\varepsilon>0$ and $\varepsilon^{b}$ with $\frac{3}{5}\leq b\leq 1$, respectively. As $\varepsilon\to 0$ the formal Hilbert expansion gives the fluid limit to the quasineutral compressible Euler system.  We start from the small-amplitude rarefaction wave of the Euler system that admits a smooth approximation with a parameter $\delta\sim\varepsilon^{\frac{3}{5}-\frac{2}{5}a}$, where  the wave strength is independent of $\varepsilon$ and we take  $\frac{2}{3}\leq a\leq 1$ if $\frac{2}{3}\leq b\leq 1$ and $4-5b\leq a\leq 1$ if $\frac{3}{5}\leq b< \frac{2}{3}$. Under the scaling $(t,x)\to (\varepsilon^{-a}t,\varepsilon^{-a}x)$, for well-prepared initial data we construct the unique global classical solution to the Vlasov-Poisson-Landau system around the rarefaction wave in the vanishing  limit $\varepsilon\to 0$ and also obtain the global-in-time convergence of solutions toward the rarefaction wave with rate $\varepsilon^{\frac{3}{5}-\frac{2}{5}a}|\ln\varepsilon|$ in the $L^{\infty}_xL^2_v$ norm. The best rate is $\varepsilon^{\frac{1}{3}}|\ln\varepsilon|$ with the choice of $a=\frac{2}{3}$ and $\frac{2}{3}\leq b\leq 1$. Note that the nontrivial electric potential in the solution connects two fixed distinct states at the far fields $x=\pm\infty$ for all $t\geq 0$ and tends asymptotically as $\varepsilon\to 0$ toward a profile determined by the macro density function under the quasineutral assumption. Our strategy is based on an intricate weighted energy method capturing the quartic dissipation to give uniform bounds of  the nonlinear dynamics around rarefaction waves.
\end{abstract}



\maketitle
\thispagestyle{empty}

\setcounter{tocdepth}{1}
\tableofcontents

\section{Introduction}


The Vlasov-Poisson-Landau system is one of fundamental models governing the motion of charged particles in plasma physics, cf.~\cite{Chen-F, Krall-Trivelpiece}. Much attention has been paid to the mathematical study of the model in recent years and we refer to \cite{Bedrossian, Bedrossian-1,Bobylev-1,CDL,Chaturvedi, ChenYM, Deng,Dong-Guo, Guo-JAMS, LYZ, Lei-1, Strain-Zhu, Wang, YuHJ04} and references therein.
When the magnetic effect is absent, the dynamics of ions in collisional plasma
with slab symmetry can be described at the kinetic level by the spatially one-dimensional Vlasov-Poisson-Landau (VPL in short) system
\begin{equation}
\label{1.1}
  \left\{
\begin{array}{rl}
\dis \partial_{t}F+v_{1}\partial_{x}F-\partial_{x}\dis \phi\partial_{v_{1}}F&\dis =\frac{1}{\varepsilon}Q(F,F),
\\
\dis -\lambda^{2}\partial^{2}_{x}\phi&\dis =\rho-\rho_{\mathrm{e}}(\phi), \quad \rho=\int_{\mathbb{R}^{3}}F\,dv.
\end{array} \right.
\end{equation}
Here, the unknown $F=F(t,x,v)\geq0$ is the density
distribution function of ions with velocity $v=(v_{1},v_{2},v_{3})\in\mathbb{R}^{3}$ at position
$x\in\mathbb{R}$ and time $t\geq0$.  
The slab symmetry with respect to the first coordinate $x$ in the spatial domain $\mathbb{R}^{3}$
has been assumed. 
The collision term is given by the Landau operator:
\begin{equation*}
Q(F_{1},F_{2})(v)=\nabla_{v}\cdot\int_{\mathbb{R}^{3}}\Phi(v-v_{*})\left\{F_{1}(v_{*})\nabla_{v}F_{2}(v)-F_{2}(v)\nabla_{v_{*}}F_{1}(v_{*})\right\}\,dv_{*},
\end{equation*}
where for the Landau collision kernel $\Phi(z)$ with $z=v-v_\ast$, we consider only the case of physically most realistic Coulomb interactions throughout the paper, namely,
\begin{equation}
\label{1.2b}
\Phi_{ij}(z)=\frac{1}{|z|}\left(\delta_{ij}-\frac{z_iz_j}{|z|^{2}}\right),\quad 1\leq i,j\leq 3,
\end{equation}
with $\delta_{ij}$ being the Kronecker delta. The self-consistent electric potential $\phi=\phi(t,x)$ is induced by the total charges through the Poisson equation.
The density of electrons $\rho_{\mathrm{e}}(\cdot)$ is a given function and depends only on the potential in
the form of an analogue of the classical Boltzmann relation 
\begin{equation}
\label{1.3b}
\rho_\mathrm{e}(\phi)=e^{\frac{\phi}{A_{\mathrm{e}}}}
\end{equation}
with a constant $A_\mathrm{e}>0$ related to the temperature of electrons; more general assumptions on $\rho_{\mathrm{e}}(\cdot)$ will be addressed in Section \ref{sec2.1} later. We refer to \cite{Bastdos-Golse2018, Charles, Duan-Liu-2015,Guo-Pausader, GGPS} for the derivation of the model for ions dynamics as well as relevant studies under the Boltzmann relation \eqref{1.3b}.

Recall that the problem on the incompressible viscous electro-magneto hydrodynamics of collisional kinetic models described by the Vlasov-Maxwell-Boltzmann system has been extensively studied in \cite{ASR} by Ars\'enio and Saint-Raymond. Motivated by some recent work \cite{Duan-Yang-Yu-1, DYY-VMB, Duan-Yang-Yu-2, Duan-Yang-Yu, Duan-Yu1, DY31}, we rather focus on the compressible fluid limit of such models related to the Landau collision operator.  
In the VPL system \eqref{1.1} above, two parameters $\varepsilon>0$ and  $\lambda>0$ denote the Knudsen number and the Debye length, respectively. The asymptotic problems as  $\varepsilon\to 0$ and or $\lambda\to 0$ can be described in Figure \ref{fig.asp} below:
\begin{figure}[h]
        \scalebox{1}{\begin{tikzpicture}
            \begin{scope}[align=flush center, anchor=mid]
                \node (VPL) at (0, 0) [draw, thick, fill=white, fill opacity=0.7,
                minimum width=2cm, minimum height=0.8cm] {VPL system};
                
                \node (phantom-top) [right=of VPL.east] {};
                
                \node (Landau) [right=of phantom-top, draw, fill=white, fill opacity=0.7,
                minimum width=2cm, minimum height=0.8cm] {Landau~with~a~nonlocal~force};
                
                \node (phantom-right) [below=of VPL] {};
                
                \node (EulerPoisson) [below=of phantom-right, draw, fill=white, fill opacity=0.7,
                minimum width=2cm, minimum height=0.8cm] {Euler-Poisson};
                
                \node (phantom-bottom) [right=of EulerPoisson] {};
                
                \node (QuasineutralEuler) [below=of Landau.south, right=of phantom-bottom,
                draw, fill=white, fill opacity=0.7, minimum width=2cm, minimum height=0.8cm]
                {Quasineutral Euler system};
            \end{scope}
            
            \begin{scope}[->, shorten <=2pt, shorten >=2pt, align=flush center]
                \draw (VPL) -- (Landau) node[midway, above]{$\lambda \to 0$};
                \draw (VPL) -- (EulerPoisson) node[midway, left] {$\varepsilon \to 0$};
                \draw (Landau.south -| QuasineutralEuler.north) -- (QuasineutralEuler)
                node [midway, right]{$\varepsilon \to 0$};
                \draw (EulerPoisson) -- (QuasineutralEuler) node[pos=0.45, above] {$\lambda \to 0$};
                \draw (VPL.south east) -- (QuasineutralEuler.north west)
                node [pos=0.4, right]{$\lambda=\varepsilon^b \to 0$, \\ $\varepsilon \to 0$};
            \end{scope}
        \end{tikzpicture}}
        \caption{Asymptotic limits of the VPL system}
        \label{fig.asp}	
    \end{figure}
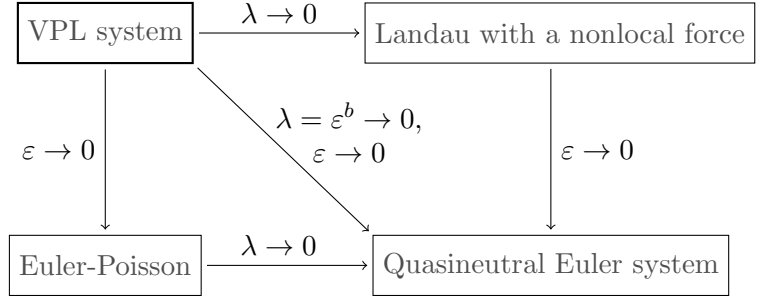

\noindent Here, $\varepsilon\to 0$ is the fluid dynamic limit while $\lambda\to 0$ is the quasineutral limit. In particular, at the formal level,   
the VPL system \eqref{1.1} tends to the following compressible Euler-Poisson system as $\varepsilon\to 0$:
 \begin{equation}
	\label{1.4A}
	\left\{
	\begin{aligned}
		\partial_{t}\rho +\partial_{x}(\rho u_1)&=0,
		\\
		\partial_{t}(\rho  u_{1})+\partial_{x}(\rho  u_{1}^{2})+\partial_{x}P+\rho \partial_{x}\phi&=0,\\
		-\lambda^{2}\partial^{2}_{x}\phi&=\rho -\rho_{\mathrm{e}}(\phi),
	\end{aligned} \right.
\end{equation}
in the isentropic case when
\begin{equation}\notag
S:=-\frac{2}{3}\ln\rho+\ln(\frac{4\pi}{3}\theta)+1\text{ is a constant}
\end{equation}
and hence
$P=\frac{2}{3}\rho\theta=\frac{e^{S-1}}{2\pi}\rho^{5/3}$,
where we have taken the gas constant as $R=2/3$. We refer to \cite{DYY-VMB, Guo-Jang, Lei-1} and references therein for relevant work on the topic. Here, in connection with the current work we point out that Guo-Jang \cite{Guo-Jang} gave the first rigorous proof of fluid limit from the Vlasov-Poisson-Boltzmann system to the compressible Euler-Poisson system in the full space of three dimensions via the Hilbert expansion. On the other hand, the VPL system \eqref{1.1} tends to the following model of the Landau equation with a density-dependent force as $\lambda\to 0$:
\begin{equation}
\label{le.ddf}
\partial_{t}F+v_{1}\partial_{x}F-\partial_{x}\rho_\mathrm{e}^{-1}(\int_{\mathbb{R}^{3}} F\,dv)\partial_{v_{1}}F=\frac{1}{\varepsilon}Q(F,F).
\end{equation} 
When there is no collision, it is related to the quasineutral limit from the Vlasov-Poisson system to the so-called kinetic Euler equations and the mathematical studies of such topic can be traced back to \cite{BrGr, Gr}; see recent progress \cite{GPI20, Han-Kwan, Han-Kwan-2015, HKI-JDE, HKI-CMS} as well as a nice survey \cite{GPI} by Griffin-Pickering and Iacobelli. Though, it seems there are few results on the quasineutral limit in the presence of collisions. 

In the paper, we are interested in the simultaneous limits 
\begin{equation}
\label{ll}
\varepsilon\to 0~~\text{and}~~\lambda\to 0\quad \text{ with } \frac{1}{\lambda^{2}} \gg \frac{1}{\varepsilon}, 
\end{equation}
which corresponds to the regime where the strength of Vlasov force is much stronger than the strength of collision; see a recent work \cite{Bobylev-1} by Bobylev and Potapenko for physical discussions. For the technical reason,
 we consider only the special case when $\lambda$ is $\varepsilon$-dependent in the sense of
\begin{equation}
\label{1.4B}
\lambda=\varepsilon^{b},\quad \mbox{for}\quad \frac{3}{5}\leq b\leq 1.
\end{equation}
Some further comments will be given in Remark \ref{rem2.5} later on. The formal asymptotic limit under \eqref{ll} gives the quasineutral Euler system of the conservative form again in the isentropic case:
 \begin{equation}
	\label{qnEuler}
	\left\{
	\begin{aligned}
		\partial_{t}\rho+\partial_{x}(\rho u_1)&=0,
		\\
		\partial_{t}(\rho u_{1})+\partial_{x}(\rho  u_{1}^{2})+\partial_{x}[P(\rho)+P^\phi(\rho)]&=0.
	\end{aligned} \right.
\end{equation}
Here $P^\phi(\cdot)$ is the density-dependent pressure induced under the quasineutral assumption such that
\begin{equation}\notag
\frac{dP^\phi(\rho)}{d\rho}=\rho \frac{d\rho_\mathrm{e}^{-1}(\rho)}{d\rho},
\end{equation}
where $\rho_\mathrm{e}^{-1}$ is the inverse function of $\rho_\mathrm{e}$, for instance, one has $\rho_\mathrm{e}^{-1}(\rho)=A_\mathrm{e}\ln \rho$ and $P^\phi=A_\mathrm{e}\rho$ in case of the classical Boltzmann relation \eqref{1.3b}. Observe that \eqref{qnEuler} also can be formally obtained by taking the quasineutral limit $\lambda\to 0$ from the Euler-Poisson system \eqref{1.4A}  
or the fluid limit $\varepsilon\to 0$ from the Landau model with forces \eqref{le.ddf}; we refer to \cite{Bobylev-1, Cordier, Han-Kwan-2013} for relevant topics. 

The main goal of this paper is to rigorously establish the convergence of the VPL system \eqref{1.1} to the 
quasineutral Euler system \eqref{qnEuler}  for solutions with rarefaction waves whose strength is small but independent of $\varepsilon$ under the assumption \eqref{1.4B}. We first give an imprecise statement of the main result in order not to go into too much details on energy norms for the sense of convergence.  We postpone the rigorous description to the next section when the explicit construction of rarefaction wave and its smooth approximation is made and more notations on norms and function spaces with high-order derivatives are introduced. 

To the end, we always denote an equilibrium associated with fluid quantities $(\rho,u,\theta)$ as 
\begin{equation*}
M_{[\rho,u,\theta]}=\frac{\rho}{(2\pi R\theta)^{3/2}}\exp(-\frac{|v-u|^2}{2R\theta}).
\end{equation*}  
Our main result is roughly stated as follows.      

\begin{theorem}\label{thm.rs}
Assume \eqref{1.2b} and \eqref{1.4B}. Consider the one-dimensional Cauchy problem on the VPL system \eqref{1.1} with the Riemann data
\begin{equation}
\label{thm.rd}
F(0,x,v)=F_0^R(x,v):=
\begin{cases}
M_{[\rho_+,u_+,\theta_+]},\quad x>0,
\\
M_{[\rho_-,u_-,\theta_-]},\quad x<0,
\end{cases}
\end{equation}
where the constant states $(\rho_\pm,u_\pm,\theta_\pm)$ with $u_{\pm}=(u_{1\pm},0,0)$ and ${\rho_+^{2/3}}/{\theta_+}={\rho^{2/3}_{-}}/{\theta_{-}}$ satisfy the Rankine-Hugoniot condition for the quasineutral Euler system \eqref{qnEuler} in the isentropic case such that $(\rho_+,u_{1+},\theta_+)$ is connected to $(\rho_-,u_{1-},\theta_-)$ by the Riemann rarefaction wave solution $(\rho^{R},u^R,\theta^R)(x/t)$ with $u^{R}=(u^R_1,0,0)$ and $\theta^R=\frac{e^{S-1}}{2\pi} (\rho^R)^{2/3}$. Then, there is $\varepsilon_0>0$ such that if $0<\varepsilon< \varepsilon_0$, $0<|\rho_+-\rho_-|+|u_{1+}-u_{1-}|+|\theta_+-\theta_-|< \varepsilon_0$ and $|\rho_\pm-1|+|u_{1\pm}|+|\theta_\pm-\frac{3}{2}|< \varepsilon_0$, then there exists a class of smooth approximate $(\varepsilon,a)$-dependent initial data $F^{\varepsilon,a}_0(x,v)$ with 
$F^{\varepsilon,a}_0(x,v)\to F^R_0(x,v)$ as $\varepsilon\to 0$ in a sense,
such that the Cauchy problem on \eqref{1.1} with $F(0,x,v)=F^{\varepsilon,a}_0(x,v)$ admits a unique global classical solution $[F^{\varepsilon,a,b}(t,x,v),\phi^{\varepsilon,a,b}(t,x)]$ with $(t,x,v)\in [0,\infty)\times \R\times \R^3$ satisfying that for any $t>0$,
$F^{\varepsilon,a,b}(t,x,v)\to M_{[\rho^{R},u^R,\theta^R](x/t)}$ as $\varepsilon\to 0$ in the sense that 
 for any $\ell>0$, 
 \begin{multline}
	\label{thm.con}
	\lim\limits_{\varepsilon\to 0^+}\sup\limits_{t\geq \ell}
	(\|\frac{F^{\varepsilon,a,b}(t,x,v)-M_{[\rho^{R},u^{R},\theta^{R}](x/t)}(v)}{\sqrt{\mu}}\|_{L_{x}^{\infty}L_{v}^{2}}\\
	+\|\phi^{\varepsilon,a,b}(t,x)-\phi^{R}(\frac{x}{t})\|_{L^{\infty}_{x}})
	=0,
\end{multline}	
with the estimate on rate of convergence:
\begin{multline}
	\label{thm.rate}
	\sup_{t\geq \ell}\|\frac{F^{\varepsilon,a,b}(t,x,v)-M_{[\rho^{R},u^{R},\theta^{R}](x/t)}(v)}{\sqrt{\mu}}\|_{L_{x}^{\infty}L_{v}^{2}}\\
	+\sup_{t\geq \ell}\|\phi^{\varepsilon,a,b}(t,x)-\phi^{R}(\frac{x}{t})\|_{L^{\infty}_{x}}
	\leq C_{\ell}\varepsilon^{\frac{3}{5}-\frac{2}{5}a}|\ln\varepsilon|,
\end{multline}
where $\mu=M_{[1,0,3/2]}(v)$ is a global Maxwellian, $\phi^R=\rho_\mathrm{e}^{-1}(\rho^R)$ is the asymptotic profile of the potential defined through the quasineutral assumption,  $C_\ell$ is a constant independent of $\varepsilon$, and 
\begin{multline}
\label{def.ab}
(a,b)\in \mathcal{S}:=\{(a,b)\in [\frac{2}{3},1]\times [\frac{3}{5},1]:  \text{$\frac{2}{3}\leq a\leq 1$ if $\frac{2}{3}\leq b\leq 1$} \\
\text{and $4-5b\leq a\leq 1$ if $\frac{3}{5}\leq b< \frac{2}{3}$}\}.
\end{multline}
Note that the best rate   in \eqref{thm.rate} is $\varepsilon^{\frac{1}{3}}|\ln\varepsilon|$ with the choice of $a=\frac{2}{3}$ and $\frac{2}{3}\leq b\leq 1$. 

\end{theorem}

Theorem \ref{thm.rs} above turns out to provide an explicit way for constructing smooth approximation solutions with rarefaction wave to the VPL system  \eqref{1.1} with Riemann data \eqref{thm.rd} in terms of the fluid limit  to smooth approximation solutions with rarefaction wave to the quasineutral compressible Euler equations \eqref{qnEuler} with the corresponding Riemann data. We refer to Remark \ref{rmk2.7} later on for an explicit construction of smooth approximation initial data. Note that it is unclear how to directly solve in a lower-regularity function space the Riemann problem with rarefaction wave \eqref{1.1} and  \eqref{thm.rd} for the VPL system and it is also the case even for the pure Landau or Boltzmann equation, which seems largely open so far.

The proof of Theorem \ref{thm.rs} is based on a scaling argument 
\begin{equation}
\label{def.sc}
(t,x)\to (\frac{t}{\varepsilon^a},\frac{x}{\varepsilon^a})
\end{equation} 
with a suitable choice of the parameter $a$. Note that the rarefaction wave solution is invariant under the scaling \eqref{def.sc} so that such scaling essentially helps change the order of the Knudsen number from $\varepsilon$ to $\varepsilon^{1-a}$; see \eqref{prefeq} later on where the Debye length is also changed correspondingly from $\varepsilon^b$ to $\varepsilon^{b-a}$. Related to the above scaling, it is also a key to make a suitable construction of smooth approximation solutions with rarefaction wave for the quasineutral compressible Euler equations \eqref{qnEuler}. In particular, this is related to the choice of another parameter $\delta$ depending on $\varepsilon$ such that $\delta\to 0$ as $\varepsilon\to 0$; see \eqref{1.25} later on. This kind of argument was first developed by Xin \cite{Xin} to justify the vanishing viscosity limit of the one-dimensional Navier-Stokes equations with rarefaction waves, where $a=\frac{3}{4}$ was chosen and the convergence rate obtained is $\varepsilon^{\frac{1}{4}}|\ln \varepsilon|$. The first result on the compressible Euler limit with rarefaction waves from the one-dimensional Boltzmann equation for the hard sphere model was given by Xin-Zeng \cite{XinZeng}, where $a=1$ was taken and the rate is $\varepsilon^{\frac{1}{5}}|\ln \varepsilon|$. The result was later improved by Li \cite{Li} with the choice of $a=\frac{2}{3}$ and the rate given by $\varepsilon^{\frac{1}{3}}|\ln \varepsilon|^2$. 

In a recent work \cite{Duan-Yang-Yu}, the same authors of this paper considered the similar problem on the construction of smooth solutions with rarefaction waves for the one-dimensional Landau equation for Coulomb potentials. We made use of the more general range on the scaling parameter as $\frac{2}{3}\leq a\leq 1$ and obtained the convergence rate as $\varepsilon^{\frac{3}{5}-\frac{2}{5}a}|\ln \varepsilon|$. In this paper, our result stated in Theorem \ref{thm.rs}, in particular, the obtained rate in \eqref{thm.rate}, is analogous to that of \cite{Duan-Yang-Yu}. Extra efforts are made for overcoming the effect of the nontrivial potential force which connects two distinct end states.  We explain the choice of the range for $a$ in Section \ref{sec.3.4} due to the techniques of proof and also point out that $a=\frac{2}{3}$ is sharp in terms of Remark \ref{rmk.ea}.  

Moreover, the quasineutral limit $\lambda\to 0$ under the assumption \eqref{ll} is simultaneously involved. As indicated by \cite{Bobylev-1}, the vanishing limit \eqref{ll} corresponds to a regime where  the strength of Vlasov force is much stronger than the strength of collision and hence the collisions are weak relative to the self-consistent force. Furthermore, the restriction \eqref{1.4B} with $\frac{3}{5}\leq b\leq 1$ arises from the techniques of proof. We will explain all the difficulties of the proof in Section \ref{sec.2.3}. It would be also interesting to study the asymptotic problem on \eqref{1.1} in the limits $\varepsilon\to 0$ and $\lambda\to 0$ when collisions dominate self-consistent forces, that is $\frac{1}{\lambda^{2}} \ll \frac{1}{\varepsilon}$. 

Another motivation of studying the convergence rate in \eqref{thm.rate} is inspired by our recent work \cite{DYY-VMB} where the compressible fluid limit to the Euler-Maxwell system from the Vlasov-Maxwell-Boltzmann system for solutions to the Cauchy problems around constant equilibrium was justified. The same approach in \cite{DYY-VMB}  together with techniques in \cite{Duan-Yang-Yu-2,Duan-Yang-Yu} should be able to provide a similar result for the compressible limit from the VPL system under consideration to the Euler-Poisson system for ions dynamics. However, the convergence rate in the small Knudsen number in  \cite{DYY-VMB} holds up to an almost global time interval $[0,\varepsilon^{-\alpha}]$ with some $\alpha>0$. Therefore, by showing \eqref{thm.rate} we provide an example in the current work for the fluid limit from the VPL system such that the convergence can be established uniformly global in time as long as the initial layer is avoided.              

The rest of this paper is organized as follows.  
In Section \ref{sec.2}, we first present the macro-micro
decomposition of the solution for the VPL system and we construct the rarefaction wave for the corresponding compressible quasineutral Euler system.
Then we give the precise statement of the main resulst in Theorem \ref{thm2.1} and we 
also show the main ideas of the proofs of Theorem \ref{thm2.1}.
In Section \ref{sec.3}, we reformulate the VPL system and make the a priori assumption
such that we can perform energy analysis conveniently. Some basic estimates are obtained
in Section \ref{sec.4} in order to proceed the proof of the main theorem. Sections \ref{sec.5}, \ref{sec.6} and \ref{sec.7}
are the main parts of the proof for establishing the a priori estimates
including zeroth order energy estimates, high order energy estimates and weighted energy
estimates, respectively.  In Section \ref{sec.8} we are devoted to proving Theorem \ref{thm2.1} so as to conclude Theorem \ref{thm.rs}. In the appendix Section \ref{sec.9},
we give the details of deriving an identity \eqref{4.7} for completeness.

\section{Main results}\label{sec.2}
In this section, we give the precise statement of main results on the limit from the VPL system \eqref{1.1} under the assumptions \eqref{ll} and \eqref{1.4B} to the  compressible quasineutral Euler equations \eqref{qnEuler} for solutions around rarefaction waves. For the purpose, we supplement \eqref{1.1} with initial data 
\begin{equation}
\label{1.1id}
F(0,x,v)=F_{0}(x,v)
\end{equation}
that connects two distinct global Maxwellians
at the far fields $x=\pm\infty$ such that
\begin{equation}
	\label{1.4a}
	F_{0}(x,v)\to M_{[\rho_{\pm},u_{\pm},\theta_{\pm}]}(v):=\frac{\rho_\pm}{{(2\pi R\theta_\pm)^{\frac{3}{2}}}}\exp\big(-\frac{|v-u_\pm|^2}{2R\theta_\pm}\big), \ \ \mbox{as}\ \ x\to\pm\infty,
\end{equation}
where $\rho_\pm>0$, $\theta_\pm>0$ and $u_{1\pm}$ with $u_{\pm}=(u_{1\pm},0,0)$ are constants. 
For the boundary data of $\phi$ at $x=\pm\infty$, we assume that 
\begin{equation}
	\label{1.5a}
	\lim_{x\to\pm\infty}\phi(t,x)=\phi_{\pm}, \quad\rho_{\pm}=\rho_{\mathrm{e}}(\phi_{\pm}),
\end{equation}
where $\rho_{\pm}=\rho_{\mathrm{e}}(\phi_{\pm})$ corresponds to the quasineutral assumption at the far fields.
It should be pointed out that the ending states $\rho_{\pm}$ and thus $\phi_\pm$ can be distinct. Note that later on we will construct a class of smooth initial data $F_0(x,v)$ that can approximate the Riemann data 
\eqref{thm.rd}.  

\subsection{Assumption on $\rho_\mathrm{e}(\cdot)$}\label{sec2.1}
Specifically, throughout this paper we make the following assumptions, cf. \cite{Duan-Liu-2015}

\medskip
\noindent($\mathcal{A}$):~~~~~$\rho_{\mathrm{e}}(\phi):(\phi_{m},\phi_{M})\to (\rho_{m},\rho_{M})$ is a positive smooth function with
$$
\rho_{m}:=\inf_{\phi_{m}<\phi<\phi_{M}}\rho_{\mathrm{e}}(\phi),\quad\rho_{M}:=\sup_{\phi_{m}<\phi<\phi_{M}}\rho_{\mathrm{e}}(\phi),
$$
satisfying the following:
\begin{itemize}
\item{($\mathcal{A}_{1}$):~~$\rho_{\mathrm{e}}(0)=1$ with $0\in(\phi_{m},\phi_{M})$;}
\item{($\mathcal{A}_{2}$):~~$\rho_{\mathrm{e}}(\phi)>0$, ~~$\rho'_{\mathrm{e}}(\phi)>0$ for each $\phi\in(\phi_{m},\phi_{M})$;}
\item{($\mathcal{A}_{3}$):~~$\rho_{\mathrm{e}}(\phi)\rho''_{\mathrm{e}}(\phi)\leq[\rho'_{\mathrm{e}}(\phi)]^{2}$ for each $\phi\in(\phi_{m},\phi_{M})$.}
\end{itemize}
\medskip

\noindent The assumption ($\mathcal{A}_{1}$)
just means that the electron density has been normalized to be unit when the
potential is zero, since the electric potential 
can be up to an arbitrary constant.
The other two assumptions ($\mathcal{A}_{2}$) and ($\mathcal{A}_{3}$) guarantee that \eqref{1.21} and \eqref{4.34b}
hold true in the later energy analysis.
A typical example satisfying ($\mathcal{A}$) can be given as
\begin{equation}
\label{1.3}
\rho_{\mathrm{e}}(\phi)=\left(1+\frac{\gamma_{\mathrm{e}}-1}{\gamma_{\mathrm{e}}}\frac{\phi}{A_{\mathrm{e}}}\right)^{\frac{1}{\gamma_{\mathrm{e}}-1}},\quad
\phi_{m}=-\frac{\gamma_{\mathrm{e}}}{\gamma_{\mathrm{e}}-1}A_{\mathrm{e}}, \quad \phi_{M}=+\infty,
\end{equation}
with $\gamma_{\mathrm{e}}\geq1$ and $A_{\mathrm{e}}>0$ being constants. Note that $\rho_{\mathrm{e}}(\phi)\to \exp({\phi}/{A_{\mathrm{e}}})$
and $\phi_{m}\to -\infty$ as $\gamma_{\mathrm{e}}\to 1$ , which corresponds to the classical Boltzmann relation. 
The classical Boltzmann relation has been extensively used in the mathematical
study of both the fluid dynamic equation, cf. \cite{Guo-Pausader,Suzuki},
and the kinetic Vlasov-type equations, cf. \cite{Charles,Han-Kwan,Han-Kwan-2015}.
In fact, \eqref{1.3} can
be deduced from the momentum equation of the isentropic Euler-Poisson system for
the fluid of electrons with the adiabatic exponent $\gamma_{e}$ under the zero-limit of electron
mass, namely, $\partial_{x}(A_{\mathrm{e}}\rho^{\gamma}_{\mathrm{e}})=\rho_{\mathrm{e}}\partial_{x}\phi$, cf. \cite{Bastdos-Golse2018,Sentis}.

\subsection{Macro-micro decomposition}
Instead of using either Hilbert expansion or Chapman-Enskog expansion, we present the macro-micro decomposition of the solution for the VPL system
with respect to the local Maxwellian, that was initiated by Liu-Yu \cite{Liu-Yu} and developed by Liu-Yang-Yu \cite{Liu-Yang-Yu}
for the Boltzmann equation.

For a given solution $F(t,x,v)$ of the VPL system \eqref{1.1}, we define five macroscopic (fluid) quantities:
the mass density, momentum and
energy density, respectively, given as
\begin{equation}
\label{1.5}
	\left\{
\begin{aligned}
\rho(t,x)&=\int_{\mathbb{R}^{3}}\psi_{0}(v)F(t,x,v)\,dv,
\\
\rho u_{i}(t,x)&=\int_{\mathbb{R}^{3}}\psi_{i}(v)F(t,x,v)\,dv, \quad \mbox{for $i=1,2,3$,}
\\
\rho(\mathcal{E}+\frac{1}{2}|u|^{2})(t,x)&=\int_{\mathbb{R}^{3}}\psi_{4}(v)F(t,x,v)\,dv,
\end{aligned} \right.
\end{equation}
and the corresponding local Maxwellian:
\begin{equation}
\label{1.7}
M=M_{[\rho,u,\theta](t,x)}(v):=\frac{\rho(t,x)}{(2\pi R\theta(t,x))^{3/2}}\exp\big(-\frac{|v-u(t,x)|^{2}}{2R\theta(t,x)}\big).
\end{equation}
Here $u(t,x)=(u_{1},u_{2},u_{3})(t,x)$ is the bulk velocity, $\mathcal{E}(t,x)>0$
is the internal energy depending on the temperature $\theta(t,x)$ by $\mathcal{E}=\frac{3}{2}R\theta=\theta$ with 
$R=\frac{2}{3}$ being taken for convenience, and $\psi_{i}(v)$ $(i=0,1,2,3,4)$ are
five collision invariants given by
$$
\psi_{0}(v)=1, \quad \psi_{i}(v)=v_{i}~(i=1,2,3),\quad \psi_{4}(v)=\frac{1}{2}|v|^{2},
$$
satisfying
\begin{equation}\notag
\int_{\mathbb{R}^{3}}\psi_{i}(v)Q(F,F)\,dv=0,\quad \mbox{for $i=0,1,2,3,4$.}
\end{equation}
Let $\langle h,g\rangle =\int_{\mathbb{R}^{3}}h(v)g(v)\,d v$ denote the $L^{2}_{v}(\mathbb{R}^{3})$ inner product.
The macroscopic kernel space is spanned by the following orthonormal basis:
\begin{equation}
\label{1.11b}
	\left\{
\begin{aligned}
\chi_{0}(v)&=\frac{1}{\sqrt{\rho}}M,\\
\chi_{i}(v)&=\frac{v_{i}-u_{i}}{\sqrt{R\rho\theta}}M, \quad \mbox{for $i=1,2,3$,}
\\
\chi_{4}(v)&=\frac{1}{\sqrt{6\rho}}(\frac{|v-u|^{2}}{R\theta}-3)M,
\\
\langle \chi_{i},\frac{\chi_{j}}{M}\rangle&=\delta_{ij},
\quad \mbox{for ~~$i,j=0,1,2,3,4$}.
\end{aligned} \right.
\end{equation}
In view of \eqref{1.11b} above, the macroscopic projection $P_{0}$ and microscopic projection  $P_{1}$
can be respectively defined as
\begin{equation}
\label{1.10}
P_{0}h=\sum_{i=0}^{4}\langle h,\frac{\chi_{i}}{M}\rangle\chi_{i},\quad P_{1}h=h-P_{0}h.
\end{equation}
Note that a function $h(v)$ is called microscopic or non-fluid if
\begin{equation}\notag
\langle h(v),\psi_{i}(v)\rangle=0, \quad \mbox{for $i=0,1,2,3,4$}.
\end{equation}

Using the notations above,
the solution $F$ of the VPL system \eqref{1.1} can be decomposed into the
macroscopic (fluid) part, i.e., the local Maxwellian $M$ defined in \eqref{1.7}, and the microscopic (non-fluid) part, i.e. $G$:
\begin{equation}
\label{1.12}
F=M+G, \quad P_{0}F=M, \quad P_{1}F=G.
\end{equation}
 By the fact that $Q(M,M)=0$, applying \eqref{1.12}, we can rewrite the first equation of \eqref{1.1} as
\begin{equation}
\label{1.13}
(M+G)_{t}+v_{1}(M+G)_{x}-\phi_{x}\partial_{v_{1}}(M+G)=\frac{1}{\varepsilon}L_{M}G+\frac{1}{\varepsilon}Q(G,G),
\end{equation}
where the linearized Landau operator $L_{M}$ around the local Maxwellian $M$ is defined as
\begin{equation}
\label{1.8}
L_{M}h:=Q(h,M)+Q(M,h).
\end{equation}
Note that the null space $\mathcal{N}$ of $L_{M}$ is spanned by $\chi_{i}~(i=0,1,2,3,4)$.

Multiplying \eqref{1.13} by the collision invariants $\psi_{i}(v)$ $(i=0,1,2,3,4)$
and integrating the resulting equations with respect to $v$ over $\mathbb{R}^{3}$, we get
\begin{equation}
\label{1.14}
	\left\{
\begin{aligned}
&\rho_{t}+(\rho u_{1})_{x}=0,
\\
&(\rho u_{1})_{t}+(\rho u_{1}^{2})_{x}+P_{x}+\rho\phi_{x}=-\int_{\mathbb{R}^{3}} v^{2}_{1}G_{x}\,dv,
\\
&(\rho u_{i})_{t}+(\rho u_{1}u_{i})_{x}=-\int_{\mathbb{R}^{3}} v_{1}v_{i}G_{x}\,dv, ~~i=2,3,
\\
&[\rho (\theta+\frac{|u|^{2}}{2})]_{t}+[\rho u_{1}(\theta+\frac{|u|^{2}}{2})+Pu_{1}]_{x}+\rho u_{1}\phi_{x}
=-\int_{\mathbb{R}^{3}} \frac{1}{2}v_{1}|v|^{2}G_{x}\,dv,
\end{aligned} \right.
\end{equation}
with the pressure $P=\frac{2}{3}\rho\theta$.

Applying $P_{1}$ to \eqref{1.13}, we obtain
\begin{equation}
\label{1.15}
G_{t}+P_{1}(v_{1}G_{x})+P_{1}(v_{1}M_{x})-\phi_{x}\partial_{v_{1}}G
=\frac{1}{\varepsilon}L_{M}G+\frac{1}{\varepsilon}Q(G,G).
\end{equation}
Since $L_{M}$
is invertible on $\mathcal{N}^\perp$, we can rewrite \eqref{1.15} to present $G$ as
\begin{equation}
\label{1.16}
G=\varepsilon L^{-1}_{M}[P_{1}(v_{1}M_{x})]+L^{-1}_{M}\Theta,
\end{equation}
with
\begin{equation*}
\Theta:=\varepsilon G_{t}+\varepsilon P_{1}(v_{1}G_{x})-\varepsilon\phi_{x}\partial_{v_{1}}G-Q(G,G).
\end{equation*}
Substituting \eqref{1.16} into \eqref{1.14}, together with the Poisson equation, 
we obtain the following fluid-type system in the Navier-Stokes-Poisson form
\begin{equation}
	\label{1.18}
	\left\{
\begin{aligned}
		&\rho_{t}+(\rho u_{1})_{x}=0,
		\\
		&(\rho u_{1})_{t}+(\rho u_{1}^{2})_{x}+P_{x}+\rho\phi_{x}=\frac{4}{3}\varepsilon(\mu(\theta)u_{1x})_{x}-(\int_{\mathbb{R}^{3}} v^{2}_{1}L^{-1}_{M}\Theta \,dv)_{x},
		\\
		&(\rho u_{i})_{t}+(\rho u_{1}u_{i})_{x}=\varepsilon(\mu(\theta)u_{ix})_{x}-(\int_{\mathbb{R}^{3}} v_{1}v_{i}L^{-1}_{M}\Theta \,dv)_{x}, ~~i=2,3,
		\\
		&[\rho (\theta+\frac{|u|^{2}}{2})]_{t}+[\rho u_{1}(\theta+\frac{|u|^{2}}{2})+Pu_{1}]_{x}+\rho u_{1}\phi_{x}\\
		&=\varepsilon(\kappa(\theta)\theta_{x})_{x}+\frac{4}{3}\varepsilon(\mu(\theta)u_{1}u_{1x})_{x}
		\\
		&\quad+\varepsilon(\mu(\theta)u_{2}u_{2x})_{x}+\varepsilon(\mu(\theta)u_{3}u_{3x})_{x}
		-\frac{1}{2}(\int_{\mathbb{R}^{3}}v_{1}|v|^{2}L^{-1}_{M}\Theta \,dv)_{x},
		\\
		&-\varepsilon^{2b}\phi_{xx}=\rho-\rho_{\mathrm{e}}(\phi).
	\end{aligned} \right.
\end{equation}
Here the viscosity coefficient $\mu(\theta)>0$ and the heat conductivity coefficient $\kappa(\theta)>0$, that both are smooth functions depending only on $\theta$,
are represented by 
\begin{eqnarray*}
	\mu(\theta)=&&- R\theta\int_{\mathbb{R}^{3}}\hat{B}_{ij}(\frac{v-u}{\sqrt{R\theta}})
	B_{ij}(\frac{v-u}{\sqrt{R\theta}})dv>0,\quad i\neq j,
	\notag\\
	\kappa(\theta)=&&-R^{2}\theta\int_{\mathbb{R}^{3}}\hat{A}_{j}(\frac{v-u}{\sqrt{R\theta}})
	A_{j}(\frac{v-u}{\sqrt{R\theta}})dv>0,
\end{eqnarray*}
where  $\hat{A}_{j}(\cdot)$ and $\hat{B}_{ij}(\cdot)$ are Burnett functions (cf. \cite{Bastdos-Golse,Guo-CPAM}) defined as
\begin{equation}
	\label{7.1}
	\hat{A}_{j}(v)=\frac{|v|^{2}-5}{2}v_{j}\quad \mbox{and} \quad \hat{B}_{ij}(v)=v_{i}v_{j}-\frac{1}{3}\delta_{ij}|v|^{2} \quad \mbox{for }  i,j=1,2,3,
\end{equation}
and $A_{j}(\cdot)$ and $B_{ij}(\cdot)$ satisfying $P_{0}A_{j}(\cdot)=0$ and $P_{0}B_{ij}(\cdot)=0$ are given by
\begin{equation}
	\label{7.2}
	A_{j}(v)=L^{-1}_{M}[\hat{A}_{j}(v)M]\quad
	\mbox{and} \quad B_{ij}(v)=L^{-1}_{M}[\hat{B}_{ij}(v)M].
\end{equation}

\subsection{Quasineutral Euler system and rarefaction waves}
When the Knudsen number $\varepsilon$, the Debye length $\lambda$ and the microscopic part $\Theta$ all are
set to be zero, the fluid-type system \eqref{1.18} is formally reduced to the 1D compressible quasineutral Euler system
\begin{equation}
\label{1.4}
\left\{
\begin{aligned}
		&\rho_{t}+(\rho u_{1})_{x}=0,
		\\
		&(\rho u_{1})_{t}+(\rho u_{1}^{2})_{x}+P_{x}+\rho\phi_{x}=0,
		\\
		&(\rho u_{i})_{t}+(\rho u_{1}u_{i})_{x}=0, ~~i=2,3,
		\\
		&[\rho (\theta+\frac{|u|^{2}}{2})]_{t}+[\rho u_{1}(\theta+\frac{|u|^{2}}{2})
		+Pu_{1}]_{x}+\rho u_{1}\phi_{x}=0,
		\\
		&\rho=\rho_{\mathrm{e}}(\phi),
	\end{aligned} \right.
\end{equation}
with the pressure $P=\frac{2}{3}\rho\theta$, where the last algebraic equation corresponds to the quasineutral relation. 

Our goal is to prove that the solution $F(t,x,v)$ of \eqref{1.1} tends to the local Maxwellian $M_{[\rho^{R},u^{R},\theta^{R}](x/t)}(v)$ as $\varepsilon\to 0$,
where $(\rho^{R},u^{R},\theta^{R})(x/t)$  with $u^{R}(x/t)=(u_{1}^{R}(x/t),0,0)$
is defined to be the center-rarefaction wave solution to the system \eqref{1.4}
with the following Riemann initial data
\begin{equation}
\label{1.19}
(\rho,u,\theta)(t,x)|_{t=0}=(\rho^{R}_{0},u^{R}_{0},\theta^{R}_{0})(x)
=\begin{cases}
(\rho_{+},u_{+},\theta_{+}),\quad x>0,
\\
(\rho_{-},u_{-},\theta_{-}),\quad x<0.
\end{cases}
\end{equation}
Here $\rho_\pm>0$, $\theta_\pm>0$ and $u_{1\pm}$ with $u_{\pm}=(u_{1\pm},0,0)$ are the same constants as in \eqref{1.4a}.
Note that $\rho^{-1}_{\mathrm{e}}(\cdot)$ exists due to the assumptions ($\mathcal{A}_{1}$) and ($\mathcal{A}_{2}$), and hence the quasineutral equation $\rho=\rho_{\mathrm{e}}(\phi)$ implies $\phi=\rho^{-1}_{\mathrm{e}}(\rho)$.
We expect that as $\varepsilon\to 0$ the electric potential $\phi(t,x)$ of \eqref{1.1} tends to
$$
\phi^{R}(\frac{x}{t})=\rho^{-1}_{\mathrm{e}}(\rho^{R}(\frac{x}{t})).
$$

Let's now construct the rarefaction wave $(\rho^{R},u^{R},\theta^{R})(x/t)$  with $u^{R}(x/t)=(u_{1}^{R}(x/t),0,0)$
for the Riemann problem \eqref{1.4} and \eqref{1.19}. Define a macroscopic entropy $S:=-\frac{2}{3}\ln\rho+\ln(\frac{4\pi}{3}\theta)+1$. Then we rewrite the system \eqref{1.4} in terms of $(\rho,u,S)$ with $u=(u_{1},0,0)$ as
\begin{equation}
\label{1.20}
\left\{
\begin{aligned}
&\rho_{t}+(\rho u_{1})_{x}=0,
\\
&u_{1t}+u_{1}u_{1x}+\frac{1}{\rho}P_{x}+[\rho_{\mathrm{e}}^{-1}(\rho)]_{x}=0,
\\
&S_{t}+u_{1}S_{x}=0,
	\end{aligned} \right.
\end{equation}
with $P=\frac{2}{3}\rho\theta=\frac{\exp(S-1)}{2\pi}\rho^{5/3}$. Note that the above system is formally reduced to \eqref{qnEuler} in the isentropic case.
As in \cite{Duan-Liu-2015}, we define a pressure generated by the potential force, as
$$
P^{\phi}(\rho):=\int^{\rho}\frac{\varrho}{\rho'_{\mathrm{e}}(\rho^{-1}_{\mathrm{e}}(\varrho))}\,d\varrho,
$$
which is a positive, increasing, and convex function of $\rho\in(\rho_{m},\rho_{M})$
under the quasineutral assumption $\rho=\rho_{\mathrm{e}}(\phi)$ and the assumptions ($\mathcal{A}_{2}$) and ($\mathcal{A}_{3}$).
Note that $(P^{\phi}(\rho))_{x}=\rho\phi_{x}$ due to $\phi=\rho^{-1}_{\mathrm{e}}(\rho)$ and $\rho'_{\mathrm{e}}(\phi)\phi_\rho=1$. 
Under the assumptions ($\mathcal{A}_{2}$) and ($\mathcal{A}_{3}$), for any $ \rho\in(\rho_{m},\rho_{M})$, it holds that
\begin{equation}
\label{1.21}
\partial_{\rho}P^{\phi}(\rho)=\frac{\rho_{\mathrm{e}}(\phi)}{\rho'_{\mathrm{e}}(\phi)}>0, \quad
\partial^{2}_{\rho}P^{\phi}(\rho)=\frac{[\rho'_{\mathrm{e}}(\phi)]^{2}-\rho_{\mathrm{e}}(\phi)\rho''_{\mathrm{e}}(\phi)}{[\rho'_{\mathrm{e}}(\phi)]^{3}}\geq0.
\end{equation}
It is straightforward to check that the system \eqref{1.20} has three distinct eigenvalues
\begin{eqnarray*}
\lambda_{i}(\rho,u_{1},S)&=&u_{1}+(-1)^{\frac{i+1}{2}}\sqrt{\partial_{\rho}P(\rho,S)+\partial_{\rho}P^{\phi}(\rho)},~i=1,3,\\
\lambda_{2}(\rho,u_{1},S)&=&u_{1}. 
\end{eqnarray*}
In terms of the two Riemann invariants
of the third eigenvalue $\lambda_{3}(\rho,u_{1},S)$,
we define the 3-rarefaction wave curve for the given
left constant state $(\rho_{-},u_{1-},\theta_{-})$ with $\rho_{-}>0$ and  $\theta_{-}>0$ as below
\begin{multline*}
R_{3}(\rho_{-},u_{1-},\theta_{-})\equiv \big\{(\rho,u_{1},\theta)\in\mathbb{R}_{+}\times\mathbb{R}\times\mathbb{R}_{+}\mid
\frac{\rho^{2/3}}{\theta}=\frac{\rho^{2/3}_{-}}{\theta_{-}},
\notag\\
u_{1}-u_{1-}=\int^{\rho}_{\rho_{-}}\frac{\sqrt{\partial_{\varrho}P(\varrho,S_{*})+\partial_{\varrho}P^{\phi}(\varrho)}}{\varrho}\,d\varrho,
~~\rho>\rho_{-},~~u_{1}>u_{1-}\big\}.
\end{multline*}
Here and to the end, $S_{*}=-\frac{2}{3}\ln\rho_{\pm}+\ln(\frac{4}{3}\pi\theta_{\pm})+1$ is a constant.

Without loss of generality, we consider only the simple 3-rarefaction wave
to the quasineutral Euler system \eqref{1.4} and \eqref{1.19} in this paper, and  the case for 1-rarefaction wave can be treated similarly.
As in \cite{Matsumura-Nishihara,Xin}, we need to consider the Riemann problem on the inviscid Burgers equation
\begin{equation}
\label{1.23}
\begin{cases}
\omega_{t}+\omega\omega_{x}=0,
\\
\omega(0,x)=
\begin{cases}
\omega_{-},\quad x<0,
\\
\omega_{+},\quad x>0.
\end{cases}
\end{cases}
\end{equation}
If two constants $\omega_{-}<\omega_{+}$ are chosen,
then \eqref{1.23} admits a centered rarefaction wave solution
$\omega^{R}(x,t)=\omega^{R}(z)$ with $z=\frac{x}{t}$ connecting $\omega_{-}$ and $\omega_{+}$ in the form of
\begin{equation*}
\omega^{R}(z)=
\left\{
\begin{aligned}
&\omega_{-},\quad z\leq\omega_{-},
\\
&z,\quad\omega_{-}< z\leq\omega_{+},
\\
&\omega_{+},\quad z>\omega_{+}.
	\end{aligned} \right.
\end{equation*}
For $(\rho_{+},u_{1+},\theta_{+})\in R_{3}(\rho_{-},u_{1-},\theta_{-})$,
the 3-rarefaction wave $(\rho^{R},u^{R},\theta^{R})(z)$ with
$u^{R}(z)=(u_{1}^{R},0,0)(z)$ and $\phi^{R}(z)$ to the Riemann problem \eqref{1.4} and \eqref{1.19}
can be defined explicitly by
\begin{equation}
\label{1.24}
\left\{
\begin{aligned}
&\lambda_{3}(\rho^{R}(z),u^{R}_{1}(z),S_{*})=
\begin{cases}
\lambda_{3}(\rho_{-},u_{1-},S_{*}),\quad z\leq \lambda_{3}(\rho_{-},u_{1-},S_{*}),
\\[1mm]
z, \quad  \lambda_{3}(\rho_{-},u_{1-},S_{*})<z\leq \lambda_{3}(\rho_{+},u_{1+},S_{*}),
\\[1mm]
\lambda_{3}(\rho_{+},u_{1+},S_{*}),\quad z>\lambda_{3}(\rho_{+},u_{1+},S_{*}),
\end{cases}
\\
&u^{R}_{1}(z)-u_{1-}=\int^{\rho^{R}(z)}_{\rho_{-}}\sqrt{\frac{5}{3}\frac{1}{2\pi e}\varrho^{-\frac{4}{3}}e^{S_{*}}
+\varrho^{-1}\frac{d}{d\rho}\rho^{-1}_{\mathrm{e}}(\varrho)}\,d\varrho,
\\
&\theta^{R}(z)=\frac{3}{2}\frac{1}{2\pi e}e^{S_{*}}(\rho^{R})^{\frac{2}{3}}(z),\\
&  \phi^{R}(z)=\rho^{-1}_{\mathrm{e}}(\rho^{R}(z)).
	\end{aligned} \right.
\end{equation}

Since the above 3-rarefaction wave is only Lipschitz continuous, we need to construct an approximate smooth rarefaction wave
to the 3-rarefaction wave defined in \eqref{1.24}. The approximate smooth rarefaction wave can be constructed
by the Burgers equation
\begin{equation}
\label{1.25}
\left\{
\begin{array}{rl}
&\overline{\omega}_{t}+\overline{\omega}\,\overline{\omega}_{x}=0,
\\
&\overline{\omega}(0,x)=\overline{\omega}_{\delta}(x)=\overline{\omega}(\frac{x}{\delta})=\frac{\omega_{+}+\omega_{-}}{2}+\frac{\omega_{+}
-\omega_{-}}{2}\tanh(\frac{x}{\delta}),
	\end{array} \right.
\end{equation}
where $\delta>0$  is a small constant depending on the Knudsen number $\varepsilon>0$ and $\tanh z:=\frac{e^{2z}-1}{e^{2z}+1}$ with $z\in \R$ is the usual hyperbolic tangent.
In fact, as given in \eqref{3.21} later on, we will choose
$\delta=\frac{1}{k}\varepsilon^{\frac{3}{5}-\frac{2}{5}a}$ 
for a suitably small constant $k>0$ independent of $\varepsilon$; see Figure \ref{fig.ab2} below.
\begin{figure}[h]
        \scalebox{.8}{\begin{tikzpicture}
            \draw[->, thick](0,0)--(0,3.5) node[above] {$\overline{\omega}_\delta$} coordinate(y axis);
            \draw[->, thick](-5, 0)--(0,0)--(5,0) node[right] {$x$} coordinate(x axis);
            
            \draw[domain=-5:5,smooth, very thick] plot(\x,{(exp(\x) - exp(-\x)) / (exp(\x) + exp(-\x)) + 2});
            \draw (1 cm,2pt) -- (1 cm,0pt) node[anchor=north] {$\delta$};
            \draw (-1 cm,2pt) -- (-1 cm,0pt) node[anchor=north] {$-\delta$};
            \draw[dash pattern=on 2pt off 3pt] (1, 0) -- (1, 3.2);
            \draw[dash pattern=on 2pt off 3pt] (-1, 0) -- (-1, 3.2);
            \draw[dash pattern=on 2pt off 3pt] (-5, 3.2) -- (5, 3.2) node[above] (5, 3.2) {$\omega_{+} = \lambda_3(\rho_{+}, u_{1+}, S_{*})$};
            \draw[dash pattern=on 2pt off 3pt] (-5, 0.8) -- (5, 0.8) node at (-5, 0.5) {$\omega_{-} = \lambda_3(\rho_{-}, u_{1-}, S_{*})$};	
            \node at (0, -1) {$S_{*} = \ln (\frac{4\pi }{3}\frac{\theta_{-}}{\rho_{-}^{2/3}}) + 1 = \ln (\frac{4\pi }{3}\frac{\theta_{+}}{\rho_{+}^{2/3}})+ 1$,
                ~~~~~~~$\delta = \frac{1}{k} \varepsilon^{\frac{3}{5}-\frac{2}{5}a}$};
        \end{tikzpicture}}	
        \caption{Smooth approximation of rarefaction waves}
        \label{fig.ab2}
    \end{figure}
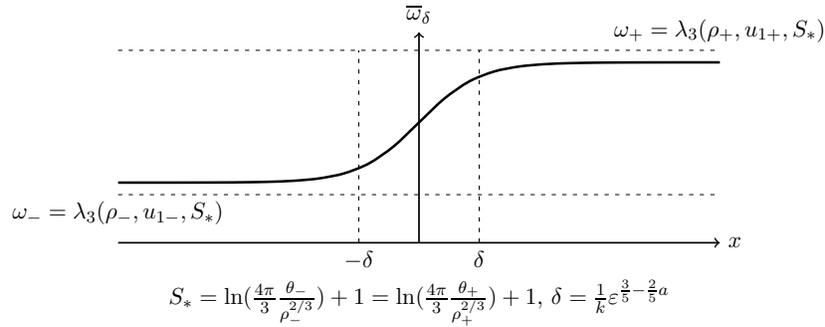

Corresponding to \eqref{1.24}, we define the approximate smooth 3-rarefaction wave  $(\bar{\rho}_{\delta},\bar{u}_{\delta},\bar{\theta}_{\delta})(t,x)$ with $\bar{u}_{\delta}=(\bar{u}_{1\delta},0,0)$ and $\bar{\phi}_{\delta}(t,x)$ by
\begin{equation}
\label{1.27}
\left\{
\begin{aligned}
&\overline{\omega}_{\delta}(t,x)=\lambda_{3}(\bar{\rho}_{\delta}(t,x),\bar{u}_{1\delta}(t,x),S_{*}),\ (\rho_{+},u_{1+},\theta_{+})\in R_{3}(\rho_{-},u_{1-},\theta_{-}),
\\
&\bar{u}_{1\delta}(z)-u_{1-}=\int^{\bar{\rho}_{\delta}(t,x)}_{\rho_{-}}\sqrt{\frac{5}{3}\frac{1}{2\pi e}\varrho^{-\frac{4}{3}}e^{S_{*}}
+\varrho^{-1}(\frac{d}{d\rho}\rho^{-1}_{\mathrm{e}}(\varrho))}\,d\varrho,
\\
&\bar{\theta}_{\delta}(t,x)
=\frac{3}{2}\frac{1}{2\pi e}e^{S_{*}}\bar{\rho}_{\delta}^{\frac{2}{3}}(t,x),\\
& \bar{\phi}_{\delta}(t,x)=\rho^{-1}_{\mathrm{e}}[\bar{\rho}_{\delta}(t,x)],
\\
&\lim\limits_{x\to\pm\infty}(\bar{\rho}_{\delta},\bar{u}_{1\delta},\bar{\theta}_{\delta})(t,x)=(\rho_{\pm},u_{1\pm},\theta_{\pm}),
\end{aligned} \right.
\end{equation}
with $\overline{\omega}_{\delta}(t,x)$ being the solution to the Burgers equation \eqref{1.25}.
From now on, we would omit the explicit dependence of $(\bar{\rho}_{\delta},\bar{u}_{\delta},\bar{\theta}_{\delta},\bar{\phi}_{\delta})(t,x)$ on $\delta$ and denote it by $(\bar{\rho},\bar{u},\bar{\theta},\bar{\phi})(t,x)$ for simplicity. Then the approximate smooth 3-rarefaction wave $(\bar{\rho},\bar{u},\bar{\theta})(t,x)$ with $\bar{u}=(\bar{u}_{1},0,0)$ and $\bar{\phi}(t,x)$ satisfy the compressible  quasineutral Euler system
\begin{equation}
\label{1.28}
\left\{
\begin{aligned}
&\bar{\rho}_{t}+(\bar{\rho}\bar{u}_{1})_{x}=0,
\\
&(\bar{\rho} \bar{u}_{1})_{t}+(\bar{\rho}\bar{u}_{1}^{2})_{x}+\bar{P}_{x}+\bar{\rho}\bar{\phi}_{x}=0,
\\
&(\bar{\rho} \bar{u}_{i})_{t}+(\bar{\rho}\bar{u}_{1}\bar{u}_{i})_{x}=0,\quad i=\mbox{2,3},
\\
&[\bar{\rho}(\bar{\theta}+\frac{|\bar{u}|^{2}}{2})]_{t}+[\bar{\rho}\bar{u}_{1}(\bar{\theta}+\frac{|\bar{u}|^{2}}{2})
+\bar{P}\bar{u}_{1}]_{x}+\bar{\rho}\bar{u}_{1}\bar{\phi}_{x}=0,
\\
&\bar{\rho}=\rho_{\mathrm{e}}(\bar{\phi}),
\end{aligned} \right.
\end{equation}
with $\bar{P}=\frac{2}{3}\bar{\rho}\bar{\theta}$. It should  be emphasized that
the initial function $\overline{\omega}_{\delta}$ in \eqref{1.25} connecting two distinct constant states is strictly increasing and smooth, so  $\overline{\omega}(t,x)$ and $(\bar{\rho},\bar{u},\bar{\theta},\bar{\phi})(t,x)$ are smooth in $t$ and $x$ up to any order by 
 \cite{Matsumura-Nishihara,Xin}. Their relevant properties are outlined in Section \ref{sec.4}.

\subsection{Notation, weight and norm}
First of all, we introduce
\begin{equation}
\label{2.1}
\overline{M}:=M_{[\bar{\rho},\bar{u},\bar{\theta}](t,x)}(v)
=\frac{\bar{\rho}(t,x)}{(2\pi R\bar{\theta}(t,x))^{3/2}}\exp\big(-\frac{|v-\bar{u}(t,x)|^{2}}{2R\bar{\theta}(t,x)}\big),
\end{equation}
with its  fluid quantities $(\bar{\rho},\bar{u},\bar{\theta})(t,x)$ being the approximate smooth rarefaction wave
constructed in \eqref{1.27}.
For convenience of the proof, we fix a normalized global Maxwellian with the fluid constant state
$(1,0,\frac{3}{2})$ by
\begin{equation}
\label{2.3}
\mu:=M_{[1,0,\frac{3}{2}]}(v)=(2\pi)^{-\frac{3}{2}}\exp\big(-\frac{|v|^{2}}{2}\big)
\end{equation}
as a reference equilibrium, and choose both the far-field  states $(\rho_+,u_+,\theta_+)$ and $(\rho_-,u_-,\theta_-)$
in \eqref{1.19} to be close enough to the constant state $(1,0,\frac{3}{2})$ such that the approximate smooth rarefaction wave further satisfies that
\begin{equation}
\label{2.2}
\sup_{t\geq 0,x\in \R}(|\bar{\rho}(t,x)-1|+|\bar{u}(t,x)|+|\bar{\theta}(t,x)-\frac{3}{2}|)<\eta_{0},
\end{equation}
for some suitably small  constant $\eta_{0}>0$  independent of $\varepsilon$.

\begin{figure}[h]
        \begin{tikzpicture}
            \usetikzlibrary{decorations.pathmorphing}
            \begin{scope}[thick, discont/.style={semithick, decoration={zigzag, segment length=4pt, amplitude=2pt}, decorate}]
                \draw (0, 0) -- (0.3, 0);
                \draw[discont] (0.3, 0) -- (0.6, 0);
                \draw[->] (0.6, 0) -- (3.5, 0) node[right] {$b$};
                \draw (0, 0) -- (0, 0.3);
                \draw[discont] (0, 0.3) -- (0, 0.6);
                \draw[->] (0, 0.6) -- (0, 3.5) node[above] {$a$};
            \end{scope}
            
            \foreach \x/\xtext in {1.2/\frac{3}{5}, 1.5/\frac{2}{3}, 3/1}
                \draw (\x, 1pt) -- (\x, -1pt) node[anchor=north] {$\xtext$};
            
            \foreach \y/\ytext in {1.5/\frac{2}{3}, 3/1}
                \draw (1pt, \y) -- (-1pt, \y) node[anchor=east] {$\ytext$};
            
            \filldraw[fill=black, opacity=0.7] (1.2, 3) -- (3, 3) -- (3, 1.5) -- (1.5, 1.5) -- cycle;
            
            \begin{scope}[dash pattern=on 2pt off 3pt]
                \draw (1.2, 0) -- (1.2, 3);
                \draw (1.5, 0) -- (1.5, 1.5);
                \draw (3, 0) -- (3, 1.5);
                \draw (0, 1.5) -- (1.5, 1.5);
                \draw (0, 3) -- (1.2, 3);
            \end{scope}
            \draw[thick] (1.5, 1.5) -- (1.5, 3);
        \end{tikzpicture}
        \caption{Ranges of parameters $a$ and $b$}
        \label{fig.ab}
    \end{figure}
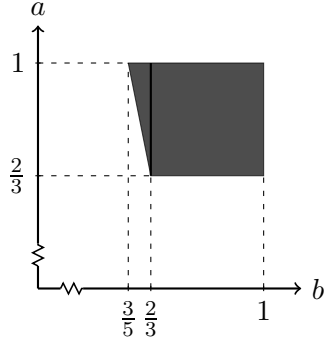  

As in \cite{Duan-Yang-Yu}, we define the scaled independent variables
\begin{equation}
	\label{3.1}
	y=\frac{x}{\varepsilon^{a}}, \quad \tau=\frac{t}{\varepsilon^{a}},
\end{equation}
with the parameter $a$ satisfying $(a,b)\in \mathcal{S}$ as in \eqref{def.ab}, see Figure \ref{fig.ab}, that is, 
\begin{equation}
	\label{3.2a}
\left\{
\begin{array}{rl}
&\frac{2}{3}\leq a\leq1, \quad \quad \quad \mbox{when}\quad \frac{2}{3}\leq b\leq1,
		\\
&4-5b\leq a\leq 1, \quad \mbox{when}\quad \frac{3}{5}\leq b\leq\frac{2}{3}.
\end{array} \right.
\end{equation}
Here the range of the parameter $a$ is determined by the requirements \eqref{6.3b} in the later proof under the conditions
\eqref{1.4B}, \eqref{3.21} and \eqref{3.22}; see also Remark \ref{rmk.ea} for a rough explanation. 
We remark that the scaling transformation in \eqref{3.1} is crucial in the energy estimates.

Therefore, the hydrodynamic limit problem on \eqref{1.1} with \eqref{1.4B}, \eqref{1.1id}, \eqref{1.4a} and \eqref{1.5a}, in particular to establish convergence rates in $\varepsilon$, is transferred to showing the global-in-time dynamical stability of the rarefaction waves to the following VPL system under the scaling \eqref{3.1}:
\begin{equation}
	\label{prefeq}
	\left\{
	\begin{aligned}
		F_{\tau}+v_{1}F_{y}-\phi_{y}\partial_{v_{1}}F&=\varepsilon^{a-1}Q(F,F),
		\\
		-\varepsilon^{2b-2a}\phi_{yy}&=\rho-\rho_{\mathrm{e}}(\phi), \quad \rho=\int_{\mathbb{R}^{3}}F\,dv,
	\end{aligned} \right.
\end{equation}
still satisfying 
\begin{equation}
\label{prefeqid1}
F(0,y,v)=F_0(y,v),
\end{equation}
and 
\begin{equation}
	\label{prefeqid2}
	\lim_{y\to\pm\infty}F_{0}(y,v)=M_{[\rho_{\pm},u_{\pm},\theta_{\pm}]}(v), \quad \lim_{y\to\pm\infty}\phi(t,y)=\phi_{\pm}, \quad \rho_{\pm}=\rho_{\mathrm{e}}(\phi_{\pm}).
\end{equation}
Note that initial data $F_0(y,v)=F_0(\frac{x}{\varepsilon^a},v)$ is allowed to depend on parameters $\varepsilon$ and $a$.

Correspondingly, we  set the scaled  macroscopic perturbation
\begin{equation}
	\label{3.2}
(\widetilde{\rho},\widetilde{u},\widetilde{\theta},\widetilde{\phi})(\tau,y)
	=({\rho}-\overline{\rho},u-\overline{u},\theta-\overline{\theta},\phi-\overline{\phi})(t,x),
\end{equation}
and the scaled microscopic perturbation
\begin{equation}
	\label{3.3}
\sqrt{\mu}{f}(\tau,y,v)=G(t,x,v)-\overline{G}(t,x,v),
\end{equation}
where the correction term $\overline{G}(t,x,v)$ is given by
\begin{equation}
	\label{3.4}
	\overline{G}=\varepsilon^{1-a} L_{M}^{-1}\big\{P_{1}v_{1}M(\frac{|v-u|^{2}
		\bar{\theta}_{y}}{2R\theta^{2}}+\frac{(v-u)\cdot\bar{u}_{y}}{R\theta})\big\}.
\end{equation}
The reason for introducing the subtraction $G-\overline{G}$ is that the term $P_{1}(v_{1}M_{x})$ in \eqref{1.15} contains  $\|(\bar{u}_{y},\bar{\theta}_{y})(\tau)\|^{2}$
in the $L^2$ estimate under the scaling transform \eqref{3.1} and
the time decay of $\|(\bar{u}_{y},\bar{\theta}_{y})(\tau)\|^{2}$  is $\varepsilon^{a}(\delta+\varepsilon^{a}\tau)^{-1}$ by Lemma \ref{lem7.2}, which is not integrable in $\tau$ over $(0,+\infty)$. 

The following notations are needed in the energy analysis for convenience of
presentation. Throughout the paper, generic positive constants are denoted  by either $c$ or $C$ varying from line to line, and
they are independent of the parameters $\varepsilon$, $\delta$, $k$  and the time $\tau$.  We use
$\langle \cdot , \cdot \rangle$ to denote the standard $L^{2}$ inner product in $\mathbb{R}_{v}^{3}$
with its corresponding $L^{2}$ norm $|\cdot|_2$, and $( \cdot , \cdot )$ to denote $L^{2}$ inner product in
$\mathbb{R}_{y}$ or $\mathbb{R}_{y}\times \mathbb{R}_{v}^{3}$  with its corresponding $L^{2}$ norm $\|\cdot\|$.
Let  $\alpha$ and $\beta$ be a nonnegative integer and a multi-index $\beta=[\beta_{1},\beta_{2},\beta_{3}]$, respectively. Denote
	$\partial_{\beta}^{\alpha}=\partial_{y}^{\alpha}
	\partial_{v_{1}}^{\beta_{1}}\partial_{v_{2}}^{\beta_{2}}\partial_{v_{3}}^{\beta_{3}}$.
If each component of $\beta$ is not greater than the corresponding one  of
$\overline{\beta}$, we use the standard notation
$\beta\leq\overline{\beta}$. And $\beta<\overline{\beta}$ means that
$\beta\leq\overline{\beta}$ and $|\beta|<|\overline{\beta}|$.
$C^{\bar\beta}_{\beta}$ is the usual  binomial coefficient. Sometimes we use the notation
 $A\approx B$ to denote that there exists $C_{0}>1$
such that $C_{0}^{-1}B\leq A\leq C_{0}B$. We also use $h_y=\partial_yh$ to denote  the space derivative of $h$ with respect to $y$.

Motivated by \cite{Duan-Yang-Yu-1} and \cite{Guo-JAMS}, a crucial point in the
proof is to introduce a new weight function which also depends on the electric potential function $\phi(\tau,y)$.  

\begin{definition}
Let
\begin{equation}
	\label{3.12}
	w(\alpha,\beta)(\tau,v):=\langle v\rangle^{2(l-|\alpha|-|\beta|)}e^{q(\tau)\langle v\rangle},\quad  l\geq |\alpha|+|\beta|,
\end{equation}
with $\langle v\rangle=\sqrt{1+|v|^{2}}$ and
\begin{equation}
	\label{3.13}
	q(\tau):=q_{1}-q_{2}\int^{\tau}_{0}q_{3}(s)\,ds,
\end{equation}
where the constants $q_{1}>0$ and $q_{2}>0$ will be chosen in the proof later, see also Theorem \ref{thm2.1},  and the function $q_{3}(\tau)$, depending on the electric potential and  the Knudsen number is given by
\begin{equation}
	\label{3.14}
	q_{3}(\tau):=\varepsilon^{1-a}(\|\phi_\tau\|^{2}_{L_{y}^{\infty}}+\|\phi_y\|^{2}_{L_{y}^{\infty}}
	+\|\phi_{yy}\|^{2}+\|\phi_{yyy}\|^{2}).
\end{equation}
In addition, we require that $q_3(\cdot)$ is integrable in time satisfying
\begin{equation*}
	q_\infty:=q_{1}-q_{2}\int^{\infty}_{0}q_{3}(s)\,ds>0,
\end{equation*}
so that  $q(\tau)$ is a strictly positive continuous function monotonically decreasing from $q_1>0$ at $\tau=0$ to $q_\infty>0$ as $\tau\to+\infty$.
\end{definition}
It should be pointed out that such weight function \eqref{3.12} is a little different from  those used in \cite{Guo-JAMS,Duan-Yu1} because of the specific choice of $q_3(\tau)$ depending on norms of the electric potential function $\phi(\tau,y)$ as in \eqref{3.13} and \eqref{3.14}. Our weight is very useful to deal with the large-velocity growth in the nonlinear electric potential term when the time-decay rate of electric potential is unavailable. In particular, the weighted energy estimate induces a quartic energy dissipation from the time derivative of such weight function; see \eqref{def.qedrf} and \eqref{2.5} later on. 
 
\begin{remark}\label{rem.q12}
We emphasize that since $\phi=\bar{\phi}+\widetilde{\phi}$ for the background profile $\bar{\phi}$ and the perturbation $\widetilde{\phi}$, the non-negative functional $q_3(\tau)$ in \eqref{3.14} contains two parts
	\begin{eqnarray*}
	q_3^I(\tau)&&:=\varepsilon^{1-a}(\|\bar{\phi}_\tau\|^{2}_{L_{y}^{\infty}}+\|\bar{\phi}_{y}\|^{2}_{L_{y}^{\infty}}
	+\|\bar{\phi}_{yy}\|^{2}+\|\bar{\phi}_{yyy}\|^{2}),
	\\
	\quad q_3^{I\!I}(\tau)&&:=\varepsilon^{1-a}(\|\widetilde{\phi}_\tau\|^{2}_{L_{y}^{\infty}}+\|\widetilde{\phi}_{y}\|^{2}_{L_{y}^{\infty}}
	+\|\widetilde{\phi}_{yy}\|^{2}+\|\widetilde{\phi}_{yyy}\|^{2}).
\end{eqnarray*}
Here the electric potential profile part  $q_3^I(\tau)$ has a fast decay as $\varepsilon^{2a}(\delta+\varepsilon^{a}\tau)^{-2}$  in terms of Lemma \ref{lem7.2}.
And the perturbation part $q_3^{I\!I}(\tau)$ is part of the energy dissipation functional \eqref{3.18} that can be proved to be integrable over $(0,+\infty)$
	in terms of \eqref{3.22}
	but it does not enjoy any explicit time decay because of the techniques of the proof.
\end{remark}
 Corresponding to the reference global Maxwellian $\mu$ in \eqref{2.3},
the Landau collision frequency is given by
\begin{equation}
	\label{3.15}
	\sigma^{ij}(v):=\Phi_{ij}\ast \mu=\int_{{\mathbb R}^3}\Phi_{ij}(v-v_{\ast})\mu(v_{\ast})\,dv_{\ast}.
\end{equation}
Here $[\sigma^{ij}(v)]_{1\leq i,j\leq 3}$ is a positive-definite self-adjoint matrix.
We denote the weighted $L^2$ norms as
$$
|\partial^\alpha_\beta
g|^2_{w}\equiv|w(\alpha,\beta)\partial^\alpha_\beta
g|^2_{2}:=\int_{{\mathbb
R}^3}w^{2}(\alpha,\beta)|\partial^\alpha_\beta g(y,v)|^2\,dv,
$$
and
$$
\|\partial^\alpha_\beta
g\|^2_{w}\equiv\|w(\alpha,\beta)\partial^\alpha_\beta
g\|^2_{2}:=\int_{\mathbb R}\int_{{\mathbb
R}^3}w^{2}(\alpha,\beta)|\partial^\alpha_\beta g(y,v)|^2\,dv\,dy.
$$
With \eqref{3.15}, we define the weighted dissipation norms as
$$
|g|^2_{\sigma,w}:=\sum_{i,j=1}^3\int_{{\mathbb
		R}^3}w^{2}[\sigma^{ij}\partial_{v_i}g\partial_{v_j}g+\sigma^{ij}\frac{v_i}{2}\frac{v_j}{2}g^2]\,dv,\ \mbox{and}\ \ \|g\|_{\sigma,w}:=\big\||g|_{\sigma,w}\big\|.
$$
Also $|g|_{\sigma}=|g|_{\sigma,1}$ and $\|g\|_{\sigma}=\|g\|_{\sigma,1}$.
From \cite[Lemma 5]{Strain-Guo-2008}, one has
\begin{equation}
\label{3.16}
	|g|_{\sigma,w}\approx |w\langle v\rangle^{-\frac{1}{2}}g|_2+\big|w\langle v\rangle^{-\frac{3}{2}}\nabla_vg\cdot\frac{v}{|v|}\big|_2+\big|w\langle v\rangle^{-\frac{1}{2}}\nabla_vg\times\frac{v}{|v|}\big|_2.
\end{equation}

In order to construct the global-in-time solutions for the VPL system near local Maxwellians as in \eqref{2.1},
a key point is to establish uniform energy estimates on the macroscopic perturbation $(\widetilde{\rho},\widetilde{u},\widetilde{\theta},\widetilde{\phi})(\tau,y)$
and the microscopic perturbation $f(\tau,y,v)$. For small $q(\tau)$ in \eqref{3.12} and the parameters $a$ and $b$ in \eqref{3.2a},  we define the following instant energy functional $\mathcal{E}_{2,l,q}(\tau)$ as
\begin{eqnarray}
	\label{3.17}
	\mathcal{E}_{2,l,q}(\tau):=&&
	\sum_{|\alpha|\leq1}(\|\partial^{\alpha}(\widetilde{\rho},\widetilde{u},\widetilde{\theta})(\tau)\|^{2}
	+\|\partial^{\alpha}f(\tau)\|_{w}^{2})+\sum_{|\alpha|+|\beta|\leq 2,|\beta|\geq1}\|\partial^{\alpha}_{\beta}f(\tau)\|_{w}^{2}
	\notag\\
	&&\hspace{0.5cm}
	+\varepsilon^{2-2a}\sum_{|\alpha|=2}(\|\partial^{\alpha}(\widetilde{\rho},\widetilde{u},\widetilde{\theta})(\tau)\|^{2}+\|\partial^{\alpha}f(\tau)\|_{w}^{2})+E(\tau),
\end{eqnarray}
with
\begin{eqnarray*}
	E(\tau):=&&\sum_{|\alpha|\leq1}(\|\partial^{\alpha}\widetilde{\phi}(\tau)\|^{2}
	+\varepsilon^{2b-2a}\|\partial^{\alpha}\widetilde{\phi}_{y}(\tau)\|^{2})
	\notag\\
	&&\quad+\varepsilon^{2-2a}\sum_{|\alpha|=2}
	(\|\partial^{\alpha}\widetilde{\phi}(\tau)\|^{2}+\varepsilon^{2b-2a}\|\partial^{\alpha}\widetilde{\phi}_{y}(\tau)\|^{2}).
\end{eqnarray*}
Correspondingly, the normal {\bf quadric} energy dissipation rate functional $\mathcal{D}_{2,l,q}(\tau)$ is given by
\begin{eqnarray}
	\label{3.18}
	\mathcal{D}_{2,l,q}(\tau):=&&\varepsilon^{1-a}\sum_{1\leq|\alpha|\leq 2}(\|\partial^{\alpha}(\widetilde{\rho},\widetilde{u},\widetilde{\theta})(\tau)\|^{2}
	+\|\partial^{\alpha}\widetilde{\phi}(\tau)\|^{2}
	+\varepsilon^{2b-2a}\|\partial^{\alpha}\widetilde{\phi}_{y}(\tau)\|^{2})
	\notag\\
	&&+\varepsilon^{a-1}(\sum_{|\alpha|\leq 1}\|\partial^{\alpha}f(\tau)\|_{\sigma,w}^{2}
	+\sum_{|\alpha|+|\beta|\leq 2,|\beta|\geq1}\|\partial^{\alpha}_{\beta}f(\tau)\|_{\sigma,w}^{2})\notag\\
	&&+\varepsilon^{1-a}\sum_{|\alpha|=2}\|\partial^{\alpha}f(\tau)\|_{\sigma,w}^{2},
\end{eqnarray}
while the {\bf quartic} energy dissipation rate functional is given by
\begin{equation}
\label{def.qedrf}
q_2q_3(\tau)\mathcal{H}_{2,l,q}(\tau)
\end{equation} 
where $q_2>0$ is a constant, $q_3(\tau)$ is defined in \eqref{3.14} and $\mathcal{H}_{2,l,q}(\tau)$ is denoted as 
\begin{eqnarray}
		\label{4.60}
		\mathcal{H}_{2,l,q}(\tau):=&&\sum_{|\alpha|\leq 1}\|\langle v\rangle^{\frac{1}{2}}\partial^{\alpha}f(\tau)\|_{w}^{2}+
		\varepsilon^{2-2a}\sum_{|\alpha|=2}\|\langle v\rangle^{\frac{1}{2}}\partial^{\alpha}f(\tau)\|_{w}^{2}
		\notag\\
		&&\hspace{1cm}+\sum_{|\alpha|+|\beta|\leq 2,|\beta|\geq1}\|\langle v\rangle^{\frac{1}{2}}\partial^{\alpha}_{\beta}f(\tau)\|_{w}^{2}.
	\end{eqnarray}
Here we restrict the highest order derivatives in all functionals to the second order for simplicity. Much higher orders can be considered so that smooth solutions can be obtained.

\subsection{Main theorem}\label{sec.2.1}
With the above preparation of notations, the main result can be stated as follows.  It is  more precise than that in Theorem \ref{thm.rs} since the convergence for both initial data and solutions as $\varepsilon\to 0$ can be made more explicit in terms of  energy functional $\mathcal{E}_{2,l,q}$ in \eqref{3.17}. 

\begin{theorem}\label{thm2.1}
Let $\lambda=\varepsilon^{b}$ with $b\in[\frac{3}{5},1]$ in \eqref{1.1} satisfying \eqref{ll}.
Let $(\rho^{R},u^{R},\theta^{R},\phi^{R})(\frac{x}{t})$ be the Riemann solution to the compressible quasineutral Euler system \eqref{1.4}-\eqref{1.19} defined by \eqref{1.24} with the wave strength $\delta_r=|\rho_{+}-\rho_{-}|+|u_{+}-u_{-}|+|\theta_{+}-\theta_{-}|$ small enough.
Under the scaling transformation $(\tau,y)=(\varepsilon^{-a}t,\varepsilon^{-a}x)$
with the parameter $a$ satisfying \eqref{3.2a}, if the initial data $F_{0}(y,v)\geq0$ satisfies 
\begin{equation}
\label{3.20}
\mathcal{E}_{2,l,q}(\tau)\mid_{\tau=0}\,\leq\, k^{\frac{1}{3}}\varepsilon^{\frac{2}{3}},
\end{equation}
for a small constant $k>0$ independent of $\varepsilon$, where we have choosen the constants $q_{1}= k^{\frac{1}{24}}$ 
and $q_{2}=k^{-\frac{1}{12}}$ in \eqref{3.12} and \eqref{3.13}, then there exists a sufficiently small constant $\varepsilon_0$ with $\varepsilon_0< k$
such that for each $\varepsilon\in (0,\varepsilon_0)$,
the Cauchy problem  on the Vlasov-Poisson-Landau system \eqref{prefeq} with \eqref{prefeqid1} and \eqref{prefeqid2}
 admits a unique global-in-time solution $[F(\tau,y,v),\phi(\tau,y)]$ satisfying that 
$F(\tau,y,v)\geq 0$ and 
 \begin{equation}
 	\label{2.5}
 	\sup_{\tau\geq 0}\mathcal{E}_{2,l,q}(\tau) +\int^{\infty}_{0}\mathcal{D}_{2,l,q}(\tau)\,d\tau +q_2\int_0^\infty q_3(\tau)\mathcal{H}_{2,l,q}(\tau)\,d\tau\leq k^{\frac{1}{6}}\varepsilon^{\frac{6}{5}-\frac{4}{5}a}.
 \end{equation}
\end{theorem}

Several remarks are in order.
\begin{remark}
To the best of our knowledge, Theorem \ref{thm2.1}  (or its rough version Theorem \ref{thm.rs}) seems to provide the first result regarding the hydrodynamic limit
with rarefaction waves for the kinetic equation under the effect of the electrostatic potential force. Moreover, the approach developed in this paper 
can  be applied to  other types of kinetic equations, such as Vlasov-Poisson-Boltzmann system with or without angular cutoff.
\end{remark}
\begin{remark}\label{rem2.5}
	We choose the Debye length $\lambda=\varepsilon^{b}$  with $b\in[\frac{3}{5},1]$ in Theorem \ref{thm2.1} in order to get a uniform convergence rate in $\varepsilon$. The range of the parameters $b$ is determined by the requirements \eqref{6.4b} in the later proof due to the slow time-decay rate in the non-trivial profile of the electric potential and the appearance of linear term in the electric potential; see \eqref{4.23} and \eqref{7.33}, respectively.
	This  phenomena is caused by the solution $F$ connecting two distinct global Maxwellians and the electric potential $\phi$
	connecting two  distinct constant states.
\end{remark}

\begin{remark}
The energy inequality \eqref{2.5} holds true at $\tau=0$ under the choice of the initial data \eqref{3.20}.	 Back to the original $(t,x)$ variable, the estimate \eqref{thm.rate} shows that the uniform convergence rate in small Knudsen number $\varepsilon>0$ can be variable with respect to the scaling parameter $a$
satisfying \eqref{3.2a}. From the proof, the obtained convergence rate in \eqref{2.5} seems optimal; see the detailed explanations of this point in
subsection \ref{sec.3.4}. In particular, choosing $a=\frac{2}{3}$ can give the fastest convergence rate $\varepsilon^{\frac{1}{3}}|\ln\varepsilon|$.
This is somewhat mysterious given that the $\varepsilon^{\frac{1}{3}}$ threshold is also sharp for Landau damping in the weakly collisional limit of kinetic equations
in \cite{Chaturvedi} for the VPL system and  \cite{Bedrossian} for the Vlasov-Poisson-Fokker-Planck system.
\end{remark}

\begin{remark}\label{rmk2.7}
We point out that the set of nonnegative smooth initial data $F_0$ satisfying \eqref{3.20} is nonempty. In fact, the initial value \eqref{1.1id} can be specifically chosen as
	\begin{equation}
		\label{2.50b}
		F_0(x,v)=M_{[\bar{\rho},\bar{u},\bar{\theta}]}(0,x,v), 
	\end{equation}
with $\phi_0(x)=\bar{\phi}(0,x)$, which automatically satisfies \eqref{3.20}, \eqref{1.4a} and \eqref{1.5a}. Here $(\bar{\rho},\bar{u},\bar{\theta},\bar{\phi})(0,x)$ are smooth in $x$ up to any order 
satisfying \eqref{1.27}.
	Indeed, in view of \eqref{2.50b} and the decomposition $F(t,x,v)=M_{[\rho,u,\theta]}(t,x,v)+G(t,x,v)$
	we have $M_{[\rho,u,\theta]}(0,x,v)=M_{[\bar{\rho},\bar{u},\bar{\theta}]}(0,x,v)$ and $G(0,x,v)=0$.
	This implies that $(\rho,u,\theta)(0,x)=(\bar{\rho},\bar{u},\bar{\theta})(0,x)$ and $\overline{G}(0,x,v)+\sqrt{\mu}f(0,x,v)=0$.
	Then it holds that
		$(\widetilde{\rho},\widetilde{u},\widetilde{\theta},\widetilde{\phi})(0,y)
		=({\rho}-\bar{\rho},u-\bar{u},\theta-\bar{\theta},\phi-\bar{\phi})(0,x)=0$
	and 
	$$
	f(0,x,v)=-\frac{\overline{G}(0,x,v)}{\sqrt{\mu}}.
	$$
	Under the scaling transformation \eqref{3.1}, we have from \eqref{3.17}, \eqref{7.24}, \eqref{7.25},  Lemma \ref{lem7.2} and 
	\eqref{3.21} that
	\begin{eqnarray*}
		\mathcal{E}_{2,l,q}(0)=&& 
		\sum_{|\alpha|\leq1}\|\partial^{\alpha}[\frac{\overline{G}(0)}{\sqrt{\mu}}]\|_{w}^{2}+\sum_{|\alpha|+|\beta|\leq 2,|\beta|\geq1}\|\partial^{\alpha}_{\beta}[\frac{\overline{G}(0)}{\sqrt{\mu}}]\|_{w}^{2}
		\notag\\
		&&
		+\varepsilon^{2-2a}\sum_{|\alpha|=2}\|\partial^{\alpha}[\frac{\overline{G}(0)}{\sqrt{\mu}}]\|_{w}^{2}\leq Ck\varepsilon^{\frac{7}{5}-\frac{3}{5}a}
		\leq Ck\varepsilon^{\frac{6}{5}-\frac{4}{5}a}.
	\end{eqnarray*}
	Therefore, one can use \eqref{2.50b} as a smooth approximation to Riemann data \eqref{thm.rd} in case of rarefaction waves.
\end{remark}

\begin{remark}
	Those  results 	
	\cite{Guo-JAMS,Chaturvedi,Dong-Guo,Duan-Yu1,Strain-Zhu,Wang}
for  the well-posedness theories of the VPL system rely crucially on the strong explicit time-decay rate of the solutions 
	to close the energy estimates. However, the time-decay rate of solutions is unavailable in the current problem with rarefaction wave,
	so we have developed a new weight function as in  \eqref{3.12} to overcome similar difficulties, see the detailed explanations of this point in subsection \ref{sec.2.3}.
	We remark that the explicit time rates of convergence to viscous rarefaction wave is still open
	for the compressible Navier-Stokes equations.
\end{remark}

\subsection{Other relevant literature}
We review some relevant literature including part of works mentioned before. First of all, let two parameters $\varepsilon>0$ and $\lambda>0$ be fixed to be unit. In the absence of self-consistent forces, the VPL system is reduced to the Landau equation. There are extensive studies on the pure Landau equation; among them, we only mention \cite{CaMi, Degond-Lemou, Guo-2002, KGH}. In particular, \cite{Guo-2002} gave the first global existence result for classical solutions of the Landau equation around global Maxwellians in a torus domain $\mathbb{T}^3$. Later, the result in \cite{Guo-2002} was extent in \cite{Guo-JAMS} to the VPL system in $\mathbb{T}^3$ through a more delicate energy method. The corresponding stability result for the VPL system in the whole space $\R^3$ was obtained by \cite{Strain-Zhu}; see also \cite{Duan-Yang-Zhao,Wang}. We mention two other recent important results \cite{Bedrossian-1, Dong-Guo} for the VPL system.
Moreover, regarding solutions around non-trivial asymptotic profiles in large time, the first and third authors of this paper \cite{Duan-Yu1, DY31} also constructed global classical solutions to the two-species VPL system around local Maxwellians with rarefaction waves either in the one-dimensional line with slab symmetry or in the three-dimensional infinite channel domain $\R\times \mathbb{T}^2$. The large time behavior of the electric potential in \cite{Duan-Yu1, DY31}  is trivial and the nontrivial case was considered in \cite{Duan-Liu-2015} for the Vlasov-Poisson-Boltzmann system describing the dynamics of only ions similar to \eqref{1.1}.


There is  numerous literature for the fluid limit problem as $\varepsilon\to 0$ in the context of the Boltzmann equation. We refer readers to \cite{Golse,Saint-Raymond} and references therein for many classical results. Here we mention  \cite{Caflish,Nishida, Guo-Jang-Jiang-2010} for the compressible fluid limit. In the setting of one space dimension, regarding the limit of the cutoff Boltzmann equation to the compressible Euler system which admits
specific solutions of basic wave patterns, we refer to \cite{Yu1} for shock wave, \cite{XinZeng} for rarefaction wave,
\cite{Huang-Wang-Yang} for contact discontinuity, and \cite{Hang-Wang2-Yang} for the composition of basic wave patterns.
These results are proved by using an energy method based on the macro-micro decomposition initiated by \cite{Liu-Yu} and developed by \cite{Liu-Yang-Yu}. 

For the fluid limit problem, it is much harder to treat the Landau equation than the cutoff Boltzmann equation since the Landau collision operator
exposes the velocity diffusion property. In the incompressible regime, a diffusive expansion to the Landau equation
in the framework of classical solutions in high-order Sobolev spaces was developed by \cite{Guo-CPAM};
see also two recent works \cite{CaMi-1} and \cite{Rachid}.
In the compressible regime, the compressible Euler limit for smooth solutions to the Landau equation near global Maxwellians was recently studied by the same authors of this paper \cite{Duan-Yang-Yu-2} and  \cite{Lei-1} via different energy methods. Regarding solutions with wave patterns, we also obtained  in \cite{Duan-Yang-Yu} an analogous result of \cite{XinZeng} for the convergence rate in small Knudsen number in case of rarefaction waves. As mentioned before, we follow the strategy of  \cite{Duan-Yang-Yu} to deal with the current problem on the convergence rate.


When $\lambda=0$ is further taken,  the limiting system  is  the compressible quasineutral Euler system that is \eqref{qnEuler} in the isentropic case. Hence, the compressible quasineutral Euler system  can be formally derived from the VPL system as both $\varepsilon$ and $\lambda$ tend to zero. However, the rigorous mathematical justification of
establishing such limit  still remains unknown, although it has been extensively studied for the Euler-Poisson system
and Vlasov-Poisson system, see for instance \cite{Cordier,Han-Kwan,Han-Kwan-2013,Han-Kwan-2015,Herda} and references therein. The current work is motivated to address the issue.

There is also a hard direction to study the asymptotic limit of the VPL system in the weak-collision regime corresponding to $\varepsilon\to +\infty$ or equivalently $1/{\varepsilon}\to 0$. Recently, by combining Guo's weighted energy method with the hypocoercive estimates and the commuting-vector-field techinque, a significant progress on the Landau damping of the VPL system  in the weakly collisional regime 
was obtained in \cite{Chaturvedi} for  the case $\varepsilon\to +\infty$ and $\lambda=1$; see also \cite{Bedrossian} for an earlier study of the relevant toy model. The non-cutoff Boltzmann case was recently investigated in \cite{BCD}. Readers may refer to \cite{Bedrossian,BCD, Chaturvedi} for more discussions on the subject. It should be very interesting to see whether or not the current approach basing on the macro-micro decomposition is also applicable in the weak-collision limit.

\subsection{Ideas of the proof}\label{sec.2.3}
We now outline some of the main difficulties and techniques involved in the paper.
The main analytical tool is the combination of techniques we have developed in \cite{Duan-Yang-Yu} for the Landau equation
with small Knudsen number and some new techniques 
for treating the major mathematical difficulties created by the nonlinear electric potential term
under the appropriate  scaling transformation \eqref{3.1}.
Compared to the previous work \cite{Duan-Yang-Yu}, the main difficulty in the proof for the VPL
system is to treat the electric potential $\phi$ and the Debye length $\lambda$ on the uniform energy estimates.
 
Since the linearized Landau operator with Coulomb interaction has no spectral gap
that results in the very weak velocity dissipation, the large-velocity growth in the nonlinear electric potential terms are hard to control.
To overcome those difficulties,  \cite{Guo-JAMS} introduced a useful weight function
$\langle v\rangle^{2(l-|\alpha|-|\beta|)}e^{\phi}$ and then established the following energy estimate for the solution of
\begin{equation}
\label{2.50a}
\frac{d}{dt}\widetilde{\mathcal{E}}_{2,l,q}(t)+\widetilde{\mathcal{D}}_{2,l,q}(t)
\leq C(\|\partial_{t}\phi(t)\|_{L^{\infty}}+\|\nabla_{x}\phi(t)\|_{L^{\infty}})\widetilde{\mathcal{E}}_{2,l,q}(t),	
\end{equation}
where $\widetilde{\mathcal{E}}_{2,l,q}(t)$ and $\widetilde{\mathcal{D}}_{2,l,q}(t)$ denote some energy functional and dissipation
functional respectively. In order to close this estimate, one key point of the proof is to obtain
the strong explicit time-decay rate of the electric potential such that 
\begin{equation}
	\label{2.50A}
\int^{+\infty}_0(\|\partial_{t}\phi(t)\|_{L^{\infty}}+\|\nabla_{x}\phi(t)\|_{L^{\infty}})\,dt\leq C.	
\end{equation}
A more detailed explanation of this point can be found in the
introduction of \cite{Guo-JAMS}. This type of method has been applied to many nice works, 
see \cite{Chaturvedi,Dong-Guo,Duan-Yu1,Duan-Yang-Zhao,Strain-Zhu,Wang} and references therein.

However, the approach in \cite{Guo-JAMS} cannot be directly applied to the current problem,
because the time-decay rate of the electric potential  is unavailable and the similar estimate \eqref{2.50A} is not longer true.
Therefore we have to develop new ideas. First of all, similar to the one in \cite{Guo-JAMS},  we use the
weight function $e^{\frac{1}{2}\phi}$ to  cancel the nonlinear electric potential term
$\frac{v_{1}}{2}\partial_{y}\phi\partial_{\beta}^{\alpha}f$
with the linear streaming term $v_{1}\partial_{\beta}^{\alpha}\partial_{y}f$, and then one has to deal with the following trouble term
\begin{equation}
 \label{2.8}
 \int_{\mathbb{R}}\int_{\mathbb{R}^{3}}(|\partial_{\tau}\phi|+|\partial_{y}\phi|)|\partial_{\beta}^{\alpha}f|^{2}\,dv\,dy.
\end{equation}
This term can no longer be controlled by the similar term on the right hand side of \eqref{2.50a}. 
Fortunately, we observe that \eqref{2.8} can be controlled by
\begin{multline}
\label{2.9}
\eta\varepsilon^{a-1}\int_{\mathbb{R}}\int_{\mathbb{R}^{3}}|\langle v\rangle^{-\frac{1}{2}}\partial_{\beta}^{\alpha}f|^{2}\,dv\,dy\\
+C_{\eta}\varepsilon^{1-a}\|(\partial_{\tau}\phi,\partial_{y}\phi)\|^{2}_{L_{y}^{\infty}}\int_{\mathbb{R}}\int_{\mathbb{R}^{3}}|\langle v\rangle^{\frac{1}{2}}\partial_{\beta}^{\alpha}f|^{2}\,dv\,dy.
\end{multline} 
The first term in \eqref{2.9} can be absorbed by  the linearized Landau operator associated with \eqref{3.16}.
To control the second term in \eqref{2.9}, we design a new important velocity-time-$\varepsilon$ dependent weight function
$w(\alpha,\beta)(\tau,v)$ given in \eqref{3.12}.
The factor $\exp\{q(\tau)\langle v\rangle^{2}\}$ in \eqref{3.12} is used to induce an extra
{\it quartic} energy dissipation term
\begin{multline}
	\label{2.54A}
q_2q_3(\tau)\big\{\sum_{|\alpha|\leq 1}\|\langle v\rangle^{\frac{1}{2}}\partial^{\alpha}f(\tau)\|_{w}^{2}+
\sum_{|\alpha|+|\beta|\leq 2,|\beta|\geq1}\|\langle v\rangle^{\frac{1}{2}}\partial^{\alpha}_{\beta}f(\tau)\|_{w}^{2}\\
+\varepsilon^{2-2a}\sum_{|\alpha|=2}\|\langle v\rangle^{\frac{1}{2}}\partial^{\alpha}f(\tau)\|_{w}^{2}\big\},
\end{multline}
for absorbing the second term in \eqref{2.9}. One of the key observations is that
the function $q_{3}(\tau)$ in $q(\tau)$ constructed by \eqref{3.14} is integrable about all time $\tau$ by using the high-order dissipation rate of 
$\widetilde{\phi}$ and the time-decay properties of the profile $\bar{\phi}$.
Note that the extra dissipation \eqref{2.54A} is quartic due to the dependence of $q_3(\tau)$ on the normal energy dissipation as mentioned in Remark \ref{rem.q12}.
The other factor $\langle v\rangle^{2(l-|\alpha|-|\beta|)}$ in \eqref{3.12} is used to take care of the derivative estimates of the free streaming term and
the velocity derivative of $f$. We would emphasize that such weighted function looks
more robust with possible applications to many other problems in the similar situation where the time-decay of solutions is very slow or unvaivalable.

Since the electric potential may take distinct constant states at both far fields,
and then the profile of the electric potential is constructed on the basis of the quasineutral assumption as $\lambda\to 0$,
and hence the potential profile $\bar{\phi}$ has a slow time-decay, see Lemma \ref{lem7.1}.  It is quite nontrivial to estimate the coupling
term $\partial_x\phi$ in the momentum equation as in \eqref{1.18}. For instance, we need to carry out the energy estimates on the integral term in \eqref{4.22}. For the purpose,  the key point is to use the good dissipative property from the Poisson equation after expanding the electron density function $\rho_{\mathrm{e}}(\phi)$ to the third order so that we can observe some new cancellations and obtain the higher order nonlinear terms, see the estimate in \eqref{4.24}.

As mentioned in \cite{Duan-Yang-Yu}, we need to subtract $\overline{G}(t,x,v)$ in \eqref{3.4} from $G(t,x,v)$ to remove the slow time-decay inhomogeneous terms,
because those terms induce the energy term  $\|(\bar{u}_{y},\bar{\theta}_{y})(\tau)\|^{2}$
in the $L^2$ estimate under the scaling transform \eqref{3.1} and
the time decay of $\|(\bar{u}_{y},\bar{\theta}_{y})(\tau)\|^{2}$  is $\varepsilon^{a}(\delta+\varepsilon^{a}\tau)^{-1}$ by Lemma \ref{lem7.2}, which is not integrable in $\tau$ over $(0,+\infty)$.  Note that the inverse of linearized operator $L^{-1}_M$ defined as \eqref{1.16}  is very complicated  due to the velocity diffusion effect of collision operator.
In order to handle the terms involving $L^{-1}_M$, we make use of the Burnett functions and velocity-decay properties to deal with them. This makes that the estimates
can be obtained in a clear way, see the identities \eqref{4.11} and \eqref{4.12} for details. 

Similar to the one in \cite{Duan-Yang-Yu,Duan-Yu1}, we use the decomposition 
$$
F=M+\overline{G}+\sqrt{\mu}f
$$ 
to improve the previous decomposition in \cite{Liu-Yu,Liu-Yang-Yu} such that some similar
basic estimates on the linearized Landau operator $\mathcal{L}$ and nonlinear
Landau operator $\Gamma$ around global Maxwellians in \cite{Strain-Guo-2008,Guo-JAMS} can be
adopted in a convenient way for the current problem, for instance, see Lemma \ref{lem7.5} and Lemma \ref{lem7.6}.
Note that the unknown $f$ is purely microscopic and it is essentially different from the one in \cite{Guo-2002,Guo-JAMS}, 
and hence  we can obtain the nice trilinear estimate, such as \eqref{7.30} in the one-dimensional
setting. However, whenever $f$ involves any macroscopic part, it seems impossible to obtain the trilinear estimate \eqref{7.30} in
case of the one-dimensional whole space $\mathbb{R}$  because the zero order macroscopic part is no longer in the dissipation rate
and  the Sobolev imbedding inequality in one dimension is quite different with that in three dimensions.

The last important ingredient of the proof is  based on a scaling transformation of the independent variables takes the form of $y=\varepsilon^{-a}x$, $\tau=\varepsilon^{-a}t$ involving with a free parameter $a$, which is related to the parameter $b$ in the Debye length;  see the detailed explanations of this point in \eqref{3.2a}. Under this scaling transformation, the obtained convergence rate in \eqref{thm.rate} seems optimal, but the calculation is more complicated due to two free parameters.
In particular, one has to deal with some difficulties caused by the higher order derivatives estimates such as \eqref{4.89A}.
Because of $\varepsilon^{a-1}$ factor in front of the fluid part $\|\partial^{\alpha}(\widetilde{\rho},\widetilde{u},\widetilde{\theta})\|^{2}$ in \eqref{4.89A}, one has to multiply the estimate \eqref{4.89A} by $\varepsilon^{2-2a}$
so that the fluid term can be controlled by $\mathcal{D}_{2,l,q}(\tau)$  in \eqref{3.18}. Hence, this motives how we include the Knudsen number $\varepsilon$ to the highest-order derivative $|\alpha|=2$ for the energy functional $\mathcal{E}_{2,l,q}(\tau)$ in \eqref{3.17}.
In the end, the desired goal is to obtain the uniform a priori estimate \eqref{2.5} and derive the convergence rate in \eqref{thm.rate}.

\section{A priori assumption}\label{sec.3}
In this section, we reformulate the VPL system and make the a priori assumption for the solutions 
under the scaling transformation \eqref{3.1}.

\subsection{Reformulation of equations}
To prove Theorem \ref{thm2.1}, recall that we are concerned with 
\begin{equation}
	\label{3.5}
	\left\{
	\begin{aligned}
		F_{\tau}+v_{1}F_{y}-\phi_{y}\partial_{v_{1}}F&=\varepsilon^{a-1}Q(F,F),
		\\
		-\varepsilon^{2b-2a}\phi_{yy}&=\rho-\rho_{\mathrm{e}}(\phi), \quad \rho=\int_{\mathbb{R}^{3}}F\,dv,
	\end{aligned} \right.
\end{equation}
with 
\begin{equation}
	\label{3.7b}
	\lim_{y\to\pm\infty}F_{0}(y,v)=M_{[\rho_{\pm},u_{\pm},\theta_{\pm}]}(v), \quad \lim_{y\to\pm\infty}\phi(t,y)=\phi_{\pm}.
\end{equation}

Next we derive the equations of $f$ in \eqref{3.3} and $(\widetilde{\rho},\widetilde{u},\widetilde{\theta},\widetilde{\phi})$ in \eqref{3.2}.
From \eqref{1.15} and \eqref{3.1}, one has
\begin{equation}
	\label{3.6}
	G_{\tau}+P_{1}(v_{1}G_{y})+P_{1}(v_{1}M_{y})-\phi_{y}\partial_{v_{1}}G=\varepsilon^{a-1}L_{M}G+\varepsilon^{a-1}Q(G,G).
\end{equation}
Define
\begin{equation}
	\label{3.8}
	\Gamma(h,g):=\frac{1}{\sqrt{\mu}}Q(\sqrt{\mu}h,\sqrt{\mu}g),
	\quad \mathcal{L}h:=\Gamma(\sqrt{\mu},h)+\Gamma(h,\sqrt{\mu}).
\end{equation}
Then we obtain the equation for $f$ in terms of $G=\overline{G}+\sqrt{\mu}f$ and \eqref{3.6} as
\begin{eqnarray}
	\label{3.7}
	&&	f_{\tau}+v_{1}f_{y}+\frac{v_{1}}{2}\phi_{y}f-\phi_{y}\partial_{v_{1}}f
	-\varepsilon^{a-1}\mathcal{L}f
	\notag\\
	=&&\varepsilon^{a-1}\{\Gamma(f,\frac{M-\mu}{\sqrt{\mu}})+
	\Gamma(\frac{M-\mu}{\sqrt{\mu}},f)+\Gamma(\frac{G}{\sqrt{\mu}},\frac{G}{\sqrt{\mu}})\}+\frac{P_{0}(v_{1}\sqrt{\mu}f_{y})}{\sqrt{\mu}}
	\notag\\
	&&-\frac{1}{\sqrt{\mu}}P_{1}\{v_{1}M(\frac{|v-u|^{2}
		\widetilde{\theta}_{y}}{2R\theta^{2}}+\frac{(v-u)\cdot\widetilde{u}_{y}}{R\theta})\}\notag\\
&&	+\frac{\phi_{y}\partial_{v_{1}}\overline{G}}{\sqrt{\mu}}-\frac{P_{1}(v_{1}\overline{G}_{y})}{\sqrt{\mu}}
	-\frac{\overline{G}_{\tau}}{\sqrt{\mu}}.
\end{eqnarray}
Here the following two identities have been used:
\begin{equation*}
	P_{1}(v_{1}M_{y})=P_{1}\{v_{1}M(\frac{|v-u|^{2}
		\widetilde{\theta}_{y}}{2R\theta^{2}}+\frac{(v-u)\cdot\widetilde{u}_{y}}{R\theta})\}
	+\varepsilon^{a-1}L_{M}\overline{G},
\end{equation*}
and
\begin{equation*}
	\frac{1}{\sqrt{\mu}}L_{M}(\sqrt{\mu}f)
	=\mathcal{L}f+\Gamma(f,\frac{M-\mu}{\sqrt{\mu}})+
	\Gamma(\frac{M-\mu}{\sqrt{\mu}},f).
\end{equation*}
For the linearized operator  $\mathcal{L}$ in \eqref{3.8}, one has the following standard facts \cite{Guo-2002}. First, 
$\mathcal{L}$ is self-adjoint and non-positive, and its null space $\ker\mathcal{L}$ is spanned by the basis $\{\sqrt{\mu},v\sqrt{\mu},|v|^{2}\sqrt{\mu}\}$. Moreover, there exists $c_{1}>0$ such that
\begin{equation}
	\label{3.9}
	-\langle\mathcal{L}g, g \rangle\geq c_{1}|g|^{2}_{\sigma}
\end{equation}
for any $g\in (\ker\mathcal{L})^{\perp}$. We would point out that $f\in (\ker\mathcal{L})^{\perp}$ in \eqref{3.7} is purely
microscopic since $G$ and $\overline{G}$ are purely microscopic.

On the other hand, subtracting \eqref{1.28} from \eqref{1.18} and using \eqref{3.1},
we obtain the equations for $(\widetilde{\rho},\widetilde{u},\widetilde{\theta},\widetilde{\phi})$ as 
\begin{equation}
	\label{3.10}
	\left\{
	\begin{aligned}
		&\widetilde{\rho}_{\tau}+(\rho\widetilde{u}_{1})_{y}+\bar{u}_{1}\widetilde{\rho}_{y}+\bar{u}_{1y}\widetilde{\rho}=0,
		\\
		&\widetilde{u}_{1\tau}+\bar{u}_{1}\widetilde{u}_{1y}+\widetilde{u}_{1}u_{1y}
		+\frac{P_{y}}{\rho}-\frac{\bar{P}_{y}}{\bar{\rho}}+\widetilde{\phi}_{y}\\
		&\quad=\varepsilon^{1-a}\frac{4}{3\rho}(\mu(\theta)u_{1y})_{y}-\frac{1}{\rho}(\int_{\mathbb{R}^{3}} v^{2}_{1}L^{-1}_{M}\Theta\,dv)_{y},
		\\
		&\widetilde{u}_{i\tau}+u_{1}\widetilde{u}_{iy}
		=\varepsilon^{1-a}\frac{1}{\rho}(\mu(\theta)u_{iy})_{y}-\frac{1}{\rho}(\int_{\mathbb{R}^3} v_{1}v_{i}L^{-1}_{M}\Theta\,dv)_{y}, ~~i=2,3,
		\\
		&\widetilde{\theta}_{\tau}+\frac{2}{3}\theta\widetilde{u}_{1y}+\frac{2}{3}\widetilde{\theta}\bar{u}_{1y}
		+\theta_{y}\widetilde{u}_{1}+\bar{u}_{1}\widetilde{\theta}_{y}\\
		&\quad=\varepsilon^{1-a}\frac{1}{\rho}(\kappa(\theta)\theta_{y})_{y}
		+\varepsilon^{1-a}\frac{4}{3\rho}\mu(\theta)u^{2}_{1y}
		 +\varepsilon^{1-a}\frac{1}{\rho}\mu(\theta)(u^{2}_{2y}+u^{2}_{3y})\\
		&\hspace{0.9cm}+\frac{1}{\rho}u\cdot(\int_{\mathbb{R}^3} v_{1}v L^{-1}_{M}\Theta\, dv)_{y}
		-\frac{1}{\rho}(\int_{\mathbb{R}^3}v_{1}\frac{|v|^{2}}{2}L^{-1}_{M}\Theta\, dv)_{y},
		\\
		&-\varepsilon^{2b-2a}\phi_{yy}=\widetilde{\rho}+\rho_{\mathrm{e}}(\bar{\phi})-\rho_{\mathrm{e}}(\phi),
	\end{aligned} \right.
\end{equation}
where $\Theta$ is given by
\begin{equation}
	\label{3.11}
	\Theta:=\varepsilon^{1-a}G_{\tau}+\varepsilon^{1-a}P_{1}(v_{1}G_{y})-\varepsilon^{1-a}\phi_{y}\partial_{v_{1}}G-Q(G,G).
\end{equation}
Similarly, from \eqref{1.28}, \eqref{1.14} and \eqref{3.1}, we get
\begin{equation}
	\label{4.31}
	\left\{
	\begin{aligned}
		&\widetilde{\rho}_{\tau}+(\rho\widetilde{u}_{1})_{y}+\bar{u}_{1}\widetilde{\rho}_{y}+\bar{u}_{1y}\widetilde{\rho}=0,
		\\
		&\widetilde{u}_{1\tau}+\widetilde{u}_{1}u_{1y}+\bar{u}_{1}\widetilde{u}_{1y}\\
		&\qquad\qquad+\frac{2}{3}\widetilde{\theta}_{y}
		+\frac{2}{3}\rho_{y}\frac{\bar{\rho}\widetilde{\theta}-\widetilde{\rho}\bar{\theta}}
		{\rho\bar{\rho}}+\frac{2\bar{\theta}}{3\bar{\rho}}\widetilde{\rho}_{y}+\widetilde{\phi}_{y}
		=-\frac{1}{\rho}\int_{\mathbb{R}^3}  v^{2}_{1}G_{y}\,dv,
		\\
		&\widetilde{u}_{i\tau}+\widetilde{u}_{1}\widetilde{u}_{iy}+\bar{u}_{1}\widetilde{u}_{iy}
		=-\frac{1}{\rho}\int_{\mathbb{R}^3}  v_{1}v_{i}G_{y}\,dv, ~~i=2,3,
		\\
		&\widetilde{\theta}_{\tau}+\frac{2}{3}\theta\widetilde{u}_{1y}+\frac{2}{3}\widetilde{\theta}\bar{u}_{1y}
		+\theta_{y}\widetilde{u}_{1}+\bar{u}_{1}\widetilde{\theta}_{y}=-\frac{1}{\rho}\int_{\mathbb{R}^3} \frac{1}{2}v_{1}v\cdot(v-2u)G_{y}\,dv.
	\end{aligned} \right.
\end{equation}

\subsection{A priori assumption}\label{sec.3.4}

The local existence of solutions for the Cauchy problem \eqref{3.5} and \eqref{3.7b} of the VPL system can be proved by following the same strategy as in \cite{Guo-JAMS} and details of the proof are omitted for brevity.
To obtain the global-in-time existence of solutions, we mainly focus on establishing the uniform a priori estimates of solutions.
For this, we first choose $\delta$ in \eqref{1.25} as
\begin{equation}
	\label{3.21}
	\delta=\frac{1}{k}\varepsilon^{\frac{3}{5}-\frac{2}{5}a},
\end{equation}
for a small constant $k>0$ as in \eqref{3.20} independent of $\varepsilon$, and we further let $\varepsilon>0$ 
be arbitrarily chosen such that $k\gg \varepsilon$ and $0<\delta<\delta_{0}$ for a small constant $\delta_{0}>0$. Then, for
an arbitrary time $\tau_{1}\in(0,+\infty)$, 
we make the a priori assumption:
\begin{equation}
	\label{3.22}
	\sup_{0\leq\tau\leq\tau_{1}}\mathcal{E}_{2,l,q}(\tau)+\int^{\tau_{1}}_{0}\mathcal{D}_{2,l,q}(\tau)\,d\tau
	\leq k^{\frac{1}{6}}\varepsilon^{\frac{6}{5}-\frac{4}{5}a}.
\end{equation}
Here $\mathcal{E}_{2,l,q}(\tau)$ and $\mathcal{D}_{2,l,q}(\tau)$ are given by \eqref{3.17} and \eqref{3.18} respectively, and $a$ satisfies \eqref{3.2a}.

With \eqref{3.21}, \eqref{3.22}, \eqref{3.13}, \eqref{3.14}, Lemma \ref{lem7.2} and the assumptions in Theorem \ref{thm2.1} in hand,
for any $\tau\in(0,\tau_{1}]$ and given a constant $C_{1}>0$ with $C_1 k^{\frac{1}{24}}<1$, we can claim that
\begin{equation}\label{3.23}
\left\{\begin{aligned}
&\int_0^{\tau}q_3(s)\,ds\leq C_1k^{\frac{1}{6}},\\ &\varepsilon^{1-a}q_{2}q_{3}(\tau)+q_2\int_0^\tau q_3(s)\,ds\leq 2C_1k^{\frac{1}{12}},\\
&\sup_{0\leq \tau\leq \tau_{1}}q(\tau)\in(0,q_{1}).
\end{aligned}\right.
\end{equation}
	Here we have used \eqref{4.56} such that
	$$
	\|\widetilde{\phi}_\tau\|^{2}_{L_{y}^{\infty}}\leq C\|(\widetilde{\rho}_{y},\widetilde{u}_{y})\|^2+C\|(\widetilde{\rho}_{yy},\widetilde{u}_{yy})\|^2
	+C\varepsilon^{\frac{7}{5}+\frac{1}{15}a}(\delta+\varepsilon^{a}\tau)^{-\frac{4}{3}}.
	$$
Using the following 1D embedding inequality
\begin{equation*}
	\|g\|_{L^{\infty}(\mathbb{R})}\leq \sqrt{2}\|g\|^{\frac{1}{2}}\|g'\|^{\frac{1}{2}},
	\quad \mbox{for any}\  g=g(y)\in H^{1}(\mathbb{R}),
\end{equation*}
and \eqref{3.22} as well as  \eqref{2.2}, it holds uniformly in all $(t,x)$ that
\begin{equation}\label{3.24}
\left\{\begin{aligned}
&	|\rho(t,x)-1|+|u(t,x)|+|\theta(t,x)-\frac{3}{2}|<C(\eta_{0}+k^{\frac{1}{12}}\varepsilon^{\frac{3}{5}-\frac{2}{5}a}),\\
&	1<\theta(t,x)<3.
\end{aligned}\right.
\end{equation}

In what follows we formally explain why one has to choose such $\delta$ as \eqref{3.21} and the a priori assumption as \eqref{3.22}. Indeed, if one assumes that
$$
\sup_{0\leq\tau\leq\tau_{1}}\mathcal{E}_{2,l,q}(\tau)+\int^{\tau_{1}}_{0}\mathcal{D}_{2,l,q}(\tau)\,d\tau\leq O(1)\varepsilon^{p}
$$
with a constant $p\geq0$, then this together with the 1D embedding inequality and Lemma \ref{lem7.1} gives rise to
\begin{eqnarray*}
	&&\|(\rho,u,\theta,\phi)(t,x)-(\rho^{R},u^{R},\theta^{R},\phi^{R})(\frac{x}{t})\|_{L^{\infty}(\mathbb{R})}
	\notag\\
	&&\leq\|(\widetilde{\rho},\widetilde{u},\widetilde{\theta},\widetilde{\phi})(\tau,y)\|_{L^{\infty}(\mathbb{R})}
	+\|(\bar{\rho},\bar{u},\bar{\theta},\bar{\phi})(t,x)-(\rho^{R},u^{R},\theta^{R},\phi^{R})(\frac{x}{t})\|_{L^{\infty}(\mathbb{R})}
	\notag\\
	&&\leq C\varepsilon^{\frac{p}{2}}+Ct^{-1}\delta(\ln(1+t)+|\ln\delta|),
\end{eqnarray*}
for any $t>0$. Choosing $\delta=O(1)\varepsilon^{\frac{p}{2}}$, we can see that
the above estimate in the vanishing Knudsen number $\varepsilon>0$ is optimal.
Another important ingredient of the proof is to deal with the slow time decay of the term in \eqref{4.14A} in the way that
\begin{eqnarray*}
	\varepsilon^{1-a}\int^{\infty}_{0}\|\widetilde{\theta}\|^{\frac{2}{3}}\|\bar{\theta}_{yy}\|^{\frac{4}{3}}_{L^{1}}\,d\tau
	&&\leq C\varepsilon^{1-a}\int^{\infty}_{0}\varepsilon^{\frac{p}{3}}\varepsilon^{\frac{4}{3}a}(\delta+\varepsilon^{a}\tau)^{-\frac{4}{3}}\,d\tau
	\notag\\
	&&\leq C\varepsilon^{1-a+\frac{1}{3}p+\frac{1}{3}a}\delta^{-\frac{1}{3}}=O(1)\varepsilon^{1-\frac{2}{3}a+\frac{1}{6}p},
\end{eqnarray*}
where in the last identity we have used $\delta=O(1)\varepsilon^{\frac{p}{2}}$.
To close the a priori assumption, one has to require that
\begin{equation*}
	\varepsilon^{1-\frac{2}{3}a+\frac{1}{6}p}\leq \varepsilon^{p}, \quad \mbox{that~~~is}\quad p\leq\frac{6}{5}-\frac{4}{5}a.
\end{equation*}
Note that the convergence rate is fastest by taking $p=\frac{6}{5}-\frac{4}{5}a$. We thereby
obtain the sharp convergence rate under the condition of \eqref{3.22} and \eqref{3.21}.

\begin{remark}\label{rmk.ea}
Given the optimal choice of $p=\frac{6}{5}-\frac{4}{5}a$, we further explain that $a=\frac{2}{3}$ is sharp for closing the energy estimate. In fact, recall that the Knudsen number for the scaled problem \eqref{3.5} is given as $\varepsilon^{1-a}$. Hence, the macroscopic dissipation based on the Navier-Stokes equations much take the form of
$$
\varepsilon^{1-a}\sum_{1\leq|\alpha|\leq 2}\|\partial^{\alpha}(\widetilde{\rho},\widetilde{u},\widetilde{\theta})(\tau)\|^{2}
$$ 
as in \eqref{3.18}. In the  derivative estimate, one need to treat bounds of some cubic nonlinear terms, for instance, see \eqref{4.54a}. Therefore, it is necessary to require that 
$
\varepsilon^{\frac{p}{2}-(1-a)}
$
is bounded for any small $\varepsilon>0$, which then implies $\frac{p}{2}-(1-a)\geq 0$ and hence $a\geq \frac{2}{3}$.  
\end{remark}

\section{Basic estimates}\label{sec.4}
In this section, we give some  necessary estimates that will be used in the later energy analysis. We emphasize that in all estimates below, the generic constant $C$ at different places is independent of the
time $\tau_1$ as well as all small parameters $\varepsilon$, $k$, $\delta$ and $\eta_{0}$. 
\subsection{Properties of rarefaction wave}
We first list the properties of the smooth approximate 3-rarefaction wave $(\bar{\rho},\bar{u},\bar{\theta},\bar{\phi})(t,x)$ constructed in \eqref{1.27}.
\begin{lemma}\label{lem7.1}
	Let $\delta_r=|\rho_{+}-\rho_{-}|+|u_{+}-u_{-}|+|\theta_{+}-\theta_{-}|$ be the wave strength of the approximate rarefaction wave
	$(\bar{\rho},\bar{u},\bar{\theta},\bar{\phi})(t,x)$ constructed in \eqref{1.27}, then $(\bar{\rho},\bar{u},\bar{\theta},\bar{\phi})(t,x)$ is smooth in $t$ and $x$, and the following  holds true:
	\\
	\mbox{(i)}~~$\bar{u}_{2}=\bar{u}_{3}=0$,\quad $\bar{u}_{1x}>0$ and $\rho_{-}<\bar{\rho}(t,x)<\rho_{+}$,\quad
	$u_{1-}<\bar{u}_{1}(t,x)<u_{1+}$,
	and
	$$
	\bar{\theta}_{x}=\frac{1}{2\pi e}e^{S_{*}}\bar{\rho}^{-\frac{1}{3}}\bar{\rho}_{x}
	=\frac{1}{2\pi e}e^{S_{*}}\bar{\rho}^{-\frac{1}{3}}\frac{1}{\sqrt{\frac{5}{3}\frac{1}{2\pi e}
			\bar{\rho}^{-\frac{4}{3}}e^{S_{*}}
			+\bar{\rho}^{-1}\big(\frac{d}{d\rho}\rho^{-1}_{\mathrm{e}}(\bar{\rho})\big)}}
	\bar{u}_{1x},$$
	\\
	for any $x\in\mathbb{R},~t\geq 0$.
	\\
	\mbox{(ii)}~~The following estimates hold for all $t> 0$, $\delta>0$ and $p\in[1,+\infty]$,
	\begin{eqnarray*}
		&&\|\partial_{x}(\bar{\rho},\bar{u}_{1},\bar{\theta},\bar{\phi})(t,x)\|_{L^{p}(\mathbb{R}_{x})}\leq C(\delta+t)^{-1+\frac{1}{p}},
		\notag\\
		&&\|\partial_{x}^{j}(\bar{\rho},\bar{u}_{1},\bar{\theta},\bar{\phi})(t,x)\|_{L^{p}(\mathbb{R}_{x})}\leq C\delta^{-j+1+\frac{1}{p}}(\delta+t)^{-1},\quad j\geq 2.
	\end{eqnarray*}
	\mbox{(iii)}~~For $t>0$ and $\delta\in(0,\delta_{0})$ with $\delta_{0}\in(0,1)$, it holds that
	$$
	\|(\bar{\rho},\bar{u},\bar{\theta},\bar{\phi})(t,x)-(\rho^{R},u^{R},\theta^{R},\phi^{R})(\frac{x}{t})\|_{L^{\infty}(\mathbb{R}_{x})}\leq C\delta t^{-1}\{\ln(1+t)+|\ln\delta|\}.
	$$
\end{lemma}
\begin{proof}
	The proof of Lemma \ref{lem7.1} can be found in \cite{Duan-Liu-2015,Huang-Li-Wang,Xin}.
\end{proof}

Since the transformation $y=\varepsilon^{-a}x$ and $\tau=\varepsilon^{-a}t$ is
considered through the proof, we give the following lemma, which is equivalent to Lemma \ref{lem7.1} (ii).
\begin{lemma}\label{lem7.2}
	For any $\tau> 0$, $\delta>0$ and $p\in[1,+\infty]$, the following estimates hold 
	\begin{eqnarray*}
		&&\|\partial_{y}(\bar{\rho},\bar{u}_{1},\bar{\theta},\bar{\phi})(\tau,y)\|_{L^{p}(\mathbb{R}_{y})}\leq C
		\varepsilon^{a(1-\frac{1}{p})}(\delta+\varepsilon^{a}\tau)^{-1+\frac{1}{p}},
		\notag\\
		&&\|\partial_{y}^{j}(\bar{\rho},\bar{u}_{1},\bar{\theta},\bar{\phi})(\tau,y)\|_{L^{p}(\mathbb{R}_{y})}\leq C\varepsilon^{a(j-\frac{1}{p})}\delta^{-j+1+\frac{1}{p}}(\delta+\varepsilon^{a}\tau)^{-1},\quad j\geq 2.
	\end{eqnarray*}
\end{lemma}
It should be pointed out that  the temporal derivatives of $(\bar{\rho},\bar{u}_{1},\bar{\theta},\bar{\phi})$ in Lemma
\ref{lem7.1} (ii) and Lemma \ref{lem7.2} obviously hold in terms of the quasineutral Euler system \eqref{1.28} and the elementary inequalities.

\subsection{Properties of $\mathcal{L}$ and $\Gamma$}
The following two lemmas are about the weighted  estimates on the linearized Landau  operator $\mathcal{L}$ and the nonlinear collision terms $\Gamma(g_1,g_2)$ defined in \eqref{3.8}. They can be proved by a straightforward modification of the arguments used in \cite[Lemmas 9]{Strain-Guo-2008}
and \cite[Lemmas 2.2-2.3]{Wang}, and we thus omit their proofs for brevity.
\begin{lemma}\label{lem7.5}
	Let $w=w(\alpha,\beta)$ in \eqref{3.12} with $0<q(\tau)\ll 1$.
	For any $\eta>0$ small enough, there exists $C_\eta>0$ such that
	\begin{multline}
		\label{7.5}
		-\langle\partial^\alpha_\beta\mathcal{L}g,w^2(\alpha,\beta)\partial^\alpha_\beta g\rangle
		\geq |w(\alpha,\beta)\partial^\alpha_\beta g|_\sigma^2\\-\eta\sum_{|\beta_1|=|\beta|}|w(\alpha,\beta_1)\partial^\alpha_{\beta_1} g|_\sigma^2
		-C_\eta\sum_{|\beta_1|<|\beta|}|w(\alpha,\beta_1)\partial^\alpha_{\beta_1} g|_\sigma^2.
	\end{multline}
	Let $|\beta| = 0$,  then there exists $c_{2}>0$ such that
	\begin{equation}
		\label{7.6}
		-\langle\partial^\alpha\mathcal{L}g,w^2(\alpha,0)\partial^\alpha g\rangle\geq c_{2}|w(\alpha,0)\partial^\alpha g|_\sigma^2-C_\eta|\chi_{\eta}(v)\partial^\alpha g|_2^2,
	\end{equation}
	where $\chi_\eta(v)$ is a general cutoff function depending on $\eta$.
\end{lemma}
\begin{lemma}\label{lem7.6}
	Let $w=w(\alpha,\beta)$ in \eqref{3.12} with $0<q(\tau)\ll 1$. For any $\epsilon>0$ small enough, one has
	\begin{multline}
		\label{7.8}
		\langle\partial^\alpha_\beta \Gamma(g_1,g_2), w^2(\alpha,\beta)   g_3\rangle\\
		\leq
		C\sum_{\alpha_1\leq\alpha}\sum_{\bar{\beta}\leq\beta_1\leq\beta}|\mu^\epsilon\partial^{\alpha_1}_{\bar{\beta}}g_1|_2|w(\alpha,\beta)  \partial^{\alpha-\alpha_1}_{\beta-\beta_1}g_2|_{\sigma}|w(\alpha,\beta)    g_3|_{\sigma}.
	\end{multline}	
	Moreover, the following estimate holds
	\begin{equation}
		\label{7.7}
		\langle\partial^\alpha \Gamma(g_1,g_2), g_3\rangle\leq C\sum_{\alpha_1\leq\alpha}|\mu^\epsilon\partial^{\alpha_1}g_1|_2
		| \partial^{\alpha-\alpha_1}g_2|_\sigma|  g_3|_\sigma.
	\end{equation}
\end{lemma}
\subsection{Estimates on linear collision terms}
In what follows we derive the estimates of the linear collision terms appearing in \eqref{3.5} and \eqref{3.7}.
\begin{lemma}\label{lem7.7}
	Let $|\alpha|+|\beta|\leq 2$ with $|\beta|\geq1$ and $w=w(\alpha,\beta)$ defined in \eqref{3.12}.
	Suppose that  \eqref{3.21}, \eqref{3.22} and \eqref{3.23} hold. For any $\eta>0$, there exists $C_\eta>0$ such that
	\begin{eqnarray}
		\label{7.9}
		&&\varepsilon^{a-1}(|(\partial^\alpha_\beta \Gamma(\frac{M-\mu}{\sqrt{\mu}},f), e^{\phi}w^2(\alpha,\beta)  h)|
		+|(\partial^\alpha_\beta \Gamma(f,\frac{M-\mu}{\sqrt{\mu}}),e^{\phi}w^2(\alpha,\beta)  h)|)
		\notag\\
		&&\hspace{0.5cm}\leq C\eta\varepsilon^{a-1}\|w(\alpha,\beta)h\|_{\sigma}^2
		+C_{\eta}(\eta_0+k^{\frac{1}{12}}\varepsilon^{\frac{3}{5}-\frac{2}{5}a})\mathcal{D}_{2,l,q}(\tau).
	\end{eqnarray}
If $|\beta|=0$ and $|\alpha|\leq 1$, for any $\eta>0$, one has
	\begin{eqnarray}
		\label{7.10}
		&&\varepsilon^{a-1}(|(\partial^\alpha \Gamma(\frac{M-\mu}{\sqrt{\mu}},f),e^{\phi}w^2(\alpha,0)h)|
		+|(\partial^\alpha \Gamma(f,\frac{M-\mu}{\sqrt{\mu}}),e^{\phi}w^2(\alpha,0)h)|)
		\notag\\
		&&\hspace{0.5cm}\leq C\eta\varepsilon^{a-1}\|w(\alpha,0)h\|^{2}_{\sigma}
		+C_{\eta}(\eta_0+k^{\frac{1}{12}}\varepsilon^{\frac{3}{5}-\frac{2}{5}a})\mathcal{D}_{2,l,q}(\tau).
	\end{eqnarray}
\end{lemma}
\begin{proof}
Using \eqref{3.2a} and the smallness of $k>0$, $\eta_{0}>0$ and $\varepsilon$, we get
\begin{equation}
\label{4.12A}\eta_0+k^{\frac{1}{12}}\varepsilon^{\frac{3}{5}-\frac{2}{5}a}<1.
\end{equation}		
This together with \eqref{3.22}, \eqref{2.2} and $\bar{\phi}=\rho^{-1}_{\mathrm{e}}(\bar{\rho})$  yields that
	\begin{equation}
		\label{4.45a}
		\sup_{t\geq 0,x\in \R}|\phi|\leq \sup_{t\geq 0,x\in \R}|\widetilde{\phi}|+\sup_{t\geq 0,x\in \R}|\bar{\phi}|\leq C,\quad \mbox{and}\quad
		e^{\phi}\approx 1.
	\end{equation}
From \eqref{7.8} and \eqref{4.45a}, it is straightforward to get 
	\begin{eqnarray}
		\label{7.11}
		&&\varepsilon^{a-1}|(\partial^\alpha_\beta \Gamma(\frac{M-\mu}{\sqrt{\mu}},f), e^{\phi}w^2(\alpha,\beta)h)|
		\notag\\
		&&\leq C\varepsilon^{a-1}\sum_{\alpha_1\leq\alpha}\sum_{\bar{\beta}\leq\beta_1\leq\beta}\int_{\mathbb R}|\mu^\epsilon\partial^{\alpha_1}_{\bar{\beta}}(\frac{M-\mu}{\sqrt{\mu}})|_2| w(\alpha,\beta) \partial^{\alpha-\alpha_1}_{\beta-\beta_1}f|_{\sigma}|  w(\alpha,\beta)h|_{\sigma}\,dy
		\notag\\
		&&\leq \eta\varepsilon^{a-1}\|w(\alpha,\beta)h\|_{\sigma}^2+
		C_\eta\sum_{\alpha_1\leq\alpha}\sum_{\bar{\beta}\leq\beta_1\leq\beta}I^{\alpha,\alpha_{1}}_{\beta,\beta_{1}}(M,f).	
	\end{eqnarray}
In this expression we have used
	\begin{equation*}
		I^{\alpha,\alpha_{1}}_{\beta,\beta_{1}}(M,f)=
		\varepsilon^{a-1}\int_{\mathbb R}|\mu^\epsilon\partial^{\alpha_1}_{\bar{\beta}}(\frac{M-\mu}{\sqrt{\mu}})|^2_2| w(\alpha,\beta) \partial^{\alpha-\alpha_1}_{\beta-\beta_1}f|^2_{\sigma}\,dy.
	\end{equation*}
	Thanks to \eqref{3.23}, it holds that $q(\tau)<q_1$ in $w(\alpha,\beta)$ defined by \eqref{3.12}. Then, for any $\beta'\geq0$  and any $m>0$, one has
	\begin{eqnarray*}
		&&| \langle v\rangle^{m} w(\alpha,\beta)\partial _{\beta'}(\frac{M-\mu}{\sqrt{\mu}})|_{\sigma}^2+| \langle v\rangle^{m}
		w(\alpha,\beta)\partial _{\beta'}(\frac{M-\mu}{\sqrt{\mu}})|_{2}^2
		\notag\\
		&&\quad\leq C_m\sum_{|\beta'|\leq|\beta''|\leq|\beta'|+1}\int_{{\mathbb R}^3}\mu^{-3q_1}
		|\partial _{\beta''}(\frac{M-\mu}{\sqrt{\mu}})|^2\,dv.
	\end{eqnarray*}
By separating the velocity into two parts $|v|\geq R$ and $|v|\leq R$ with a large constant $R>0$, then using
\eqref{3.24} and choosing a small constant $q_{1}>0$, we deduce that
	$$
	\int_{|v|\geq R}\mu^{-3q_1}|\partial _{\beta''}(\frac{M-\mu}{\sqrt{\mu}})|^2\,dv\leq C(\eta_0+k^{\frac{1}{12}}\varepsilon^{\frac{3}{5}-\frac{2}{5}a})^{2},
	$$
	and
	$$
	\int_{|v|\leq R}\mu^{-3q_1}|\partial _{\beta''}(\frac{M-\mu}{\sqrt{\mu}})|^2\,dv\leq C(|\rho-1|+|u|+|\theta-\frac{3}{2}|)^2\leq C
	(\eta_0+k^{\frac{1}{12}}\varepsilon^{\frac{3}{5}-\frac{2}{5}a})^{2}.
	$$
So, from the above bound, for any $\beta'\geq0$ and $m>0$, we obtain
\begin{equation}
		\label{7.12}
		| \langle v\rangle^{m}w(\alpha,\beta)\partial _{\beta'}(\frac{M-\mu}{\sqrt{\mu}})|_{\sigma}+| \langle v\rangle^{m}w(\alpha,\beta)
		\partial _{\beta'}(\frac{M-\mu}{\sqrt{\mu}})|_{2}\leq C(\eta_0+k^{\frac{1}{12}}\varepsilon^{\frac{3}{5}-\frac{2}{5}a}).
	\end{equation}
Similarly, the following estimate holds for any $m>0$ and sufficiently small $\epsilon_1>0$, 
	\begin{equation}
		\label{7.25aa}
		|\langle v\rangle^{m}w(\alpha,\beta)\mu^{-\frac{1}{2}}M^{1-\epsilon_1}|_{2}+|\langle v\rangle^{m}w(\alpha,\beta)\mu^{-\frac{1}{2}}M^{1-\epsilon_1}|_{\sigma}\leq C.
	\end{equation}

We next turn to compute carefully the right hand terms of \eqref{7.11}. Since  we only consider $|\alpha|+|\beta|\leq 2$ and $|\beta|\geq1$, one has $|\alpha|\leq1$ in $I^{\alpha,\alpha_{1}}_{\beta,\beta_{1}}(M,f)$.
	If $|\alpha_{1}|=0$, then $w(\alpha,\beta)\leq w(\alpha-\alpha_{1},\beta-\beta_1)$, and we use this together with \eqref{7.12}, \eqref{3.18} and \eqref{4.12A} to obtain
	\begin{eqnarray}
		\label{7.13}
		I^{\alpha,\alpha_{1}}_{\beta,\beta_{1}}(M,f)&&\leq C(\eta_0+k^{\frac{1}{12}}\varepsilon^{\frac{3}{5}-\frac{2}{5}a})^2\varepsilon^{a-1}\|w(\alpha,\beta)\partial^{\alpha-\alpha_{1}}_{\beta-\beta_{1}}f\|^2_{\sigma}
		\notag\\
		&&\leq C(\eta_0+k^{\frac{1}{12}}\varepsilon^{\frac{3}{5}-\frac{2}{5}a})\mathcal{D}_{2,l,q}(\tau).
	\end{eqnarray}
If $|\alpha_{1}|=|\alpha|=1$,  then $|\alpha-\alpha_1|+|\beta-\beta_{1}|\leq 1$ and $w(\alpha,\beta)\leq w(\alpha-\alpha_1+\alpha_2,\beta-\beta_1)$ for $|\alpha_2|\leq1$.
So, from these facts, the 1D embedding inequality,  \eqref{7.25aa} and \eqref{3.18}, we get
\begin{eqnarray}
	\label{7.14}
	I^{\alpha,\alpha_{1}}_{\beta,\beta_{1}}(M,f)&&\leq C\varepsilon^{a-1}\|\partial^{\alpha_{1}}(\rho,u,\theta)\|^{2}
	\||w(\alpha,\beta)\partial^{\alpha-\alpha_1}_{\beta-\beta_{1}}f|^2_{\sigma}\|_{L_y^{\infty}}
	\notag\\
	&&\leq 
	Ck^{\frac{1}{12}}\varepsilon^{\frac{3}{5}-\frac{2}{5}a}\varepsilon^{a-1}
	\|\partial^{\alpha-\alpha_1}_{\beta-\beta_{1}}f\|_{\sigma,w}
	\|\partial^{\alpha-\alpha_1}_{\beta-\beta_{1}}f_{y}\|_{\sigma,w}
	\notag\\
	&&\leq Ck^{\frac{1}{12}}\varepsilon^{\frac{3}{5}-\frac{2}{5}a}\mathcal{D}_{2,l,q}(\tau),
\end{eqnarray}	
where in the second inequality we have used	Lemma \ref{lem7.2}, \eqref{3.22} and $\delta=\frac{1}{k}\varepsilon^{\frac{3}{5}-\frac{2}{5}a}$ in \eqref{3.21} such that
	\begin{eqnarray}
		\label{4.13A}
		\|(\rho_y,u_y,\theta_y)\|^{2}&&\leq 2\|(\bar{\rho}_y,\bar{u}_y,\bar{\theta}_y)\|^{2}
		+2\|\partial^{\alpha_{1}}(\widetilde{\rho}_y,\widetilde{u}_y,\widetilde{\theta}_y)\|^{2}
		\notag\\
		&&\leq C\varepsilon^{a}\delta^{-1}+Ck^{\frac{1}{6}}\varepsilon^{\frac{6}{5}-\frac{4}{5}a}
		\notag\\
		&&\leq  Ck\varepsilon^{a}\varepsilon^{-\frac{3}{5}+\frac{2}{5}a}+Ck^{\frac{1}{12}}\varepsilon^{\frac{3}{5}-\frac{2}{5}a}
		\leq Ck^{\frac{1}{12}}\varepsilon^{\frac{3}{5}-\frac{2}{5}a}.
	\end{eqnarray}
Consequently, plugging \eqref{7.14} and \eqref{7.13} into \eqref{7.11} gives
	\begin{multline}
		\label{7.16}
		\varepsilon^{a-1}|(\partial^\alpha_\beta \Gamma(\frac{M-\mu}{\sqrt{\mu}},f), e^{\phi}w^2(\alpha,\beta)h)|\\
		\leq \eta\varepsilon^{a-1}\|w(\alpha,\beta)h\|^{2}_{\sigma}
		+C_{\eta}(\eta_0+k^{\frac{1}{12}}\varepsilon^{\frac{3}{5}-\frac{2}{5}a})\mathcal{D}_{2,l,q}(\tau).
	\end{multline}
	For the second term on the left hand side of \eqref{7.9}, thanks to \eqref{7.8} and  \eqref{4.45a}, it follows that
	\begin{eqnarray}
		\label{4.14B}
		&&\varepsilon^{a-1}|(\partial^\alpha_\beta \Gamma(f,\frac{M-\mu}{\sqrt{\mu}}), e^{\phi}w^2(\alpha,\beta)h)|
		\notag\\
		&&\leq C\varepsilon^{a-1}\sum_{\alpha_1\leq\alpha}\sum_{\bar{\beta}\leq\beta_1\leq\beta}\int_{\mathbb R}|\mu^\epsilon\partial^{\alpha_1}_{\bar{\beta}}f|_2| w(\alpha,\beta) \partial^{\alpha-\alpha_1}_{\beta-\beta_1}(\frac{M-\mu}{\sqrt{\mu}})|_{\sigma}|  w(\alpha,\beta)h|_{\sigma}\,dy
		\notag\\
		&&\leq \eta\varepsilon^{a-1}\|w(\alpha,\beta)h\|_{\sigma}^2+
		C_\eta\sum_{\alpha_1\leq\alpha}\sum_{\bar{\beta}\leq\beta_1\leq\beta}I^{\alpha,\alpha_{1}}_{\beta,\beta_{1}}(f,M).	
	\end{eqnarray}
In this expression we have used
	\begin{equation*}
		I^{\alpha,\alpha_{1}}_{\beta,\beta_{1}}(f,M)=
		\varepsilon^{a-1}\int_{\mathbb R}|\mu^\epsilon\partial^{\alpha_1}_{\bar{\beta}}f|^2_2| w(\alpha,\beta) \partial^{\alpha-\alpha_1}_{\beta-\beta_1}(\frac{M-\mu}{\sqrt{\mu}})|^2_{\sigma}\,dy.
	\end{equation*}
	Performing the similar calculations as \eqref{7.13} and \eqref{7.14}, one gets
	\begin{equation*}
		I^{\alpha,\alpha_{1}}_{\beta,\beta_{1}}(f,M)\leq C(\eta_{0}+k^{\frac{1}{12}}\varepsilon^{\frac{3}{5}-\frac{2}{5}a})\mathcal{D}_{2,l,q}(\tau).
	\end{equation*}
	Hence, we plug this into \eqref{4.14B} to obtain
	\begin{multline*}
		\varepsilon^{a-1}|(\partial^\alpha_\beta \Gamma(f,\frac{M-\mu}{\sqrt{\mu}}), e^{\phi}w^2(\alpha,\beta)h)|\\
		\leq C\eta\varepsilon^{a-1}\|w(\alpha,\beta)h\|^{2}_{\sigma}
		+C_{\eta}(\eta_{0}+k^{\frac{1}{12}}\varepsilon^{\frac{3}{5}-\frac{2}{5}a})\mathcal{D}_{2,l,q}(\tau).
	\end{multline*}
	This, combined with \eqref{7.16}, immediately gives   the desired estimate \eqref{7.9}.
	The proof of \eqref{7.10} follows along the same lines as the proof of \eqref{7.9}, and we omit its proof for
	brevity. Thus, Lemma \ref{lem7.7} is proved.
\end{proof}

\subsection{Estimates on nonlinear collision terms}
We are now in a position to compute the nonlinear collision terms
appearing in \eqref{3.5} and \eqref{3.7}.
Note that $\overline{G}$ is contained in the nonlinear collision terms and it can be represented 
precisely by using the Burnett functions as in \eqref{7.1} and \eqref{7.2}. For this, we first give
the following lemma about the fast velocity decay of the Burnett functions and its proof can be found in \cite[Lemma 6.1]{Duan-Yu1}.

\begin{lemma}
Let $L_{\widehat{M}}$ be defined by \eqref{1.8} for any Maxwellian $\widehat{M}=M_{[\widehat{\rho},\widehat{u},\widehat{\theta}]}(v)$  and its null space be denoted by $\ker{L_{\widehat{M}}}$. Suppose that $U(v)$ is any polynomial of $\frac{v-\hat{u}}{\sqrt{R}\hat{\theta}}$ such that
$U(v)\widehat{M}\in(\ker{L_{\widehat{M}}})^{\perp}$. Then, for any $\epsilon_2\in(0,1)$ and any multi-index $\beta$, there exists a constant $C_{\beta}>0$ such that
	$$
	|\partial_{\beta}L^{-1}_{\widehat{M}}(U(v)\widehat{M})|\leq C_{\beta}(\widehat{\rho},\widehat{u},\widehat{\theta})\widehat{M}^{1-\epsilon_2}.
	$$
	Moreover, under the condition \eqref{3.24}, there exists $C_{\beta}>0$ such that
	\begin{equation}
		\label{7.4}
		|\partial_{\beta}A_{j}(\frac{v-u}{\sqrt{R\theta}})|+|\partial_{\beta}B_{ij}(\frac{v-u}{\sqrt{R\theta}})|
		\leq C_{\beta}M^{1-\epsilon_2},
	\end{equation}
	where $A_{j}(\cdot)$ and $B_{ij}(\cdot)$ are defined in \eqref{7.2}.
\end{lemma}
Then we consider the nonlinear collision terms.
\begin{lemma}
	\label{lem7.8}
	Let $|\alpha|+|\beta|\leq 2$ with $|\beta|\geq1$ and $w=w(\alpha,\beta)$ defined in \eqref{3.12}.
	Suppose that  \eqref{3.21}, \eqref{3.22} and \eqref{3.23} hold. For any $\eta>0$, there exists $C_\eta>0$ such that
	\begin{multline}
		\label{7.18}
		\varepsilon^{a-1}|(\partial^\alpha_\beta\Gamma(\frac{G}{\sqrt{\mu}},\frac{G}{\sqrt{\mu}}),e^{\phi} w^2(\alpha,\beta) h)|
		\leq  C\eta\varepsilon^{a-1}\|w(\alpha,\beta)h\|^{2}_{\sigma}\\
		+C_{\eta}k^{\frac{1}{12}}\varepsilon^{\frac{3}{5}-\frac{2}{5}a}\mathcal{D}_{2,l,q}(\tau)
		+C_{\eta}\varepsilon^{\frac{7}{5}+\frac{1}{15}a}(\delta+\varepsilon^{a}\tau)^{-\frac{4}{3}}.
	\end{multline}
If $|\beta|=0$ and $|\alpha|\leq 1$, then for any $\eta>0$, it holds that
	\begin{multline}
		\label{7.19}
		\varepsilon^{a-1}|(\partial^\alpha \Gamma(\frac{G}{\sqrt{\mu}},\frac{G}{\sqrt{\mu}}), e^{\phi}w^2(\alpha,0) h)|
	\leq  C\eta\varepsilon^{a-1}\|w(\alpha,0)h\|^{2}_{\sigma}\\
		+C_{\eta}k^{\frac{1}{12}}\varepsilon^{\frac{3}{5}-\frac{2}{5}a}\mathcal{D}_{2,l,q}(\tau)
		+C_{\eta}\varepsilon^{\frac{7}{5}+\frac{1}{15}a}(\delta+\varepsilon^{a}\tau)^{-\frac{4}{3}}.
	\end{multline}
\end{lemma}
\begin{proof}
	Let $|\alpha|+|\beta|\leq 2$ and $|\beta|\geq1$, then we deduce from $G=\overline{G}+\sqrt{\mu}f$ that
	\begin{equation}
		\label{7.20}
		\partial^\alpha_\beta\Gamma(\frac{G}{\sqrt{\mu}},\frac{G}{\sqrt{\mu}})=
		\partial^\alpha_\beta\Gamma(f,f)
		+\partial^\alpha_\beta\Gamma(\frac{\overline{G}}{\sqrt{\mu}},f)+\partial^\alpha_\beta\Gamma(f,\frac{\overline{G}}{\sqrt{\mu}})
		+\partial^\alpha_\beta\Gamma(\frac{\overline{G}}{\sqrt{\mu}},\frac{\overline{G}}{\sqrt{\mu}}).
	\end{equation}
	We take the inner product of \eqref{7.20} with $\varepsilon^{a-1}e^{\phi} w^2(\alpha,\beta) h$
	and then compute each term.  Using \eqref{7.8} and \eqref{4.45a}, one obtains that
	\begin{eqnarray*}
		&&\varepsilon^{a-1}|(\partial^\alpha_\beta \Gamma(f,f),e^{\phi}w^2(\alpha,\beta)h)|
		\notag\\
		&&\leq C\varepsilon^{a-1}\sum_{\alpha_1\leq\alpha}\sum_{\bar{\beta}\leq\beta_1\leq\beta}\int_{\mathbb R}|\mu^\epsilon\partial^{\alpha_1}_{\bar{\beta}}f|_2|w(\alpha,\beta) \partial^{\alpha-\alpha_1}_{\beta-\beta_1}f|_{\sigma}| w(\alpha,\beta) h|_{\sigma}\,dy
		\notag\\
		&&\leq \eta\varepsilon^{a-1}\|w(\alpha,\beta)h\|^{2}_{\sigma}
		+C_{\eta}\sum_{\alpha_1\leq\alpha}\sum_{\bar{\beta}\leq\beta_1\leq\beta}I^{\alpha,\alpha_{1}}_{\beta,\beta_{1}}(f,f),
	\end{eqnarray*}	
where we have denoted
	$$
	I^{\alpha,\alpha_{1}}_{\beta,\beta_{1}}(f,f)=\varepsilon^{a-1}\int_{\mathbb R}|\mu^\epsilon\partial^{\alpha_1}_{\bar{\beta}}f|^2_2|w(\alpha,\beta) \partial^{\alpha-\alpha_1}_{\beta-\beta_1}f|^2_{\sigma}\,dy.
	$$	
	Since we consider $|\alpha|+|\beta|\leq 2$ with $|\beta|\geq1$, then $|\alpha|\leq1$.	
	If $|\alpha-\alpha_1|+|\beta-\beta_1|=0$,  it follows from \eqref{3.22}, \eqref{3.18} and \eqref{4.12A}	that
	\begin{eqnarray}
		\label{4.21A}
		I^{\alpha,\alpha_{1}}_{\beta,\beta_{1}}(f,f)
		&&\leq \varepsilon^{a-1}\|\mu^\epsilon\partial^{\alpha_1}_{\bar{\beta}}f\|^2_2\||w(\alpha,\beta) \partial^{\alpha-\alpha_1}_{\beta-\beta_1}f|^2_{\sigma}\|_{L_y^{\infty}}
		\notag\\
		&&\leq Ck^{\frac{1}{6}}\varepsilon^{\frac{6}{5}-\frac{4}{5}a}\varepsilon^{a-1}\|w(\alpha,\beta) \partial^{\alpha-\alpha_1}_{\beta-\beta_1}f\|_{\sigma}\|w(\alpha,\beta) \partial^{\alpha-\alpha_1}_{\beta-\beta_1}f_y\|_{\sigma}
		\notag\\
		&&\leq Ck^{\frac{1}{12}}\varepsilon^{\frac{3}{5}-\frac{2}{5}a}\mathcal{D}_{2,l,q}(\tau).
	\end{eqnarray}	
	Here we have used the fact that $w(\alpha,\beta)\leq w(\alpha_2,0)$ when $|\alpha_2|\leq 1$ and $|\beta|\geq 1$.
	
	If $|\alpha-\alpha_1|+|\beta-\beta_1|=1$, then $|\alpha_1|+|\beta_1|\leq1$. If $|\alpha_1|+|\beta_1|=0$, one has $|\alpha|+|\beta|=1$ and
	$w(\alpha,\beta)=w(\alpha-\alpha_{1},\beta-\beta_{1})$, and hence a simple computation derives
	\begin{eqnarray}
		\label{4.22A}
		I^{\alpha,\alpha_{1}}_{\beta,\beta_{1}}(f,f)
		&&\leq \varepsilon^{a-1}\||\mu^\epsilon\partial^{\alpha_1}_{\bar{\beta}}f|^2_2\|_{L_y^{\infty}}\|w(\alpha,\beta) \partial^{\alpha-\alpha_1}_{\beta-\beta_1}f\|^2_{\sigma}
		\notag\\
		&&\leq Ck^{\frac{1}{12}}\varepsilon^{\frac{3}{5}-\frac{2}{5}a}\mathcal{D}_{2,l,q}(\tau).
	\end{eqnarray}	
	If $|\alpha_1|+|\beta_1|=1$, then $|\alpha|+|\beta|=2$ and
	$w(\alpha,\beta)\leq w(\alpha-\alpha_{1}+\alpha_2,\beta-\beta_{1})$ for $|\alpha_2|\leq1$.
	In this case, if $|\alpha-\alpha_1|=0$, we use the similar arguments as \eqref{4.21A} to get the same bound.
	If $|\alpha-\alpha_1|=1$, we use the similar arguments as \eqref{4.22A} to get the same bound.
	
	If $|\alpha-\alpha_1|+|\beta-\beta_1|=2$, then $|\alpha_1|+|\beta_1|=0$ and $w(\alpha,\beta)=w(\alpha-\alpha_{1},\beta-\beta_{1})$,
	we use the similar arguments as \eqref{4.22A} to get the same bound.
	
As a consequence, the above bounds give
	\begin{equation}
		\label{7.30}
		\varepsilon^{a-1}|(\partial^\alpha_\beta \Gamma(f,f),e^{\phi}w^2(\alpha,\beta)h)|
		\leq\eta\varepsilon^{a-1}\|w(\alpha,\beta)h\|^{2}_{\sigma}
		+C_{\eta}k^{\frac{1}{12}}\varepsilon^{\frac{3}{5}-\frac{2}{5}a}\mathcal{D}_{2,l,q}(\tau).
	\end{equation}		
	For the second term of \eqref{7.20}, we have from \eqref{7.8} and \eqref{4.45a} that
	\begin{multline*}
		\varepsilon^{a-1}|(\partial^\alpha_\beta \Gamma(\frac{\overline{G}}{\sqrt{\mu}},f),e^{\phi} w^2(\alpha,\beta)h)|\\
		\leq C\eta\varepsilon^{a-1}\|w(\alpha,\beta)h\|^{2}_{\sigma}
		+C_{\eta}\sum_{\alpha_1\leq\alpha}\sum_{\bar{\beta}\leq\beta_1\leq\beta}I^{\alpha,\alpha_{1}}_{\beta,\beta_{1}}(\overline{G},f).
	\end{multline*}
	Here the term $I^{\alpha,\alpha_{1}}_{\beta,\beta_{1}}(\overline{G},f)$ is defined by
	$$
	I^{\alpha,\alpha_{1}}_{\beta,\beta_{1}}(\overline{G},f)=\varepsilon^{a-1}\int_{\mathbb R}|\mu^\epsilon\partial^{\alpha_1}_{\bar{\beta}}(\frac{\overline{G}}{\sqrt{\mu}})|^2_2| w(\alpha,\beta) \partial^{\alpha-\alpha_1}_{\beta-\beta_1}f|^2_{\sigma}\,dy.
	$$		
	By \eqref{7.1} and \eqref{7.2}, we can rewrite \eqref{3.4} as
	\begin{equation*}
		\overline{G}=\varepsilon^{1-a}\{\frac{\sqrt{R}\bar{\theta}_{y}}{\sqrt{\theta}}A_{1}(\frac{v-u}{\sqrt{R\theta}})
		+\bar{u}_{1y}B_{11}(\frac{v-u}{\sqrt{R\theta}})\},
	\end{equation*}
	which implies that for $\beta_{1}=(1,0,0)$,
	\begin{equation}
		\label{7.22}
		\partial_{\beta_{1}}\overline{G}=\varepsilon^{1-a}\{\frac{\sqrt{R}\bar{\theta}_{y}}{\sqrt{\theta}}
		\partial_{v_{1}}A_{1}(\frac{v-u}{\sqrt{R\theta}})(\frac{1}{\sqrt{R\theta}})
		+\bar{u}_{1y}\partial_{v_{1}}B_{11}(\frac{v-u}{\sqrt{R\theta}})\frac{1}{\sqrt{R\theta}}\}.
	\end{equation}
	Then, it is straightforward to get
	\begin{eqnarray}
		\label{7.23}
		\overline{G}_y&&=\varepsilon^{1-a}\Big\{\frac{\sqrt{R}\bar{\theta}_{yy}}{\sqrt{\theta}}A_{1}(\frac{v-u}{\sqrt{R\theta}})
		-\frac{\sqrt{R}\bar{\theta}_{y}\theta_{y}}{2\sqrt{\theta^{3}}}A_{1}(\frac{v-u}{\sqrt{R\theta}})
		\notag\\
		&&\quad-\frac{\sqrt{R}\bar{\theta}_{y}}{\sqrt{\theta}}
		\nabla_{v}A_{1}(\frac{v-u}{\sqrt{R\theta}})\cdot\frac{u_{y}}{\sqrt{R\theta}}
		-\frac{\sqrt{R}\bar{\theta}_{y}\theta_{y}}{\sqrt{\theta}}
		\nabla_{v}A_{1}(\frac{v-u}{\sqrt{R\theta}})\cdot\frac{v-u}{2\sqrt{R\theta^{3}}}
		\notag\\
		&&\quad+\bar{u}_{1yy}B_{11}(\frac{v-u}{\sqrt{R\theta}})
		-\frac{\bar{u}_{1y}u_{y}}{\sqrt{R\theta}}\cdot\nabla_{v}B_{11}(\frac{v-u}{\sqrt{R\theta}})\notag\\
		&&\quad-\frac{\bar{u}_{1y}\theta_{y}(v-u)}{2\sqrt{R\theta^{3}}}\cdot\nabla_{v}B_{11}(\frac{v-u}{\sqrt{R\theta}})
		\Big\}.
	\end{eqnarray}
	And $\overline{G}_{\tau}$ has the similar expression as \eqref{7.23}. Thus, for any $|\bar{\beta}|\geq0$ and $m>0$, by 
	similar expansion as \eqref{7.22} and \eqref{7.23}, we derive from \eqref{7.4} and \eqref{7.25aa} that
	\begin{equation}
		\label{7.24}
		|\langle v\rangle^{m} \partial _{\bar{\beta}}(\frac{\overline{G}}{\sqrt{\mu}})|_{\sigma,w}
		\leq C\varepsilon^{1-a}|(\bar{u}_{1y},\bar{\theta}_{y})|,
	\end{equation}
	and for $|\bar{\alpha}|=1$,
	\begin{equation}
		\label{7.25}
		| \langle v\rangle^{m} \partial^{\bar{\alpha}}_{\bar{\beta}}(\frac{\overline{G}}{\sqrt{\mu}})|_{\sigma,w}
		\leq C\varepsilon^{1-a}\{|\partial^{\bar{\alpha}}(\bar{u}_{1y},\bar{\theta}_{y})|
		+|(\bar{u}_{1y},\bar{\theta}_{y})||\partial^{\bar{\alpha}}(u,\theta)|\}.
	\end{equation}
	If $|\alpha_{1}|=0$, we use \eqref{7.24}, Lemma \ref{lem7.2} and $\delta=k^{-1}\varepsilon^{\frac{3}{5}-\frac{2}{5}a}$ in \eqref{3.21} to obtain
	\begin{multline}
		\label{4.27A}
		\||\mu^\epsilon\partial^{\alpha_1}_{\bar{\beta}}(\frac{\overline{G}}{\sqrt{\mu}})|^2_2\|_{L_y^{\infty}}
		\leq C\varepsilon^{2-2a}\|(\bar{u}_{1y},\bar{\theta}_{y})\|^2_{L_y^{\infty}}
		\\	
		\leq C\varepsilon^{2-2a}\varepsilon^{2a}(\delta+\varepsilon^{a}\tau)^{-2}
		\leq C\varepsilon^{\frac{8}{5}+\frac{4}{15}a}(\delta+\varepsilon^{a}\tau)^{-\frac{4}{3}},
	\end{multline}
	which yields immediately that
	\begin{eqnarray*}
		I^{\alpha,\alpha_{1}}_{\beta,\beta_{1}}(\overline{G},f)&&\leq \varepsilon^{a-1}\||\mu^\epsilon\partial^{\alpha_1}_{\bar{\beta}}(\frac{\overline{G}}{\sqrt{\mu}})|^2_2\|_{L_y^{\infty}}\| w(\alpha,\beta) \partial^{\alpha-\alpha_1}_{\beta-\beta_1}f\|^2_{\sigma}
		\notag\\
		&&\leq C\varepsilon^{\frac{8}{5}+\frac{4}{15}a}(\delta+\varepsilon^{a}\tau)^{-\frac{4}{3}}\mathcal{D}_{2,l,q}(\tau)
		\notag\\
		&&\leq Ck^{\frac{1}{12}}\varepsilon^{\frac{3}{5}-\frac{2}{5}a}\mathcal{D}_{2,l,q}(\tau).
	\end{eqnarray*}	
	If $|\alpha_{1}|=1$, then $|\alpha|=1$ and $w(\alpha,\beta)\leq w(\alpha-\alpha_{1}+\alpha_{2},\beta)$ for $|\alpha_{2}|\leq1$, and hence it follows that
	\begin{multline*}
		I^{\alpha,\alpha_{1}}_{\beta,\beta_{1}}(\overline{G},f)\leq \varepsilon^{a-1}\|\mu^\epsilon\partial^{\alpha_1}_{\bar{\beta}}(\frac{\overline{G}}{\sqrt{\mu}})\|^2_2
		\|| w(\alpha,\beta) \partial^{\alpha-\alpha_1}_{\beta-\beta_1}f|^2_{\sigma}\|_{L_y^{\infty}}\\
		\leq Ck^{\frac{1}{12}}\varepsilon^{\frac{3}{5}-\frac{2}{5}a}\mathcal{D}_{2,l,q}(\tau).
	\end{multline*}
Here we have used \eqref{7.25}, Lemma \ref{lem7.2}, \eqref{3.21} and \eqref{4.13A} such that	
	\begin{eqnarray}
		\label{4.28A}
		\|\mu^\epsilon\partial^{\alpha_1}_{\bar{\beta}}(\frac{\overline{G}}{\sqrt{\mu}})\|^2_2	
		&&\leq C\varepsilon^{2-2a}\{\|\partial^{\alpha_1}(\bar{u}_{1y},\bar{\theta}_{y})\|^2
		+\|(\bar{u}_{1y},\bar{\theta}_{y})\|^2_{L_y^{\infty}}\|\partial^{\alpha_1}(u,\theta)\|^2\}
		\notag\\	
		&&\leq C\varepsilon^{2-2a}\{\varepsilon^{3a}\delta^{-1}(\delta+\varepsilon^{a}\tau)^{-2}
		+\varepsilon^{2a}(\delta+\varepsilon^{a}\tau)^{-2}k^{\frac{1}{12}}\varepsilon^{\frac{3}{5}-\frac{2}{5}a}\}
		\notag\\
		&&\leq C\varepsilon^{\frac{8}{5}+\frac{4}{15}a}(\delta+\varepsilon^{a}\tau)^{-\frac{4}{3}}.
	\end{eqnarray}
Consequently, with these facts, we have
	\begin{multline}
		\label{7.29}
		\varepsilon^{a-1}|(\partial^\alpha_\beta \Gamma(\frac{\overline{G}}{\sqrt{\mu}},f),e^{\phi} w^2(\alpha,\beta)h)|\\
		\leq C\eta\varepsilon^{a-1}\|w(\alpha,\beta)h\|^{2}_{\sigma}
		+C_{\eta}k^{\frac{1}{12}}\varepsilon^{\frac{3}{5}-\frac{2}{5}a}\mathcal{D}_{2,l,q}(\tau).
	\end{multline}
	Performing the similar calculations as \eqref{7.29}, it holds that
	\begin{multline}
		\label{4.30A}
		\varepsilon^{a-1}|(\partial^\alpha_\beta \Gamma(f,\frac{\overline{G}}{\sqrt{\mu}}),e^{\phi} w^2(\alpha,\beta)h)|\\
		\leq C\eta\varepsilon^{a-1}\|w(\alpha,\beta)h\|^{2}_{\sigma}
		+C_{\eta}k^{\frac{1}{12}}\varepsilon^{\frac{3}{5}-\frac{2}{5}a}\mathcal{D}_{2,l,q}(\tau).
	\end{multline}
	
	Lastly we consider the last term in \eqref{7.20}. As before, we have from \eqref{7.8} and \eqref{4.45a} that
	\begin{multline*}
		\varepsilon^{a-1}|(\partial^\alpha_\beta \Gamma(\frac{\overline{G}}{\sqrt{\mu}},\frac{\overline{G}}{\sqrt{\mu}}),e^{\phi} w^2(\alpha,\beta)h)|\\
		\leq C\eta\varepsilon^{a-1}\|w(\alpha,\beta)h\|^{2}_{\sigma}
		+C_{\eta}\sum_{\alpha_1\leq\alpha}\sum_{\bar{\beta}\leq\beta_1\leq\beta}I^{\alpha,\alpha_{1}}_{\beta,\beta_{1}}(\overline{G},\overline{G}).
	\end{multline*}
In this expression we have used
	$$
	I^{\alpha,\alpha_{1}}_{\beta,\beta_{1}}(\overline{G},\overline{G})=\varepsilon^{a-1}\int_{\mathbb R}|\mu^\epsilon\partial^{\alpha_1}_{\bar{\beta}}(\frac{\overline{G}}{\sqrt{\mu}})|^2_2| w(\alpha,\beta) \partial^{\alpha-\alpha_1}_{\beta-\beta_1}(\frac{\overline{G}}{\sqrt{\mu}})|^2_{\sigma}\,dy.
	$$	
If  $|\alpha|=0$, we deduce from \eqref{7.24}, \eqref{7.25}, Lemma \ref{lem7.2} and \eqref{3.21} that
	\begin{eqnarray*}
		I^{\alpha,\alpha_{1}}_{\beta,\beta_{1}}(\overline{G},\overline{G})\leq
		&&\varepsilon^{a-1}\||\mu^\epsilon\partial^{\alpha_1}_{\bar{\beta}}(\frac{\overline{G}}{\sqrt{\mu}})|^2_2\|_{L_y^{\infty}}\| w(\alpha,\beta) \partial^{\alpha-\alpha_1}_{\beta-\beta_1}(\frac{\overline{G}}{\sqrt{\mu}})\|^2_{\sigma}
		\notag\\
		\leq&& C\varepsilon^{a-1}\varepsilon^{2-2a}\varepsilon^{2a}(\delta+\varepsilon^{a}\tau)^{-2}
		\varepsilon^{2-2a}\varepsilon^{a}(\delta+\varepsilon^{a}\tau)^{-1}
		\notag\\
		=&&C\varepsilon^{3}(\delta+\varepsilon^{a}\tau)^{-3}\leq C\varepsilon^{\frac{7}{5}+\frac{1}{15}a}(\delta+\varepsilon^{a}\tau)^{-\frac{4}{3}}.
	\end{eqnarray*}
If  $|\alpha|=1$, then $|\alpha_1|=0$ or $|\alpha_1|=1$. When $|\alpha_1|=0$, by \eqref{4.27A} and \eqref{4.28A}, we conclude that
	\begin{eqnarray*}
		I^{\alpha,\alpha_{1}}_{\beta,\beta_{1}}(\overline{G},\overline{G})\leq
		&&\varepsilon^{a-1}\||\mu^\epsilon\partial^{\alpha_1}_{\bar{\beta}}(\frac{\overline{G}}{\sqrt{\mu}})|^2_2\|_{L_y^{\infty}}\| w(\alpha,\beta) \partial^{\alpha-\alpha_1}_{\beta-\beta_1}(\frac{\overline{G}}{\sqrt{\mu}})\|^2_{\sigma}
		\notag\\
		\leq&&C\varepsilon^{\frac{7}{5}+\frac{1}{15}a}(\delta+\varepsilon^{a}\tau)^{-\frac{4}{3}}.
	\end{eqnarray*}
	If $|\alpha_1|=1$, this case has the same bound as the above. It follows that
	\begin{multline}
		\label{7.27}
		\varepsilon^{a-1}|(\partial^\alpha_\beta \Gamma(\frac{\overline{G}}{\sqrt{\mu}},\frac{\overline{G}}{\sqrt{\mu}}),e^{\phi} w^2(\alpha,\beta)h)|\\
		\leq C\eta\varepsilon^{a-1}\|w(\alpha,\beta)h\|_{\sigma}^{2}
		+C_{\eta}\varepsilon^{\frac{7}{5}+\frac{1}{15}a}(\delta+\varepsilon^{a}\tau)^{-\frac{4}{3}}.
	\end{multline}

In sumarry,  combining \eqref{7.27}, \eqref{4.30A}, \eqref{7.29} and \eqref{7.30}, the estimate \eqref{7.18} follows.
Exactly as in the analogous estimate in \eqref{7.18}, we can prove that  \eqref{7.7} holds and we omit the details for brevity.
	This completes the proof of Lemma \ref{lem7.8}.
\end{proof}
\begin{lemma}
	\label{lem4.8}
	Let $|\alpha|= 2$ and $w=w(\alpha,0)$ defined in \eqref{3.12}.
	Suppose that  \eqref{3.21}, \eqref{3.22} and \eqref{3.23} hold. For any $\eta>0$, there exists $C_\eta>0$ such that
	\begin{multline}
		\label{4.34A}
		|(\partial^\alpha \Gamma(\frac{G}{\sqrt{\mu}},\frac{G}{\sqrt{\mu}}), e^{\phi}w^2(\alpha,0) h)|
	\leq  C\eta\|w(\alpha,0)h\|^{2}_{\sigma}\\
		+C_{\eta}k^{\frac{1}{12}}\varepsilon^{\frac{3}{5}-\frac{2}{5}a}\mathcal{D}_{2,l,q}(\tau)
		+C_{\eta}\varepsilon^{\frac{7}{5}+\frac{1}{15}a}(\delta+\varepsilon^{a}\tau)^{-\frac{4}{3}}.
	\end{multline}
\end{lemma}
\begin{proof}
For $|\alpha|= 2$, similar to \eqref{7.20},  it is easy to see that
	\begin{equation}
		\label{4.35A}
		\partial^\alpha\Gamma(\frac{G}{\sqrt{\mu}},\frac{G}{\sqrt{\mu}})=
		\partial^\alpha\Gamma(f,f)
		+\partial^\alpha\Gamma(\frac{\overline{G}}{\sqrt{\mu}},f)+\partial^\alpha\Gamma(f,\frac{\overline{G}}{\sqrt{\mu}})
		+\partial^\alpha\Gamma(\frac{\overline{G}}{\sqrt{\mu}},\frac{\overline{G}}{\sqrt{\mu}}).
	\end{equation}
Next we focus on the estimate on the inner product of \eqref{4.35A} with $e^{\phi} w^2(\alpha,0) h$.  In view of \eqref{7.7} and \eqref{4.45a}, one gets
	\begin{equation*}
		|(\partial^\alpha \Gamma(f,f),e^{\phi}w^2(\alpha,0)h)|
		\leq \eta\|w(\alpha,0)h\|^{2}_{\sigma}
		+C_{\eta}\sum_{\alpha_1\leq\alpha}I_\alpha^{\alpha_{1}}(f,f),
	\end{equation*}	
where $I^{\alpha_{1}}_{\alpha}(f,f)$ is given by
$$
I_\alpha^{\alpha_{1}}(f,f)=\int_{\mathbb R}|\mu^\epsilon\partial^{\alpha_1}f|^2_2|w(\alpha,0) \partial^{\alpha-\alpha_1}f|^2_{\sigma}\,dy.
$$	
	If $|\alpha_1|=|\alpha|=2$, then  $w(\alpha,0)\leq w(\alpha-\alpha_1+\alpha_2,0)$ for $|\alpha_2|\leq 1$, and hence it follows
	 from these facts, \eqref{3.22} and \eqref{4.12A}	that
	\begin{eqnarray*}
		I_\alpha^{\alpha_{1}}(f,f)
		&&\leq \|\mu^\epsilon\partial^{\alpha_1}f\|^2_2\||w(\alpha,0)\partial^{\alpha-\alpha_1}f|^2_{\sigma}\|_{L_y^{\infty}}
		\notag\\
		&&\leq Ck^{\frac{1}{6}}\varepsilon^{\frac{6}{5}-\frac{4}{5}a}\varepsilon^{2a-2}\|w(\alpha,0)\partial^{\alpha-\alpha_1}f\|_{\sigma}\|w(\alpha,0)\partial^{\alpha-\alpha_1}f_y\|_{\sigma}
		\notag\\
		&&\leq Ck^{\frac{1}{12}}\varepsilon^{\frac{3}{5}-\frac{2}{5}a}\mathcal{D}_{2,l,q}(\tau).
	\end{eqnarray*}	
Similarly, the case $|\alpha_1|\leq1$ can be estimated as
	\begin{equation*}
		I_\alpha^{\alpha_{1}}(f,f)
		\leq \||\mu^\epsilon\partial^{\alpha_1}f|^2_2\|_{L_y^{\infty}}\|w(\alpha,0) \partial^{\alpha-\alpha_1}f\|^2_{\sigma}
		\leq Ck^{\frac{1}{12}}\varepsilon^{\frac{3}{5}-\frac{2}{5}a}\mathcal{D}_{2,l,q}(\tau).
	\end{equation*}	
Consequently, with these two estimates, we have 
	\begin{equation}
		\label{4.36A}
		|(\partial^\alpha \Gamma(f,f),e^{\phi}w^2(\alpha,0)h)|
		\leq\eta\|w(\alpha,0)h\|^{2}_{\sigma}
		+C_{\eta}k^{\frac{1}{12}}\varepsilon^{\frac{3}{5}-\frac{2}{5}a}\mathcal{D}_{2,l,q}(\tau).
	\end{equation}	
	
For the second term of \eqref{4.35A}, we use \eqref{7.7} and \eqref{4.45a} again to get
	\begin{equation*}
		|(\partial^\alpha \Gamma(\frac{\overline{G}}{\sqrt{\mu}},f),e^{\phi} w^2(\alpha,0)h)|
		\leq C\eta\|w(\alpha,0)h\|^{2}_{\sigma}
		+C_{\eta}\sum_{\alpha_1\leq\alpha}I^{\alpha_{1}}_{\alpha}(\overline{G},f),
	\end{equation*}
	where 
	$$
	I^{\alpha_{1}}_{\alpha}(\overline{G},f)=\int_{\mathbb R}|\mu^\epsilon\partial^{\alpha_1}(\frac{\overline{G}}{\sqrt{\mu}})|^2_2| w(\alpha,0) \partial^{\alpha-\alpha_1}f|^2_{\sigma}\,dy.
	$$		
If $|\alpha_{1}|=2$, from \eqref{7.23}, Lemma \ref{lem7.2} and $\delta=k^{-1}\varepsilon^{\frac{3}{5}-\frac{2}{5}a}$ in \eqref{3.21}, we  claim that	
	\begin{equation*}
		\|\mu^\epsilon\partial^{\alpha_1}(\frac{\overline{G}}{\sqrt{\mu}})\|^2_2\leq
		Ck^{\frac{1}{12}}\varepsilon^{\frac{3}{5}-\frac{2}{5}a},
	\end{equation*}	
which together with $w(\alpha,0)\leq w(\alpha-\alpha_{1}+\alpha_{2},0)$ for $|\alpha_{2}|\leq1$ yields that
\begin{equation*}
	I^{\alpha_{1}}_{\alpha}(\overline{G},f)\leq \|\mu^\epsilon\partial^{\alpha_1}(\frac{\overline{G}}{\sqrt{\mu}})\|^2_2
	\|| w(\alpha,0) \partial^{\alpha-\alpha_1}f|^2_{\sigma}\|_{L_y^{\infty}}
	\leq Ck^{\frac{1}{12}}\varepsilon^{\frac{3}{5}-\frac{2}{5}a}\mathcal{D}_{2,l,q}(\tau).
\end{equation*}	
	If $|\alpha_{1}|\leq1$, it is straightforward to check that
	\begin{equation*}
		I^{\alpha_{1}}_{\alpha}(\overline{G},f)\leq \||\mu^\epsilon\partial^{\alpha_1}(\frac{\overline{G}}{\sqrt{\mu}})|^2_2\|_{L_y^{\infty}}\| w(\alpha,0) \partial^{\alpha-\alpha_1}f\|^2_{\sigma}
		\leq Ck^{\frac{1}{12}}\varepsilon^{\frac{3}{5}-\frac{2}{5}a}\mathcal{D}_{2,l,q}(\tau).
	\end{equation*}	
	With the help of the above estimates, one gets
	\begin{equation}
		\label{4.38A}
		|(\partial^\alpha\Gamma(\frac{\overline{G}}{\sqrt{\mu}},f),e^{\phi} w^2(\alpha,0)h)|\leq C\eta\|w(\alpha,0)h\|^{2}_{\sigma}
		+C_{\eta}k^{\frac{1}{12}}\varepsilon^{\frac{3}{5}-\frac{2}{5}a}\mathcal{D}_{2,l,q}(\tau).
	\end{equation}
	The term $(\partial^\alpha\Gamma(f,\frac{\overline{G}}{\sqrt{\mu}}),e^{\phi} w^2(\alpha,0)h)$
	shares the same bound as \eqref{4.38A}. 
	
Lastly, we use the similar argument as \eqref{7.27} to get
\begin{equation}
		\label{4.39A}
		|(\partial^\alpha \Gamma(\frac{\overline{G}}{\sqrt{\mu}},\frac{\overline{G}}{\sqrt{\mu}}),e^{\phi} w^2(\alpha,0)h)|
		\leq C\eta\|w(\alpha,0)h\|^{2}_{\sigma}+C_{\eta}\varepsilon^{\frac{7}{5}+\frac{1}{15}a}(\delta+\varepsilon^{a}\tau)^{-\frac{4}{3}}.
	\end{equation}
	Therefore, the estimate \eqref{4.34A} follows from \eqref{4.39A}, \eqref{4.38A} and \eqref{4.36A}.
	This ends the proof of Lemma \ref{lem4.8}.
\end{proof}

\section{Zeroth order energy estimates}\label{sec.5}
In this section we consider the zeroth order energy estimates for both the fluid and non-fluid parts.

\subsection{Zeroth order estimates on fluid part}\label{sec4.1}
As pointed out in \cite{Liu-Yang-Yu,Liu-Yang-Yu-Zhao}, the zero-order energy estimates for the fluid-type system \eqref{3.10}
and \eqref{4.31} cannot be directly derived by the usual $L^2$ energy method. For this, we make use of the entropy and entropy flux motivated in \cite{Liu-Yang-Yu}
to derive the zero-order estimates on the fluid part $(\widetilde{\rho},\widetilde{u},\widetilde{\theta},\widetilde{\phi})$.
The main result is given in Lemma \ref{lem.5.1A} below. In order to make the proof of Lemma \ref{lem.5.1A} not too long, we postpone the proof of some estimates to Lemma \ref{lem.5.2a}, Lemma \ref{lem.5.2A} and Lemma \ref{lem.5.3A}.

\begin{lemma}\label{lem.5.1A}
Assume that all the conditions in Theorem \ref{thm2.1} hold.	
Let $(F(\tau),\phi(\tau))$ be a smooth solution to the Cauchy problem  on the
 VPL system \eqref{3.5} and \eqref{3.7b} for $0\leq \tau\leq T$ with $T>0$.
Let \eqref{3.22} be satisfied, then the following energy estimate holds
	\begin{eqnarray}
		\label{4.39}
		&&\frac{d}{d\tau}\Big\{\int_{\mathbb{R}}\eta(\tau,y)\, dy
		+\frac{3}{4}\big(\widetilde{\phi}^{2},\rho'_{\mathrm{e}}(\bar{\phi})\big)
		+\frac{1}{2}\big(\widetilde{\phi}^{3},\rho''_{\mathrm{e}}(\bar{\phi})\big)
		+\frac{3}{4}\varepsilon^{2b-2a}\|\widetilde{\phi}_y\|^{2}
		\notag\\
		&&\hspace{1cm}-N_1(\tau)+\kappa_{3}\varepsilon^{1-a}(\widetilde{u}_{1},\widetilde{\rho}_y)\Big\}
		+c_{3}\|\sqrt{\bar{u}_{1y}}(\widetilde{\rho},\widetilde{u}_{1},\widetilde{\theta})\|^{2}
		\notag\\
		&&\hspace{1cm}+c\varepsilon^{1-a}\big\{\|(\widetilde{\rho}_y,\widetilde{u}_y,\widetilde{\theta}_y,\widetilde{\phi}_y)\|^{2}
		+\varepsilon^{2b-2a}\|\widetilde{\phi}_{yy}\|^{2}\big\}
		\notag\\
		&&\leq  C\varepsilon^{1-a}\|f_y\|^{2}_{\sigma}
		+Ck^{\frac{1}{12}}\varepsilon^{\frac{3}{5}-\frac{2}{5}a}\mathcal{D}_{2,l,q}(\tau)
		+C\varepsilon^{\frac{7}{5}+\frac{1}{15}a}(\delta+\varepsilon^{a}\tau)^{-\frac{4}{3}},
	\end{eqnarray}
where $\kappa_3>0$ is a given constant, and $N_1(\tau)$ is denoted by
	\begin{multline}
		\label{5.23b}
		N_1(\tau)=\varepsilon^{1-a}\int_{\mathbb{R}}\int_{\mathbb{R}^{3}}\big\{\frac{3}{2}(\frac{\widetilde{\theta}}{\theta})_{y}
		(R\theta)^{\frac{3}{2}}A_{1}(\frac{v-u}{\sqrt{R\theta}})\frac{\sqrt{\mu}}{M}f \big\}\,dv\,dy
		\\
		+\varepsilon^{1-a}\sum^{3}_{i=1}\int_{\mathbb{R}}\int_{\mathbb{R}^{3}} \big\{(\frac{3}{2}\widetilde{u}_{iy}-\frac{3}{2}\frac{\widetilde{\theta}}{\theta}u_{iy})
		R\theta B_{1i}(\frac{v-u}{\sqrt{R\theta}})\frac{\sqrt{\mu}}{M}f\big\}\, dv\,dy.
	\end{multline}
\end{lemma}
\begin{proof}
	The proof is divided by three steps as follows.
	
	\medskip
	\noindent{{\bf Step 1.}}
	Inspired by \cite{Liu-Yang-Yu}, we define the macroscopic entropy $S$ by
	$$
	-\frac{3}{2}\rho S:=\int_{\mathbb{R}^{3}}M\ln M\,dv.
	$$
	Furthermore, a direct calculation shows that
	\begin{equation}
		\label{4.1}
		S=-\frac{2}{3}\ln\rho+\ln(\frac{4\pi}{3}\theta)+1,\quad
		P=\frac{2}{3}\rho\theta=\frac{1}{2\pi e}\rho^{\frac{5}{3}}\exp (S).
	\end{equation}
	Multiplying the first equation of \eqref{3.5} by $\ln M$ and integrating over $v$, then making a direct calculation, it holds that
	\begin{equation*}
		(-\frac{3}{2}\rho S)_{\tau}+(-\frac{3}{2}\rho u_{1} S)_{y}
		+\int_{\mathbb{R}^{3}} v_{1}G_{y}\ln M\,dv=0.
	\end{equation*}
	Here we have used \eqref{1.7} and \eqref{1.5} such that
	\begin{equation*}
		\int_{\mathbb{R}^{3}}\phi_{y}\partial_{v_{1}}F \ln M\,dv=\frac{\phi_{y}}{R\theta}\int_{\mathbb{R}^{3}}F(v_1-u_1)\,dv
		=\frac{\phi_{y}}{R\theta}(\rho u_1-u_1\rho)=0.
	\end{equation*}
	We denote $X$ and $Y$ as
$$
X:=(X_{0},X_{1},X_{2},X_{3},X_{4})^{t}=(\rho,\rho u_{1},\rho u_{2},\rho u_{3},\rho(\theta+\frac{|u|^{2}}{2}))^{t},
$$
and
	\begin{multline*}
		Y:=(Y_{0},Y_{1},Y_{2},Y_{3},Y_{4})^{t}\\
		=(\rho u_{1},\rho u^{2}_{1}+P,\rho u_{1}u_{2},\rho u_{1} u_{3},\rho u_{1}(\theta+\frac{|u|^{2}}{2})+Pu_{1})^{t},
	\end{multline*}
	where $(\cdot,\cdot,\cdot)^{t}$ is the transpose of the vector $(\cdot,\cdot,\cdot)$.
With these facts and \eqref{3.1} in hand, we can rewrite the conservation laws \eqref{1.18} as
	\begin{equation*}
		X_{\tau}+Y_{y}=\begin{pmatrix}
			0
			\\
			\frac{4}{3}\varepsilon^{1-a}(\mu(\theta)u_{1y})_{y}-(\int_{\mathbb{R}^{3}} v^{2}_{1}L^{-1}_{M}\Theta\, dv)_{y}-\rho\phi_{y}
			\\
			\varepsilon^{1-a}(\mu(\theta)u_{2y})_{y}-(\int_{\mathbb{R}^{3}} v_{1}v_{2}L^{-1}_{M}\Theta \,dv)_{y}
			\\
			\varepsilon^{1-a}(\mu(\theta)u_{3y})_{y}-(\int_{\mathbb{R}^{3}} v_{1}v_{3}L^{-1}_{M}\Theta \,dv)_{y}
			\\
			\varepsilon^{1-a}(\kappa(\theta)\theta_{y})_{y}+D-(\int_{\mathbb{R}^{3}}\frac{ v_{1}}{2}|v|^{2}L^{-1}_{M}\Theta\, dv)_{y}-\rho u_{1}\phi_{y}
		\end{pmatrix}.
	\end{equation*}
	Here the term $D$ is defined by
	$$
	D=\varepsilon^{1-a}\{\frac{4}{3}(\mu(\theta)u_{1}u_{1y})_{y}
	+(\mu(\theta)u_{2}u_{2y})_{y}+(\mu(\theta)u_{3}u_{3y})_{y}\}.
	$$
	Define a relative entropy-entropy flux pair $(\eta,q)(\tau,y)$ around the Maxwellian $\overline{M}=M_{[\bar{\rho},\bar{u},\bar{\theta}]}$ as
	\begin{equation}
		\label{4.3}
		\left\{
		\begin{array}{rl}
			&\eta(\tau,y):=\bar{\theta}\{-\frac{3}{2}\rho S+\frac{3}{2}\bar{\rho}\bar{S}+\frac{3}{2}\nabla_{X}(\rho S)|_{X=\bar{X}}\cdot(X-\bar{X})\},
			\\
			&q(\tau,y):=\bar{\theta}\{-\frac{3}{2}\rho u_{1}S+\frac{3}{2}\bar{\rho}\bar{u}_{1}\bar{S}
			+\frac{3}{2}\nabla_{X}(\rho S)|_{X=\bar{X}}\cdot(Y-\bar{Y})\},
		\end{array} \right.
	\end{equation}
	with $\bar{S}=-\frac{2}{3}\ln\bar{\rho}+\ln(\frac{4\pi}{3}\bar{\theta})+1$
	and $\bar{X}=(\bar{\rho},\bar{\rho}\bar{ u}_{1},\bar{\rho}\bar{ u}_{2},\bar{\rho}\bar{ u}_{3},\bar{\rho}(\bar{\theta}
	+\frac{1}{2}|\bar{u}|^{2}))^{t}$.
	Using \eqref{4.1} and \eqref{4.3}, it is straightforward to check that
	\begin{equation*}
		(\rho S)_{X_{0}}=S+\frac{|u|^{2}}{2\theta}-\frac{5}{3}, \quad
		(\rho S)_{X_{i}}=-\frac{u_{i}}{\theta}, \quad \mbox{i=1,2,3}, \quad (\rho S)_{X_{4}}=\frac{1}{\theta},
	\end{equation*}
	and
	\begin{equation}
		\label{4.4}
		\left\{
		\begin{array}{rl}
			&\eta(\tau,y)=\frac{3}{2}\{\rho\theta-\bar{\theta}\rho S+\rho[(\bar{S}-\frac{5}{3})\bar{\theta}
			+\frac{|u-\bar{u}|^{2}}{2}]+\frac{2}{3}\bar{\rho}\bar{\theta}\}
			\\
			&\hspace{1cm}=\rho\bar{\theta}\Psi(\frac{\bar{\rho}}{\rho})+\frac{3}{2}\rho\bar{\theta}\Psi(\frac{\theta}{\bar{\theta}})
			+\frac{3}{4}\rho|u-\bar{u}|^{2},
			\\
			&q(\tau,y)=u_{1}\eta(\tau,y)+(u_{1}-\bar{u}_{1})(\rho\theta-\bar{\rho}\bar{\theta}),
		\end{array} \right.
	\end{equation}
	where $\Psi(s):=s-\ln s-1$ is a strictly convex function around $s=1$. Using \eqref{4.4}, \eqref{3.2}, \eqref{3.24} and the properties of convex function $\Psi(s)$, we know that there exists a constant $C>1$ such that
	\begin{equation}
		\label{4.5}
		C^{-1}|(\widetilde{\rho},\widetilde{u},\widetilde{\theta})|^{2}\leq \eta(\tau,y)\leq C|(\widetilde{\rho},\widetilde{u},\widetilde{\theta})|^{2}.
	\end{equation}
	With these, from a direct but tedious computation, we have the following identity 
	\begin{eqnarray}
		\label{4.6}
		&&\eta_{\tau}(\tau,y)+q_{y}(\tau,y)-\nabla_{[\bar{\rho},\bar{u},\bar{S}]}\eta(\tau,y)\cdot(\bar{\rho},\bar{u},\bar{S})_{\tau}
		-\nabla_{[\bar{\rho},\bar{u},\bar{S}]}q(\tau,y)\cdot(\bar{\rho},\bar{u},\bar{S})_{y}
		\notag\\
		&&+\varepsilon^{1-a}\big\{\frac{2\bar{\theta}}{\theta}\mu(\theta)\widetilde{u}^{2}_{1y}
		+\frac{3\bar{\theta}}{2\theta}\sum_{i=2}^{3}\mu(\theta)\widetilde{u}^{2}_{iy}
		+\frac{3\bar{\theta}}{2\theta^{2}}\kappa(\theta)\widetilde{\theta}^{2}_{y}\big\}
		+\frac{3}{2}\rho\widetilde{u}_{1}\phi_{y}+(\cdot\cdot\cdot)_{y}
		\notag\\
		=&&\varepsilon^{1-a}\big\{\frac{3\kappa(\theta)}{2\theta^{2}}
		(\bar{\theta}_{y}+\widetilde{\theta}_{y})\bar{\theta}_{y}\widetilde{\theta}
		-\frac{3\bar{\theta}}{2\theta^{2}}\kappa(\theta)\widetilde{\theta}_{y}\bar{\theta}_{y}\big\}\notag\\
		&&+\varepsilon^{1-a}\big\{\frac{2\mu(\theta)}{\theta}(\bar{u}_{1y}+\widetilde{u}_{1y})
		\bar{u}_{1y}\widetilde{\theta}-\frac{2\bar{\theta}}{\theta}\mu(\theta)\widetilde{u}_{1y}\bar{u}_{1y}\big\}
		\notag\\
		&&+\big\{\frac{3}{2}(\frac{\widetilde{\theta}}{\theta})_{y}
		\int_{\mathbb{R}^{3}} (\frac{1}{2}v_{1}|v|^{2}-u\cdot vv_{1})L^{-1}_{M}\Theta\, dv\big\}\notag\\
		&&+\big\{(\frac{3}{2}\widetilde{u}_{y}-\frac{3}{2}\frac{\widetilde{\theta}}{\theta}u_{y})\cdot
		\int_{\mathbb{R}^{3}} v_{1}vL^{-1}_{M}\Theta\, dv\big\}
		\notag\\
		:=&&I^1_1+I^2_1+I^3_1+I^4_1.
	\end{eqnarray}
	Here and in the sequel, the notation $(\cdot\cdot\cdot)_{y}$ represents the term in the conservative form so that
	it vanishes after integration.
	
	To estimate the identity \eqref{4.6}, we first claim that there exists $c_{3}>0$ such that
	\begin{eqnarray}
		\label{4.7}
		&&-\nabla_{[\bar{\rho},\bar{u},\bar{S}]}\eta(\tau,y)\cdot(\bar{\rho},\bar{u},\bar{S})_{\tau}
		-\nabla_{[\bar{\rho},\bar{u},\bar{S}]}q(\tau,y)\cdot(\bar{\rho},\bar{u},\bar{S})_{y}
		\notag\\
		&&\geq c_{3}\bar{u}_{1y}(\widetilde{\rho}^{2}+\widetilde{u}_{1}^{2}+\widetilde{\theta}^{2})-\frac{3}{2}\rho\widetilde{u}_{1}\bar{\phi}_{y},
	\end{eqnarray}
	whose proof is given in the appendix Section \ref{sec.9}.

	Since both $\mu(\theta)$ and $\kappa(\theta)$ are smooth functions of $\theta$, there exists a constant $C>1$ such that $\mu(\theta),\kappa(\theta)\in[C^{-1},C]$. Plugging \eqref{4.7} into \eqref{4.6} and integrating the resulting equation with respect to $y$ over $\mathbb{R}$, we obtain
	\begin{multline}
		\label{4.9}
		\frac{d}{d\tau}\int_{\mathbb{R}}\eta(\tau,y) \,dy
		+c_{3}\|\sqrt{\bar{u}_{1y}}(\widetilde{\rho},\widetilde{u}_{1},\widetilde{\theta})\|^{2}
		+c\varepsilon^{1-a}\|(\widetilde{u}_{y},\widetilde{\theta}_{y})\|^{2}\\
		\leq \sum_{i=1}^{4}\int_{\mathbb{R}} I^{i}_{1} \,dy-\frac{3}{2}\int_{\mathbb{R}}\rho\widetilde{u}_{1}\widetilde{\phi}_{y} \,dy.
	\end{multline}
Next we focus on the estimates on the right hand side of \eqref{4.9}.
	Recall the definition $I^1_1$ in \eqref{4.6}. We deduce from the integration by parts  that
	\begin{eqnarray*}
		\int_{\mathbb{R}}I^{1}_{1} \,dy
		&&=\varepsilon^{1-a}\int_{\mathbb{R}}\big\{\frac{3\kappa(\theta)}{2\theta^{2}}
		(\bar{\theta}_{y}+\widetilde{\theta}_{y})\bar{\theta}_{y}\widetilde{\theta}
		-\frac{3\bar{\theta}}{2\theta^{2}}\kappa(\theta)\widetilde{\theta}_{y}\bar{\theta}_{y}\big\}\,dy
		\notag\\
		&&\leq C\varepsilon^{1-a}\int_{\mathbb{R}}|
		\frac{3\kappa(\theta)}{2\theta^{2}}(\widetilde{\theta}_{y}+\bar{\theta}_{y})
		\bar{\theta}_{y}\widetilde{\theta}+(\frac{3\bar{\theta}\kappa(\theta)}{2\theta^{2}})_{y}\widetilde{\theta}\bar{\theta}_{y}+
		\frac{3\bar{\theta}\kappa(\theta)}{2\theta^{2}}\widetilde{\theta}\bar{\theta}_{yy}|\,dy
		\notag\\
		&&\leq C\varepsilon^{1-a}\int_{\mathbb{R}}|\widetilde{\theta}|\{|\bar{\theta}_{y}||\widetilde{\theta}_{y}|
		+|\bar{\theta}_{y}|^{2}+|\bar{\theta}_{y}||\theta_{y}|+|\bar{\theta}_{yy}|\}\,dy
		\notag\\
		&&\leq C\varepsilon^{1-a}\|\widetilde{\theta}\|_{L_{y}^{\infty}}\{\|\bar{\theta}_{yy}\|_{L^{1}}+\|\bar{\theta}_{y}\|^{2}
		+\|\widetilde{\theta}_{y}\|^2\}.
	\end{eqnarray*}
	This together with  the imbedding inequality,  Lemma \ref{lem7.2}, \eqref{3.22} and \eqref{3.18} further leads to
	\begin{align}
		\label{4.14A}
		&\int_{\mathbb{R}}I^{1}_{1} \,dy\notag\\
		&\leq \varepsilon^{1-a}\big\{\eta\|\widetilde{\theta}_{y}\|^{2}
		+C_{\eta}\|\widetilde{\theta}\|^{\frac{2}{3}}\|\bar{\theta}_{yy}\|^{\frac{4}{3}}_{L^{1}}
		+C_{\eta}\|\widetilde{\theta}\|^{\frac{2}{3}}\|\bar{\theta}_{y}\|^{\frac{8}{3}}
		+ C\|\widetilde{\theta}\|^{\frac{1}{2}}\|\widetilde{\theta}_{y}\|^{\frac{1}{2}}\|\widetilde{\theta}_{y}\|^2\big\}
		\notag\\
		&\leq \eta\varepsilon^{1-a}\|\widetilde{\theta}_{y}\|^{2}
		+C_{\eta}\varepsilon^{1-a}(k^{\frac{1}{6}}\varepsilon^{\frac{6}{5}-\frac{4}{5}a})^{\frac{1}{3}}\varepsilon^{\frac{4}{3}a}
		(\delta+\varepsilon^{a}\tau)^{-\frac{4}{3}}+C\sqrt{\mathcal{E}_{2,l,q}(\tau)}\varepsilon^{1-a}\|\widetilde{\theta}_{y}\|^2
		\notag\\
		&\leq \eta\varepsilon^{1-a}\|\widetilde{\theta}_{y}\|^{2}
		+C_{\eta}\varepsilon^{\frac{7}{5}+\frac{1}{15}a}(\delta+\varepsilon^{a}\tau)^{-\frac{4}{3}}+
		Ck^{\frac{1}{12}}\varepsilon^{\frac{3}{5}-\frac{2}{5}a}\mathcal{D}_{2,l,q}(\tau).
	\end{align}
	The estimate for $I^{2}_{1}$ can be handed in the same way as \eqref{4.14A}, namely
	\begin{eqnarray*}
		\int_{\mathbb{R}}I^{2}_{1} \,dy
		\leq\eta\varepsilon^{1-a}\|(\widetilde{u}_{1y},\widetilde{\theta}_{y})\|^{2}
		+C_{\eta}\varepsilon^{\frac{7}{5}+\frac{1}{15}a}(\delta+\varepsilon^{a}\tau)^{-\frac{4}{3}}
		+Ck^{\frac{1}{12}}\varepsilon^{\frac{3}{5}-\frac{2}{5}a}\mathcal{D}_{2,l,q}(\tau).
	\end{eqnarray*}

	The estimation of $I^{3}_{1}$ and $I^{4}_{1}$ are subtle since we have to deal with the inverse operator $L^{-1}_{M}$.
	First of all, we use the self-adjoint property of $L^{-1}_{M}$, \eqref{7.1} and \eqref{7.2} to rewrite
	\begin{eqnarray}
		\label{4.11}
		\int_{\mathbb{R}^{3}} (\frac{1}{2}v_{1}|v|^{2}-u\cdot vv_{1})L^{-1}_{M}\Theta\, dv &=&
		\int_{\mathbb{R}^{3}} L^{-1}_{M}\{P_{1}(\frac{1}{2}v_{1}|v|^{2}-u\cdot vv_{1})M\}\frac{\Theta}{M} \,dv
		\notag\\
		&=&\int_{\mathbb{R}^{3}} L^{-1}_{M}\{(R\theta)^{\frac{3}{2}}\hat{A}_{1}(\frac{v-u}{\sqrt{R\theta}})M\}\frac{\Theta}{M} \,dv\notag\\
		&=&(R\theta)^{\frac{3}{2}}\int_{\mathbb{R}^{3}}A_{1}(\frac{v-u}{\sqrt{R\theta}})\frac{\Theta}{M}\, dv,
	\end{eqnarray}
	and
	\begin{eqnarray}
		\label{4.12}
		\int_{\mathbb{R}^{3}}v_{1}v_{i}L^{-1}_{M}\Theta \,dv &=&
		\int_{\mathbb{R}^{3}} L^{-1}_{M}\{P_{1}( v_{1}v_{i}M)\}\frac{\Theta}{M} \,dv
		\notag\\
		&=&\int_{\mathbb{R}^{3}} L^{-1}_{M}\{R\theta\hat{B}_{1i}(\frac{v-u}{\sqrt{R\theta}})M\}\frac{\Theta}{M}\, dv\notag\\
		&=&R\theta\int_{\mathbb{R}^{3}}B_{1i}(\frac{v-u}{\sqrt{R\theta}})\frac{\Theta}{M}\, dv.
	\end{eqnarray}
Due to \eqref{4.11} and the expression of $I^{3}_{1}$, it is easy to get
	\begin{eqnarray}
		\label{4.13}
		\int_{\mathbb{R}}I^{3}_{1} \,dy&&=\int_{\mathbb{R}}\big\{\frac{3}{2}(\frac{\widetilde{\theta}}{\theta})_{y}
		\int_{\mathbb{R}^{3}} (\frac{1}{2}v_{1}|v|^{2}-u\cdot vv_{1})L^{-1}_{M}\Theta\, dv\big\}\,dy
		\notag\\
		&&=\int_{\mathbb{R}}\int_{\mathbb{R}^{3}}\big\{\frac{3}{2}(\frac{\widetilde{\theta}}{\theta})_{y}
		(R\theta)^{\frac{3}{2}}A_{1}(\frac{v-u}{\sqrt{R\theta}})\frac{\Theta}{M}\big\} \, dv\,dy.
	\end{eqnarray}
	To bound \eqref{4.13}, we need to use \eqref{7.4}, \eqref{2.3} and \eqref{3.24} to obtain the desired estimate that 
	for any multi-index $\beta$ and $m\geq 0$,
	\begin{equation}
		\label{4.14}
		\int_{\mathbb{R}^{3}}\frac{|\langle v\rangle^{m}\sqrt{\mu}\partial_{\beta}A_{1}(\frac{v-u}{\sqrt{R\theta}})|^{2}}{M^{2}}\,dv
		+\int_{\mathbb{R}^{3}}\frac{|\langle v\rangle^{m}\sqrt{\mu}\partial_{\beta}B_{1i}(\frac{v-u}{\sqrt{R\theta}})|^{2}}{M^{2}}\,dv\leq C.
	\end{equation}
	Recall $\Theta$ in \eqref{3.11} given by
	\begin{equation*}
		\Theta=\varepsilon^{1-a}G_{\tau}+\varepsilon^{1-a}P_{1}(v_{1}G_{y})-\varepsilon^{1-a}\phi_{y}\partial_{v_{1}}G-Q(G,G).
	\end{equation*}
In view of Lemma \ref{lem7.2}, \eqref{3.22} and $\delta=\frac{1}{k}\varepsilon^{\frac{3}{5}-\frac{2}{5}a}$ in \eqref{3.21}, the following estimates hold
	\begin{equation}
		\label{4.16}
		\|\widetilde{\theta}\|^{2}\|\bar{\theta}_{y}\|^{2}_{L_{y}^{\infty}}
		\leq Ck^{\frac{1}{6}}\varepsilon^{\frac{6}{5}-\frac{4}{5}a}\varepsilon^{2a}(\delta+\varepsilon^{a}\tau)^{-2}
		\leq C\varepsilon^{\frac{7}{5}+\frac{1}{15}a}(\delta+\varepsilon^{a}\tau)^{-\frac{4}{3}},
	\end{equation}
	and
	\begin{equation}
		\label{5.15A}
		\|(\widetilde{\rho},\widetilde{u},\widetilde{\theta})\|_{L_{y}^{\infty}}\leq Ck^{\frac{1}{12}}\varepsilon^{\frac{3}{5}-\frac{2}{5}a}\leq C.
	\end{equation}
Consequently, with \eqref{5.15A} and \eqref{4.16} in hand, we deduce that
	\begin{align}
		\label{5.17b}
		&\varepsilon^{1-a}\int_{\mathbb{R}}\int_{\mathbb{R}^{3}}\big\{\frac{3}{2}(\frac{\widetilde{\theta}}{\theta})_{y}
		(R\theta)^{\frac{3}{2}}A_{1}(\frac{v-u}{\sqrt{R\theta}})\frac{\overline{G}_{\tau}}{M}\big\} \, dv\,dy	
		\notag\\
		&\leq \eta\varepsilon^{1-a}(\|\widetilde{\theta}_{y}\|^2+\|\widetilde{\theta}\theta_{y}\|^2)+
		C_\eta\varepsilon^{1-a}\|\frac{\overline{G}_{\tau}}{\sqrt{\mu}}\|^2
		\notag\\
		&\leq C\eta\varepsilon^{1-a}\|\widetilde{\theta}_{y}\|^{2}+C_\eta k^{\frac{1}{12}}\varepsilon^{\frac{3}{5}-\frac{2}{5}a}\mathcal{D}_{2,l,q}(\tau)
		+C_{\eta}\varepsilon^{\frac{7}{5}+\frac{1}{15}a}(\delta+\varepsilon^{a}\tau)^{-\frac{4}{3}}.
	\end{align}
	Here we have used \eqref{7.25} and  \eqref{5.26b} such that
	\begin{eqnarray}
		\label{L}
		\varepsilon^{1-a}\|\frac{\overline{G}_{\tau}}{\sqrt{\mu}}\|^2
		&&\leq C\varepsilon^{3-3a}\{\|(\bar{u}_{1y\tau},\bar{\theta}_{y\tau})\|^2
		+\|(\bar{u}_{1y},\bar{\theta}_{y})\|_{L_y^\infty}^2	\|(u_\tau,\theta_\tau)\|^2\}
		\notag\\
		&&\leq C\varepsilon^{3-3a}\{\varepsilon^{3a}\delta^{-1}(\delta+\varepsilon^{a}\tau)^{-2}
		+\varepsilon^{2a}(\delta+\varepsilon^{a}\tau)^{-2}\|(\widetilde{u}_\tau,\widetilde{\theta}_\tau)\|^2\}
		\notag\\
		&&\leq 
		Ck^{\frac{1}{12}}\varepsilon^{\frac{3}{5}-\frac{2}{5}a}\mathcal{D}_{2,l,q}(\tau)
		+C\varepsilon^{\frac{7}{5}+\frac{1}{15}a}(\delta+\varepsilon^{a}\tau)^{-\frac{4}{3}}.
	\end{eqnarray}
A simple computation gives
	\begin{eqnarray}
		\label{5.18b}
		&&\varepsilon^{1-a}\int_{\mathbb{R}}\int_{\mathbb{R}^{3}}\big\{\frac{3}{2}(\frac{\widetilde{\theta}}{\theta})_{y}
		(R\theta)^{\frac{3}{2}}A_{1}(\frac{v-u}{\sqrt{R\theta}})\frac{\sqrt{\mu}}{M}f_{\tau}\big\}\, dv\,dy
		\notag\\
		&&=\varepsilon^{1-a}\frac{d}{d\tau}\int_{\mathbb{R}}\int_{\mathbb{R}^{3}}\big\{\frac{3}{2}(\frac{\widetilde{\theta}}{\theta})_{y}
		(R\theta)^{\frac{3}{2}}A_{1}(\frac{v-u}{\sqrt{R\theta}})\frac{\sqrt{\mu}}{M}f \big\}\,dv\,dy
		\notag\\
		&&\hspace{0.5cm}-\varepsilon^{1-a}\int_{\mathbb{R}}\int_{\mathbb{R}^{3}}\frac{3}{2}(\frac{\widetilde{\theta}}{\theta})_{y\tau}
		\big\{(R\theta)^{\frac{3}{2}}A_{1}(\frac{v-u}{\sqrt{R\theta}})\frac{\sqrt{\mu}}{M} f\big\} \,dv\,dy
		\notag\\
		&&\hspace{0.5cm}-\varepsilon^{1-a}\int_{\mathbb{R}}\int_{\mathbb{R}^{3}}\frac{3}{2}(\frac{\widetilde{\theta}}{\theta})_{y}
		\big\{(R\theta)^{\frac{3}{2}}A_{1}(\frac{v-u}{\sqrt{R\theta}})\frac{\sqrt{\mu}}{M} \big\}_\tau f\,dv\,dy.
	\end{eqnarray}
After integration, the second term on the right hand side of \eqref{5.18b} can be bounded by 
	\begin{eqnarray}\label{ad.520}
		&&\varepsilon^{1-a}\int_{\mathbb{R}}\int_{\mathbb{R}^{3}}\frac{3}{2}(\frac{\widetilde{\theta}}{\theta})_{\tau}
		\big\{(R\theta)^{\frac{3}{2}}A_{1}(\frac{v-u}{\sqrt{R\theta}})\frac{\sqrt{\mu}}{M} f\big\}_y \,dv\,dy
		\notag\\ 
		&&\leq \eta\varepsilon^{1-a}\|(\frac{\widetilde{\theta}}{\theta})_{\tau}\|^2+C_{\eta}\varepsilon^{1-a}\|f_{y}\|_{\sigma}^2
		+C_{\eta}\varepsilon^{1-a}\|(\rho,u,\theta)_y\|_{L_y^\infty}^2\|f\|_{\sigma}^2
		\notag\\
		&&\leq C\eta\varepsilon^{1-a}\|(\widetilde{\rho}_{y},\widetilde{u}_{y},\widetilde{\theta}_{y},\widetilde{\phi}_{y})\|^{2}+C_{\eta}\varepsilon^{1-a}\|f_{y}\|_{\sigma}^{2}
		+C_{\eta}k^{\frac{1}{12}}\varepsilon^{\frac{3}{5}-\frac{2}{5}a}\mathcal{D}_{2,l,q}(\tau)\notag\\
		&&\quad+C_{\eta}\varepsilon^{\frac{7}{5}+\frac{1}{15}a}(\delta+\varepsilon^{a}\tau)^{-\frac{4}{3}},
	\end{eqnarray}
where in the last inequality we have used \eqref{5.26b}, \eqref{5.15A}, \eqref{4.16} and the fact that
\begin{eqnarray*}
\|(\rho_y,u_y,\theta_y,\phi_y)\|_{L_y^{\infty}}
\leq Ck^{\frac{1}{12}}\varepsilon^{\frac{1}{10}a+\frac{1}{10}},
\end{eqnarray*}
due to Lemma \ref{lem7.2}, \eqref{3.22} and \eqref{3.21}. Note that \eqref{5.26b} is used to treat the time derivative estimate and we would separately treat it in Lemma \ref{lem.5.2a} later on.  The last term of \eqref{5.18b} has the same bound as in \eqref{ad.520}. It follows from 
these, \eqref{5.17b}, \eqref{5.18b} and the fact $G=\overline{G}+\sqrt{\mu}f$ that
	\begin{eqnarray*}
		&&\varepsilon^{1-a}\int_{\mathbb{R}}\int_{\mathbb{R}^{3}}\big\{\frac{3}{2}(\frac{\widetilde{\theta}}{\theta})_{y}
		(R\theta)^{\frac{3}{2}}A_{1}(\frac{v-u}{\sqrt{R\theta}})\frac{G_{\tau}}{M}\big\}\, dv\,dy
		\notag\\
		&&\leq\varepsilon^{1-a}\frac{d}{d\tau}\int_{\mathbb{R}}\int_{\mathbb{R}^{3}}\big\{\frac{3}{2}(\frac{\widetilde{\theta}}{\theta})_{y}
		(R\theta)^{\frac{3}{2}}A_{1}(\frac{v-u}{\sqrt{R\theta}})\frac{\sqrt{\mu}}{M}f \big\}\,dv\,dy\notag\\
		&&\quad+ C\eta\varepsilon^{1-a}\|(\widetilde{\rho}_{y},\widetilde{u}_{y},\widetilde{\theta}_{y},\widetilde{\phi}_{y})\|^{2}
		+C_{\eta}\varepsilon^{1-a}\|f_{y}\|_{\sigma}^{2}
		+C_{\eta}k^{\frac{1}{12}}\varepsilon^{\frac{3}{5}-\frac{2}{5}a}\mathcal{D}_{2,l,q}(\tau)\\
		&&\quad+C_{\eta}\varepsilon^{\frac{7}{5}+\frac{1}{15}a}(\delta+\varepsilon^{a}\tau)^{-\frac{4}{3}}.
	\end{eqnarray*}

	On the other hand, it is direct to verify that
	\begin{eqnarray*}
		&&\varepsilon^{1-a}\int_{\mathbb{R}}\int_{\mathbb{R}^{3}}|\big\{\frac{3}{2}(\frac{\widetilde{\theta}}{\theta})_{y}
		(R\theta)^{\frac{3}{2}}A_{1}(\frac{v-u}{\sqrt{R\theta}})
		\frac{P_{1}(v_{1}G_{y})-\phi_{y}\partial_{v_{1}}G}{M}\big\}|\, dv\,dy
		\notag\\
		&&\leq \eta\varepsilon^{1-a}(\|\widetilde{\theta}_{y}\|^2+\|\widetilde{\theta}\theta_{y}\|^2)\\
		&&\quad+
		C_\eta\varepsilon^{1-a}(\|\langle v\rangle^{-2}\frac{P_{1}(v_{1}G_{y})}{\sqrt{\mu}}\|^2
		+\|\phi_{y}\|_{L_y^{\infty}}^{2}\|\langle v\rangle^{-2}\frac{\partial_{v_{1}}G}{\sqrt{\mu}}\|^2)
		\notag\\
		&&\leq C\eta\varepsilon^{1-a}\|\widetilde{\theta}_{y}\|^{2}+C_{\eta}\varepsilon^{1-a}\|f_{y}\|^{2}_{\sigma}
		+C_{\eta}k^{\frac{1}{12}}\varepsilon^{\frac{3}{5}-\frac{2}{5}a}\mathcal{D}_{2,l,q}(\tau)\\
		&&\quad+C_{\eta}\varepsilon^{\frac{7}{5}+\frac{1}{15}a}(\delta+\varepsilon^{a}\tau)^{-\frac{4}{3}}.
	\end{eqnarray*}
	Using \eqref{3.8} and \eqref{7.7}, we have from a direct calculation that
	\begin{eqnarray*}
		\int_{\mathbb{R}}&&\big\{\frac{3}{2}(\frac{\widetilde{\theta}}{\theta})_{y}
		(R\theta)^{\frac{3}{2}}\int_{\mathbb{R}^{3}}A_{1}(\frac{v-u}{\sqrt{R\theta}})\frac{Q(G,G)}{M} dv\big\}\,dy
		\notag\\
		=&&(\frac{3}{2}(\frac{\widetilde{\theta}}{\theta})_{y}
		(R\theta)^{\frac{3}{2}}\frac{\sqrt{\mu}
			A_{1}(\frac{v-u}{\sqrt{R\theta}})}{M},\Gamma(\frac{G}{\sqrt{\mu}},\frac{G}{\sqrt{\mu}}))
		\notag\\
		\leq&& C\eta\varepsilon^{1-a}\|\widetilde{\theta}_{y}\|^{2}
		+C_{\eta}k^{\frac{1}{12}}\varepsilon^{\frac{3}{5}-\frac{2}{5}a}\mathcal{D}_{2,l,q}(\tau)
		+C_{\eta}\varepsilon^{\frac{7}{5}+\frac{1}{15}a}(\delta+\varepsilon^{a}\tau)^{-\frac{4}{3}}.
	\end{eqnarray*}
Consequently, plugging the above estimates into \eqref{4.13} gives
	\begin{eqnarray}
		\label{4.20}
		\int_{\mathbb{R}} I^{3}_{1}\,dy
		\leq&& \varepsilon^{1-a}\frac{d}{d\tau}\int_{\mathbb{R}}\int_{\mathbb{R}^{3}}\big\{\frac{3}{2}(\frac{\widetilde{\theta}}{\theta})_{y}
		(R\theta)^{\frac{3}{2}}A_{1}(\frac{v-u}{\sqrt{R\theta}})\frac{\sqrt{\mu}}{M}f \big\}\,dv\,dy\notag\\
		&&+ C\eta\varepsilon^{1-a}\|(\widetilde{\rho}_{y},\widetilde{u}_{y},\widetilde{\theta}_{y},\widetilde{\phi}_{y})\|^{2}
		+C_{\eta}\varepsilon^{1-a}\|f_{y}\|_{\sigma}^{2}\notag\\
		&&
		+C_{\eta}k^{\frac{1}{12}}\varepsilon^{\frac{3}{5}-\frac{2}{5}a}\mathcal{D}_{2,l,q}(\tau)
		+C_{\eta}\varepsilon^{\frac{7}{5}+\frac{1}{15}a}(\delta+\varepsilon^{a}\tau)^{-\frac{4}{3}}.
	\end{eqnarray}
	The estimate for $I^{4}_{1}$ can be done similarly as \eqref{4.20}. In fact, it follows that
	\begin{eqnarray*}
		\int_{\mathbb{R}} I^{4}_{1}\,dy=&&\sum^{3}_{i=1}\int_{\mathbb{R}}\int_{\mathbb{R}^{3}} \big\{(\frac{3}{2}\widetilde{u}_{iy}-\frac{3}{2}\frac{\widetilde{\theta}}{\theta}u_{iy})
		R\theta B_{1i}(\frac{v-u}{\sqrt{R\theta}})\frac{\Theta}{M}\big\} \, dv\,dy
		\notag\\
		\leq&& \varepsilon^{1-a}\frac{d}{d\tau}\sum^{3}_{i=1}\int_{\mathbb{R}}\int_{\mathbb{R}^{3}} \big\{(\frac{3}{2}\widetilde{u}_{iy}-\frac{3}{2}\frac{\widetilde{\theta}}{\theta}u_{iy})
		R\theta B_{1i}(\frac{v-u}{\sqrt{R\theta}})\frac{\sqrt{\mu}}{M}f\big\}\, dv\,dy\\
		&&+ C\eta\varepsilon^{1-a}\|(\widetilde{\rho}_{y},\widetilde{u}_{y},\widetilde{\theta}_{y},\widetilde{\phi}_{y})\|^{2}
		\notag\\
		&&+C_{\eta}\varepsilon^{1-a}\|f_{y}\|_{\sigma}^{2}
		+C_{\eta}k^{\frac{1}{12}}\varepsilon^{\frac{3}{5}-\frac{2}{5}a}\mathcal{D}_{2,l,q}(\tau)
		+C_{\eta}\varepsilon^{\frac{7}{5}+\frac{1}{15}a}(\delta+\varepsilon^{a}\tau)^{-\frac{4}{3}}.
	\end{eqnarray*}
	Collecting the above estimates from $I^{1}_{1}$ to $I^{4}_{1}$, we thereby obtain
	\begin{eqnarray}
		\label{4.21}
		\sum_{i=1}^{4}\int_{\mathbb{R}} I^{i}_{1} \,dy
		\leq&& \frac{d}{d\tau}N_1(\tau)+C\eta\varepsilon^{1-a}\|(\widetilde{\rho}_{y},\widetilde{u}_{y},\widetilde{\theta}_{y},\widetilde{\phi}_{y})\|^{2}
		\notag\\
		&&+C_{\eta}\varepsilon^{1-a}\|f_{y}\|_{\sigma}^{2}
		+C_{\eta}k^{\frac{1}{12}}\varepsilon^{\frac{3}{5}-\frac{2}{5}a}\mathcal{D}_{2,l,q}(\tau)\notag\\
		&&\quad+C_{\eta}\varepsilon^{\frac{7}{5}+\frac{1}{15}a}(\delta+\varepsilon^{a}\tau)^{-\frac{4}{3}},
	\end{eqnarray}
where $N_1(\tau)$ is defined in \eqref{5.23b}.
Moreover, we have the estimate  \eqref{5.27A} whose proof will be postponed to Lemma \ref{lem.5.2A} later. Then, by substituting \eqref{4.21} and \eqref{5.27A} into \eqref{4.9}, we see that there exists  $\kappa_1>0$ such that the following estimate holds
	\begin{align}
		\label{4.30}
		\frac{d}{d\tau}&\big\{\int_{\mathbb{R}}\eta(\tau,y)\,dy
		+\frac{3}{4}\big(\widetilde{\phi}^{2},\rho'_{\mathrm{e}}(\bar{\phi})\big)
		+\frac{1}{2}\big(\widetilde{\phi}^{3},\rho''_{\mathrm{e}}(\bar{\phi})\big)
		+\frac{3}{4}\varepsilon^{2b-2a}\|\widetilde{\phi}_{y}\|^{2}-N_1(\tau)\big\}
		\notag\\
		&\hspace{0.5cm}+c_{3}\|\sqrt{\bar{u}_{1y}}(\widetilde{\rho},\widetilde{u}_{1},\widetilde{\theta})\|^{2}
		+\kappa_1\varepsilon^{1-a}\|(\widetilde{u}_{y},\widetilde{\theta}_{y})\|^{2}
		\notag\\
		\leq& C(\eta+k)\varepsilon^{1-a}\|\widetilde{\phi}_{y}\|^2
		+C\eta\varepsilon^{1-a}\|(\widetilde{\rho}_{y},\widetilde{u}_{y},\widetilde{\theta}_{y})\|^{2}
		+C_{\eta}\varepsilon^{1-a}\|f_{y}\|_{\sigma}^{2}
		\notag\\
		&\hspace{0.5cm}+C_{\eta}k^{\frac{1}{12}}\varepsilon^{\frac{3}{5}-\frac{2}{5}a}\mathcal{D}_{2,l,q}(\tau)
		+C_{\eta}\varepsilon^{\frac{7}{5}+\frac{1}{15}a}(\delta+\varepsilon^{a}\tau)^{-\frac{4}{3}},
	\end{align}
for any small $\eta>0$. This completes the first step in our proof of \eqref{4.39}.

	\medskip
	\noindent{{\bf Step 2.}}
	It should be noted that there are no dissipation terms for $\widetilde{\rho}_{y}$ and $\widetilde{\phi}_{y}$
	in \eqref{4.30}.
	For this, we turn to the Euler-type equations \eqref{4.31}. We first take the inner product of the second equation of  
	\eqref{4.31} with $\varepsilon^{1-a}\widetilde{\rho}_{y}$  to get
	\begin{eqnarray}
		\label{4.32}
		&&\varepsilon^{1-a}(\frac{2\bar{\theta}}{3\bar{\rho}}\widetilde{\rho}_{y},\widetilde{\rho}_{y})
		+\varepsilon^{1-a}(\widetilde{u}_{1}u_{1y}+\bar{u}_{1}\widetilde{u}_{1y}
		+\frac{2}{3}\widetilde{\theta}_{y}+\frac{2}{3}\rho_{y}\frac{\bar{\rho}\widetilde{\theta}-\widetilde{\rho}\bar{\theta}}{\rho\bar{\rho}},\widetilde{\rho}_{y})
		\notag\\
		&&=-\varepsilon^{1-a}(\widetilde{u}_{1\tau},\widetilde{\rho}_{y})
		-\varepsilon^{1-a}(\widetilde{\phi}_{y},\widetilde{\rho}_{y})
		-\varepsilon^{1-a}(\frac{1}{\rho}\int_{\mathbb{R}^3} v^{2}_{1}G_{y}\,dv,\widetilde{\rho}_{y}).
	\end{eqnarray}
With Lemma \ref{lem7.2}, \eqref{3.21}, \eqref{3.22} and \eqref{3.18} in hand, we get
	\begin{eqnarray}
		\label{4.34a}
		\varepsilon^{1-a}|(\widetilde{u}_{1}u_{1y},\widetilde{\rho}_{y})|
		&&\leq\varepsilon^{1-a}\big\{\|\widetilde{u}_{1}\|_{L_{y}^{\infty}}\|\widetilde{u}_{1y}\|\|\widetilde{\rho}_{y}\|
		+\|\bar{u}_{1y}\|_{L_{y}^{\infty}}\|\widetilde{u}_{1}\|\|\widetilde{\rho}_{y}\|\big\}
		\notag\\
		&&\leq C\varepsilon^{1-a}k^{\frac{1}{12}}\varepsilon^{\frac{3}{5}-\frac{2}{5}a}\big\{\|\widetilde{u}_{1y}\|^{2}+\|\widetilde{\rho}_{y}\|^{2}
		+\|\bar{u}_{1y}\|^{2}_{L_{y}^{\infty}} \big\}
		\notag\\
		&&\leq Ck^{\frac{1}{12}}\varepsilon^{\frac{3}{5}-\frac{2}{5}a}\mathcal{D}_{2,l,q}(\tau)
		+Ck^{\frac{1}{12}}\varepsilon^{\frac{3}{5}-\frac{2}{5}a}\varepsilon^{1-a}\varepsilon^{2a}
		(\delta+\varepsilon^{a}\tau)^{-2}
		\notag\\
		&&\leq Ck^{\frac{1}{12}}\varepsilon^{\frac{3}{5}-\frac{2}{5}a}\mathcal{D}_{2,l,q}(\tau)
		+C\varepsilon^{\frac{7}{5}+\frac{1}{15}a}(\delta+\varepsilon^{a}\tau)^{-\frac{4}{3}}.
	\end{eqnarray}
Applying the Cauchy-Schwarz inequality and a similar calculation as \eqref{4.34a}, it holds that
	\begin{eqnarray*}
		&&\varepsilon^{1-a}(|(\bar{u}_{1}\widetilde{u}_{1y}+\frac{2}{3}\widetilde{\theta}_{y}
		+\frac{2}{3}\rho_{y}\frac{\bar{\rho}\widetilde{\theta}-\widetilde{\rho}\bar{\theta}}
		{\rho\bar{\rho}},\widetilde{\rho}_{y})|+|(\frac{1}{\rho}\int_{\mathbb{R}^3} v^{2}_{1}G_{y}\,dv,\widetilde{\rho}_{y})|)
		\notag\\
		&&\leq \eta\varepsilon^{1-a}\|\widetilde{\rho}_{y}\|^{2}
		+C_{\eta}\varepsilon^{1-a}\|(\widetilde{u}_{1y},\widetilde{\theta}_{y})\|^{2}+C_{\eta}\varepsilon^{1-a}\|f_{y}\|_{\sigma}^{2}
		\notag\\
		&&\hspace{0.5cm}+Ck^{\frac{1}{12}}\varepsilon^{\frac{3}{5}-\frac{2}{5}a}\mathcal{D}_{2,l,q}(\tau)
		+C\varepsilon^{\frac{7}{5}+\frac{1}{15}a}(\delta+\varepsilon^{a}\tau)^{-\frac{4}{3}},
	\end{eqnarray*}
while for the first term on the right hand side of \eqref{4.32}, we deduce that
	\begin{eqnarray*}
		-\varepsilon^{1-a}(\widetilde{u}_{1\tau},\widetilde{\rho}_{y})=&&
		-\varepsilon^{1-a}(\widetilde{u}_{1},\widetilde{\rho}_{y})_{\tau}
		-\varepsilon^{1-a}(\widetilde{u}_{1y},\widetilde{\rho}_{\tau})
		\notag\\
		=&&-\varepsilon^{1-a}(\widetilde{u}_{1},\widetilde{\rho}_{y})_{\tau}
		+\varepsilon^{1-a}(\widetilde{u}_{1y},(\rho\widetilde{u}_{1})_{y}+\bar{u}_{1}\widetilde{\rho}_{y}+\bar{u}_{1y}\widetilde{\rho})
		\notag\\
		\leq&& -\varepsilon^{1-a}(\widetilde{u}_{1},\widetilde{\rho}_{y})_{\tau}+
		\eta\varepsilon^{1-a}\|\widetilde{\rho}_{y}\|^{2}+C_{\eta}\varepsilon^{1-a}\|\widetilde{u}_{1y}\|^{2}
		\notag\\
		&&+Ck^{\frac{1}{12}}\varepsilon^{\frac{3}{5}-\frac{2}{5}a}\mathcal{D}_{2,l,q}(\tau)
		+C\varepsilon^{\frac{7}{5}+\frac{1}{15}a}(\delta+\varepsilon^{a}\tau)^{-\frac{4}{3}}.
	\end{eqnarray*}
The estimate for the second term on the right hand side of \eqref{4.32}	can be directly obtained from \eqref{5.37A} whose proof will be postponed to Lemma \ref{lem.5.3A} later.

Therefore, collecting all the above estimates and taking $\eta>0$ small enough, there exists $\kappa_2>0$ such that
	\begin{eqnarray}
		\label{4.36}
		&&\varepsilon^{1-a}(\widetilde{u}_{1},\widetilde{\rho}_{y})_{\tau}
		+\kappa_2\varepsilon^{1-a}(\|\widetilde{\rho}_{y}\|^{2}+\|\widetilde{\phi}_{y}\|^{2})
		+\kappa_2\varepsilon^{1-a}\varepsilon^{2b-2a}\|\widetilde{\phi}_{yy}\|^{2}
		\notag\\
		&&\leq C\varepsilon^{1-a}(\|(\widetilde{u}_{1y},\widetilde{\theta}_{y})\|^{2}+\|f_{y}\|_{\sigma}^{2})
		+Ck^{\frac{1}{12}}\varepsilon^{\frac{3}{5}-\frac{2}{5}a}\mathcal{D}_{2,l,q}(\tau)\notag\\
		&&\quad+C\varepsilon^{\frac{7}{5}+\frac{1}{15}a}(\delta+\varepsilon^{a}\tau)^{-\frac{4}{3}}.
	\end{eqnarray}
	
	\medskip
	\noindent{{\bf Step 3.}}
	In summary, by taking the summation of \eqref{4.30} and \eqref{4.36}$\times\kappa_{3}$ with
	$C\kappa_3<\frac{1}{4}\kappa_1$, then choosing $\eta>0$ and $k>0$ small enough such that $C(k+\eta)<\frac{1}{4}\kappa_{2}\kappa_{3}$
	and $C\eta<\frac{1}{4}\kappa_{1}$, we obtain
	the desired estimate \eqref{4.39}. This completes the proof of Lemma \ref{lem.5.1A}.
\end{proof}

Since the a priori assumption \eqref{3.22} does not contain any temporal derivatives, we need to derive 
the following estimates, which have been used
in the proof of Lemma \ref{lem.5.1A} above.
\begin{lemma}\label{lem.5.2a}
For $|\alpha|\leq 1$, it holds that
	\begin{multline}
		\label{5.26b}
		\|\partial^{\alpha}(\widetilde{\rho}_{\tau},\widetilde{u}_{\tau},\widetilde{\theta}_{\tau})\|^{2}
		\leq C\|\partial^{\alpha}(\widetilde{\rho}_y,\widetilde{u}_y,\widetilde{\theta}_y,\widetilde{\phi}_y)\|^{2}+C\|\partial^{\alpha}f_y\|_{\sigma}^{2}\\
		+Ck^{\frac{1}{12}}\varepsilon^{\frac{3}{5}a-\frac{2}{5}}\mathcal{D}_{2,l,q}(\tau)
		+C\varepsilon^{\frac{7}{5}+\frac{1}{15}a}(\delta+\varepsilon^{a}\tau)^{-\frac{4}{3}},
	\end{multline}
	and
	\begin{multline}
		\label{4.56}
		\|\partial^{\alpha}\widetilde{\phi}_{\tau}\|^{2}+\varepsilon^{2b-2a}\|\partial^{\alpha}\widetilde{\phi}_{y\tau}\|^{2}\\
		\leq C\|\partial^{\alpha}(\widetilde{\rho}_{y},\widetilde{u}_{y})\|^2+Ck^{\frac{1}{6}}\varepsilon^{\frac{6}{5}-\frac{4}{5}a}\|(\widetilde{\rho}_{y},\widetilde{u}_{y})\|^2
		+C\varepsilon^{\frac{7}{5}+\frac{1}{15}a}(\delta+\varepsilon^{a}\tau)^{-\frac{4}{3}}.
	\end{multline}
\end{lemma}
\begin{proof}
	Differentiating the second equation of \eqref{4.31} with respect to $y$	and taking the inner product
	of the resulting equation with $\widetilde{u}_{1\tau y}$, we have
	\begin{eqnarray}
		\label{5.27b}	
		&&\|\widetilde{u}_{1y\tau}\|^2\notag\\
		&&=-([\widetilde{u}_{1}u_{1y}]_y+[\bar{u}_{1}\widetilde{u}_{1y}]_y+\frac{2}{3}\widetilde{\theta}_{yy}
		+\frac{2}{3}[\rho_{y}\frac{\bar{\rho}\widetilde{\theta}-\widetilde{\rho}\bar{\theta}}
		{\rho\bar{\rho}}]_y+[\frac{2\bar{\theta}}{3\bar{\rho}}\widetilde{\rho}_{y}]_y+\widetilde{\phi}_{yy},\widetilde{u}_{1y\tau})
		\notag\\
		&&\quad-(\frac{1}{\rho}\int_{\mathbb{R}^3}  v^{2}_{1}G_{yy}\,dv,\widetilde{u}_{1y\tau})-([\frac{1}{\rho}]_y\int_{\mathbb{R}^3}  v^{2}_{1}G_{y}\,dv,\widetilde{u}_{1y\tau}).
	\end{eqnarray}	
By the Cauchy-Schwarz inequality, Lemma \ref{lem7.2}, \eqref{3.21}, \eqref{3.22}  and \eqref{3.18}, we get
	\begin{eqnarray*}
		&&|([\widetilde{u}_{1}u_{1y}]_y,\widetilde{u}_{1y\tau})|\\
		&&\leq |(\widetilde{u}_{1}u_{1yy},\widetilde{u}_{1y\tau})|+
		|(\widetilde{u}_{1y}u_{1y},\widetilde{u}_{1y\tau})|
		\notag\\
		&&\leq \eta\|\widetilde{u}_{1y\tau}\|^2+C_\eta(\|\widetilde{u}_{1}\|^2_{L^\infty}\|\widetilde{u}_{1yy}\|^2+
		\|\widetilde{u}_{1}\|^2\|\bar{u}_{1yy}\|^2_{L^\infty})+C_\eta
		\|\widetilde{u}_{1y}\|^2\|u_{1y}\|^2_{L^\infty}
		\notag\\
		&&\leq \eta\|\widetilde{u}_{1y\tau}\|^2+C_\eta k^{\frac{1}{12}}\varepsilon^{\frac{3}{5}a-\frac{2}{5}}\mathcal{D}_{2,l,q}(\tau)
		+C_{\eta}\varepsilon^{\frac{7}{5}+\frac{1}{15}a}(\delta+\varepsilon^{a}\tau)^{-\frac{4}{3}},
	\end{eqnarray*}
	and
	\begin{eqnarray*}
		|(\frac{1}{\rho}\int_{\mathbb{R}^3}  v^{2}_{1}G_{yy}\,dv,\widetilde{u}_{1y\tau})|
		&&\leq\eta\|\widetilde{u}_{1y\tau}\|^2
		+C_{\eta}\|\langle v\rangle^{-\frac{1}{2}}f_{yy}\|^{2}+C_{\eta}\|\frac{\overline{G}_{yy}}{\sqrt{\mu}}\|^2
		\notag\\
		&&\leq\eta\|\widetilde{u}_{1y\tau}\|^2
		+C_{\eta}\|f_{yy}\|_\sigma^{2}+C_\eta k^{\frac{1}{12}}\varepsilon^{\frac{3}{5}a-\frac{2}{5}}\mathcal{D}_{2,l,q}(\tau)\\
		&&\quad+C_{\eta}\varepsilon^{\frac{7}{5}+\frac{1}{15}a}(\delta+\varepsilon^{a}\tau)^{-\frac{4}{3}}.
	\end{eqnarray*}	
	The other terms in \eqref{5.27b} can be treated similarly. We thus arrive at	
	\begin{eqnarray*}	
		\|\widetilde{u}_{1y\tau}\|^2\leq&&C\eta\|\widetilde{u}_{1y\tau}\|^2
		+C_{\eta}\|(\widetilde{\rho}_{yy},\widetilde{u}_{yy},\widetilde{\theta}_{yy},\widetilde{\phi}_{yy})\|^{2}
		+C_{\eta}\|f_{yy}\|_\sigma^{2}
		\notag\\
		&&+C_\eta k^{\frac{1}{12}}\varepsilon^{\frac{3}{5}a-\frac{2}{5}}\mathcal{D}_{2,l,q}(\tau)
		+C_{\eta}\varepsilon^{\frac{7}{5}+\frac{1}{15}a}(\delta+\varepsilon^{a}\tau)^{-\frac{4}{3}}.
	\end{eqnarray*}	
	Similar estimates also hold for $\widetilde{\rho}_{y\tau}$, $\widetilde{u}_{2y\tau}$, $\widetilde{u}_{3y\tau}$ and $\widetilde{\theta}_{y\tau}$. So we can choose
	$\eta>0$ small enough such that
	\begin{multline}
		\label{5.28b}
		\|(\widetilde{\rho}_{y\tau},\widetilde{u}_{y\tau},\widetilde{\theta}_{y\tau})\|^{2}
		\leq C\|(\widetilde{\rho}_{yy},\widetilde{u}_{yy},\widetilde{\theta}_{yy},\widetilde{\phi}_{yy})\|^{2}+C\|f_{yy}\|_{\sigma}^{2}
		\\
		+Ck^{\frac{1}{12}}\varepsilon^{\frac{3}{5}a-\frac{2}{5}}\mathcal{D}_{2,l,q}(\tau)
		+C\varepsilon^{\frac{7}{5}+\frac{1}{15}a}(\delta+\varepsilon^{a}\tau)^{-\frac{4}{3}}.
	\end{multline}	
	Multiplying \eqref{4.31} by $\widetilde{\rho}_{\tau}$, $\widetilde{u}_{1\tau}$, $\widetilde{u}_{i\tau}$ with $i=2,3$ and
	$\widetilde{\theta}_{\tau}$ respectively, then integrating the resulting equations with respect to $y$ over $\mathbb{R}$ and adding them together,
	we can prove \eqref{5.26b} holds for $|\alpha|=0$. This together with  \eqref{5.28b} gives \eqref{5.26b}.
	
	We are going to prove \eqref{4.56}. By the last equation of \eqref{3.10}, one has
	\begin{equation}
		\label{5.33b}
		-\varepsilon^{2b-2a}(\widetilde{\phi}_{yy\tau},\widetilde{\phi}_{\tau})
		=\varepsilon^{2b-2a}(\bar{\phi}_{yy\tau},\widetilde{\phi}_{\tau})+
		(\widetilde{\rho}_{\tau},\widetilde{\phi}_{\tau})
		+([\rho_{\mathrm{e}}(\bar{\phi})-\rho_{\mathrm{e}}(\phi)]_{\tau},\widetilde{\phi}_{\tau}).
	\end{equation}
	In view of Lemma \ref{lem7.2}, \eqref{3.21} and \eqref{3.22}, we obtain
	\begin{eqnarray}
		\label{4.36a}
		\varepsilon^{2b-2a}|(\bar{\phi}_{yy\tau},\widetilde{\phi}_{\tau})|
		&&\leq \eta\|\widetilde{\phi}_{\tau}\|^{2}
		+C_{\eta}\varepsilon^{4b-4a}\|\bar{\phi}_{yy\tau}\|^{2}
		\notag\\
		&&\leq \eta\|\widetilde{\phi}_{\tau}\|^{2}
		+C_{\eta}\varepsilon^{4b-4a}\varepsilon^{5a}\delta^{-3}(\delta+\varepsilon^{a}\tau)^{-2}
		\notag\\
		&&\leq  \eta\|\widetilde{\phi}_{\tau}\|^{2}
		+C_{\eta}\varepsilon^{4b+a}(\frac{1}{k}\varepsilon^{\frac{3}{5}-\frac{2}{5}a})^{-\frac{11}{3}}(\delta+\varepsilon^{a}\tau)^{-\frac{4}{3}}
		\notag\\
		&&\leq  \eta\|\widetilde{\phi}_{\tau}\|^{2}
		+C_{\eta}\varepsilon^{\frac{7}{5}+\frac{1}{15}a}(\delta+\varepsilon^{a}\tau)^{-\frac{4}{3}},
	\end{eqnarray}
where in the last inequality we have used 
	\begin{equation*}
		\varepsilon^{4b+a}(\varepsilon^{\frac{3}{5}-\frac{2}{5}a})^{-\frac{11}{3}}\leq\varepsilon^{\frac{7}{5}+\frac{1}{15}a}.
	\end{equation*}
On the other hand, we get form the first equation of \eqref{3.10} together with \eqref{4.16} and \eqref{5.15A} that
	\begin{eqnarray}
		\label{5.35b}
		|(\widetilde{\rho}_{\tau},\widetilde{\phi}_{\tau})|\leq&& \eta\|\widetilde{\phi}_{\tau}\|^{2}+C_\eta\|(\rho\widetilde{u}_{1})_{y}+\bar{u}_{1}\widetilde{\rho}_{y}+\bar{u}_{1y}\widetilde{\rho}\|^2
		\notag\\
		\leq&&\eta\|\widetilde{\phi}_{\tau}\|^{2}+C_\eta\|(\widetilde{\rho}_{y},\widetilde{u}_{y})\|^2
		+C_{\eta}\varepsilon^{\frac{7}{5}+\frac{1}{15}a}(\delta+\varepsilon^{a}\tau)^{-\frac{4}{3}},
	\end{eqnarray}
while for the last term of \eqref{5.33b}, we deduce from the Taylor's formula that
	\begin{eqnarray*}
		&&([\rho_{\mathrm{e}}(\bar{\phi})-\rho_{\mathrm{e}}(\phi)]_{\tau},\widetilde{\phi}_{\tau})\\
		=&&-(\rho'_{\mathrm{e}}(\bar{\phi})\widetilde{\phi}_{\tau},\widetilde{\phi}_{\tau})
		-([\rho'_{\mathrm{e}}(\phi)-\rho'_{\mathrm{e}}(\bar{\phi})]\phi_{\tau},\widetilde{\phi}_{\tau})
		\notag\\
		\leq&& -(\rho'_{\mathrm{e}}(\bar{\phi})\widetilde{\phi}_{\tau},\widetilde{\phi}_{\tau})+C(\eta+k^{\frac{1}{12}}\varepsilon^{\frac{3}{5}-\frac{2}{5}a})\|\widetilde{\phi}_{\tau}\|^{2}
		+C_\eta\varepsilon^{\frac{7}{5}+\frac{1}{15}a}(\delta+\varepsilon^{a}\tau)^{-\frac{4}{3}}.
	\end{eqnarray*}
	Owing to these and the fact $\rho'_{\mathrm{e}}(\bar{\phi})>c$, we can choose $\eta>0$ small enough such that
	\begin{equation}
		\label{4.38}
		\|\widetilde{\phi}_{\tau}\|^{2}+\varepsilon^{2b-2a}\|\widetilde{\phi}_{y\tau}\|^{2}
		\leq C\|(\widetilde{\rho}_{y},\widetilde{u}_{y})\|^2
		+C\varepsilon^{\frac{7}{5}+\frac{1}{15}a}(\delta+\varepsilon^{a}\tau)^{-\frac{4}{3}},
	\end{equation}
which gives the estimate \eqref{4.56} for the case $|\alpha|=0$.
	
	Similar to \eqref{5.33b}, it is clear to see that
	\begin{equation*}
		-\varepsilon^{2b-2a}(\widetilde{\phi}_{yyy\tau},\widetilde{\phi}_{y\tau})
		=\varepsilon^{2b-2a}(\bar{\phi}_{yyy\tau},\widetilde{\phi}_{y\tau})+
		(\widetilde{\rho}_{y\tau},\widetilde{\phi}_{y\tau})
		+([\rho_{\mathrm{e}}(\bar{\phi})-\rho_{\mathrm{e}}(\phi)]_{y\tau},\widetilde{\phi}_{y\tau}).
	\end{equation*}
Following the similar calculations as \eqref{4.36a} and \eqref{5.35b} respectively, we have
	$$
	\varepsilon^{2b-2a}|(\bar{\phi}_{yyy\tau},\widetilde{\phi}_{y\tau})|
	\leq \eta\|\widetilde{\phi}_{y\tau}\|^{2}
	+C_{\eta}\varepsilon^{\frac{7}{5}+\frac{1}{15}a}(\delta+\varepsilon^{a}\tau)^{-\frac{4}{3}},
	$$
	and
	\begin{multline}
	|(\widetilde{\rho}_{y\tau},\widetilde{\phi}_{y\tau})|
	\leq \eta\|\widetilde{\phi}_{y\tau}\|^{2}+C_\eta\|(\widetilde{\rho}_{yy},\widetilde{u}_{yy})\|^2
	+C_\eta k^{\frac{1}{6}}\varepsilon^{\frac{6}{5}-\frac{4}{5}a}\|(\widetilde{\rho}_{y},\widetilde{u}_{y})\|^2\\
	+C_{\eta}\varepsilon^{\frac{7}{5}+\frac{1}{15}a}(\delta+\varepsilon^{a}\tau)^{-\frac{4}{3}}.
	\end{multline}
By the Taylor's formula and \eqref{4.38}, we claim that
	\begin{eqnarray*}
		&&([\rho_{\mathrm{e}}(\bar{\phi})-\rho_{\mathrm{e}}(\phi)]_{y\tau},\widetilde{\phi}_{y\tau})\\
		=&&-(\rho'_{\mathrm{e}}(\bar{\phi})\widetilde{\phi}_{y\tau},\widetilde{\phi}_{y\tau})
		-([\rho'_{\mathrm{e}}(\bar{\phi})]_y\widetilde{\phi}_{\tau},\widetilde{\phi}_{y\tau})
		\notag\\
		&&-([\rho'_{\mathrm{e}}(\phi)-\rho'_{\mathrm{e}}(\bar{\phi})]\phi_{y\tau},\widetilde{\phi}_{y\tau})
		-([\rho'_{\mathrm{e}}(\phi)-\rho'_{\mathrm{e}}(\bar{\phi})]_y\phi_{\tau},\widetilde{\phi}_{y\tau})
		\notag\\
		\leq&& -(\rho'_{\mathrm{e}}(\bar{\phi})\widetilde{\phi}_{y\tau},\widetilde{\phi}_{y\tau})+C(\eta+k^{\frac{1}{12}}\varepsilon^{\frac{3}{5}-\frac{2}{5}a})\|\widetilde{\phi}_{y\tau}\|^{2}
		\notag\\
		&&+C_\eta k^{\frac{1}{6}}\varepsilon^{\frac{6}{5}-\frac{4}{5}a}\|\widetilde{\phi}_{\tau}\|^2+C_\eta\varepsilon^{\frac{7}{5}+\frac{1}{15}a}(\delta+\varepsilon^{a}\tau)^{-\frac{4}{3}}.
	\end{eqnarray*}
At the end, from  the above bounds and the fact $\rho'_{\mathrm{e}}(\bar{\phi})>c$, we can choose $\eta>0$ small enough such that
	\begin{multline*}
		\|\widetilde{\phi}_{y\tau}\|^{2}+\varepsilon^{2b-2a}\|\widetilde{\phi}_{yy\tau}\|^{2}\\
		\leq C\|(\widetilde{\rho}_{yy},\widetilde{u}_{yy})\|^2+Ck^{\frac{1}{6}}\varepsilon^{\frac{6}{5}-\frac{4}{5}a}\|(\widetilde{\rho}_{y},\widetilde{u}_{y})\|^2
		+C\varepsilon^{\frac{7}{5}+\frac{1}{15}a}(\delta+\varepsilon^{a}\tau)^{-\frac{4}{3}},
	\end{multline*}
which gives the estimate \eqref{4.56} for the case $|\alpha|=1$. Combining this and \eqref{4.38} together gives \eqref{4.56}, 
we thus complete the proof of Lemma \ref{lem.5.2a}.
\end{proof}
For deducing \eqref{4.30} in the proof of Lemma \ref{lem.5.1A} above, we have used the following estimate.
\begin{lemma}\label{lem.5.2A}
It holds that
	\begin{align}
		\label{5.27A}
		-\frac{3}{2}\int_{\mathbb{R}}\rho\widetilde{u}_{1}\widetilde{\phi}_{y} \, dy
		&\leq-\frac{3}{4}\varepsilon^{2b-2a}\frac{d}{d\tau}\|\widetilde{\phi}_{y}\|^{2}
		-\frac{3}{4}\frac{d}{d\tau}\big(\widetilde{\phi}^{2},\rho'_{\mathrm{e}}(\bar{\phi})\big)
		-\frac{1}{2}\frac{d}{d\tau}\big(\widetilde{\phi}^{3},\rho''_{\mathrm{e}}(\bar{\phi})\big)
		\notag\\
		&\quad+C(\eta+k)\varepsilon^{1-a}\|\widetilde{\phi}_{y}\|^2+C_{\eta}k^{\frac{1}{12}}\varepsilon^{\frac{3}{5}-\frac{2}{5}a}\mathcal{D}_{2,l,q}(\tau)\notag\\
		&\quad+C_{\eta}\varepsilon^{\frac{7}{5}+\frac{1}{15}a}(\delta+\varepsilon^{a}\tau)^{-\frac{4}{3}}.
	\end{align}
	
\end{lemma}
\begin{proof}
	Recall the last term of \eqref{4.9}. By integration by parts, from the first and last equations of \eqref{3.10}, we get
	\begin{align}
		\label{4.22}
		-\frac{3}{2}\int_{\mathbb{R}}\rho\widetilde{u}_{1}\widetilde{\phi}_{y}\,dy
		&=\frac{3}{2}\int_{\mathbb{R}}(\rho\widetilde{u}_{1})_{y}\widetilde{\phi}\,dy
		=-\frac{3}{2}\int_{\mathbb{R}}[\widetilde{\rho}_{\tau}+(\bar{u}_{1}\widetilde{\rho})_{y}]
		\widetilde{\phi}\,dy
		\notag\\
		&=\frac{3}{2}\varepsilon^{2b-2a}\big(\widetilde{\phi},\widetilde{\phi}_{yy\tau}\big)
		+\frac{3}{2}\varepsilon^{2b-2a}\big(\widetilde{\phi},\bar{\phi}_{yy\tau}\big)+I_{2}+I_{3}.
	\end{align}
	Here $I_{2}$ and $I_{3}$ are given as
	$$
	I_{2}=\frac{3}{2}(\widetilde{\phi},[\rho_{\mathrm{e}}(\bar{\phi})-\rho_{\mathrm{e}}(\phi)]_{\tau}), \quad
	I_{3}=-\frac{3}{2}(\widetilde{\phi},[\widetilde{\rho}\bar{u}_{1}]_{y}).
	$$
For the first term in \eqref{4.22}, we have from an integration by parts that
	\begin{equation}
		\label{4.25aa}
		\frac{3}{2}\varepsilon^{2b-2a}(\widetilde{\phi},\widetilde{\phi}_{yy\tau})
		=-\frac{3}{4}\varepsilon^{2b-2a}\frac{d}{d\tau}\|\widetilde{\phi}_{y}\|^{2}.
	\end{equation}
For the second term in \eqref{4.22}, according to  Lemma \ref{lem7.2}, \eqref{3.21} and \eqref{3.22}, we deduce that
	\begin{eqnarray}
		\label{4.23}
		&&\varepsilon^{2b-2a}|(\widetilde{\phi},\bar{\phi}_{yy\tau})|\notag\\
		&&\leq C\varepsilon^{2b-2a}\|\widetilde{\phi}\|_{L_{y}^{\infty}}\|\bar{\phi}_{yy\tau}\|_{L^{1}}
		\notag\\
		&&\leq C\varepsilon^{2b-2a}
		\|\widetilde{\phi}\|^{\frac{1}{2}}\|\widetilde{\phi}_{y}\|^{\frac{1}{2}}\varepsilon^{2a}\delta^{-1}(\delta+\varepsilon^{a}\tau)^{-1}
		\notag\\
		&&\leq \eta(\varepsilon^{\frac{1-a}{4}}\|\widetilde{\phi}_{y}\|^{\frac{1}{2}})^{4}
		+C_{\eta}\big\{\varepsilon^{\frac{a-1}{4}}\varepsilon^{2b-2a}\|\widetilde{\phi}\|^{\frac{1}{2}}
		\varepsilon^{2a}\delta^{-1}(\delta+\varepsilon^{a}\tau)^{-1}\big\}^{\frac{4}{3}}
		\notag\\
		&&\leq \eta\varepsilon^{1-a}\|\widetilde{\phi}_{y}\|^{2}
		+C_{\eta}\big\{\varepsilon^{\frac{a-1}{4}+2b}(k^{\frac{1}{6}}\varepsilon^{\frac{6}{5}-\frac{4}{5}a})^{\frac{1}{4}}
		k\varepsilon^{-\frac{3}{5}+\frac{2}{5}a}\big\}^{\frac{4}{3}}(\delta+\varepsilon^{a}\tau)^{-\frac{4}{3}}
		\notag\\
		&&\leq \eta\varepsilon^{1-a}\|\widetilde{\phi}_{y}\|^{2}+C_{\eta}\varepsilon^{\frac{7}{5}+\frac{1}{15}a}(\delta+\varepsilon^{a}\tau)^{-\frac{4}{3}}.
	\end{eqnarray}
	Here we have used the inequality
	\begin{equation*}
		\big\{\varepsilon^{\frac{a-1}{4}+2b}(k^{\frac{1}{6}}\varepsilon^{\frac{6}{5}-\frac{4}{5}a})^{\frac{1}{4}}
		k\varepsilon^{-\frac{3}{5}+\frac{2}{5}a}\big\}^{\frac{4}{3}}
		=\big\{k^{\frac{25}{24}}\varepsilon^{\frac{9}{20}a-\frac{11}{20}+2b}\big\}^{\frac{4}{3}}\leq k^{\frac{25}{18}}\varepsilon^{\frac{7}{5}+\frac{1}{15}a},
	\end{equation*}
	by requiring that
	\begin{equation}
		\label{4.25a}
		\varepsilon^{b}\leq \varepsilon^{\frac{4}{5}-\frac{1}{5}a}\quad \Leftrightarrow\quad  \quad a\geq4-5b.
	\end{equation}

	To compute the term $I_{2}$ in \eqref{4.22},
	we first have from the Taylor's formula with an integral remainder that
	\begin{equation}
		\label{4.24}
		\rho_{\mathrm{e}}(\bar{\phi})-\rho_{\mathrm{e}}(\phi)=-\rho'_{\mathrm{e}}(\bar{\phi})\widetilde{\phi}-\frac{1}{2}\rho''_{\mathrm{e}}(\bar{\phi})\widetilde{\phi}^{2}
		\underbrace{-\int^{\phi}_{\bar{\phi}}\frac{(\varrho-\phi)^{2}}{2}\rho'''_{e}(\varrho)\,d\varrho}_{J_{1}}.
	\end{equation}
	Here we expand $\rho_{\mathrm{e}}(\phi)$ around the asymptotic profile up to the third order so that we can observe some new
	cancellations and obtain the higher order nonlinear terms. In fact, from \eqref{4.24}, we can write
	$$
	I_{2}=-\frac{3}{2}(\widetilde{\phi},[\rho'_{\mathrm{e}}(\bar{\phi})\widetilde{\phi}]_{\tau})
	-\frac{3}{4}(\widetilde{\phi},[\rho''_{\mathrm{e}}(\bar{\phi})\widetilde{\phi}^{2}]_{\tau})
	+\frac{3}{2}(\widetilde{\phi},J_{1\tau}).
	$$
	By a direct calculation and using $\bar{\rho}_{\tau}=-[\bar{\rho}\bar{u}_{1}]_{y}$, one has the following identities
	\begin{eqnarray*}
		-\frac{3}{2}(\widetilde{\phi},[\rho'_{\mathrm{e}}(\bar{\phi})\widetilde{\phi}]_{\tau})=&&
		-\frac{3}{4}\frac{d}{d\tau}(\widetilde{\phi}^{2},\rho'_{\mathrm{e}}(\bar{\phi}))
		-\frac{3}{4}(\widetilde{\phi}^{2},[\rho'_{\mathrm{e}}(\bar{\phi})]_{\tau})
		\notag\\
		=&&-\frac{3}{4}\frac{d}{d\tau}(\widetilde{\phi}^{2},\rho'_{\mathrm{e}}(\bar{\phi}))+\underbrace{\frac{3}{4}(\widetilde{\phi}^{2},
			\rho''_{\mathrm{e}}(\bar{\phi})\frac{d\bar{\phi}}{d\bar{\rho}}[\bar{\rho}\bar{u}_{1}]_{y})}_{J_{2}}.
	\end{eqnarray*}
	Here it should be noted that $J_{2}$ cannot be directly controlled for the time being,
	and its estimate should be postponed to the subsequent estimates on $I_{3}$ by an exact cancellation with other terms.
Observe that
	\begin{multline*}
		-\frac{3}{4}(\widetilde{\phi},[\rho''_{\mathrm{e}}(\bar{\phi})\widetilde{\phi}^{2}]_{\tau})
		=-\frac{1}{2}\frac{d}{d\tau}(\widetilde{\phi}^{3},\rho''_{\mathrm{e}}(\bar{\phi}))+\frac{1}{4}(\widetilde{\phi}^{3},
		\rho'''_{e}(\bar{\phi})\frac{d\bar{\phi}}{d\bar{\rho}}[\bar{\rho}\bar{u}_{1}]_{y})
		\\
		\leq-\frac{1}{2}\frac{d}{d\tau}(\widetilde{\phi}^{3},\rho''_{\mathrm{e}}(\bar{\phi}))+ \eta\varepsilon^{1-a}\|\widetilde{\phi}_{y}\|^{2}
		+C_{\eta}\varepsilon^{\frac{7}{5}+\frac{1}{15}a}(\delta+\varepsilon^{a}\tau)^{-\frac{4}{3}},
	\end{multline*}
	where Lemma \ref{lem7.2}, \eqref{3.21} and \eqref{3.22} have been used such that
	\begin{eqnarray}
		\label{4.25}
		|(\widetilde{\phi}^{3},
		\rho'''_{e}(\bar{\phi})\frac{d\bar{\phi}}{d\bar{\rho}}[\bar{\rho}\bar{u}_{1}]_{y})|
		&&\leq C\|(\bar{\rho}_{y},\bar{u}_{1y})\|_{L_{y}^{\infty}}
		\|\widetilde{\phi}\|_{L_{y}^{\infty}}\|\widetilde{\phi}\|^{2}
		\notag\\
		&&\leq \eta(\varepsilon^{\frac{1-a}{4}}\|\widetilde{\phi}_{y}\|^{\frac{1}{2}})^{4}
		+C_{\eta}\big\{\varepsilon^{\frac{a-1}{4}}\varepsilon^{a}(\delta+\varepsilon^{a}\tau)^{-1}\|\widetilde{\phi}\|^{\frac{5}{2}}\big\}^{\frac{4}{3}}
		\notag\\
		&&\leq \eta\varepsilon^{1-a}\|\widetilde{\phi}_{y}\|^{2}
		+C_{\eta}\varepsilon^{\frac{7}{5}+\frac{1}{15}a}(\delta+\varepsilon^{a}\tau)^{-\frac{4}{3}}.
	\end{eqnarray}
To estimate $\frac{3}{2}\big(\widetilde{\phi},J_{1\tau}\big)$, we first observe that
	\begin{equation*}
		J_{1}\sim\widetilde{\phi}^{3},\quad J_{1\tau}=\phi_{\tau}\int^{\phi}_{\bar{\phi}}(\varrho-\phi)\rho'''_{e}(\varrho)\,d\varrho
		+\frac{1}{2}\widetilde{\phi}^{2}\bar{\phi}_{\tau}\rho'''_{e}(\bar{\phi})\sim\phi_{\tau}\widetilde{\phi}^{2}
		+\bar{\phi}_{\tau}\widetilde{\phi}^{2}.
	\end{equation*}
	Using this, \eqref{4.38} and a similar calculation as \eqref{4.25}, we have
	\begin{eqnarray}
		\label{4.27}
		|(\widetilde{\phi},J_{1\tau})|&&\leq C\|\widetilde{\phi}\|_{L_{y}^{\infty}}^{2}\|\widetilde{\phi}\|\|\widetilde{\phi}_{\tau}\|
		+C\|\widetilde{\phi}\|_{L_{y}^{\infty}}\|\bar{\phi}_{\tau}\|_{L_{y}^{\infty}}\|\widetilde{\phi}\|^{2}
		\notag\\
		&&\leq C\eta\varepsilon^{1-a}\|\widetilde{\phi}_{y}\|^{2}+C_{\eta}k^{\frac{1}{12}}\varepsilon^{\frac{3}{5}-\frac{2}{5}a}\mathcal{D}_{2,l,q}(\tau)\notag\\
		&&\quad+C_{\eta}\varepsilon^{\frac{7}{5}+\frac{1}{15}a}(\delta+\varepsilon^{a}\tau)^{-\frac{4}{3}}.
	\end{eqnarray}
	Therefore, we can conclude from the above estimates that
	\begin{eqnarray*}
		I_{2}\leq &&-\frac{3}{4}\frac{d}{d\tau}\big(\widetilde{\phi}^{2},\rho'_{\mathrm{e}}(\bar{\phi})\big)
		-\frac{1}{2}\frac{d}{d\tau}\big(\widetilde{\phi}^{3},\rho''_{\mathrm{e}}(\bar{\phi})\big)+J_{2}
		+C\eta\varepsilon^{1-a}\|\widetilde{\phi}_{y}\|^{2}
		\nonumber\\
		&&+C_{\eta}k^{\frac{1}{12}}\varepsilon^{\frac{3}{5}-\frac{2}{5}a}\mathcal{D}_{2,l,q}(\tau)
		+C_{\eta}\varepsilon^{\frac{7}{5}+\frac{1}{15}a}(\delta+\varepsilon^{a}\tau)^{-\frac{4}{3}}.
	\end{eqnarray*}

	As for the term $I_{3}$ in \eqref{4.22}, we first use the last equation of \eqref{3.10} to obtain
	\begin{eqnarray*}
		I_{3}&&=-\frac{3}{2}\big(\widetilde{\phi},[\widetilde{\rho}\bar{u}_{1}]_{y}\big)=
		-\frac{3}{2}(\widetilde{\phi},\widetilde{\rho}_{y}\bar{u}_{1})
		-\frac{3}{2}(\widetilde{\phi},\widetilde{\rho}\bar{u}_{1y})
		\notag\\
		&&=\frac{3}{2}\Big(\widetilde{\phi},\big[\varepsilon^{2b-2a}(\partial^{3}_{y}\widetilde{\phi}
		+\partial^{3}_{y}\bar{\phi})+(\rho_{\mathrm{e}}(\bar{\phi})-\rho_{\mathrm{e}}(\phi))_{y}\big ]\bar{u}_{1}\Big)\notag\\
		&&\quad+\frac{3}{2}\Big(\widetilde{\phi},\big[\varepsilon^{2b-2a}(\partial^{2}_{y}\widetilde{\phi}+\partial^{2}_{y}\bar{\phi})
		+\rho_{\mathrm{e}}(\bar{\phi})-\rho_{\mathrm{e}}(\phi)\big]
		\bar{u}_{1y}\Big)
		\notag\\
		&&=\frac{3}{4}\varepsilon^{2b-2a}(\bar{u}_{1y},\widetilde{\phi}_{y}^{2})-\frac{3}{2}\varepsilon^{2b-2a}
		(\widetilde{\phi}_{y},\partial^{2}_{y}\bar{\phi}\bar{u}_{1})
		+\underbrace{\frac{3}{2}\big(\widetilde{\phi},[\rho_{\mathrm{e}}(\bar{\phi})-\rho_{\mathrm{e}}(\phi)]_{y}\bar{u}_{1}\big)}_{J_{3}}\notag\\
		&&\quad+\underbrace{\frac{3}{2}\big(\widetilde{\phi},[\rho_{\mathrm{e}}(\bar{\phi})-\rho_{\mathrm{e}}(\phi)]\bar{u}_{1y}\big)}_{J_{4}},
	\end{eqnarray*}
	where in the last identity we have used the following identities
	$$
	\frac{3}{2}\varepsilon^{2b-2a}\big[(\widetilde{\phi},\widetilde{\phi}_{yyy}\bar{u}_{1})
	+(\widetilde{\phi},\widetilde{\phi}_{yy}\bar{u}_{1y})\big]
	=\frac{3}{4}\varepsilon^{2b-2a}(\bar{u}_{1y},\widetilde{\phi}_{y}^{2})
	$$
	and
	$$
	\frac{3}{2}\varepsilon^{2b-2a}\big[(\widetilde{\phi},\bar{\phi}_{yyy}\bar{u}_{1})
	+(\widetilde{\phi},\bar{\phi}_{yy}\bar{u}_{1y})\big]
	=-\frac{3}{2}\varepsilon^{2b-2a}(\widetilde{\phi}_{y},\bar{\phi}_{yy}\bar{u}_{1}).
	$$
	By using Lemma \ref{lem7.2}, \eqref{3.21} and $a\geq4-5b$ in \eqref{4.25a}, one gets
	\begin{eqnarray*}
		\frac{3}{4}\varepsilon^{2b-2a}|(\bar{u}_{1y},\widetilde{\phi}_{y}^{2})|
		&&\leq C\varepsilon^{2b-2a}\|\bar{u}_{1y}\|_{L_{y}^{\infty}}\|\widetilde{\phi}_{y}\|^{2}
		\notag\\
		&&\leq Ck\varepsilon^{2b-\frac{8}{5}+\frac{2}{5}a}\varepsilon^{1-a}\|\widetilde{\phi}_{y}\|^{2}
		\leq Ck\varepsilon^{1-a}\|\widetilde{\phi}_{y}\|^{2}.
	\end{eqnarray*}
	The second term in $I_{3}$ can  be handed in the same way as in \eqref{4.23}, which is bounded by
	\begin{eqnarray*}
		\frac{3}{2}\varepsilon^{2b-2a}|(\widetilde{\phi}_{y},\bar{\phi}_{yy}\bar{u}_{1})|
		&&=\frac{3}{2}\varepsilon^{2b-2a}|(\widetilde{\phi},[\bar{\phi}_{yy}\bar{u}_{1}]_{y})|
		\\
		&&\leq \eta\varepsilon^{1-a}\|\widetilde{\phi}_{y}\|^{2}
		+C_{\eta}\varepsilon^{\frac{7}{5}+\frac{1}{15}a}(\delta+\varepsilon^{a}\tau)^{-\frac{4}{3}}.
	\end{eqnarray*}
	
	In order to estimate $J_{3}$ and $J_{4}$ in $I_{3}$, we first use \eqref{4.24} to expand them as
	\begin{eqnarray*}
		J_{3}=&&-\frac{3}{2}(\widetilde{\phi},[\rho'_{\mathrm{e}}(\bar{\phi})\widetilde{\phi}]_{y}\bar{u}_{1})
		-\frac{3}{4}(\widetilde{\phi},[\rho''_{\mathrm{e}}(\bar{\phi})\widetilde{\phi}^{2}]_{y}\bar{u}_{1})
		+\frac{3}{2}(\widetilde{\phi},J_{1y}\bar{u}_{1})
		\\
		=&&-\frac{3}{4}(\widetilde{\phi}^{2},\rho''_{\mathrm{e}}(\bar{\phi})\frac{d\bar{\phi}}{d\bar{\rho}}\bar{\rho}_{y}\bar{u}_{1})
		+\frac{3}{4}(\widetilde{\phi}^{2},\rho'_{\mathrm{e}}(\bar{\phi})\bar{u}_{1y})
		\\
		&&-\frac{1}{4}(\widetilde{\phi}^{3},\rho'''_{e}(\bar{\phi})\frac{d\bar{\phi}}{d\bar{\rho}}\bar{\rho}_{y}\bar{u}_{1})
		+\frac{1}{2}(\widetilde{\phi}^{3},\rho''_{\mathrm{e}}(\bar{\phi})\bar{u}_{1y})
		+\frac{3}{2}(\widetilde{\phi},J_{1y}\bar{u}_{1}),
	\end{eqnarray*}
	and
	\begin{equation*}
		J_{4}=-\frac{3}{2}(\widetilde{\phi}^{2},\rho'_{\mathrm{e}}(\bar{\phi})\bar{u}_{1y})
		-\frac{3}{4}(\widetilde{\phi}^{3},\rho''_{\mathrm{e}}(\bar{\phi})\bar{u}_{1y})
		+\frac{3}{2}(\widetilde{\phi},J_{1}\bar{u}_{1y}).
	\end{equation*}
	Summing up the above estimates, we obtain
	\begin{eqnarray*}
		&&J_{2}+J_{3}+J_{4}\\
		&&=\underbrace{\frac{3}{4}(\widetilde{\phi}^{2},\rho''_{\mathrm{e}}(\bar{\phi})\frac{d\bar{\phi}}{d\bar{\rho}}\bar{\rho}\bar{u}_{1y})
			-\frac{3}{4}(\widetilde{\phi}^{2},\rho'_{\mathrm{e}}(\bar{\phi})\bar{u}_{1y})}_{J_{5}}
		\notag\\
		&&\quad+\underbrace{\frac{3}{2}(\widetilde{\phi},J_{1y}\bar{u}_{1}\big)+\frac{3}{2}\big(\widetilde{\phi},J_{1}\bar{u}_{1y})
			-\frac{1}{4}(\widetilde{\phi}^{3},\rho'''_{\mathrm{e}}(\bar{\phi})\frac{d\bar{\phi}}{d\bar{\rho}}\bar{\rho}_{y}\bar{u}_{1})
			-\frac{1}{4}(\widetilde{\phi}^{3},\rho''_{\mathrm{e}}(\bar{\phi})\bar{u}_{1y})}_{J_{6}}.
	\end{eqnarray*}
	In view of the assumptions $(\mathcal{A}_{2})$, $(\mathcal{A}_{3})$, $\bar{\rho}=\rho_{\mathrm{e}}(\bar{\phi})$ and $\bar{u}_{1y}>0$, we have
	\begin{equation}
		\label{4.34b}
		J_{5}=\frac{3}{4}(\widetilde{\phi}^{2}\bar{u}_{1y},\frac{\rho''_{\mathrm{e}}(\bar{\phi})\rho_{\mathrm{e}}(\bar{\phi})
			-[\rho'_{\mathrm{e}}(\bar{\phi})]^{2}}{\rho'_{\mathrm{e}}(\bar{\phi})})\leq 0.
	\end{equation}
Note that the terms in $J_{6}$ can be handed similarly as \eqref{4.27} and \eqref{4.25}. It then follows that
	$$
	J_{6}\leq C\eta\varepsilon^{1-a}\|\widetilde{\phi}_{y}\|^{2}
	+C_{\eta}k^{\frac{1}{12}}\varepsilon^{\frac{3}{5}-\frac{2}{5}a}\mathcal{D}_{2,l,q}(\tau)
	+C_{\eta}\varepsilon^{\frac{7}{5}+\frac{1}{15}a}(\delta+\varepsilon^{a}\tau)^{-\frac{4}{3}}.
	$$
	Combining the estimates $I_{2}$ and $I_{3}$ gives rise to
	\begin{eqnarray*}
		I_{2}+I_{3}\leq&&
		-\frac{3}{4}\frac{d}{d\tau}\big(\widetilde{\phi}^{2},\rho'_{\mathrm{e}}(\bar{\phi})\big)
		-\frac{1}{2}\frac{d}{d\tau}\big(\widetilde{\phi}^{3},\rho''_{\mathrm{e}}(\bar{\phi})\big)
		+C(\eta+k)\varepsilon^{1-a}\|\widetilde{\phi}_{y}\|^2
		\notag\\
		&&+C_{\eta}k^{\frac{1}{12}}\varepsilon^{\frac{3}{5}-\frac{2}{5}a}\mathcal{D}_{2,l,q}(\tau)
		+C_{\eta}\varepsilon^{\frac{7}{5}+\frac{1}{15}a}(\delta+\varepsilon^{a}\tau)^{-\frac{4}{3}}.
	\end{eqnarray*}

Consequently, plugging this, \eqref{4.25aa} and \eqref{4.23} into \eqref{4.22},  the desired estimate \eqref{5.27A} follows.
This hence completes the proof of Lemma \ref{lem.5.2A}.
\end{proof}

For deducing \eqref{4.36} in the proof of Lemma \ref{lem.5.1A} above, the following estimate has been used.

\begin{lemma}\label{lem.5.3A}
It holds that
	\begin{eqnarray}
		\label{5.37A}
		-\varepsilon^{1-a}(\widetilde{\rho}_{y},\widetilde{\phi}_{y})
		\leq&& -\varepsilon^{1-a}\varepsilon^{2b-2a}\|\widetilde{\phi}_{yy}\|^{2}-c\varepsilon^{1-a}\|\widetilde{\phi}_{y}\|^{2}
		\notag\\
		&&+Ck^{\frac{1}{12}}\varepsilon^{\frac{3}{5}-\frac{2}{5}a}\mathcal{D}_{2,l,q}(\tau)
		+C\varepsilon^{\frac{7}{5}+\frac{1}{15}a}(\delta+\varepsilon^{a}\tau)^{-\frac{4}{3}}.
	\end{eqnarray}	
\end{lemma}
\begin{proof}
Differentiating the last equation of \eqref{3.10} with respect to $y$ and taking the inner product of the resulting equations with $\widetilde{\phi}_{y}$, we have
	\begin{equation*}
		-(\widetilde{\rho}_{y},\widetilde{\phi}_{y})=\varepsilon^{2b-2a}\big[(\widetilde{\phi}_{yyy},\widetilde{\phi}_{y})
		+(\bar{\phi}_{yyy},\widetilde{\phi}_{y})\big]
		+([\rho_{\mathrm{e}}(\bar{\phi})-\rho_{\mathrm{e}}(\phi)]_{y},\widetilde{\phi}_{y}).
	\end{equation*}
	Performing the similar calculations as \eqref{4.36a}, we obtain
	\begin{equation*}
		\varepsilon^{1-a}\varepsilon^{2b-2a}|(\bar{\phi}_{yyy},\widetilde{\phi}_{y})|
		\leq \eta\varepsilon^{1-a}\|\widetilde{\phi}_{y}\|^{2}
		+C_{\eta}\varepsilon^{\frac{7}{5}+\frac{1}{15}a}(\delta+\varepsilon^{a}\tau)^{-\frac{4}{3}}.
	\end{equation*}
	Using the Taylor's formula with an integral remainder, we obtain
	\begin{equation}
		\label{4.34}
		\rho_{\mathrm{e}}(\bar{\phi})-\rho_{\mathrm{e}}(\phi)=-\rho'_{\mathrm{e}}(\bar{\phi})\widetilde{\phi}
		\underbrace{-\int^{\phi}_{\bar{\phi}}(\varrho-\phi)\rho''_{\mathrm{e}}(\varrho)\,d\varrho}_{J_{7}},
	\end{equation}
which yields immediately that
	\begin{equation}
		\label{4.35}
		J_{7}\sim\widetilde{\phi}^{2},
		\quad J_{7y}=\phi_{y}\int^{\phi}_{\bar{\phi}}\rho''_{\mathrm{e}}(\varrho)\,d\varrho
		-\bar{\phi}_{y}\widetilde{\phi}\rho''_{\mathrm{e}}(\bar{\phi})\sim \phi_{y}\widetilde{\phi}
		-\bar{\phi}_{y}\widetilde{\phi}.
	\end{equation}
In view of \eqref{4.35}, Lemma \ref{lem7.2}, \eqref{3.21} and \eqref{3.22}, we get
	\begin{eqnarray*}
		&&\varepsilon^{1-a}(|(\rho''_{\mathrm{e}}(\bar{\phi})\bar{\phi}_{y}\widetilde{\phi},\widetilde{\phi}_{y})|+|(J_{7y},\widetilde{\phi}_{y})|)
		\notag\\
		&&\leq C\varepsilon^{1-a}(\|\bar{\phi}_{y}\|_{L_{y}^{\infty}}\|\widetilde{\phi}\|\|\widetilde{\phi}_{y}\|
		+\|\widetilde{\phi}\|_{L_{y}^{\infty}}\|\widetilde{\phi}_{y}\|^{2})
		\notag\\
		&&\leq Ck^{\frac{1}{12}}\varepsilon^{\frac{3}{5}-\frac{2}{5}a}\mathcal{D}_{2,l,q}(\tau)
		+C\varepsilon^{\frac{7}{5}+\frac{1}{15}a}(\delta+\varepsilon^{a}\tau)^{-\frac{4}{3}}.
	\end{eqnarray*}
Similarly, the following estimate holds
	\begin{eqnarray*}
		&&\varepsilon^{1-a}\big([\rho_{\mathrm{e}}(\bar{\phi})-\rho_{\mathrm{e}}(\phi)]_{y},\widetilde{\phi}_{y}\big)
		\notag\\
		&&=-\varepsilon^{1-a}\big[(\rho'_{\mathrm{e}}(\bar{\phi})\widetilde{\phi}_{y},\widetilde{\phi}_{y})
		+(\rho''_{\mathrm{e}}(\bar{\phi})\bar{\phi}_{y}\widetilde{\phi},\widetilde{\phi}_{y})-(J_{7y},\widetilde{\phi}_{y})\big]
		\notag\\
		&&\leq-\varepsilon^{1-a}(\rho'_{\mathrm{e}}(\bar{\phi})\widetilde{\phi}_{y},\widetilde{\phi}_{y})
		+Ck^{\frac{1}{12}}\varepsilon^{\frac{3}{5}-\frac{2}{5}a}\mathcal{D}_{2,l,q}(\tau)
		+C\varepsilon^{\frac{7}{5}+\frac{1}{15}a}(\delta+\varepsilon^{a}\tau)^{-\frac{4}{3}}.
	\end{eqnarray*}
Therefore,	collecting all the above estimates and the fact that $\rho'_{\mathrm{e}}(\bar{\phi})>c$, we obtain \eqref{5.37A} 
by taking $\eta>0$ small enough. This hence completes the proof of Lemma \ref{lem.5.3A}.
\end{proof}

\subsection{Zeroth order estimates on non-fluid part}
In what follows we are devoted to completing the zeroth order energy estimates for the microscopic component $f$
by using the properties of the linearized operator.
\begin{lemma}\label{lem.5.4A}
	Under the conditions listed in Lemma \ref{lem.5.1A}, it holds that
	\begin{eqnarray}
		\label{4.42}
		&&\frac{1}{2}\frac{d}{d\tau}\|e^{\frac{\phi}{2}}f\|^{2}+c\varepsilon^{a-1}\|f\|_{\sigma}^{2}\notag\\
		\leq&& C\varepsilon^{1-a}(\|(\widetilde{u}_{y},\widetilde{\theta}_{y})\|^{2}+\|f_{y}\|_{\sigma}^{2})
		+C(\eta_{0}+k^{\frac{1}{12}}\varepsilon^{\frac{3}{5}-\frac{2}{5}a})\mathcal{D}_{2,l,q}(\tau)
		\notag\\
		&&+C\varepsilon^{\frac{7}{5}+\frac{1}{15}a}(\delta+\varepsilon^{a}\tau)^{-\frac{4}{3}}
		+Cq_3(\tau)\mathcal{H}_{2,l,q}(\tau),
	\end{eqnarray}
	where the function $\mathcal{H}_{2,l,q}(\tau)$ is defined by \eqref{4.60}.
\end{lemma}
\begin{proof}
	By taking the inner product of equation \eqref{3.7} with $e^{\phi} f$ over $\mathbb{R}\times{\mathbb R}^3$, one gets
	\begin{eqnarray}
		\label{4.40}
		&&\frac{1}{2}\frac{d}{d\tau}(f,e^{\phi} f)-\frac{1}{2}(f,e^{\phi}\phi_{\tau}f)+(v_{1}f_{y}+\frac{v_{1}}{2}\phi_{y} f,e^{\phi} f)
		-\varepsilon^{a-1}(\mathcal{L}f,e^{\phi} f)
		\notag\\
		&&=\varepsilon^{a-1}(\Gamma(f,\frac{M-\mu}{\sqrt{\mu}})+
		\Gamma(\frac{M-\mu}{\sqrt{\mu}},f),e^{\phi} f)+\varepsilon^{a-1}(\Gamma(\frac{G}{\sqrt{\mu}},\frac{G}{\sqrt{\mu}}),e^{\phi} f)
		\notag\\
		&&\quad+(\frac{P_{0}(v_{1}\sqrt{\mu}f_{y})}{\sqrt{\mu}},e^{\phi} f)-(\frac{1}{\sqrt{\mu}}P_{1}\{v_{1}M(\frac{|v-u|^{2}
			\widetilde{\theta}_{y}}{2R\theta^{2}}+\frac{(v-u)\cdot\widetilde{u}_{y}}{R\theta})\},e^{\phi} f)
		\notag\\
		&&\quad \quad\quad+(\frac{\phi_{y}\partial_{v_{1}}\overline{G}}{\sqrt{\mu}},e^{\phi} f)-(\frac{P_{1}(v_{1}\overline{G}_{y})}{\sqrt{\mu}},e^{\phi} f)-(\frac{\overline{G}_{\tau}}{\sqrt{\mu}},e^{\phi} f).
	\end{eqnarray}
	Here we have used the fact that
	$$
	-(\phi_{y}\partial_{v_{1}}f,e^{\phi} f)=-\frac{1}{2}(\phi_{y}e^{\phi},\partial_{v_{1}}[f^2])=0.
	$$
	
	We estimate each term in \eqref{4.40}. The second term in \eqref{4.40} is the main difficulty. We can use \eqref{4.45a},
	\eqref{3.16}, \eqref{3.14} and \eqref{4.60} to bound
	\begin{eqnarray}
		\label{4.41}
		|(f,e^{\phi}\phi_{\tau}f)|
		&&\leq \eta\varepsilon^{a-1}\|\langle v\rangle^{-\frac{1}{2}}f\|^2
		+C_{\eta}\varepsilon^{1-a}\|\phi_\tau\|_{L^\infty_y}^2\|\langle v\rangle^{\frac{1}{2}}f\|^2
		\notag\\
		&&\leq C\eta\varepsilon^{a-1}\|f\|_{\sigma}^{2}+C_{\eta}q_3(\tau)\|\langle v\rangle^{\frac{1}{2}}f\|^{2}
		\notag\\
		&&\leq C\eta\varepsilon^{a-1}\|f\|_{\sigma}^{2}+C_{\eta}q_3(\tau)\mathcal{H}_{2,l,q}(\tau).
	\end{eqnarray}
	For the third term on the left hand side of \eqref{4.40}, we observe that
	\begin{equation*}
		(v_{1}f_{y}+\frac{v_{1}}{2}\phi_{y} f,e^{\phi} f)=(v_{1}[e^{\frac{\phi}{2}}f]_{y},e^{\frac{\phi}{2}} f)=0,
	\end{equation*}
while for the last term on the left hand side of \eqref{4.40}, combining \eqref{3.9} and \eqref{4.45a} gives
	$$
	-\varepsilon^{a-1}(\mathcal{L} f,e^{\phi}f)\geq c\varepsilon^{a-1}\| f\|^2_{\sigma}.
	$$
	Next, we go to bound the right hand side of  \eqref{4.40}. It follows from \eqref{7.10} and \eqref{7.19}  that
	\begin{eqnarray*}
		&&\varepsilon^{a-1}|(\Gamma(f,\frac{M-\mu}{\sqrt{\mu}})+
		\Gamma(\frac{M-\mu}{\sqrt{\mu}},f),e^{\phi} f)|+\varepsilon^{a-1}|(\Gamma(\frac{G}{\sqrt{\mu}},\frac{G}{\sqrt{\mu}}),e^{\phi} f)|
		\notag\\
		&&\leq C\eta\varepsilon^{a-1}\|f\|^{2}_{\sigma}
		+C_{\eta}(\eta_{0}+k^{\frac{1}{12}}\varepsilon^{\frac{3}{5}-\frac{2}{5}a})\mathcal{D}_{2,l,q}(\tau)
		+C_{\eta}\varepsilon^{\frac{7}{5}+\frac{1}{15}a}(\delta+\varepsilon^{a}\tau)^{-\frac{4}{3}}.
	\end{eqnarray*}
	By employing \eqref{7.25aa}, \eqref{1.10}, \eqref{4.45a} and \eqref{3.16}, one gets
	\begin{eqnarray*}
		&&|(\frac{P_{0}(v_{1}\sqrt{\mu}f_{y})}{\sqrt{\mu}},e^{\phi} f)|\\
		&&
		=|\sum_{i=0}^{4}(\mu^{-\frac{1}{2}}\langle v_{1}\sqrt{\mu}f_{y},\frac{\chi_{i}}{M}\rangle\chi_{i},e^{\phi}f)|
		\notag\\
		&& \leq \eta\varepsilon^{a-1}\|\langle v\rangle^{-\frac{1}{2}}f\|^2+C_{\eta}\varepsilon^{1-a}\sum_{i=0}^{4}\|\langle v\rangle^{\frac{1}{2}}\mu^{-\frac{1}{2}}\langle v_{1}\sqrt{\mu}f_{y},\frac{\chi_{i}}{M}\rangle\chi_{i}\|^2
		\notag\\
		&&\leq C\eta\varepsilon^{a-1}\|f\|_{\sigma}^{2}+C_{\eta}\varepsilon^{1-a}\|f_{y}\|_{\sigma}^{2},
	\end{eqnarray*}
	and
	\begin{multline*}
		|(\frac{1}{\sqrt{\mu}}P_{1}\{v_{1}M(\frac{|v-u|^{2}
			\widetilde{\theta}_{y}}{2R\theta^{2}}+\frac{(v-u)\cdot\widetilde{u}_{y}}{R\theta})\},e^{\phi}f)|\\
		\leq C\eta\varepsilon^{a-1}\|f\|_{\sigma}^{2}+C_{\eta}\varepsilon^{1-a}\|(\widetilde{u}_{y},\widetilde{\theta}_{y})\|^{2}.
	\end{multline*}
In view of \eqref{4.45a}, \eqref{3.16} and \eqref{L}, it is straightforward to get
	\begin{eqnarray*}
		|(\frac{\overline{G}_{\tau}}{\sqrt{\mu}},e^{\phi}f)|
		&&\leq 	\eta\varepsilon^{a-1}\|\langle v\rangle^{-\frac{1}{2}}f\|^2
		+C_{\eta}\varepsilon^{1-a}\|\langle v\rangle^{\frac{1}{2}}\frac{\overline{G}_{\tau}}{\sqrt{\mu}}\|^2
		\notag\\
		&&\leq  C\eta\varepsilon^{a-1}\|f\|_{\sigma}^{2}
		+C_{\eta}k^{\frac{1}{12}}\varepsilon^{\frac{3}{5}-\frac{2}{5}a}\mathcal{D}_{2,l,q}(\tau)
		+C_\eta\varepsilon^{\frac{7}{5}+\frac{1}{15}a}(\delta+\varepsilon^{a}\tau)^{-\frac{4}{3}}.
	\end{eqnarray*}
	Similarly, 
	\begin{eqnarray*}
		&&|(\frac{P_{1}(v_{1}\overline{G}_{y})}{\sqrt{\mu}},e^{\phi}f)|+
		|(\frac{\phi_{y}\partial_{v_{1}}\overline{G}}{\sqrt{\mu}},e^{\phi} f)|
		\notag\\
		&&\leq  C\eta\varepsilon^{a-1}\|f\|_{\sigma}^{2}
		+C_{\eta}k^{\frac{1}{12}}\varepsilon^{\frac{3}{5}-\frac{2}{5}a}\mathcal{D}_{2,l,q}(\tau)
		+C_\eta\varepsilon^{\frac{7}{5}+\frac{1}{15}a}(\delta+\varepsilon^{a}\tau)^{-\frac{4}{3}}.
	\end{eqnarray*}
Consequently, substituting the above  estimates into \eqref{4.40} and taking a small $\eta>0$, the estimate \eqref{4.42} follows.
	We thus finish the proof of Lemma \ref{lem.5.4A}.
\end{proof}

Finally, combining Lemma \ref{lem.5.4A} together with Lemma \ref{lem.5.1A}, we immediately have the following
estimates on the zeroth order energy norm for both the fluid and non-fluid parts.
\begin{lemma}\label{lem4.2}
	It holds that
	\begin{eqnarray}
		\label{4.44}
		&&\{\|(\widetilde{\rho},\widetilde{u},\widetilde{\theta})(\tau)\|^{2}+\|f(\tau)\|^{2}\}
		+\{\|\widetilde{\phi}(\tau)\|^{2}+\varepsilon^{2b-2a}\|\widetilde{\phi}_{y}(\tau)\|^{2}\}
		\notag\\
		&&\hspace{0.5cm}+\int^{\tau}_{0}\|\sqrt{\bar{u}_{1y}}(\widetilde{\rho},\widetilde{u}_{1},\widetilde{\theta})(s)\|^{2}\,ds
		+\varepsilon^{a-1}\int^{\tau}_{0}\|f(s)\|_{\sigma}^{2}\,ds
		\notag\\
		&&\hspace{1cm}+\varepsilon^{1-a}\int^{\tau}_{0}\big\{\|(\widetilde{\rho}_y,\widetilde{u}_y,\widetilde{\theta}_y)(s)\|^{2}
		+\|\widetilde{\phi}_{y}(s)\|^{2}
		+\varepsilon^{2b-2a}\|\widetilde{\phi}_{yy}(s)\|^{2}\big\}\,ds
		\notag\\
		&&\leq C\varepsilon^{2(1-a)}\|(\widetilde{\rho}_{y},\widetilde{u}_y,\widetilde{\theta}_y)(\tau)\|^{2}
		+Ck^{\frac{1}{3}}\varepsilon^{\frac{6}{5}-\frac{4}{5}a}
		+C\varepsilon^{1-a}\int^{\tau}_{0}\|f_y(s)\|^{2}_{\sigma}\,ds
		\notag\\
		&&\quad+C(\eta_{0}+k^{\frac{1}{12}}\varepsilon^{\frac{3}{5}-\frac{2}{5}a})\int^{\tau}_{0}\mathcal{D}_{2,l,q}(s)\,ds\notag\\
		&&\quad+C\int^{\tau}_{0}q_{3}(s)
		\mathcal{H}_{2,l,q}(s)\,ds.
	\end{eqnarray}
\end{lemma}
\begin{proof}
	By taking suitably large constant $C_2>0$ and then adding $\eqref{4.39}\times C_2$ to \eqref{4.42}, we get  
	\begin{eqnarray}
		\label{4.43}
		&&\frac{d}{d\tau}\Big\{C_2\big\{\int_{\mathbb{R}}\eta(\tau,y) \,dy
		+\frac{3}{4}\big(\widetilde{\phi}^{2},\rho'_{\mathrm{e}}(\bar{\phi})\big)
		+\frac{1}{2}\big(\widetilde{\phi}^{3},\rho''_{\mathrm{e}}(\bar{\phi})\big)
		+\frac{3}{4}\varepsilon^{2b-2a}\|\widetilde{\phi}_{y}\|^{2}
		\notag\\
		&&\hspace{0.5cm}-N_1(\tau)
		+\kappa_{3}\varepsilon^{1-a}(\widetilde{u}_{1},\widetilde{\rho}_{y})\big\}
		+\frac{1}{2}(f,e^{\phi} f)\Big\}+c\|\sqrt{\bar{u}_{1y}}(\widetilde{\rho},\widetilde{u}_{1},\widetilde{\theta})\|^{2}
		\notag\\
		&&\hspace{0.5cm}+c\varepsilon^{1-a}\big\{\|(\widetilde{\rho}_y,\widetilde{u}_y,\widetilde{\theta}_y)\|^{2}+\|\widetilde{\phi}_{y}\|^{2}
		+\varepsilon^{2b-2a}\|\widetilde{\phi}_{yy}\|^{2}\big\}+c\varepsilon^{a-1}\|f\|_{\sigma}^{2}
		\notag\\
		&&\leq  C\varepsilon^{1-a}\|f_y\|^{2}_{\sigma}
		+C(\eta_{0}+k^{\frac{1}{12}}\varepsilon^{\frac{3}{5}-\frac{2}{5}a})\mathcal{D}_{2,l,q}(\tau)
		+C\varepsilon^{\frac{7}{5}+\frac{1}{15}a}(\delta+\varepsilon^{a}\tau)^{-\frac{4}{3}}\notag\\
		&&\quad+Cq_3(\tau)\mathcal{H}_{2,l,q}(\tau).
	\end{eqnarray}
	Integrating \eqref{4.43} with respect to $\tau$ and using $\eta(\tau,y)\approx|(\widetilde{\rho},\widetilde{u},\widetilde{\theta})|^{2}$ in \eqref{4.5}
	and the fact that
	\begin{equation*}
		N_1(\tau)\leq \eta\|f\|^{2}+C_\eta\varepsilon^{2-2a}\|(\widetilde{u}_y,\widetilde{\theta}_y)\|^{2}
		+C_\eta k^{\frac{1}{3}}\varepsilon^{\frac{6}{5}-\frac{4}{5}a},
	\end{equation*}
	in terms of \eqref{5.23b}, as well as
	\begin{equation}
		\label{5.50A}
		\mathcal{E}_{2,l,q}(\tau)\mid_{\tau=0}\,\leq\, k^{\frac{1}{3}}\varepsilon^{\frac{2}{3}}
		\leq Ck^{\frac{1}{3}}\varepsilon^{\frac{6}{5}-\frac{4}{5}a},
	\end{equation}
	we can prove \eqref{4.44} holds by choosing $\eta>0$ small enough. Hence, Lemma \ref{lem4.2} is proved.
\end{proof}

\section{High order energy estimates}\label{sec.6}
In this section we consider the high order energy estimates
for both the fluid and non-fluid parts.

\subsection{First order derivative estimates on fluid part}
As in Section \ref{sec4.1}, the proof on the first order derivative estimates
of the fluid part $(\widetilde{\rho},\widetilde{u},\widetilde{\theta},\widetilde{\phi})$ is based on the fluid-type systems \eqref{3.10} and \eqref{4.31}.
\begin{lemma}\label{lem6.1}
	Under the conditions listed in Lemma \ref{lem.5.1A}, it holds that
	\begin{align}
		\label{4.57}
		&\frac{d}{d\tau}\Big(\int_{\mathbb{R}}\big(\frac{2\bar{\theta}}{3\bar{\rho}\rho}|\widetilde{\rho}_y|^{2}
		+|\widetilde{u}_y|^{2}+\frac{1}{\theta}|\widetilde{\theta}_y|^{2}\big)\,dy
		+\varepsilon^{2b-2a}\|\frac{1}{\sqrt{\rho}}\widetilde{\phi}_{yy}\|^{2}\notag\\
		&\qquad\qquad\qquad\qquad\qquad\qquad+\big(\frac{1}{\rho}\widetilde{\phi}_{y},\rho'_{\mathrm{e}}(\bar{\phi})\widetilde{\phi}_{y}\big)
		+\kappa_6\varepsilon^{1-a}(\widetilde{u}_{1y},\widetilde{\rho}_{yy})\Big\}
		\notag\\
		&\quad-\frac{d}{d\tau}N_2(\tau)+c\varepsilon^{1-a}\big\{\|(\widetilde{\rho}_{yy},\widetilde{u}_{yy},\widetilde{\theta}_{yy})\|^{2}
		+ (\|\widetilde{\phi}_{yy}\|^{2}+\varepsilon^{2b-2a}\|\widetilde{\phi}_{yyy}\|^{2})\big\}
		\notag\\
		&\leq C\varepsilon^{1-a}\|f_{yy}\|^{2}_{\sigma}
		+C(k^{\frac{1}{12}}\varepsilon^{\frac{3}{5}-\frac{2}{5}a}+k^{\frac{1}{12}}\varepsilon^{\frac{3}{5}a-\frac{2}{5}})\mathcal{D}_{2,l,q}(\tau)\notag\\
		&\quad+C\varepsilon^{\frac{7}{5}+\frac{1}{15}a}(\delta+\varepsilon^{a}\tau)^{-\frac{4}{3}},
	\end{align}
where $\kappa_6>0$ is a given constant, and $N_2(\tau)$ is denoted by   
	\begin{eqnarray}
		\label{6.12b}
		N_2(\tau)=&&\varepsilon^{1-a}\sum^{3}_{i=1}\int_{\mathbb{R}}\int_{\mathbb{R}^{3}}\frac{1}{\rho}\widetilde{u}_{iyy}[R\theta B_{1i}(\frac{v-u}{\sqrt{R\theta}})
		\frac{\sqrt{\mu}}{M}f]_{y}\, dv\,dy
		\notag\\
		&&+\varepsilon^{1-a}\sum^{3}_{i=1}\int_{\mathbb{R}}\int_{\mathbb{R}^{3}}(\frac{1}{\theta}\widetilde{\theta}_{y})_{y}\frac{1}{\rho}u_{iy}R\theta B_{1i}(\frac{v-u}{\sqrt{R\theta}})\frac{\sqrt{\mu}}{M}f\, dv\,dy
		\notag\\
		&&+\varepsilon^{1-a}\int_{\mathbb{R}}\int_{\mathbb{R}^{3}}(\frac{1}{\theta}\widetilde{\theta}_{y})_{y}\frac{1}{\rho}[(R\theta)^{\frac{3}{2}}A_{1}(\frac{v-u}{\sqrt{R\theta}})\frac{\sqrt{\mu}}{M}f]_y\, dv\,dy.
	\end{eqnarray} 
\end{lemma}
\begin{proof}	
	The proof is divided by three steps as follows.
	
	\medskip
	\noindent{{\bf Step 1.}}	
Differentiating \eqref{3.10} with respect to $y$, one has
	\begin{equation}
		\label{4.45}
		\left\{
		\begin{array}{rl}
			&\widetilde{\rho}_{y\tau}+\rho \widetilde{u}_{1yy}=Q_{1}
			\\
			&\widetilde{u}_{1y \tau}+\frac{2}{3}\widetilde{\theta}_{yy}+\frac{2\bar{\theta}}{3\bar{\rho}}\widetilde{\rho}_{yy}
			+\widetilde{\phi}_{yy}\\
			&\qquad=\varepsilon^{1-a}[\frac{4}{3\rho}\mu(\theta)\widetilde{u}_{1yy}]_{y}+Q_{2}-[\frac{1}{\rho}(\int_{\mathbb{R}^{3}} v^{2}_{1}L^{-1}_{M}\Theta \,dv)_{y}]_{y},
			\\
			&\widetilde{u}_{iy \tau}=\varepsilon^{1-a}[\frac{1}{\rho}\mu(\theta)\widetilde{u}_{iyy}]_{y}+Q_{i+2}-[\frac{1}{\rho}(\int_{\mathbb{R}^3} v_{1}v_{i}L^{-1}_{M}\Theta\, dv)_{y}]_{y}, ~~i=2,3,
			\\
			&\widetilde{\theta}_{y \tau}+\frac{2}{3}\theta
			\widetilde{u}_{1yy}
			=\varepsilon^{1-a}[\frac{1}{\rho}\kappa(\theta)\widetilde{\theta}_{yy}]_{y}+Q_{5}-[\frac{1}{\rho}
			(\int_{\mathbb{R}^3}v_{1}\frac{|v|^{2}}{2}L^{-1}_{M}\Theta\, dv)_{y}]_{y}\\
&\hspace{4cm}			+[\frac{1}{\rho}u\cdot(\int_{\mathbb{R}^3} v_{1}v L^{-1}_{M}\Theta \,dv)_{y}]_{y},
		\end{array} \right.
	\end{equation}
	where we have denoted that
	\begin{equation*}
		\left\{
		\begin{array}{rl}
			&Q_{1}=-\rho_{y}\widetilde{u}_{1y}-[\rho_{y}\widetilde{u}_{1}]_{y}-[\bar{u}_{1}\widetilde{\rho}_{y}]_{y}
			-[\bar{u}_{1y}\widetilde{\rho}]_{y},
			\\
			&Q_{2}=-(\frac{2\bar{\theta}}{3\bar{\rho}})_{y}\widetilde{\rho}_{y}-[u_{1}\widetilde{u}_{1y}]_{y}
			-[\widetilde{u}_{1}\bar{u}_{1y}]_{y}-[\frac{2}{3}\rho_{y}(\frac{\theta}{\rho}-\frac{\bar{\theta}}{\bar{\rho}})]_{y}
			\\
			&\hspace{1cm}+\varepsilon^{1-a}[\frac{4}{3\rho}(\mu(\theta))_{y}u_{1y}]_{y}+\varepsilon^{1-a}[\frac{4}{3\rho}\mu(\theta)\bar{u}_{1yy}]_{y},
			\\
			&Q_{1+i}=-[\widetilde{u}_{1}\widetilde{u}_{iy}]_{y}-[\bar{u}_{1}\widetilde{u}_{iy}]_{y}+
			\varepsilon^{1-a}[\frac{1}{\rho}(\mu(\theta))_{y}\widetilde{u}_{iy}]_{y}, \quad i=2,3,
			\\
			&Q_{5}=-\frac{2}{3}\theta_{y}\widetilde{u}_{1y}
			-[\frac{2}{3}\widetilde{\theta}\bar{u}_{1y}]_{y}-[u_{1}\widetilde{\theta}_{y}]_{y}
			-[\bar{\theta}_{y}\widetilde{u}_{1}]_{y}+\varepsilon^{1-a}[\frac{1}{\rho}\kappa(\theta)\bar{\theta}_{yy}]_{y}
			\\
			&\hspace{1cm}+\varepsilon^{1-a}[\frac{1}{\rho}(\kappa(\theta))_{y}\theta_{y}]_{y}
			+\varepsilon^{1-a}[\frac{4}{3\rho}\mu(\theta)u^{2}_{1y}]_{y}+\varepsilon^{1-a}[\frac{1}{\rho}\mu(\theta)(\widetilde{u}^{2}_{2y}+\widetilde{u}^{2}_{3y})]_{y}.
		\end{array} \right.
	\end{equation*}
	
	Multiplying \eqref{4.45}$_{1}$ by $\frac{2\bar{\theta}}{3\bar{\rho}\rho}\widetilde{\rho}_{y}$, \eqref{4.45}$_{2}$ by $\widetilde{u}_{1y}$,  \eqref{4.45}$_{3}$ by $\widetilde{u}_{iy}$ (i=2,3) and \eqref{4.45}$_{4}$ by $\frac{1}{\theta}\widetilde{\theta}_{y}$, then adding the resulting equations together,
	we can arrive at
	\begin{eqnarray}
		\label{4.47}
		&&\frac{1}{2}(\frac{2\bar{\theta}}{3\bar{\rho}\rho}\widetilde{\rho}^{2}_{y})_{\tau}
		+\frac{1}{2}(\widetilde{u}^{2}_{y})_{\tau}
		+\frac{1}{2}(\frac{1}{\theta}\widetilde{\theta}^{2}_{y})_{\tau}+\varepsilon^{1-a}\frac{4}{3\rho}\mu(\theta)\widetilde{u}^{2}_{1yy}
		\notag\\
		&&\hspace{0.5cm}+\varepsilon^{1-a}\sum^{3}_{i=2}\frac{1}{\rho}\mu(\theta)\widetilde{u}^{2}_{iyy}
		+\varepsilon^{1-a}\frac{1}{\rho\theta}\kappa(\theta)\widetilde{\theta}^{2}_{yy}+\widetilde{\phi}_{yy}\widetilde{u}_{1y}
		+(\cdot\cdot\cdot)_{y}
		\notag\\
		&&=\frac{1}{2}(\frac{2\bar{\theta}}{3\bar{\rho}\rho})_{\tau}\widetilde{\rho}^{2}_{y}
		+\frac{1}{2}(\frac{1}{\theta})_{\tau}\widetilde{\theta}^{2}_{y}
		+(\frac{2\bar{\theta}}{3\bar{\rho}})_{y}\widetilde{\rho}_{y}\widetilde{u}_{1y}
		-\varepsilon^{1-a}\frac{1}{\rho}\kappa(\theta)\widetilde{\theta}_{yy}(\frac{1}{\theta})_{y}\widetilde{\theta}_{y}
		\notag\\
		&&\hspace{0.5cm}+Q_{1}\frac{2\bar{\theta}}{3\bar{\rho}\rho}\widetilde{\rho}_{y}
		+Q_{2}\widetilde{u}_{1y}+Q_{3}\widetilde{u}_{2y}+Q_{4}\widetilde{u}_{3y}+Q_{5}\frac{1}{\theta}\widetilde{\theta}_{y}+H,
	\end{eqnarray}
	where the term $H$ is defined by
	\begin{eqnarray}
		\label{4.48}
		H=&\frac{1}{\rho}(\int_{\mathbb{R}^3} v_{1}vL^{-1}_{M}\Theta\,dv)_{y}\cdot\widetilde{u}_{yy}
		+\frac{1}{\rho}u_{y}\cdot\int_{\mathbb{R}^3} v_{1}v L^{-1}_{M}\Theta\,dv(\frac{1}{\theta}\widetilde{\theta}_{y})_{y}
		\notag\\
		&+\frac{1}{\rho}(\int_{\mathbb{R}^3}(v_{1}\frac{|v|^{2}}{2}-v_{1}v\cdot u)L^{-1}_{M}\Theta\,dv)_{y}(\frac{1}{\theta}\widetilde{\theta}_{y})_{y}.
	\end{eqnarray}
	We shall  integrate \eqref{4.47} with respect to $y$ over $\mathbb{R}$ and then
	estimate each term. For the first on the right hand side of \eqref{4.47}, we get from the embedding inequality,
	Lemma \ref{lem7.2}, $\delta=\frac{1}{k}\varepsilon^{\frac{3}{5}-\frac{2}{5}a}$ in \eqref{3.21}, \eqref{3.22} and \eqref{3.18} that
	\begin{align}
		\label{4.54a}
		&|(\frac{\widetilde{\rho}_{y}^{2}}{2},(\frac{2\bar{\theta}}{3\bar{\rho}\rho})_{\tau})|\notag\\
		&\leq C\|(\bar{\rho}_{\tau},\bar{\theta}_{\tau})\|_{L_{y}^{\infty}}\|\widetilde{\rho}_{y}\|^{2}
		+C\|\widetilde{\rho}_{y}\|_{L_{y}^{\infty}}\|\widetilde{\rho}_{\tau}\|\|\widetilde{\rho}_{y}\|
		\notag\\
		&\leq C\varepsilon^{a}\delta^{-1}\|\widetilde{\rho}_{y}\|^{2}
		+ Ck^{\frac{1}{12}}\varepsilon^{\frac{3}{5}-\frac{2}{5}a}(\|\widetilde{\rho}_{\tau}\|^{2}+\|\widetilde{\rho}_{yy}\|^{2}+\|\widetilde{\rho}_{y}\|^{2})
		\notag\\
		&\leq C\varepsilon^{a}(\frac{1}{k}\varepsilon^{\frac{3}{5}-\frac{2}{5}a})^{-1}
		\varepsilon^{a-1}\mathcal{D}_{2,l,q}(\tau)+Ck^{\frac{1}{12}}\varepsilon^{\frac{3}{5}a-\frac{2}{5}}
		\mathcal{D}_{2,l,q}(\tau)\notag\\
		&\quad+C\varepsilon^{\frac{7}{5}+\frac{1}{15}a}(\delta+\varepsilon^{a}\tau)^{-\frac{4}{3}}
		\notag\\
		&=C(k\varepsilon^{\frac{12}{5}a-\frac{8}{5}}+k^{\frac{1}{12}}\varepsilon^{\frac{3}{5}a-\frac{2}{5}})
		\mathcal{D}_{2,l,q}(\tau)+C\varepsilon^{\frac{7}{5}+\frac{1}{15}a}(\delta+\varepsilon^{a}\tau)^{-\frac{4}{3}}
		\notag\\
		&\leq Ck^{\frac{1}{12}}\varepsilon^{\frac{3}{5}a-\frac{2}{5}}\mathcal{D}_{2,l,q}(\tau)+C\varepsilon^{\frac{7}{5}+\frac{1}{15}a}(\delta+\varepsilon^{a}\tau)^{-\frac{4}{3}}.
	\end{align}
	Here we have used \eqref{5.26b} such that
	\begin{eqnarray*}
		Ck^{\frac{1}{12}}\varepsilon^{\frac{3}{5}-\frac{2}{5}a}\|\widetilde{\rho}_{\tau}\|^{2}
		\leq Ck^{\frac{1}{12}}\varepsilon^{\frac{3}{5}a-\frac{2}{5}}\mathcal{D}_{2,l,q}(\tau)
		+C\varepsilon^{\frac{7}{5}+\frac{1}{15}a}(\delta+\varepsilon^{a}\tau)^{-\frac{4}{3}}.
	\end{eqnarray*}
	Similar as for obtaining \eqref{4.54a}, one has
	\begin{equation*}
		|(\frac{1}{2}(\frac{1}{\theta})_{\tau},\widetilde{\theta}^{2}_{y})|+|(\widetilde{u}_{1y},(\frac{2\bar{\theta}}{3\bar{\rho}})_{y}\widetilde{\rho}_{y})|
		\leq Ck^{\frac{1}{12}}\varepsilon^{\frac{3}{5}a-\frac{2}{5}}\mathcal{D}_{2,l,q}(\tau)+C\varepsilon^{\frac{7}{5}+\frac{1}{15}a}(\delta+\varepsilon^{a}\tau)^{-\frac{4}{3}}.
	\end{equation*}
	By \eqref{4.13A} and \eqref{3.18}, it follows that
	\begin{equation}
		\label{4.56a}
		\varepsilon^{1-a}|(\frac{1}{\rho}\kappa(\theta)\widetilde{\theta}_{yy},(\frac{1}{\theta})_{y}\widetilde{\theta}_{y})|
		\leq C\varepsilon^{1-a}\|\theta_{y}\|\|\widetilde{\theta}_{yy}\|\|\widetilde{\theta}_{y}\|_{L^{\infty}}
		\leq Ck^{\frac{1}{12}}\varepsilon^{\frac{3}{5}-\frac{2}{5}a}\mathcal{D}_{2,l,q}(\tau).
	\end{equation}
	To estimate $Q_{1}$ in \eqref{4.47}, from Lemma \ref{lem7.2}, \eqref{3.21}, \eqref{3.22} and \eqref{3.18}, one has
	\begin{eqnarray}
		\label{4.57a}
		|(\bar{\rho}_{yy}\widetilde{u}_{1},\frac{2\bar{\theta}}{3\bar{\rho}\rho}\widetilde{\rho}_{y})|
		&&\leq C\|\widetilde{u}_{1}\|_{L_{y}^{\infty}}(\|\widetilde{\rho}_{y}\|^{2}+\|\bar{\rho}_{yy}\|^{2})
		\notag\\
		&&\leq Ck^{\frac{1}{12}}\varepsilon^{\frac{3}{5}-\frac{2}{5}a}\big\{
		\varepsilon^{a-1}\mathcal{D}_{2,l,q}(\tau)+
		\varepsilon^{3a}\delta^{-1}(\delta+\varepsilon^{a}\tau)^{-2}\big\}
		\notag\\
		&&\leq Ck^{\frac{1}{12}}\varepsilon^{\frac{3}{5}a-\frac{2}{5}}\mathcal{D}_{2,l,q}(\tau)
		+C\varepsilon^{\frac{7}{5}+\frac{1}{15}a}(\delta+\varepsilon^{a}\tau)^{-\frac{4}{3}},
	\end{eqnarray}
which together with the similar calculations as \eqref{4.54a} yields that
	\begin{equation*}
		|(Q_{1},\frac{2\bar{\theta}}{3\bar{\rho}\rho}\widetilde{\rho}_{y})|\leq Ck^{\frac{1}{12}}\varepsilon^{\frac{3}{5}a-\frac{2}{5}}\mathcal{D}_{2,l,q}(\tau)
		+C\varepsilon^{\frac{7}{5}+\frac{1}{15}a}(\delta+\varepsilon^{a}\tau)^{-\frac{4}{3}}.
	\end{equation*}
	Likewise, to control the terms containing $Q_{2}$-$Q_{5}$ in \eqref{4.47}, we start with
	\begin{eqnarray*}
		&&\varepsilon^{1-a}|([\frac{4}{3\rho}\mu(\theta)\bar{u}_{1yy}]_{y},\widetilde{u}_{1y})|=
		\varepsilon^{1-a}|(\frac{4}{3\rho}\mu(\theta)\bar{u}_{1yy},\widetilde{u}_{1yy})|
		\notag\\
		&&\leq C\varepsilon^{1-a}(k^{\frac{1}{12}}\varepsilon^{\frac{3}{5}a-\frac{2}{5}}\|\widetilde{u}_{1yy}\|^{2}
		+k^{-\frac{1}{12}}\varepsilon^{\frac{2}{5}-\frac{3}{5}a}\|\bar{u}_{1yy}\|^{2})
		\notag\\
		&&\leq Ck^{\frac{1}{12}}\varepsilon^{\frac{3}{5}a-\frac{2}{5}}\mathcal{D}_{2,l,q}(\tau)
		+Ck^{-\frac{1}{12}}\varepsilon^{\frac{7}{5}-\frac{8}{5}a}\varepsilon^{3a}\delta^{-1}(\delta+\varepsilon^{a}\tau)^{-2}
		\notag\\
		&&\leq Ck^{\frac{1}{12}}\varepsilon^{\frac{3}{5}a-\frac{2}{5}}\mathcal{D}_{2,l,q}(\tau)
		+C\varepsilon^{\frac{7}{5}+\frac{1}{15}a}(\delta+\varepsilon^{a}\tau)^{-\frac{4}{3}},
	\end{eqnarray*}
	and then this together with the similar calculations as  \eqref{4.57a}, \eqref{4.56a} and \eqref{4.54a} leads us to
	\begin{eqnarray*}
		&&|(Q_{2},\widetilde{u}_{1y})+(Q_{3},\widetilde{u}_{2y})+(Q_{4},\widetilde{u}_{3y})+(Q_{5},\frac{1}{\theta}\widetilde{\theta}_{y})|
		\notag\\
		&&\leq C(k^{\frac{1}{12}}\varepsilon^{\frac{3}{5}-\frac{2}{5}a}+k^{\frac{1}{12}}\varepsilon^{\frac{3}{5}a-\frac{2}{5}})\mathcal{D}_{2,l,q}(\tau)
		+C\varepsilon^{\frac{7}{5}+\frac{1}{15}a}(\delta+\varepsilon^{a}\tau)^{-\frac{4}{3}}.
	\end{eqnarray*}
	
	In order to estimate the term $H$ given in \eqref{4.48}, we only consider the first term in $H$ since other terms in $H$ can be treated similarly.
	We first use \eqref{4.12} to write
	\begin{equation}
		\label{4.50}
		\int_{\mathbb{R}}\frac{1}{\rho}\widetilde{u}_{yy}\cdot(\int_{\mathbb{R}^3} v_{1}vL^{-1}_{M}\Theta\, dv)_{y}\,dy
		=\sum^{3}_{i=1}\int_{\mathbb{R}}\int_{\mathbb{R}^{3}}\frac{1}{\rho}\widetilde{u}_{iyy}\big[R\theta B_{1i}(\frac{v-u}{\sqrt{R\theta}})
		\frac{\Theta}{M}\big]_{y}\,dv\,dy.
	\end{equation}
	Consider \eqref{4.50} associated with $\Theta$ given in \eqref{3.11}. We first take an integration by parts to get
	\begin{eqnarray}
		\label{6.10b}
		&&\varepsilon^{1-a}\sum^{3}_{i=1}\int_{\mathbb{R}}\int_{\mathbb{R}^{3}}\frac{1}{\rho}\widetilde{u}_{iyy}\big[R\theta B_{1i}(\frac{v-u}{\sqrt{R\theta}})
		\frac{\sqrt{\mu}}{M}f_\tau\big]_{y}\,dv\,dy
		\notag\\
		=&&\varepsilon^{1-a}\frac{d}{d\tau}\sum^{3}_{i=1}\int_{\mathbb{R}}\int_{\mathbb{R}^{3}}\frac{1}{\rho}\widetilde{u}_{iyy}\big[R\theta B_{1i}(\frac{v-u}{\sqrt{R\theta}})
		\frac{\sqrt{\mu}}{M}f\big]_{y}\, dv\,dy
		\notag\\
		&&-\varepsilon^{1-a}\sum^{3}_{i=1}\int_{\mathbb{R}}\int_{\mathbb{R}^{3}}\frac{1}{\rho}\widetilde{u}_{iyy\tau}\big[R\theta B_{1i}(\frac{v-u}{\sqrt{R\theta}})
		\frac{\sqrt{\mu}}{M}f\big]_{y} \, dv \,dy
		\notag\\
		&&-\varepsilon^{1-a}\sum^{3}_{i=1}\int_{\mathbb{R}}\int_{\mathbb{R}^{3}}[\frac{1}{\rho}]_\tau\widetilde{u}_{iyy}\big[R\theta B_{1i}(\frac{v-u}{\sqrt{R\theta}})
		\frac{\sqrt{\mu}}{M}f\big]_{y} \, dv \,dy
		\notag\\
		&&-\varepsilon^{1-a}\sum^{3}_{i=1}\int_{\mathbb{R}}\int_{\mathbb{R}^{3}}\frac{1}{\rho}\widetilde{u}_{iyy}\big[\{R\theta B_{1i}(\frac{v-u}{\sqrt{R\theta}})
		\frac{\sqrt{\mu}}{M}\}_\tau f\big]_{y} \, dv \,dy.
	\end{eqnarray}
	By integration by parts, \eqref{5.26b}, \eqref{4.14}, \eqref{3.16}, Lemma \ref{lem7.2}, \eqref{3.21}, \eqref{3.22} and \eqref{3.18}, the second term of \eqref{6.10b} is bounded by
	\begin{eqnarray*}
		&&\varepsilon^{1-a}\sum^{3}_{i=1}\int_{\mathbb{R}}\int_{\mathbb{R}^{3}}\widetilde{u}_{iy\tau}
		\big\{\frac{1}{\rho}[R\theta B_{1i}(\frac{v-u}{\sqrt{R\theta}})
		\frac{\sqrt{\mu}}{M}f]_y\big\}_{y} \, dv \,dy
		\notag\\	
		\leq&&\eta\varepsilon^{1-a}\|\widetilde{u}_{y\tau}\|^2+C_\eta\varepsilon^{1-a}
		\sum^{3}_{i=1}\int_{\mathbb{R}}\int_{\mathbb{R}^{3}}|\big\{\frac{1}{\rho}[R\theta B_{1i}(\frac{v-u}{\sqrt{R\theta}})
		\frac{\sqrt{\mu}}{M}  f]_{y} \big\}_{y}|^2 \, dv\,dy
		\notag\\
		\leq&& \eta\varepsilon^{1-a}\|(\widetilde{\rho}_{yy},\widetilde{u}_{yy},\widetilde{\theta}_{yy},\widetilde{\phi}_{yy})\|^{2}
		+C_{\eta}\varepsilon^{1-a}\|f_{yy}\|^{2}_{\sigma}
		+C_{\eta}k^{\frac{1}{12}}\varepsilon^{\frac{3}{5}-\frac{2}{5}a}\mathcal{D}_{2,l,q}(\tau)\\
		&&+C_{\eta}\varepsilon^{\frac{7}{5}+\frac{1}{15}a}(\delta+\varepsilon^{a}\tau)^{-\frac{4}{3}}.
	\end{eqnarray*}
	The last two term of \eqref{6.10b} have the same bound as the above, it follows that
	\begin{eqnarray*}
		&&\varepsilon^{1-a}\sum^{3}_{i=1}\int_{\mathbb{R}}\int_{\mathbb{R}^{3}}\frac{1}{\rho}\widetilde{u}_{iyy}\big[R\theta B_{1i}(\frac{v-u}{\sqrt{R\theta}})
		\frac{\sqrt{\mu}}{M}f_\tau\big]_{y}\,dv\,dy
		\notag\\
		\leq&&\varepsilon^{1-a}\frac{d}{d\tau}\sum^{3}_{i=1}\int_{\mathbb{R}}\int_{\mathbb{R}^{3}}\frac{1}{\rho}\widetilde{u}_{iyy}\big[R\theta B_{1i}(\frac{v-u}{\sqrt{R\theta}})
		\frac{\sqrt{\mu}}{M}f\big]_{y}\, dv\,dy\\
		&&+C\eta\varepsilon^{1-a}\|(\widetilde{\rho}_{yy},\widetilde{u}_{yy},\widetilde{\theta}_{yy},\widetilde{\phi}_{yy})\|^{2}
		\notag\\
		&&+C_{\eta}\varepsilon^{1-a}\|f_{yy}\|^{2}_{\sigma}
		+C_{\eta}k^{\frac{1}{12}}\varepsilon^{\frac{3}{5}-\frac{2}{5}a}\mathcal{D}_{2,l,q}(\tau)
		+C_{\eta}\varepsilon^{\frac{7}{5}+\frac{1}{15}a}(\delta+\varepsilon^{a}\tau)^{-\frac{4}{3}}.
	\end{eqnarray*}
	Following the similar arguments as \eqref{5.17b} yields that
	\begin{eqnarray*}
		&&\varepsilon^{1-a}\sum^{3}_{i=1}\int_{\mathbb{R}}\int_{\mathbb{R}^{3}}\frac{1}{\rho}\widetilde{u}_{iyy}[R\theta B_{1i}(\frac{v-u}{\sqrt{R\theta}})
		\frac{\overline{G}_{\tau}}{M}]_y \, dv \,dy	
		\notag\\
		&&\leq\eta\varepsilon^{1-a}\|\widetilde{u}_{yy}\|^{2}
		+C_{\eta}k^{\frac{1}{12}}\varepsilon^{\frac{3}{5}-\frac{2}{5}a}\mathcal{D}_{2,l,q}(\tau)
		+C_{\eta}\varepsilon^{\frac{7}{5}+\frac{1}{15}a}(\delta+\varepsilon^{a}\tau)^{-\frac{4}{3}}.
	\end{eqnarray*}
Thanks to $G=\overline{G}+\sqrt{\mu}f$,	it follows from the above two estimates that
	\begin{eqnarray*}
		&&\varepsilon^{1-a}\sum^{3}_{i=1}\int_{\mathbb{R}}\int_{\mathbb{R}^{3}}\frac{1}{\rho}\widetilde{u}_{iyy}[R\theta B_{1i}(\frac{v-u}{\sqrt{R\theta}})
		\frac{G_{\tau}}{M}]_y \, dv \,dy
		\notag\\
		\leq&&\varepsilon^{1-a}\frac{d}{d\tau}\sum^{3}_{i=1}\int_{\mathbb{R}}\int_{\mathbb{R}^{3}}\frac{1}{\rho}\widetilde{u}_{iyy}\big[R\theta B_{1i}(\frac{v-u}{\sqrt{R\theta}})\frac{\sqrt{\mu}}{M}f\big]_{y}\, dv\,dy\\
		&&+C\eta\varepsilon^{1-a}\|(\widetilde{\rho}_{yy},\widetilde{u}_{yy},\widetilde{\theta}_{yy},\widetilde{\phi}_{yy})\|^{2}
		\notag\\
		&&+C_{\eta}\varepsilon^{1-a}\|f_{yy}\|^{2}_{\sigma}
		+C_{\eta}k^{\frac{1}{12}}\varepsilon^{\frac{3}{5}-\frac{2}{5}a}\mathcal{D}_{2,l,q}(\tau)
		+C_{\eta}\varepsilon^{\frac{7}{5}+\frac{1}{15}a}(\delta+\varepsilon^{a}\tau)^{-\frac{4}{3}}.
	\end{eqnarray*}
	In view of \eqref{4.14}, \eqref{7.24}, \eqref{3.16}, Lemma \ref{lem7.2}, \eqref{3.21}, \eqref{3.22} and \eqref{3.18}, one has
	\begin{eqnarray*}
		&&\sum^{3}_{i=1}|\int_{\mathbb{R}}\int_{\mathbb{R}^{3}}\frac{1}{\rho}\widetilde{u}_{iyy}\big[R\theta B_{1i}(\frac{v-u}{\sqrt{R\theta}})
		\frac{\Theta_1}{M}\big]_{y} \,dv\,dy|
		\notag\\
		&&\leq \eta\varepsilon^{1-a}\|\widetilde{u}_{yy}\|^{2}
		+C_{\eta}\varepsilon^{1-a}\|f_{yy}\|^{2}_{\sigma}
		+C_{\eta}k^{\frac{1}{12}}\varepsilon^{\frac{3}{5}-\frac{2}{5}a}\mathcal{D}_{2,l,q}(\tau)\\
		&&\quad+C_{\eta}\varepsilon^{\frac{7}{5}+\frac{1}{15}a}(\delta+\varepsilon^{a}\tau)^{-\frac{4}{3}},
	\end{eqnarray*}
where we have denoted
$$
\Theta_1=\varepsilon^{1-a}P_{1}(v_{1}G_{y})-\varepsilon^{1-a}\phi_{y}\partial_{v_{1}}G-Q(G,G).
$$
Recall $\Theta=\varepsilon^{1-a}G_\tau+\Theta_1$ as in \eqref{3.11}. From the above estimates, we then conclude that
	\begin{eqnarray*}
		&&\int_{\mathbb{R}}\frac{1}{\rho}\widetilde{u}_{yy}\cdot(\int_{\mathbb{R}^3} v_{1}vL^{-1}_{M}\Theta\, dv)_{y}\,dy
		\notag\\
		\leq&&\varepsilon^{1-a}\frac{d}{d\tau}\sum^{3}_{i=1}\int_{\mathbb{R}}\int_{\mathbb{R}^{3}}\frac{1}{\rho}\widetilde{u}_{iyy}[R\theta B_{1i}(\frac{v-u}{\sqrt{R\theta}})
		\frac{\sqrt{\mu}}{M}f]_{y}\, dv\,dy\notag\\
		&&+C\eta\varepsilon^{1-a}\|(\widetilde{\rho}_{yy},\widetilde{u}_{yy},\widetilde{\theta}_{yy},\widetilde{\phi}_{yy})\|^{2}
		+C_{\eta}\varepsilon^{1-a}\|f_{yy}\|^{2}_{\sigma}
		+C_{\eta}k^{\frac{1}{12}}\varepsilon^{\frac{3}{5}-\frac{2}{5}a}\mathcal{D}_{2,l,q}(\tau)\notag\\
		&&
		+C_{\eta}\varepsilon^{\frac{7}{5}+\frac{1}{15}a}(\delta+\varepsilon^{a}\tau)^{-\frac{4}{3}}.	
	\end{eqnarray*}
	Likewise, it holds that
	\begin{eqnarray*}
		&&\int_{\mathbb{R}}(\frac{1}{\theta}\widetilde{\theta}_{y})_{y}\frac{1}{\rho}u_{y}\cdot\int_{\mathbb{R}^3} v_{1}v L^{-1}_{M}\Theta\,dv\,dy\\
&&		+\int_{\mathbb{R}}(\frac{1}{\theta}\widetilde{\theta}_{y})_{y}\frac{1}{\rho}(\int_{\mathbb{R}^3}(v_{1}\frac{|v|^{2}}{2}-v_{1}v\cdot u)L^{-1}_{M}\Theta\,dv)_{y}\,dy
		\notag\\
		=&&\sum^{3}_{i=1}\int_{\mathbb{R}}\int_{\mathbb{R}^{3}}(\frac{1}{\theta}\widetilde{\theta}_{y})_{y}\frac{1}{\rho}u_{iy}R\theta B_{1i}(\frac{v-u}{\sqrt{R\theta}})\frac{\Theta}{M}\, dv\,dy\\
		&&+\int_{\mathbb{R}}\int_{\mathbb{R}^{3}}(\frac{1}{\theta}\widetilde{\theta}_{y})_{y}\frac{1}{\rho}[(R\theta)^{\frac{3}{2}}A_{1}(\frac{v-u}{\sqrt{R\theta}})\frac{\Theta}{M}]_{y}\, dv\,dy
		\notag\\
		\leq&&\varepsilon^{1-a}\frac{d}{d\tau}\sum^{3}_{i=1}\int_{\mathbb{R}}\int_{\mathbb{R}^{3}}(\frac{1}{\theta}\widetilde{\theta}_{y})_{y}\frac{1}{\rho}u_{iy}R\theta B_{1i}(\frac{v-u}{\sqrt{R\theta}})\frac{\sqrt{\mu}}{M}f\, dv\,dy
		\notag\\
		&&+\varepsilon^{1-a}\frac{d}{d\tau}\int_{\mathbb{R}}\int_{\mathbb{R}^{3}}(\frac{1}{\theta}\widetilde{\theta}_{y})_{y}\frac{1}{\rho}[(R\theta)^{\frac{3}{2}}A_{1}(\frac{v-u}{\sqrt{R\theta}})\frac{\sqrt{\mu}}{M}f]_{y}\, dv\,dy\\
		&&+C\eta\varepsilon^{1-a}\|(\widetilde{\rho}_{yy},\widetilde{u}_{yy},\widetilde{\theta}_{yy},\widetilde{\phi}_{yy})\|^{2}
		\notag\\
		&&+C_{\eta}\varepsilon^{1-a}\|f_{yy}\|^{2}_{\sigma}
		+C_{\eta}k^{\frac{1}{12}}\varepsilon^{\frac{3}{5}-\frac{2}{5}a}\mathcal{D}_{2,l,q}(\tau)
		+C_{\eta}\varepsilon^{\frac{7}{5}+\frac{1}{15}a}(\delta+\varepsilon^{a}\tau)^{-\frac{4}{3}}.
	\end{eqnarray*}
	With these, we can claim that
	\begin{eqnarray*}
		\int_{\mathbb{R}}H\,dy
		\leq&&\frac{d}{d\tau}N_2(\tau)+C\eta\varepsilon^{1-a}\|(\widetilde{\rho}_{yy},\widetilde{u}_{yy},\widetilde{\theta}_{yy},\widetilde{\phi}_{yy})\|^{2}
		\notag\\
		&&+C_{\eta}\varepsilon^{1-a}\|f_{yy}\|^{2}_{\sigma}
		+C_{\eta}k^{\frac{1}{12}}\varepsilon^{\frac{3}{5}-\frac{2}{5}a}\mathcal{D}_{2,l,q}(\tau)
		+C_{\eta}\varepsilon^{\frac{7}{5}+\frac{1}{15}a}(\delta+\varepsilon^{a}\tau)^{-\frac{4}{3}}.
	\end{eqnarray*}
	Here we recall that $N_2(\tau)$ is defined in \eqref{6.12b}.
	In summary, integrating \eqref{4.47} with respect to $y$ over $\mathbb{R}$ and collecting the above estimates,
for any small $\eta>0$ and a given constant $\kappa_4>0$, we deduce  that
	\begin{align}
		\label{4.51}
		&\frac{d}{d\tau}\Big\{\int_{\mathbb{R}}\big(\frac{2\bar{\theta}}{3\bar{\rho}\rho}|\widetilde{\rho}_{y}|^{2}
		+|\widetilde{u}_{y}|^{2}+\frac{1}{\theta}|\widetilde{\theta}_{y}|^{2}\big)\,dy
		+\varepsilon^{2b-2a}\|\frac{1}{\sqrt{\rho}}\widetilde{\phi}_{yy}\|^{2}+\big(\frac{1}{\rho}\widetilde{\phi}_{y},\rho'_{\mathrm{e}}(\bar{\phi})\widetilde{\phi}_{y}\big)\Big\}
		\notag\\
		&\hspace{1cm}-\frac{d}{d\tau}N_2(\tau)+\kappa_4\varepsilon^{1-a}\|(\widetilde{u}_{yy},\widetilde{\theta}_{yy})\|^{2} 
		\notag\\
		&\leq C\eta\varepsilon^{1-a}\|(\widetilde{\rho}_{yy},\widetilde{u}_{yy},\widetilde{\theta}_{yy},\widetilde{\phi}_{yy})\|^{2}
		+C_{\eta}\varepsilon^{1-a}\|f_{yy}\|^{2}_{\sigma}\notag\\
		&\quad+C(k^{\frac{1}{12}}\varepsilon^{\frac{3}{5}-\frac{2}{5}a}+k^{\frac{1}{12}}\varepsilon^{\frac{3}{5}a-\frac{2}{5}})\mathcal{D}_{2,l,q}(\tau)
		\notag\\
		&\quad
		+C\varepsilon^{\frac{7}{5}+\frac{1}{15}a}(\delta+\varepsilon^{a}\tau)^{-\frac{4}{3}},
	\end{align}
	where we have used \eqref{4.49} whose proof will be postponed to Lemma \ref{lem.6.2} later. 
	
	\medskip
	\noindent{{\bf Step 2.}}
To get the dissipation term for $\widetilde{\rho}_{yy}$. Differentiating the second equation of \eqref{4.31} with respect to $y$
	and taking the inner product of the resulting equation with $\varepsilon^{1-a}\widetilde{\rho}_{yy}$, we have
	\begin{eqnarray}
		\label{6.14b}
		&&\varepsilon^{1-a}(\widetilde{u}_{1y\tau}+(\widetilde{u}_{1}u_{1y})_{y}+(\bar{u}_{1}\widetilde{u}_{1y})_{y}
		+\frac{2\bar{\theta}}{3\bar{\rho}}\widetilde{\rho}_{yy}
		+(\frac{2\bar{\theta}}{3\bar{\rho}})_{y}\widetilde{\rho}_{y},\widetilde{\rho}_{yy})
		\notag\\
		&&=\varepsilon^{1-a}(-\frac{2}{3}\widetilde{\theta}_{yy}
		-(\frac{2}{3}\rho_{y}\frac{\bar{\rho}\widetilde{\theta}-\widetilde{\rho}\bar{\theta}}
		{\rho\bar{\rho}})_{y}-(\frac{1}{\rho}\int_{\mathbb{R}^{3}} v^{2}_{1}G_{y}\,dv)_{y},\widetilde{\rho}_{yy})\notag\\
		&&\quad-\varepsilon^{1-a}(\widetilde{\phi}_{yy},\widetilde{\rho}_{yy}).
	\end{eqnarray}
Using the analogous	calculations as \eqref{6.14b} and \eqref{4.81a} whose proof will be postponed to Lemma \ref{lem.6.3} later,
it is straightforward to check that
	\begin{eqnarray}
		\label{4.53}
		&&\varepsilon^{1-a}(\widetilde{u}_{1y},\widetilde{\rho}_{yy})_{\tau}+\kappa_{5}\varepsilon^{1-a}(\|\widetilde{\rho}_{yy}\|^{2}
		+\|\widetilde{\phi}_{yy}\|^{2}+\varepsilon^{2b-2a}\|\widetilde{\phi}_{yyy}\|^{2})
		\notag\\
		&&\leq C\varepsilon^{1-a}(\|(\widetilde{u}_{yy},\widetilde{\theta}_{yy})\|^{2}
		+\|f_{yy}\|_{\sigma}^{2})+Ck^{\frac{1}{12}}\varepsilon^{\frac{3}{5}-\frac{2}{5}a}\mathcal{D}_{2,l,q}(\tau)\notag\\
		&&\quad+C\varepsilon^{\frac{7}{5}+\frac{1}{15}a}(\delta+\varepsilon^{a}\tau)^{-\frac{4}{3}},
	\end{eqnarray}
for $k_5>0$. Here the details are omitted for brevity.
	
	\medskip
	\noindent{{\bf Step 3.}}
	In summary, adding \eqref{4.51} to \eqref{4.53}$\times\kappa_{6}$ with
	$C\kappa_6<\frac{1}{4}\kappa_4$, then choosing $\eta>0$ small enough such that $C\eta<\frac{1}{4}\kappa_{5}\kappa_{6}$, we can get \eqref{4.57}.
	This hence completes the proof of Lemma \ref{lem6.1}.
\end{proof}

For deducing \eqref{4.51} in the proof of Lemma \ref{lem6.1} above, the following estimate has been used.
\begin{lemma}\label{lem.6.2}
	It holds that
	\begin{eqnarray}
		\label{4.49}
		(\widetilde{\phi}_{yy},\widetilde{u}_{1y})
		\geq&& \frac{1}{2}\varepsilon^{2b-2a}\frac{d}{d\tau}\|\frac{1}{\sqrt{\rho}}\widetilde{\phi}_{yy}\|^{2}+ \frac{1}{2}\frac{d}{d\tau}(\frac{1}{\rho}\widetilde{\phi}_{y},\rho'_{\mathrm{e}}(\bar{\phi})\widetilde{\phi}_{y})
		\notag\\
		&&-Ck^{\frac{1}{12}}\varepsilon^{\frac{3}{5}a-\frac{2}{5}}\mathcal{D}_{2,l,q}(\tau)
		-C\varepsilon^{\frac{7}{5}+\frac{1}{15}a}(\delta+\varepsilon^{a}\tau)^{-\frac{4}{3}}.
	\end{eqnarray}	
\end{lemma}
\begin{proof}
	In view of the first equation of \eqref{3.10}, it is straightforward to get
	\begin{equation}
		\label{4.60b}
		(\widetilde{\phi}_{yy},\widetilde{u}_{1y})=	-(\frac{1}{\rho}\widetilde{\phi}_{yy},\widetilde{\rho}_{\tau})-(\frac{1}{\rho}\widetilde{\phi}_{yy},
		\rho_{y} \widetilde{u}_{1}+\bar{u}_{1y}\widetilde{\rho})-(\frac{1}{\rho}\widetilde{\phi}_{yy},\bar{u}_{1}\widetilde{\rho}_{y}).
	\end{equation}
	For the first term on the right hand side of \eqref{4.60b}, we use the last equation of \eqref{3.10} to write
	$$
	-(\frac{1}{\rho}\widetilde{\phi}_{yy},\widetilde{\rho}_{\tau})=(\frac{1}{\rho}\widetilde{\phi}_{yy},
	\varepsilon^{2b-2a}(\widetilde{\phi}_{yy\tau}+\bar{\phi}_{yy\tau})
	+[\rho_{\mathrm{e}}(\bar{\phi})-\rho_{\mathrm{e}}(\phi)]_{\tau}).
	$$
According to  the first equation of \eqref{3.10}, Lemma \ref{lem7.2}, \eqref{3.21}, \eqref{3.22} and \eqref{3.18}, one has
	\begin{eqnarray}
		\label{4.60aa}
		\varepsilon^{2b-2a}|([\frac{1}{\rho}]_{\tau}\widetilde{\phi}_{yy},\widetilde{\phi}_{yy})|
		&&\leq C\varepsilon^{2b-2a}(\|\bar{\rho}_{\tau}\|_{L_{y}^{\infty}}\|\widetilde{\phi}_{yy}\|^{2}
		+\|\widetilde{\rho}_{\tau}\|\|\widetilde{\phi}_{yy}\|\|\widetilde{\phi}_{yy}\|_{L_{y}^{\infty}})
		\notag\\
		&&\leq C\varepsilon^{a-1}(\|\bar{\rho}_{\tau}\|_{L_{y}^{\infty}}
		+\|(\rho\widetilde{u}_{1})_{y}+(\bar{u}_{1}\widetilde{\rho})_y\|)\mathcal{D}_{2,l,q}(\tau)
		\notag\\
		&&\leq C(k^{\frac{1}{12}}\varepsilon^{\frac{3}{5}a-\frac{2}{5}}+k\varepsilon^{\frac{12}{5}a-\frac{8}{5}})\mathcal{D}_{2,l,q}(\tau)
		\notag\\
		&&\leq Ck^{\frac{1}{12}}\varepsilon^{\frac{3}{5}a-\frac{2}{5}}\mathcal{D}_{2,l,q}(\tau),
	\end{eqnarray}
which yields immediately that
	\begin{eqnarray}
		\label{4.67-2}
		&&\varepsilon^{2b-2a}(\frac{1}{\rho}\widetilde{\phi}_{yy},\widetilde{\phi}_{yy\tau})\notag\\
		&&=\frac{1}{2}\varepsilon^{2b-2a}(\frac{1}{\rho}\widetilde{\phi}_{yy},\widetilde{\phi}_{yy})_{\tau}
		-\frac{1}{2}\varepsilon^{2b-2a}([\frac{1}{\rho}]_{\tau}\widetilde{\phi}_{yy},\widetilde{\phi}_{yy})
		\notag\\
		&&\geq \frac{1}{2}\varepsilon^{2b-2a}\frac{d}{d\tau}\|\frac{1}{\sqrt{\rho}}\widetilde{\phi}_{yy}\|^{2}
		-Ck^{\frac{1}{12}}\varepsilon^{\frac{3}{5}a-\frac{2}{5}}\mathcal{D}_{2,l,q}(\tau).
	\end{eqnarray}
	On the other hand, it holds that
	\begin{eqnarray}
		\label{4.61a}
		&&\varepsilon^{2b-2a}|(\frac{1}{\rho}\widetilde{\phi}_{yy},\bar{\phi}_{yy\tau})|\notag\\
		&&\leq C\varepsilon^{2b-2a}
		\big\{k^{\frac{1}{12}}\varepsilon^{\frac{3}{5}a-\frac{2}{5}}\varepsilon^{1-a}\|\widetilde{\phi}_{yy}\|^{2}
		+k^{-\frac{1}{12}}\varepsilon^{\frac{2}{5}-\frac{3}{5}a}\varepsilon^{a-1}\|\bar{\phi}_{yy\tau}\|^{2}\big\}
		\notag\\
		&&\leq Ck^{\frac{1}{12}}\varepsilon^{\frac{3}{5}a-\frac{2}{5}}\mathcal{D}_{2,l,q}(\tau)
		+Ck^{-\frac{1}{12}}\varepsilon^{2b+\frac{17}{5}a-\frac{3}{5}}
		(\frac{1}{k}\varepsilon^{\frac{3}{5}-\frac{2}{5}a})^{-\frac{11}{3}}(\delta+\varepsilon^{a}\tau)^{-\frac{4}{3}}
		\notag\\
		&&\leq Ck^{\frac{1}{12}}\varepsilon^{\frac{3}{5}a-\frac{2}{5}}\mathcal{D}_{2,l,q}(\tau)
		+C\varepsilon^{\frac{7}{5}+\frac{1}{15}a}(\delta+\varepsilon^{a}\tau)^{-\frac{4}{3}},
	\end{eqnarray}
	where in the last inequality we have required that
	\begin{equation}
		\label{4.68b}
		\varepsilon^{2b+\frac{17}{5}a-\frac{3}{5}}
		(\varepsilon^{\frac{3}{5}-\frac{2}{5}a})^{-\frac{11}{3}}\leq \varepsilon^{\frac{7}{5}+\frac{1}{15}a}, \quad \mbox{that~~is} \quad a\geq \frac{21-10b}{24}.
	\end{equation}
	In light of \eqref{4.34}, the following identity holds true
	$$
	(\frac{1}{\rho}\widetilde{\phi}_{yy},[\rho_{\mathrm{e}}(\bar{\phi})-\rho_{\mathrm{e}}(\phi)]_{\tau})
	=-(\frac{1}{\rho}\widetilde{\phi}_{yy},\rho'_{\mathrm{e}}(\bar{\phi})\widetilde{\phi}_{\tau})
	-(\frac{1}{\rho}\widetilde{\phi}_{yy},\rho''_{\mathrm{e}}(\bar{\phi})\bar{\phi}_{\tau}\widetilde{\phi})
	+(\frac{1}{\rho}\widetilde{\phi}_{yy},J_{7\tau}).
	$$
	By Lemma \ref{lem7.2}, \eqref{3.21}, \eqref{3.22}, \eqref{3.18}, \eqref{5.26b} and \eqref{4.38}, we deduce that
	\begin{eqnarray*}
		|(\frac{1}{\rho}\widetilde{\phi}_{yy},\rho''_{\mathrm{e}}(\bar{\phi})\bar{\phi}_{\tau}\widetilde{\phi})|
		&&\leq C\|\widetilde{\phi}\|(\|\widetilde{\phi}_{yy}\|^{2}+\|\bar{\phi}_{\tau}\|^{2}_{L_{y}^{\infty}})
		\notag\\
		&&\leq Ck^{\frac{1}{12}}\varepsilon^{\frac{3}{5}a-\frac{2}{5}}\mathcal{D}_{2,l,q}(\tau)
		+C\varepsilon^{\frac{7}{5}+\frac{1}{15}a}(\delta+\varepsilon^{a}\tau)^{-\frac{4}{3}},
	\end{eqnarray*}
	and
	\begin{eqnarray*}
		&&-(\frac{1}{\rho}\widetilde{\phi}_{yy},\rho'_{\mathrm{e}}(\bar{\phi})\widetilde{\phi}_{\tau})\\
		=&&
		\frac{1}{2}\frac{d}{d\tau}(\frac{1}{\rho}\widetilde{\phi}_{y},\rho'_{\mathrm{e}}(\bar{\phi})\widetilde{\phi}_{y})
		-\frac{1}{2}(\widetilde{\phi}_{y},[\frac{1}{\rho}\rho'_{\mathrm{e}}(\bar{\phi})]_{\tau}\widetilde{\phi}_{y})
		+(\widetilde{\phi}_{y},[\frac{1}{\rho}\rho'_{\mathrm{e}}(\bar{\phi})]_{y}\widetilde{\phi}_{\tau})
		\notag\\
		\geq&& \frac{1}{2}\frac{d}{d\tau}(\frac{1}{\rho}\widetilde{\phi}_{y},\rho'_{\mathrm{e}}(\bar{\phi})\widetilde{\phi}_{y})
		-Ck^{\frac{1}{12}}\varepsilon^{\frac{3}{5}a-\frac{2}{5}}\mathcal{D}_{2,l,q}(\tau).
	\end{eqnarray*}
	The similar calculations as the above and \eqref{4.35} leads us to
	\begin{eqnarray*}
		|(\frac{1}{\rho}\widetilde{\phi}_{yy},J_{7\tau})|&&\leq|(\frac{1}{\rho}\widetilde{\phi}_{yy},\phi_{\tau}\widetilde{\phi})|+|(\frac{1}{\rho}\widetilde{\phi}_{yy},\bar{\phi}_{\tau}\widetilde{\phi})|
		\notag\\
		&&\leq Ck^{\frac{1}{12}}\varepsilon^{\frac{3}{5}a-\frac{2}{5}}\mathcal{D}_{2,l,q}(\tau)
		+C\varepsilon^{\frac{7}{5}+\frac{1}{15}a}(\delta+\varepsilon^{a}\tau)^{-\frac{4}{3}}.
	\end{eqnarray*}
	It follows from the above four estimates that
	\begin{multline}
		\label{4.71b}
		(\frac{1}{\rho}\widetilde{\phi}_{yy},[\rho_{\mathrm{e}}(\bar{\phi})-\rho_{\mathrm{e}}(\phi)]_{\tau})\\
		\geq \frac{1}{2}\frac{d}{d\tau}(\frac{1}{\rho}\widetilde{\phi}_{y},\rho'_{\mathrm{e}}(\bar{\phi})\widetilde{\phi}_{y})
		-Ck^{\frac{1}{12}}\varepsilon^{\frac{3}{5}a-\frac{2}{5}}\mathcal{D}_{2,l,q}(\tau)
		-C\varepsilon^{\frac{7}{5}+\frac{1}{15}a}(\delta+\varepsilon^{a}\tau)^{-\frac{4}{3}}.
	\end{multline}
	Combining \eqref{4.67-2}, \eqref{4.61a} and \eqref{4.71b}, one has
	\begin{eqnarray*}
		-(\frac{1}{\rho}\widetilde{\phi}_{yy},\widetilde{\rho}_{\tau})\geq&&
		\frac{1}{2}\varepsilon^{2b-2a}\frac{d}{d\tau}\|\frac{1}{\sqrt{\rho}}\widetilde{\phi}_{yy}\|^{2}+ \frac{1}{2}\frac{d}{d\tau}\big(\frac{1}{\rho}\widetilde{\phi}_{y},\rho'_{\mathrm{e}}(\bar{\phi})\widetilde{\phi}_{y}\big)
		\notag\\
		&&-Ck^{\frac{1}{12}}\varepsilon^{\frac{3}{5}a-\frac{2}{5}}\mathcal{D}_{2,l,q}(\tau)
		-C\varepsilon^{\frac{7}{5}+\frac{1}{15}a}(\delta+\varepsilon^{a}\tau)^{-\frac{4}{3}}.
	\end{eqnarray*}
	The second term on the right hand side of \eqref{4.60b} are bounded by
$$|(\frac{1}{\rho}\widetilde{\phi}_{yy},
		\rho_{y} \widetilde{u}_{1}+\bar{u}_{1y}\widetilde{\rho})|
		\leq Ck^{\frac{1}{12}}\varepsilon^{\frac{3}{5}a-\frac{2}{5}}\mathcal{D}_{2,l,q}(\tau)
		+C\varepsilon^{\frac{7}{5}+\frac{1}{15}a}(\delta+\varepsilon^{a}\tau)^{-\frac{4}{3}}.
		$$
To estimate	the last term of \eqref{4.60b}, we first use the last equation of \eqref{3.10} to write
	\begin{equation*}
		-(\frac{\bar{u}_{1}}{\rho}\widetilde{\phi}_{yy},\widetilde{\rho}_{y})
		=(\frac{\bar{u}_{1}}{\rho}\widetilde{\phi}_{yy},\varepsilon^{2b-2a}[\widetilde{\phi}_{yyy}
		+\bar{\phi}_{yyy}]+[\rho_{\mathrm{e}}(\bar{\phi})-\rho_{\mathrm{e}}(\phi)]_{y}).
	\end{equation*}
	
	Using the integration by parts and the similar calculations  as \eqref{4.60aa} and \eqref{4.61a} respectively, one has
	\begin{equation*}
		\varepsilon^{2b-2a}|(\frac{\bar{u}_{1}}{\rho}\widetilde{\phi}_{yy},\widetilde{\phi}_{yyy})|
		=\frac{1}{2}\varepsilon^{2b-2a}|([\frac{\bar{u}_{1}}{\rho}]_{y}\widetilde{\phi}_{yy},\widetilde{\phi}_{yy})|
		\leq Ck^{\frac{1}{12}}\varepsilon^{\frac{3}{5}a-\frac{2}{5}}\mathcal{D}_{2,l,q}(\tau),
	\end{equation*}
	and
	\begin{equation*}
		\varepsilon^{2b-2a}|(\frac{\bar{u}_{1}}{\rho}\widetilde{\phi}_{yy},\bar{\phi}_{yyy})|
		\leq Ck^{\frac{1}{12}}\varepsilon^{\frac{3}{5}a-\frac{2}{5}}\mathcal{D}_{2,l,q}(\tau)
		+C\varepsilon^{\frac{7}{5}+\frac{1}{15}a}(\delta+\varepsilon^{a}\tau)^{-\frac{4}{3}}.
	\end{equation*}
	In addition, we have by applying \eqref{4.34} that
	\begin{eqnarray*}
		|(\frac{\bar{u}_{1}}{\rho}\widetilde{\phi}_{yy},[\rho_{\mathrm{e}}(\bar{\phi})-\rho_{\mathrm{e}}(\phi)]_{y})|
		&&=|-(\frac{\bar{u}_{1}}{\rho}\widetilde{\phi}_{yy},[\rho'_{\mathrm{e}}(\bar{\phi})\widetilde{\phi}]_{y})
		+(\frac{\bar{u}_{1}}{\rho}\widetilde{\phi}_{yy},J_{7y})|
		\notag\\
		&&\leq Ck^{\frac{1}{12}}\varepsilon^{\frac{3}{5}a-\frac{2}{5}}\mathcal{D}_{2,l,q}(\tau)
		+C\varepsilon^{\frac{7}{5}+\frac{1}{15}a}(\delta+\varepsilon^{a}\tau)^{-\frac{4}{3}}.
	\end{eqnarray*}
	With the help of the above estimates, we conclude that
	\begin{equation*}
		|(\frac{\bar{u}_{1}}{\rho}\widetilde{\phi}_{yy},\widetilde{\rho}_{y})|
		\leq Ck^{\frac{1}{12}}\varepsilon^{\frac{3}{5}a-\frac{2}{5}}\mathcal{D}_{2,l,q}(\tau)
		+C\varepsilon^{\frac{7}{5}+\frac{1}{15}a}(\delta+\varepsilon^{a}\tau)^{-\frac{4}{3}}.
	\end{equation*}
	Hence, substituting all the above estimates into \eqref{4.60b}, the desired estimate \eqref{4.49} follows.
	This completes the proof of Lemma \ref{lem.6.2}.
\end{proof}
For deducing \eqref{4.53} in the proof of Lemma \ref{lem6.1} above, we have used the following estimate.
\begin{lemma}\label{lem.6.3}
	It holds that
	\begin{eqnarray}
		\label{4.81a}
		-\varepsilon^{1-a}(\widetilde{\phi}_{yy},\widetilde{\rho}_{yy})\leq&&
		-\frac{1}{2}\varepsilon^{1-a}\varepsilon^{2b-2a}(\widetilde{\phi}_{yyy},\widetilde{\phi}_{yyy})
		-\varepsilon^{1-a}(\widetilde{\phi}_{yy},\rho'_{\mathrm{e}}(\bar{\phi})\widetilde{\phi}_{yy})
		\notag\\
		&&+Ck^{\frac{1}{12}}\varepsilon^{\frac{3}{5}-\frac{2}{5}a}\mathcal{D}_{2,l,q}(\tau)
		+C\varepsilon^{\frac{7}{5}+\frac{1}{15}a}(\delta+\varepsilon^{a}\tau)^{-\frac{4}{3}}.
	\end{eqnarray}	
\end{lemma}
\begin{proof}	
	Using the last equation of \eqref{3.10} and integration by parts, one has
	\begin{multline}
		\label{4.82a}
		-(\widetilde{\phi}_{yy},\widetilde{\rho}_{yy})\\
		=-\varepsilon^{2b-2a}[
		(\widetilde{\phi}_{yyy},\widetilde{\phi}_{yyy})
		+(\widetilde{\phi}_{yyy},\bar{\phi}_{yyy})]
		+(\widetilde{\phi}_{yy},[\rho_{\mathrm{e}}(\bar{\phi})-\rho_{\mathrm{e}}(\phi)]_{yy}).
	\end{multline} 
	Following the similar calculation as \eqref{4.36a}, it holds that
	\begin{eqnarray}
		\label{4.69a}
		&&\varepsilon^{1-a}\varepsilon^{2b-2a}|(\widetilde{\phi}_{yyy},\bar{\phi}_{yyy})|
		\notag\\
		&&\leq \eta\varepsilon^{1-a}\varepsilon^{2b-2a}\|\widetilde{\phi}_{yyy}\|^{2}
		+C_{\eta}\varepsilon^{1-a}\varepsilon^{2b-2a}\|\bar{\phi}_{yyy}\|^{2}
		\notag\\
		&&\leq \eta\varepsilon^{1-a}\varepsilon^{2b-2a}\|\widetilde{\phi}_{yyy}\|^{2}
		+C_{\eta}\varepsilon^{\frac{7}{5}+\frac{1}{15}a}(\delta+\varepsilon^{a}\tau)^{-\frac{4}{3}}.
	\end{eqnarray}
To estimate the last term in \eqref{4.82a},	we first use \eqref{4.34} to write
	\begin{multline} 
		\label{6.26b}
		(\widetilde{\phi}_{yy},[\rho_{\mathrm{e}}(\bar{\phi})-\rho_{\mathrm{e}}(\phi)]_{yy})\\
		=-(\widetilde{\phi}_{yy},\rho'_{\mathrm{e}}(\bar{\phi})\widetilde{\phi}_{yy})
		-(\widetilde{\phi}_{yy},\rho''_{\mathrm{e}}(\bar{\phi})\bar{\phi}_{y}\widetilde{\phi}_{y} +[\rho''_{\mathrm{e}}(\bar{\phi})\bar{\phi}_{y}\widetilde{\phi}]_{y})
		+(\widetilde{\phi}_{yy},J_{7yy}).
	\end{multline}
	By Lemma \ref{lem7.2}, \eqref{3.21} and \eqref{3.22}, we get
	\begin{multline*}
		\varepsilon^{1-a}|(\widetilde{\phi}_{yy},\rho''_{\mathrm{e}}(\bar{\phi})\bar{\phi}_{y}\widetilde{\phi}_{y}+[\rho''_{\mathrm{e}}(\bar{\phi})\bar{\phi}_{y}\widetilde{\phi}]_{y})|\\
		\leq Ck^{\frac{1}{12}}\varepsilon^{\frac{3}{5}-\frac{2}{5}a}\mathcal{D}_{2,l,q}(\tau)
		+C\varepsilon^{\frac{7}{5}+\frac{1}{15}a}(\delta+\varepsilon^{a}\tau)^{-\frac{4}{3}}.
	\end{multline*}
while for the last term of \eqref{6.26b}, it holds that
	\begin{equation*}
		\varepsilon^{1-a}|(\widetilde{\phi}_{yy},J_{7yy})|
		\leq Ck^{\frac{1}{12}}\varepsilon^{\frac{3}{5}-\frac{2}{5}a}\mathcal{D}_{2,l,q}(\tau)
		+C\varepsilon^{\frac{7}{5}+\frac{1}{15}a}(\delta+\varepsilon^{a}\tau)^{-\frac{4}{3}}.
	\end{equation*}
Substituting the above two estimates into \eqref{6.26b},  we deduce that
	\begin{multline}
		\label{4.70a}
		\varepsilon^{1-a}(\widetilde{\phi}_{yy},[\rho_{\mathrm{e}}(\bar{\phi})-\rho_{\mathrm{e}}(\phi)]_{yy})
		\leq -\varepsilon^{1-a}(\widetilde{\phi}_{yy},\rho'_{\mathrm{e}}(\bar{\phi})\widetilde{\phi}_{yy})
		\\
		+Ck^{\frac{1}{12}}\varepsilon^{\frac{3}{5}-\frac{2}{5}a}\mathcal{D}_{2,l,q}(\tau) +C\varepsilon^{\frac{7}{5}+\frac{1}{15}a}(\delta+\varepsilon^{a}\tau)^{-\frac{4}{3}}.
	\end{multline}
Consequently, combining \eqref{4.82a}, \eqref{4.69a}, \eqref{4.70a} and  using the fact that $\rho'_{\mathrm{e}}(\bar{\phi})>c$,
we can obtain \eqref{4.81a} by choosing  $\eta>0$  small enough. This ends up the proof of Lemma \ref{lem.6.3}.
\end{proof}

\subsection{First order derivative estimates on non-fluid part}
Next we turn to derive the first order derivative estimates for the non-fluid part $f$.
As before, the proof is based on the microscopic equation \eqref{3.7}.
\begin{lemma}\label{lem.6.4}
	Under the conditions listed in Lemma \ref{lem.5.1A}, it holds that
	\begin{eqnarray}
		\label{4.59}
		&&\frac{1}{2}\frac{d}{d\tau}\|e^{\frac{\phi}{2}}f_y\|^{2}
		+c\varepsilon^{a-1}\|f_y\|_{\sigma}^{2}\notag\\
		\leq&& C\varepsilon^{1-a}\{\|(\widetilde{u}_{yy},\widetilde{\theta}_{yy})\|^{2}+\|f_{yy}\|_{\sigma}^{2}\}
		+C(\eta_{0}+k^{\frac{1}{12}}\varepsilon^{\frac{3}{5}-\frac{2}{5}a})\mathcal{D}_{2,l,q}(\tau)
		\notag\\
		&&+C\varepsilon^{\frac{7}{5}+\frac{1}{15}a}(\delta+\varepsilon^{a}\tau)^{-\frac{4}{3}}
		+Cq_3(\tau)\mathcal{H}_{2,l,q}(\tau).
	\end{eqnarray}
\end{lemma}
\begin{proof}
	Applying $\partial^{\alpha}$ to equation \eqref{3.7} with $|\alpha|=1$ and further taking the inner product of the resulting equation with $e^{\phi}\partial^{\alpha}f$,
	we have
	\begin{align}
		\label{4.58}
		&(\partial^{\alpha}f_{\tau},e^{\phi}\partial^{\alpha}f)
		+(v_{1}\partial^{\alpha}f_{y}+\frac{v_{1}}{2}\phi_{y} \partial^{\alpha}f,e^{\phi}\partial^{\alpha}f)
		+(\frac{v_{1}}{2}\partial^{\alpha}\phi_{y} f,e^{\phi}\partial^{\alpha}f)
		\notag\\
		&-(\phi_{y} \partial_{v_{1}}\partial^{\alpha}f,e^{\phi}\partial^{\alpha}f)
		-(\partial^{\alpha}\phi_{y} \partial_{v_{1}}f,e^{\phi}\partial^{\alpha}f)
		-\varepsilon^{a-1}(\mathcal{L}\partial^{\alpha}f,e^{\phi}\partial^{\alpha}f)
		\notag\\
		=&\varepsilon^{a-1}\big\{(\partial^{\alpha}\Gamma(f,\frac{M-\mu}{\sqrt{\mu}})+
		\partial^{\alpha}\Gamma(\frac{M-\mu}{\sqrt{\mu}},f)+\partial^{\alpha}\Gamma(\frac{G}{\sqrt{\mu}},\frac{G}{\sqrt{\mu}})
		,e^{\phi}\partial^{\alpha}f )\big\}
		\notag\\
		&+(\partial^{\alpha}[\frac{P_{0}(v_{1}\sqrt{\mu}f_{y})}{\sqrt{\mu}}]
		-\partial^{\alpha}[\frac{1}{\sqrt{\mu}}P_{1}\{v_{1}M(\frac{|v-u|^{2}
			\widetilde{\theta}_{y}}{2R\theta^{2}}+\frac{(v-u)\cdot\widetilde{u}_{y}}{R\theta})\}],e^{\phi}\partial^{\alpha}f)
		\notag\\
		&+(\partial^{\alpha}[\frac{\phi_{y}\partial_{v_{1}}\overline{G}}{\sqrt{\mu}}] -\partial^{\alpha}[\frac{P_{1}(v_{1}\overline{G}_{y})}{\sqrt{\mu}}] -\partial^{\alpha}[\frac{\overline{G}_{\tau}}{\sqrt{\mu}}],e^{\phi}\partial^{\alpha}f).
	\end{align}
	Let's compute \eqref{4.58} term by term. Using the similar arguments as \eqref{4.41}, we get
	\begin{eqnarray*}
		(\partial^{\alpha}f_{\tau},e^{\phi}\partial^{\alpha}f)
		&&=\frac{1}{2}(\partial^{\alpha}f,e^{\phi}\partial^{\alpha}f)_{\tau}-\frac{1}{2}(\partial^{\alpha}f,e^{\phi}\phi_{\tau}\partial^{\alpha}f)
		\notag\\
		&&\geq \frac{1}{2}(\partial^{\alpha}f,e^{\phi}\partial^{\alpha}f)_{\tau}-C\eta\varepsilon^{a-1}\|\partial^{\alpha}f\|_{\sigma}^{2}
		-C_{\eta}q_3(\tau)\mathcal{H}_{2,l,q}(\tau).
	\end{eqnarray*}
On the other hand, we observe from an integration by parts that
	\begin{equation*}
		(v_{1}\partial^{\alpha}f_{y}+\frac{v_{1}}{2}\phi_{y}\partial^{\alpha}f,e^{\phi}\partial^{\alpha}f)=(v_{1}[e^{\frac{\phi}{2}}\partial^{\alpha}f]_{y},e^{\frac{\phi}{2}} \partial^{\alpha}f)=0.
	\end{equation*}
	The fourth term on the left hand side of \eqref{4.58} vanishes after integration by parts while for the third and fifth terms on the left hand side of \eqref{4.58}, they
	can be estimated as  follows                                                                           
	\begin{eqnarray*}
		&&|(\frac{v_{1}}{2}\partial^{\alpha}\phi_{y} f,e^{\phi}\partial^{\alpha}f)|\\
		\leq&&
		\eta\varepsilon^{a-1}\| \langle v\rangle^{-\frac{1}{2}}\partial^{\alpha}f\|^{2}
		+C_{\eta}\varepsilon^{1-a}\|\partial^{\alpha}\phi_{y}\|^2_{L_y^{\infty}}\|\langle v_{1}\rangle \langle v\rangle^{\frac{1}{2}}f\|^{2}
		\notag\\
		\leq&&
		C\eta\varepsilon^{a-1}\|\partial^{\alpha}f\|_{\sigma}^{2}
		+C_{\eta}\varepsilon^{1-a}(\|\phi_{yy}\|^{2}+\|\phi_{yyy}\|^{2})\|\langle v\rangle^{\frac{1}{2}}f\|_{w}^{2}
		\notag\\
		\leq&& C\eta\varepsilon^{a-1}\|\partial^{\alpha}f\|_{\sigma}^{2}+C_{\eta}q_3(\tau)\mathcal{H}_{2,l,q}(\tau),
	\end{eqnarray*}
	and
	\begin{equation*}
		|(\partial^{\alpha}\phi_{y} \partial_{v_{1}}f,e^{\phi}\partial^{\alpha}f)|
		\leq C\eta\varepsilon^{a-1}\|\partial^{\alpha}f\|_{\sigma}^{2}+C_{\eta}q_3(\tau)\mathcal{H}_{2,l,q}(\tau).
	\end{equation*}
	Here we have used \eqref{3.14}, \eqref{4.60} and $\langle v_{1}\rangle\leq  \langle v\rangle^{2l}$ with $l\geq\frac{1}{2}$.
	Thanks to \eqref{3.9} and \eqref{4.45a}, the last term on the left hand side of \eqref{4.58} has a lower bound by
	\begin{equation*}
		-\varepsilon^{a-1}(\mathcal{L}\partial^{\alpha}f,e^{\phi}\partial^{\alpha}f)\geq c\varepsilon^{a-1}\|\partial^{\alpha}f\|_{\sigma}^{2}.
	\end{equation*}
	
	Next we go to bound the right hand side of \eqref{4.58}, we use the similar arguments as \eqref{7.10} and \eqref{7.19} to get
	\begin{eqnarray*}
		&&\varepsilon^{a-1}|(\partial^{\alpha}\Gamma(f,\frac{M-\mu}{\sqrt{\mu}})
		+\partial^{\alpha}\Gamma(\frac{M-\mu}{\sqrt{\mu}},f)+\partial^{\alpha}\Gamma(\frac{G}{\sqrt{\mu}},\frac{G}{\sqrt{\mu}}),e^{\phi}\partial^{\alpha}f)|
		\notag\\
		&&\leq C\eta\varepsilon^{a-1}\|\partial^{\alpha}f\|^{2}_{\sigma}
		+C_{\eta}(\eta_{0}+k^{\frac{1}{12}}\varepsilon^{\frac{3}{5}-\frac{2}{5}a})\mathcal{D}_{2,l,q}(\tau)
		+C_{\eta}\varepsilon^{\frac{7}{5}+\frac{1}{15}a}(\delta+\varepsilon^{a}\tau)^{-\frac{4}{3}}.
	\end{eqnarray*}
	By applying \eqref{1.10}, \eqref{3.16}, \eqref{4.45a}, \eqref{7.25aa}, Lemma \ref{lem7.2},
	\eqref{3.18}, \eqref{3.21} and \eqref{3.22}, one has
	\begin{eqnarray}
		\label{6.30b}
		&&|(\frac{\partial^{\alpha}P_{0}(v_{1}\sqrt{\mu}f_{y})}{\sqrt{\mu}}-\frac{1}{\sqrt{\mu}}\partial^{\alpha}P_{1}\{v_{1}M(\frac{|v-u|^{2}
			\widetilde{\theta}_{y}}{2R\theta^{2}}+\frac{(v-u)\cdot\widetilde{u}_{y}}{R\theta})\},e^{\phi}\partial^{\alpha}f)|
		\notag\\
		&&\leq \eta\varepsilon^{a-1}\|\langle v\rangle^{-\frac{1}{2}}\partial^{\alpha}f\|^{2}
		+C_{\eta}\varepsilon^{1-a}\|\langle v\rangle^{\frac{1}{2}}\frac{\partial^{\alpha}P_{0}(v_{1}\sqrt{\mu}f_{y})}{\sqrt{\mu}}\|^{2}
		\notag\\
		&&\hspace{1.5cm}+C_{\eta}\varepsilon^{1-a}\|\langle v\rangle^{\frac{1}{2}}\mu^{-\frac{1}{2}}\partial^{\alpha}P_{1}\{v_{1}M(\frac{|v-u|^{2}
			\widetilde{\theta}_{y}}{2R\theta^{2}}+\frac{(v-u)\cdot\widetilde{u}_{y}}{R\theta})\}\|^{2}
		\notag\\
		&&\leq C\eta\varepsilon^{a-1}\|\partial^{\alpha}f\|^{2}_{\sigma}+C_{\eta}\varepsilon^{1-a}
		(\|\partial^{\alpha}f_{y}\|^{2}_{\sigma}+\|(\partial^{\alpha}\widetilde{u}_{y},\partial^{\alpha}\widetilde{\theta}_{y})\|^{2})\notag\\
		&&\quad+C_{\eta}k^{\frac{1}{12}}\varepsilon^{\frac{3}{5}-\frac{2}{5}a}\mathcal{D}_{2,l,q}(\tau).
	\end{eqnarray}
The remaining terms in \eqref{4.58} can be bounded by 
	\begin{eqnarray*}
		&&|(\partial^{\alpha}[\frac{\phi_{y}\partial_{v_{1}}\overline{G}}{\sqrt{\mu}}]-\partial^{\alpha}[\frac{P_{1}(v_{1}\overline{G}_{y})}{\sqrt{\mu}}]
		-\partial^{\alpha}[\frac{\overline{G}_{\tau}}{\sqrt{\mu}}],e^{\phi}\partial^{\alpha}f)|
		\notag\\
		&& \leq C\eta\varepsilon^{a-1}\|\partial^{\alpha}f\|^{2}_{\sigma}
		+C_{\eta}k^{\frac{1}{12}}\varepsilon^{\frac{3}{5}-\frac{2}{5}a}\mathcal{D}_{2,l,q}(\tau)
		+C_{\eta}\varepsilon^{\frac{7}{5}+\frac{1}{15}a}(\delta+\varepsilon^{a}\tau)^{-\frac{4}{3}}.
	\end{eqnarray*}
Therefore,	substituting the above  estimates into \eqref{4.58} and taking $\eta>0$ small enough, we can prove \eqref{4.59} holds.
	This completes the proof of Lemma \ref{lem.6.4}.
\end{proof}
From Lemma \ref{lem6.1} and Lemma \ref{lem.6.4}, we immediately have 
\begin{lemma}\label{lem.6.5}
	It holds that
	\begin{align}
		\label{4.61}
		&\|(\widetilde{\rho}_y,\widetilde{u}_y,\widetilde{\theta}_y)(\tau)\|^{2}+\|f_y(\tau)\|^{2}
		+\|\widetilde{\phi}_y(\tau)\|^{2}+\varepsilon^{2b-2a}\|\widetilde{\phi}_{yy}(\tau)\|^{2}\notag\\
		&\quad+\varepsilon^{a-1}\int^{\tau}_{0}\|f_{y}(s)\|_{\sigma}^{2}\,ds
		\notag\\
		&\hspace{0.5cm}+\varepsilon^{1-a}\int^{\tau}_{0}\big\{\|(\widetilde{\rho}_{yy},\widetilde{u}_{yy},\widetilde{\theta}_{yy})(s)\|^{2}
		+\|\widetilde{\phi}_{yy}(s)\|^{2}+\varepsilon^{2b-2a}\|\widetilde{\phi}_{yyy}(s)\|^{2}\big\}\,ds
		\notag\\
		&\leq C\varepsilon^{2(1-a)}\|(\widetilde{\rho}_{yy},\widetilde{u}_{yy},\widetilde{\theta}_{yy})(\tau)\|^{2}
		+Ck^{\frac{1}{3}}\varepsilon^{\frac{6}{5}-\frac{4}{5}a}
		+C\varepsilon^{1-a}\int^{\tau}_{0}\|f_{yy}(s)\|^{2}_{\sigma}\,ds
		\notag\\
		&\quad+C(\eta_{0}+k^{\frac{1}{12}}\varepsilon^{\frac{3}{5}-\frac{2}{5}a}
		+k^{\frac{1}{12}}\varepsilon^{\frac{3}{5}a-\frac{2}{5}})\int^{\tau}_{0}\mathcal{D}_{2,l,q}(s)\,ds\notag\\
		&\quad+C\int^{\tau}_{0}q_{3}(s)\mathcal{H}_{2,l,q}(s)\,ds.
	\end{align}
\end{lemma}
\begin{proof}
	In view of \eqref{6.12b}, \eqref{4.14}, Lemma \ref{lem7.2}, \eqref{3.21} and \eqref{3.22}, we get
	\begin{equation}
		\label{6.32b}
		N_2(\tau)
		\leq \eta\|f_y\|^{2}+C_\eta\varepsilon^{2-2a}\|(\widetilde{u}_{yy},\widetilde{\theta}_{yy})\|^{2}
		+C_\eta k^{\frac{1}{3}}\varepsilon^{\frac{6}{5}-\frac{4}{5}a}.
	\end{equation}	
 By the suitable linear combinations of  $\eqref{4.57}$ and \eqref{4.59}, we have 
 the uniform estimate \eqref{4.61} by integrating the resulting equation with respect to $\tau$, where
 \eqref{5.50A} and \eqref{6.32b} with a small $\eta>0$ have used.  This concludes the proof of Lemma \ref{lem.6.5}.
\end{proof}

\subsection{Second order derivative estimates}
In order to complete the derivative estimates of all orders, we still need to derive the second-order derivative estimates.
\begin{lemma}\label{lem.6.6}
	For any small $\eta>0$, it holds that
	\begin{align}
		\label{4.76}
		&\varepsilon^{2(1-a)}\sum_{|\alpha|=2}(\|\partial^{\alpha}(\widetilde{\rho},\widetilde{u},\widetilde{\theta})(\tau)\|^{2}
		+\|\partial^{\alpha}f(\tau)\|^{2}+\|\partial^{\alpha}\widetilde{\phi}(\tau)\|^{2}
		+\varepsilon^{2b-2a}\|\partial^{\alpha}\widetilde{\phi}_{y}(\tau)\|^{2})\notag\\
&\quad		+\varepsilon^{1-a}\sum_{|\alpha|=2}\int^{\tau}_{0}\|\partial^{\alpha}f(s)\|_{\sigma}^{2}\,ds
		\notag\\
		\leq& C(\eta+\eta_{0}+k^{\frac{1}{12}}\varepsilon^{\frac{3}{5}-\frac{2}{5}a})\varepsilon^{1-a}\sum_{|\alpha|=2}\int^{\tau}_{0}
		\big\{\|\partial^{\alpha}(\widetilde{\rho},\widetilde{u},\widetilde{\theta})(s)\|^{2}+\| \partial^\alpha f(s)\|^{2}_{\sigma}\notag\\
&\hspace{6cm}		+\|\partial^{\alpha}\widetilde{\phi}(s)\|^{2}+\varepsilon^{2(1-a)}\|\partial^{\alpha}\widetilde{\phi}_{y}(s)\|^{2}\big\}\, ds
		\notag\\
		&\quad+C_{\eta}k^{\frac{1}{3}}\varepsilon^{\frac{6}{5}-\frac{4}{5}a}+
		C_{\eta}k^{\frac{1}{12}}\varepsilon^{\frac{3}{5}-\frac{2}{5}a}\int^{\tau}_{0}\mathcal{D}_{2,l,q}(s)\,ds\notag\\
		&\quad+C_{\eta}
		\int^{\tau}_{0}q_{3}(s)\mathcal{H}_{2,l,q}(s)\,ds.
	\end{align}	
\end{lemma}
\begin{proof}
	Recall the first equation of \eqref{3.5}, we have from \eqref{3.8} and a direct calculation that
	\begin{multline}
		\label{4.62}
		(\frac{F}{\sqrt{\mu}})_{\tau}+v_{1}(\frac{F}{\sqrt{\mu}})_{y}
		+\frac{v_{1}}{2}\phi_{y}(\frac{F}{\sqrt{\mu}})
		-\phi_{y}\partial_{v_{1}}(\frac{F}{\sqrt{\mu}})
		\\
		=\varepsilon^{a-1}\big\{\mathcal{L} f+\Gamma(f,\frac{M-\mu}{\sqrt{\mu}})+\Gamma(\frac{M-\mu}{\sqrt{\mu}},f)
		+\Gamma(\frac{G}{\sqrt{\mu}},\frac{G}{\sqrt{\mu}})+\frac{L_{M}\overline{G}}{\sqrt{\mu}}\big\}.
	\end{multline}
After acting  $\partial^{\alpha}$ with $|\alpha|=2$ to equation \eqref{4.62} and further taking the inner product of the resulting equation with $e^{\phi}\frac{\partial^{\alpha}F}{\sqrt{\mu}}$  over $\mathbb{R}\times{\mathbb R}^3$, the direct energy estimate gives the identity
	\begin{align}
		\label{4.63}
		&\frac{1}{2}(\frac{\partial^\alpha F}{\sqrt{\mu}},e^{\phi}\frac{\partial^\alpha F}{\sqrt{\mu}})_{\tau}
		-\frac{1}{2}(\frac{\partial^\alpha F}{\sqrt{\mu}},e^{\phi}\phi_{\tau}\frac{\partial^\alpha F}{\sqrt{\mu}})
		+(v_1(\frac{\partial^\alpha F}{\sqrt{\mu}})_{y}+\frac{v_{1}}{2}\phi_{y}\frac{\partial^{\alpha}F}{\sqrt{\mu}},
		e^{\phi}\frac{\partial^\alpha F}{\sqrt{\mu}})
		\notag\\
		&\quad-(\phi_{y}\partial_{v_{1}}(\frac{\partial^{\alpha}F}{\sqrt{\mu}}),e^{\phi}\frac{\partial^\alpha F}{\sqrt{\mu}})\notag\\
		&\quad+\sum_{1\leq\alpha_{1}\leq\alpha}C^{\alpha_{1}}_{\alpha}(\frac{v_{1}}{2}\partial^{\alpha_{1}}\phi_{y}\frac{
			\partial^{\alpha-\alpha_{1}}F}{\sqrt{\mu}}-\partial^{\alpha_{1}}\phi_{y}\partial_{v_{1}}(\frac{\partial^{\alpha-\alpha_{1}}F}{\sqrt{\mu}}),
		e^{\phi}\frac{\partial^\alpha F}{\sqrt{\mu}})
		\notag\\
		&=\varepsilon^{a-1}(\mathcal{L} \partial^\alpha f,e^{\phi}\frac{\partial^\alpha F}{\sqrt{\mu}})
		+\varepsilon^{a-1}(\partial^\alpha\Gamma(f,\frac{M-\mu}{\sqrt{\mu}})
		+\partial^\alpha\Gamma(\frac{M-\mu}{\sqrt{\mu}},f),e^{\phi}\frac{\partial^\alpha F}{\sqrt{\mu}})
		\notag\\
		&\hspace{0.5cm}+\varepsilon^{a-1}(\partial^\alpha\Gamma(\frac{G}{\sqrt{\mu}},\frac{G}{\sqrt{\mu}}),e^{\phi}\frac{\partial^\alpha F}{\sqrt{\mu}})+\varepsilon^{a-1}(\frac{\partial^{\alpha}L_{M}\overline{G}}{\sqrt{\mu}},e^{\phi}\frac{\partial^\alpha F}{\sqrt{\mu}}).
	\end{align}
	We  estimate \eqref{4.63} term by term. We first consider the right hand terms in \eqref{4.63}.
	By a simple computation, one has the following identity
	$$
	M_{y}=M\{\frac{\rho_{y}}{\rho}+\frac{(v-u)\cdot u_{y}}{R\theta}
	+(\frac{|v-u|^{2}}{2R\theta}-\frac{3}{2})\frac{\theta_{y}}{\theta} \}.
	$$
	Then, for $|\alpha|= 2$, it holds that
	\begin{eqnarray}
		\label{4.67}
		\partial^{\alpha}M&&=M\{\frac{\partial^{\alpha}\rho}{\rho}+\frac{(v-u)\cdot\partial^{\alpha}u}{R\theta}+(\frac{|v-u|^{2}}{2R\theta}-\frac{3}{2})\frac{\partial^{\alpha}\theta}{\theta}\}+\cdot\cdot\cdot
		\notag\\
		&&=\{\mu+(M-\mu)\}
		\{\frac{\partial^{\alpha}\rho}{\rho}+\frac{(v-u)\cdot\partial^{\alpha}u}{R\theta}+(\frac{|v-u|^{2}}{2R\theta}-\frac{3}{2})\frac{\partial^{\alpha}\theta}{\theta}\}
		+\cdot\cdot\cdot
		\notag\\
		&&:=\mathbb{J}_{1}
		+\mathbb{J}_{2}+\mathbb{J}_{3}.
	\end{eqnarray}
	Here the terms $\mathbb{J}_{1}$ and $\mathbb{J}_{2}$ are the second-order derivatives of $(\rho,u,\theta)$
	with $\mu$ and $M-\mu$ and $\mathbb{J}_{3}$ is the low order derivatives of $(\rho,u,\theta)$.
	Since $\frac{\mathbb{J}_{1}}{\sqrt{\mu}}\in\ker{\mathcal{L}}$, it follows that
	$(\mathcal{L}f,e^{\phi}\frac{\mathbb{J}_{1}}{\sqrt{\mu}})=0$. For the terms $\frac{\mathbb{J}_{2}}{\sqrt{\mu}}$ and
	$\frac{\mathbb{J}_{3}}{\sqrt{\mu}}$, recalling that $\mathcal{L}f=\Gamma(\sqrt{\mu},f)+\Gamma(f,\sqrt{\mu})$,
	we then use \eqref{7.7}, \eqref{4.45a}, \eqref{7.12}, \eqref{7.25aa}, Lemma \ref{lem7.2}, \eqref{3.21} and \eqref{3.22} to get
	\begin{eqnarray*}
		&&\varepsilon^{a-1}|(\mathcal{L}\partial^{\alpha}f,e^{\phi}\frac{\mathbb{J}_{2}}{\sqrt{\mu}})|
		\leq C\varepsilon^{a-1}\|\partial^{\alpha}f\|_{\sigma}\|\frac{\mathbb{J}_{2}}{\sqrt{\mu}}\|_{\sigma}
		\notag\\
		&&\leq C(\eta_{0}+k^{\frac{1}{12}}\varepsilon^{\frac{3}{5}-\frac{2}{5}a})
		\varepsilon^{a-1}(\|\partial^{\alpha}f\|_{\sigma}^{2}+\|\partial^{\alpha}(\widetilde{\rho},\widetilde{u},\widetilde{\theta})\|^{2})\\
		&&\quad+C\varepsilon^{a-1}\varepsilon^{\frac{7}{5}+\frac{1}{15}a}(\delta+\varepsilon^{a}\tau)^{-\frac{4}{3}},
	\end{eqnarray*}
	and
	\begin{multline*}
		\varepsilon^{a-1}|(\mathcal{L}\partial^{\alpha}f,e^{\phi}\frac{\mathbb{J}_{3}}{\sqrt{\mu}})|\\
		\leq\eta\varepsilon^{a-1}\|\partial^{\alpha}f\|_{\sigma}^{2}
		+C_{\eta}k^{\frac{1}{12}}\varepsilon^{\frac{3}{5}a-\frac{2}{5}}\mathcal{D}_{2,l,q}(\tau)
		+C_{\eta}\varepsilon^{a-1}\varepsilon^{\frac{7}{5}+\frac{1}{15}a}(\delta+\varepsilon^{a}\tau)^{-\frac{4}{3}}.
	\end{multline*}
	It follows from those estimates that
	\begin{eqnarray*}
		&&\varepsilon^{a-1}|(\mathcal{L}\partial^{\alpha}f,e^{\phi}\frac{\partial^{\alpha}M}{\sqrt{\mu}})|\\
		&&\leq  C(\eta+\eta_{0}+k^{\frac{1}{12}}\varepsilon^{\frac{3}{5}-\frac{2}{5}a})\varepsilon^{a-1}(\|\partial^{\alpha}f\|_{\sigma}^{2}
		+\|\partial^{\alpha}(\widetilde{\rho},\widetilde{u},\widetilde{\theta})\|^{2})
		\notag\\
		&&\quad +C_{\eta}k^{\frac{1}{12}}\varepsilon^{\frac{3}{5}a-\frac{2}{5}}\mathcal{D}_{2,l,q}(\tau)
		+C_{\eta}\varepsilon^{a-1}\varepsilon^{\frac{7}{5}+\frac{1}{15}a}(\delta+\varepsilon^{a}\tau)^{-\frac{4}{3}}.
	\end{eqnarray*}
	On the other hand, it holds that
	\begin{multline}
		\label{6.39A}
		\varepsilon^{a-1}|(\mathcal{L}\partial^{\alpha}f,e^{\phi}\frac{\partial^{\alpha}\overline{G}}{\sqrt{\mu}})|\\
		\leq \eta\varepsilon^{a-1}\|\partial^{\alpha}f\|^{2}_{\sigma}
		+C_{\eta}k^{\frac{1}{12}}\varepsilon^{\frac{3}{5}a-\frac{2}{5}}\mathcal{D}_{2,l,q}(\tau)
		+C_{\eta}\varepsilon^{a-1}\varepsilon^{\frac{7}{5}+\frac{1}{15}a}(\delta+\varepsilon^{a}\tau)^{-\frac{4}{3}},
	\end{multline}
	and
	$$
	-\varepsilon^{a-1}(\mathcal{L}\partial^{\alpha}f,e^{\phi}\partial^{\alpha}f)
	\geq c\varepsilon^{a-1}\|\partial^{\alpha}f\|_{\sigma}^{2}.
	$$
	With the above three estimates and the fact that $F=M+\overline{G}+\sqrt{\mu}f$ in hand, we deduce that
	\begin{eqnarray}
		\label{4.89A}
		&&\varepsilon^{a-1}(\mathcal{L}\partial^{\alpha}f,e^{\phi}\frac{\partial^{\alpha}F}{\sqrt{\mu}})\notag\\
		&&\leq -c\varepsilon^{a-1}\|\partial^{\alpha}f\|_{\sigma}^{2}
		+C(\eta+\eta_{0}+k^{\frac{1}{12}}\varepsilon^{\frac{3}{5}-\frac{2}{5}a})\varepsilon^{a-1}(\|\partial^{\alpha}f\|_{\sigma}^{2}
		+\|\partial^{\alpha}(\widetilde{\rho},\widetilde{u},\widetilde{\theta})\|^{2})
		\notag\\
		&&\quad +C_{\eta}k^{\frac{1}{12}}\varepsilon^{\frac{3}{5}a-\frac{2}{5}}\mathcal{D}_{2,l,q}(\tau)
		+C_{\eta}\varepsilon^{a-1}\varepsilon^{\frac{7}{5}+\frac{1}{15}a}(\delta+\varepsilon^{a}\tau)^{-\frac{4}{3}}.
	\end{eqnarray}
	For $|\alpha|=2$, by \eqref{4.67},  \eqref{7.25}, Lemma \ref{lem7.2}, \eqref{3.21} and \eqref{3.22}, we claim that
	\begin{eqnarray}
		\label{4.68}
		\|\frac{\partial^{\alpha}F}{\sqrt{\mu}}\|_{\sigma}^{2}
		&&\leq \|\frac{\sqrt{\mu}\partial^{\alpha}f}{\sqrt{\mu}}\|_{\sigma}^{2}
		+\|\frac{\partial^{\alpha}\overline{G}}{\sqrt{\mu}}\|_{\sigma}^{2}
		+\|\frac{\partial^{\alpha}M}{\sqrt{\mu}}\|_{\sigma}^{2}
		\notag\\
		&&\leq C(\|\partial^{\alpha}f\|^{2}_{\sigma}
		+\|\partial^{\alpha}(\widetilde{\rho},\widetilde{u},\widetilde{\theta})\|^{2})
		+Ck^{\frac{1}{12}}\varepsilon^{\frac{3}{5}-\frac{2}{5}a}\mathcal{D}_{2,l,q}(\tau)\notag\\
 &&\quad		+C\varepsilon^{\frac{7}{5}+\frac{1}{15}a}(\delta+\varepsilon^{a}\tau)^{-\frac{4}{3}},
	\end{eqnarray}
which together with \eqref{7.7}, \eqref{4.45a} and \eqref{7.12} yields that
	\begin{eqnarray}
		\label{4.85}
		&&\varepsilon^{a-1}\{|(\partial^{\alpha}\Gamma(\frac{M-\mu}{\sqrt{\mu}},f),e^{\phi}\frac{\partial^{\alpha}F}{\sqrt{\mu}})|+
		|(\partial^{\alpha}\Gamma(f,\frac{M-\mu}{\sqrt{\mu}}),e^{\phi}\frac{\partial^{\alpha}F}{\sqrt{\mu}})|\}
		\notag\\
		&&\leq C(\eta+\eta_{0}+k^{\frac{1}{12}}\varepsilon^{\frac{3}{5}-\frac{2}{5}a})\varepsilon^{a-1}(\|\partial^{\alpha}f\|^{2}_{\sigma}
		+\|\partial^{\alpha}(\widetilde{\rho},\widetilde{u},\widetilde{\theta})\|^{2})
		\notag\\
		&&\hspace{0.5cm}+C_{\eta}k^{\frac{1}{12}}\varepsilon^{\frac{3}{5}a-\frac{2}{5}}\mathcal{D}_{2,l,q}(\tau)
		+C_{\eta}\varepsilon^{a-1}\varepsilon^{\frac{7}{5}+\frac{1}{15}a}(\delta+\varepsilon^{a}\tau)^{-\frac{4}{3}}.
	\end{eqnarray}
	By similar arguments as \eqref{4.34A} and \eqref{4.68}, we get
	\begin{eqnarray*}
		&&\varepsilon^{a-1}|(\partial^{\alpha}\Gamma(\frac{G}{\sqrt{\mu}},\frac{G}{\sqrt{\mu}}),e^{\phi}\frac{\partial^{\alpha}F}{\sqrt{\mu}})|\notag\\
		&&\leq  C\eta\varepsilon^{a-1}(\|\partial^{\alpha}f\|^{2}_{\sigma}
		+\|\partial^{\alpha}(\widetilde{\rho},\widetilde{u},\widetilde{\theta})\|^{2})
		\notag\\
		&&\quad +C_{\eta}k^{\frac{1}{12}}\varepsilon^{\frac{3}{5}a-\frac{2}{5}}\mathcal{D}_{2,l,q}(\tau)
		+C_{\eta}\varepsilon^{a-1}\varepsilon^{\frac{7}{5}+\frac{1}{15}a}(\delta+\varepsilon^{a}\tau)^{-\frac{4}{3}}.
	\end{eqnarray*}
	In view of \eqref{3.4}, \eqref{4.45a} and \eqref{4.68}, the last term of \eqref{4.63} is bounded by
	\begin{eqnarray}
		\label{4.97b}
		&&\varepsilon^{a-1}|(\frac{\partial^{\alpha}L_{M}\overline{G}}{\sqrt{\mu}},e^{\phi}\frac{\partial^\alpha F}{\sqrt{\mu}})|\notag\\
		&&=|(\frac{1}{\sqrt{\mu}}\partial^{\alpha}P_{1}\{v_{1}M(\frac{|v-u|^{2}
			\overline{\theta}_{y}}{2R\theta^{2}}+\frac{(v-u)\cdot\bar{u}_{y}}{R\theta})\},e^{\phi}\frac{\partial^{\alpha}F}{\sqrt{\mu}})|
		\notag\\
		&&\leq C\|\langle v\rangle^{\frac{1}{2}}\frac{1}{\sqrt{\mu}}\partial^{\alpha}P_{1}\{v_{1}M(\frac{|v-u|^{2}
			\overline{\theta}_{y}}{2R\theta^{2}}+\frac{(v-u)\cdot\bar{u}_{y}}{R\theta})\}\|
		\|\langle v\rangle^{-\frac{1}{2}}\frac{\partial^{\alpha}F}{\sqrt{\mu}}\|
		\notag\\
		&&\leq C\eta\varepsilon^{a-1}(\|\partial^{\alpha}f\|^{2}_{\sigma}
		+\|\partial^{\alpha}(\widetilde{\rho},\widetilde{u},\widetilde{\theta})\|^{2})
		+C_{\eta}k^{\frac{1}{12}}\varepsilon^{\frac{3}{5}a-\frac{2}{5}}\mathcal{D}_{2,l,q}(\tau)\notag\\
		&&\quad+C_{\eta}\varepsilon^{a-1}\varepsilon^{\frac{7}{5}+\frac{1}{15}a}(\delta+\varepsilon^{a}\tau)^{-\frac{4}{3}}.
	\end{eqnarray}

	Now, we turn to consider the terms on the left hand side of \eqref{4.63}, which is the main difficulty.
	By the fact $F=M+\overline{G}+\sqrt{\mu}f$, we deduce from \eqref{4.45a}, \eqref{4.56},  Lemma \ref{lem7.2}, \eqref{3.21} and \eqref{3.22} that
	\begin{eqnarray}
		\label{4.64}
		&&|(\frac{\partial^\alpha F}{\sqrt{\mu}},e^{\phi}\phi_{\tau}\frac{\partial^\alpha F}{\sqrt{\mu}})|\notag\\
		&&=|(e^{\phi}\phi_{\tau}\frac{\partial^{\alpha}(M+\overline{G}+\sqrt{\mu}f)}{\sqrt{\mu}},
		\frac{\partial^{\alpha}(M+\overline{G}+\sqrt{\mu}f)}{\sqrt{\mu}})|
		\notag\\
		&&\leq  |(e^{\phi}\phi_{\tau}\partial^{\alpha}f,\partial^{\alpha}f)|
		+C\|\phi_{\tau}\|_{L_{y}^{\infty}}(\|\langle v\rangle^{-\frac{1}{2}}\partial^{\alpha}f\|^{2}+
		\|\langle v\rangle^{\frac{1}{2}}\frac{\partial^{\alpha}(M+\overline{G})}{\sqrt{\mu}}\|^{2})
		\notag\\
		&&\leq C\eta\varepsilon^{a-1}(\|\partial^{\alpha}f\|^{2}_{\sigma}
		+\|\partial^{\alpha}(\widetilde{\rho},\widetilde{u},\widetilde{\theta})\|^{2})+
		C_{\eta}k^{\frac{1}{12}}\varepsilon^{\frac{3}{5}a-\frac{2}{5}}\mathcal{D}_{2,l,q}(\tau)
		\notag\\
		&&\hspace{0.5cm}+C_{\eta}\varepsilon^{a-1}\varepsilon^{\frac{7}{5}+\frac{1}{15}a}(\delta+\varepsilon^{a}\tau)^{-\frac{4}{3}}
		+C_{\eta}q_{3}(\tau)\varepsilon^{2a-2}\mathcal{H}_{2,l,q}(\tau).
	\end{eqnarray}
	Here we have used the fact that
	\begin{eqnarray*}
		|(\phi_{\tau}\partial^{\alpha}f,e^{\phi}\partial^{\alpha}f)|
		&&\leq C\eta\varepsilon^{a-1}\|\langle v\rangle^{-\frac{1}{2}}\partial^{\alpha}f\|^{2}+C_{\eta}\varepsilon^{1-a}\|\phi_{\tau}\|^2_{L_{y}^{\infty}}
		\|\langle v\rangle^{\frac{1}{2}}\partial^{\alpha}f\|^{2}
		\notag\\
		&&\leq C\eta\varepsilon^{a-1}\|\partial^{\alpha}f\|_{\sigma}^{2}+C_{\eta}q_3(\tau)\|\langle v\rangle^{\frac{1}{2}}\partial^{\alpha}f\|^{2}
		\notag\\
		&&\leq C\eta\varepsilon^{a-1}\|\partial^{\alpha}f\|_{\sigma}^{2}+C_{\eta}q_3(\tau)\varepsilon^{2a-2}\mathcal{H}_{2,l,q}(\tau),
	\end{eqnarray*}
	due to  \eqref{3.14} and \eqref{4.60}. We observe from the integration by parts that
	\begin{eqnarray}
		\label{4.88a}
		&&(v_1(\frac{\partial^\alpha F}{\sqrt{\mu}})_y+\frac{v_{1}}{2}\phi_{y}\frac{\partial^{\alpha}F}{\sqrt{\mu}},e^{\phi}\frac{\partial^\alpha F}{\sqrt{\mu}})-(\phi_{y}\partial_{v_{1}}(\frac{\partial^{\alpha}F}{\sqrt{\mu}}),e^{\phi}\frac{\partial^\alpha F}{\sqrt{\mu}})
		\notag\\
		&&=(v_{1}(e^{\frac{\phi}{2}}\frac{\partial^{\alpha}F}{\sqrt{\mu}})_{y},e^{\frac{\phi}{2}}\frac{\partial^\alpha F}{\sqrt{\mu}})+\frac{1}{2}((\frac{\partial^{\alpha}F}{\sqrt{\mu}})^{2},\partial_{v_{1}}(\phi_{y}e^{\phi}))=0.
	\end{eqnarray}
The last term on the left hand side of \eqref{4.63} is very complicated. We need to handle it carefully. We first write
	\begin{align}
		\label{4.65}
		-&(\frac{v_{1}}{2}\partial^{\alpha_{1}}\phi_{y}\frac{
			\partial^{\alpha-\alpha_{1}}F}{\sqrt{\mu}}-\partial^{\alpha_{1}}\phi_{y}\partial_{v_{1}}(\frac{\partial^{\alpha-\alpha_{1}}F}{\sqrt{\mu}}),
		e^{\phi}\frac{\partial^\alpha F}{\sqrt{\mu}})\notag\\
		= &(\partial^{\alpha_{1}}\phi_{y}\frac{\partial_{v_{1}}\partial^{\alpha-\alpha_{1}}F}{\sqrt{\mu}}
		,e^{\phi}\frac{\partial^\alpha F}{\sqrt{\mu}}).
		\notag\\
		=&
		(\partial^{\alpha_{1}}\phi_{y}\frac{\partial_{v_{1}}\partial^{\alpha-\alpha_{1}}M}{\sqrt{\mu}},
		e^{\phi}\frac{\partial^\alpha M}{\sqrt{\mu}})+(\partial^{\alpha_{1}}\phi_{y}\frac{\partial_{v_{1}}\partial^{\alpha-\alpha_{1}}M}{\sqrt{\mu}},
		e^{\phi}\frac{\partial^\alpha(\overline{G}+\sqrt{\mu}f)}{\sqrt{\mu}})
		\notag\\
		&+(\partial^{\alpha_{1}}\phi_{y}\frac{\partial_{v_{1}}\partial^{\alpha-\alpha_{1}}\overline{G}}{\sqrt{\mu}},
		e^{\phi}\frac{\partial^\alpha F}{\sqrt{\mu}})
		+(\partial^{\alpha_{1}}\phi_{y}\frac{\partial_{v_{1}}\partial^{\alpha-\alpha_{1}}(\sqrt{\mu}f)}{\sqrt{\mu}},
		e^{\phi}\frac{\partial^\alpha(M+\overline{G})}{\sqrt{\mu}})
		\notag\\
		&+(\partial^{\alpha_{1}}\phi_{y}\frac{\partial_{v_{1}}\partial^{\alpha-\alpha_{1}}(\sqrt{\mu}f)}{\sqrt{\mu}},
		e^{\phi}\frac{\partial^\alpha(\sqrt{\mu}f)}{\sqrt{\mu}})
		\notag\\
		:=&I_{4}+I_5+I_6+I_7+I_8,
	\end{align}
for $1\leq|\alpha_{1}|\leq|\alpha|=2$. We postpone the estimate of $I_4$ to Lemma \ref{lem6.7} later on. To bound $I_{5}$, for $\alpha_{1}=\alpha=2$, we observe 
	$$
	(\partial^{\alpha}\phi_{y}\partial_{v_{1}}M,
	e^{\phi}\partial^\alpha G\frac{1}{M})=-(\partial^{\alpha}\phi_{y}\frac{(v_{1}-u_{1})}{R\theta},
	e^{\phi}\partial^\alpha G)=0,
	$$
which combined with $G=\overline{G}+\sqrt{\mu}f$, \eqref{4.45a}, \eqref{4.12A} and \eqref{7.12} yields that
	\begin{eqnarray*}
		|I_{5}|=&&|(\partial^{\alpha_{1}}\phi_{y}\frac{\partial_{v_{1}}\partial^{\alpha-\alpha_{1}}M}{\sqrt{\mu}},
		e^{\phi}\frac{\partial^\alpha(\overline{G}+\sqrt{\mu}f)}{\sqrt{\mu}})|
		\notag\\
		=&&|(\partial^{\alpha}\phi_{y}\partial_{v_{1}}M,
		e^{\phi}\partial^\alpha G[\frac{1}{\mu}-\frac{1}{M}])
		+(\partial^{\alpha}\phi_{y}\partial_{v_{1}}M,e^{\phi}\partial^\alpha G\frac{1}{M})|
		\notag\\
		\leq &&C(\eta_{0}+k^{\frac{1}{12}}\varepsilon^{\frac{3}{5}-\frac{2}{5}a})\|\partial^{\alpha}\widetilde{\phi}_{y}\|\|\partial^{\alpha}f\|_{\sigma}
		+Ck^{\frac{1}{12}}\varepsilon^{\frac{3}{5}a-\frac{2}{5}}\mathcal{D}_{2,l,q}(\tau)\\
&&		+C\varepsilon^{a-1}\varepsilon^{\frac{7}{5}+\frac{1}{15}a}(\delta+\varepsilon^{a}\tau)^{-\frac{4}{3}}.
	\end{eqnarray*}
	The case $|\alpha_1|=1$ has the same bound as \eqref{6.39A}. It follows that
	\begin{eqnarray*}
		|I_{5}|
		\leq&& \eta\varepsilon^{a-1}\|\partial^{\alpha}f\|_{\sigma}^{2}+C(\eta_{0}+k^{\frac{1}{12}}\varepsilon^{\frac{3}{5}-\frac{2}{5}a})\|\partial^{\alpha}\widetilde{\phi}_{y}\|\|\partial^{\alpha}f\|_{\sigma}
		\notag\\	
		&&+C_{\eta}k^{\frac{1}{12}}\varepsilon^{\frac{3}{5}a-\frac{2}{5}}\mathcal{D}_{2,l,q}(\tau)
		+C_{\eta}\varepsilon^{a-1}\varepsilon^{\frac{7}{5}+\frac{1}{15}a}(\delta+\varepsilon^{a}\tau)^{-\frac{4}{3}}.
	\end{eqnarray*}
To bound $I_{6}$, we use \eqref{4.45a}, \eqref{4.68}, \eqref{7.24}, \eqref{7.25}, Lemma \ref{lem7.2}, \eqref{3.21} and \eqref{3.22}  to get
	\begin{eqnarray*}
		|I_{6}|&&\leq \eta\varepsilon^{a-1}\|\langle v\rangle^{-\frac{1}{2}}\frac{\partial^\alpha F}{\sqrt{\mu}}\|^{2}+ C_{\eta}\varepsilon^{1-a}\|\langle v\rangle^{\frac{1}{2}}\partial^{\alpha_{1}}\phi_{y}\frac{\partial_{v_{1}}\partial^{\alpha-\alpha_{1}}\overline{G}}{\sqrt{\mu}}\|^{2}
		\notag\\
		&&\leq C\eta\varepsilon^{a-1}(\|\partial^{\alpha}f\|^{2}_{\sigma}
		+\|\partial^{\alpha}(\widetilde{\rho},\widetilde{u},\widetilde{\theta})\|^{2})+C_{\eta}k^{\frac{1}{12}}\varepsilon^{\frac{3}{5}a-\frac{2}{5}}\mathcal{D}_{2,l,q}(\tau)\\
&&\quad		+C_{\eta}\varepsilon^{a-1}\varepsilon^{\frac{7}{5}+\frac{1}{15}a}(\delta+\varepsilon^{a}\tau)^{-\frac{4}{3}},
	\end{eqnarray*}
while for the term $I_{7}$, it holds that
	\begin{eqnarray*}
		|I_{7}|=&&|(\partial^{\alpha-\alpha_{1}}(\sqrt{\mu}f),\partial^{\alpha_{1}}\phi_{y}e^{\phi}\partial_{v_{1}}[\frac{\partial^\alpha(M+\overline{G})}{\mu}])|
		\notag\\
		\leq&& \eta\varepsilon^{a-1}\|\langle v\rangle^{\frac{1}{2}}\sqrt{\mu}\partial_{v_{1}}
		[\frac{\partial^\alpha(M+\overline{G})}{\mu}]\|^2+C_\eta\varepsilon^{1-a}
		\|\langle v\rangle^{-\frac{1}{2}}\partial^{\alpha-\alpha_{1}}f\partial^{\alpha_{1}}\phi_{y}\|^2
		\notag\\
		\leq&& C\eta\varepsilon^{a-1}\|\partial^{\alpha}(\widetilde{\rho},\widetilde{u},\widetilde{\theta})\|^{2}+C_{\eta}k^{\frac{1}{12}}\varepsilon^{\frac{3}{5}a-\frac{2}{5}}\mathcal{D}_{2,l,q}(\tau)\\
		&&+C_{\eta}\varepsilon^{a-1}\varepsilon^{\frac{7}{5}+\frac{1}{15}a}(\delta+\varepsilon^{a}\tau)^{-\frac{4}{3}}.
	\end{eqnarray*}
	As for the term $I_{8}$, if $|\alpha_1|=1$, we derive from \eqref{4.45a}, \eqref{3.16}, \eqref{3.14} and \eqref{4.60} that
	\begin{eqnarray*}
		|I_{8}|=&&|(\partial^{\alpha_1}\phi_{y}[\partial_{v_1}\partial^{\alpha-\alpha_1}f-\frac{v_1}{2}\partial^{\alpha-\alpha_1}f], e^{\phi}\partial^\alpha f)|
		\notag\\
		\leq&& C\|\partial^{\alpha_1}\phi_{y}\|_{L_{y}^{\infty}}(\|\langle v\rangle^{\frac{1}{2}}\partial_{v_1}\partial^{\alpha-\alpha_1}f\|+\|\langle v_{1}\rangle\langle v\rangle^{\frac{1}{2}}\partial^{\alpha-\alpha_1}f\|)\|\langle v\rangle^{-\frac{1}{2}}\partial^\alpha f\|
		\notag\\
		\leq&& C\eta\varepsilon^{a-1}\|\langle v\rangle^{-\frac{1}{2}}\partial^\alpha f\|^2
		+C_{\eta}q_{3}(\tau)(\|\langle v\rangle^{\frac{1}{2}}\partial_{v_1}\partial^{\alpha-\alpha_1}f\|^{2}+\|\langle v\rangle^{\frac{1}{2}}\partial^{\alpha-\alpha_1}f\|_{w}^{2})
		\notag\\
		\leq&& C\eta\varepsilon^{a-1}\| \partial^\alpha f\|^{2}_{\sigma}
		+C_{\eta}q_{3}(\tau)\mathcal{H}_{2,l,q}(\tau),
	\end{eqnarray*}
	where we have used the fact that $\langle v_{1}\rangle\leq\langle v\rangle^{2(l-|\alpha-\alpha_1|)}$
	for $|\alpha-\alpha_{1}|=1$ with $l\geq2$. Similarly, the case $\alpha_1=\alpha$ can be estimated as
	\begin{eqnarray*}
		|I_{8}|\leq&& C\|\partial^{\alpha_1}\phi_{y}\|(\|\langle v\rangle^{\frac{1}{2}}\partial_{v_1}\partial^{\alpha-\alpha_1}f\|_{L_{y}^{\infty}}+\|\langle v_{1}\rangle\langle v\rangle^{\frac{1}{2}}\partial^{\alpha-\alpha_1}f\|_{L_{y}^{\infty}})\|\langle v\rangle^{-\frac{1}{2}}\partial^\alpha f\|
		\notag\\
		\leq&& C\eta\varepsilon^{a-1}\| \partial^\alpha f\|^{2}_{\sigma}+C_{\eta}q_{3}(\tau)\mathcal{H}_{2,l,q}(\tau).
	\end{eqnarray*}
	It follows from the above two estimates that
	\begin{equation}
		\label{4.91a}
		|I_{8}|\leq C\eta\varepsilon^{a-1}\| \partial^\alpha f\|^{2}_{\sigma}
		+C_{\eta}q_{3}(\tau)\mathcal{H}_{2,l,q}(\tau).
	\end{equation}
	Combining the above estimates of $I_{4}$, $I_{5}$, $I_{6}$, $I_{7}$ and $I_{8}$ together, we have from \eqref{4.65} that
	\begin{align}
		\label{6.49A}
		& \frac{d}{d\tau}E_{1}(\tau)-\sum_{1\leq\alpha_{1}\leq\alpha}C^{\alpha_{1}}_{\alpha}(\frac{v_{1}}{2}\partial^{\alpha_{1}}\phi_{y}\frac{
			\partial^{\alpha-\alpha_{1}}F}{\sqrt{\mu}}-\partial^{\alpha_{1}}\phi_{y}\partial_{v_{1}}(\frac{\partial^{\alpha-\alpha_{1}}F}{\sqrt{\mu}}),
		e^{\phi}\frac{\partial^\alpha F}{\sqrt{\mu}})\notag\\
		&\leq
		+C(\eta+\eta_{0}+k^{\frac{1}{12}}\varepsilon^{\frac{3}{5}-\frac{2}{5}a})\big\{\varepsilon^{a-1}(\|\partial^{\alpha}(\widetilde{\rho},\widetilde{u},\widetilde{\theta})\|^{2}
		+\| \partial^\alpha f\|^{2}_{\sigma}+\|\partial^{\alpha}\widetilde{\phi}\|^{2})\notag\\
	&\hspace{7cm}	+\varepsilon^{1-a}\|\partial^{\alpha}\widetilde{\phi}_{y}\|^{2}\big\}\notag\\
&\quad	+C_{\eta}k^{\frac{1}{12}}\varepsilon^{\frac{3}{5}a-\frac{2}{5}}\mathcal{D}_{2,l,q}(\tau)	+C_{\eta}\varepsilon^{a-1}\varepsilon^{\frac{7}{5}+\frac{1}{15}a}(\delta+\varepsilon^{a}\tau)^{-\frac{4}{3}}\notag\\
&\quad		+C_{\eta}q_{3}(\tau)\mathcal{H}_{2,l,q}(\tau).
	\end{align}
Here, as mentioned before, recall that we have used \eqref{4.66} for the estimate of $I_4$ in Lemma \ref{lem6.7} later.

	Therefore, substituting \eqref{4.89A}, \eqref{4.85}-\eqref{4.88a} and \eqref{6.49A} into \eqref{4.63},
	and further multiplying it by $\varepsilon^{2(1-a)}$, we can obatin for any small $\eta>0$,
	\begin{eqnarray}
		\label{4.75}
		&&\varepsilon^{2(1-a)}\frac{d}{d\tau}\sum_{|\alpha|=2}
		(\frac{1}{2}\|e^{\frac{\phi}{2}}\frac{\partial^\alpha F}{\sqrt{\mu}}\|^{2}+E_{1}(\tau))
		+c\varepsilon^{1-a}\sum_{|\alpha|=2}\|\partial^{\alpha}f\|_{\sigma}^{2}
		\notag\\
		\leq&& C(\eta+\eta_{0}+k^{\frac{1}{12}}\varepsilon^{\frac{3}{5}-\frac{2}{5}a})\notag\\
	&&\times	
		\sum_{|\alpha|=2}\big\{\varepsilon^{1-a}
		(\|\partial^{\alpha}(\widetilde{\rho},\widetilde{u},\widetilde{\theta})\|^{2}+\| \partial^\alpha f\|^{2}_{\sigma}
		+\|\partial^{\alpha}\widetilde{\phi}\|^{2})+\varepsilon^{3(1-a)}\|\partial^{\alpha}\widetilde{\phi}_{y}\|^{2}\big\}
		\notag\\
		&&+C_{\eta}\varepsilon^{\frac{7}{5}+\frac{1}{15}a}(\delta+\varepsilon^{a}\tau)^{-\frac{4}{3}}\notag\\
		&&+ C_{\eta}k^{\frac{1}{12}}\varepsilon^{\frac{3}{5}-\frac{2}{5}a}\mathcal{D}_{2,l,q}(\tau)
		+C_{\eta}q_{3}(\tau)\mathcal{H}_{2,l,q}(\tau).
	\end{eqnarray}
Recall $E_{1}(\tau)$ given \eqref{7.39} and use $F=M+\overline{G}+\sqrt{\mu}f$, we can claim that
	\begin{eqnarray*}
		&&\varepsilon^{2(1-a)}\sum_{|\alpha|=2}(\|e^{\frac{\phi}{2}}\frac{\partial^{\alpha}F}{\sqrt{\mu}}(\tau)\|^{2}+E_{1}(\tau))
		\notag\\
		&&\geq c\varepsilon^{2(1-a)}\sum_{|\alpha|=2}(\|\partial^{\alpha}(\widetilde{\rho},\widetilde{u},\widetilde{\theta})\|^{2}
		+\|\partial^{\alpha}f\|^{2}+\|\partial^{\alpha}\widetilde{\phi}\|^{2}
		+\varepsilon^{2b-2a}\|\partial^{\alpha}\widetilde{\phi}_{y}\|^{2})\notag\\
		&&\quad-Ck^{\frac{1}{3}}\varepsilon^{\frac{6}{5}-\frac{4}{5}a},
	\end{eqnarray*}
	and
	\begin{equation*}
		\varepsilon^{2(1-a)}\sum_{|\alpha|=2}(\|e^{\frac{\phi}{2}}\frac{\partial^{\alpha}F}{\sqrt{\mu}}(0)\|^{2}
		+E_{1}(0))
		\leq Ck^{\frac{1}{3}}\varepsilon^{\frac{6}{5}-\frac{4}{5}a}.
	\end{equation*}
By integrating \eqref{4.75} with respect to $\tau$ and employing the above two estimates,
we can obtain the desired estimate \eqref{4.76}. This completes the proof of Lemma \ref{lem.6.6}.
\end{proof}
For deducing \eqref{6.49A} in the proof of Lemma \ref{lem.6.6}  above, we have used the following estimate.
\begin{lemma}\label{lem6.7} 
	Recall $I_4$ in \eqref{4.65}. It holds that
	\begin{align}
		\label{4.66}
		I_{4}\leq&
		-\frac{d}{d\tau}E_{1}(\tau)+C(\eta+\eta_{0}+k^{\frac{1}{12}}\varepsilon^{\frac{3}{5}-\frac{2}{5}a})
		\big\{\varepsilon^{a-1}(\|\partial^{\alpha}(\widetilde{\rho},\widetilde{u},\widetilde{\theta})\|^{2}
		+\|\partial^{\alpha}\widetilde{\phi}\|^{2})\notag\\
		&\qquad\qquad\qquad\qquad\qquad\qquad\qquad\qquad\qquad+\varepsilon^{1-a}\|\partial^{\alpha}\widetilde{\phi}_{y}\|^{2}\big\}
		\notag\\
		&+C_{\eta}k^{\frac{1}{12}}\varepsilon^{\frac{3}{5}a-\frac{2}{5}}\mathcal{D}_{2,l,q}(\tau)
		+C_{\eta}\varepsilon^{a-1}\varepsilon^{\frac{7}{5}+\frac{1}{15}a}(\delta+\varepsilon^{a}\tau)^{-\frac{4}{3}}.
	\end{align}
	Here the function $E_{1}(\tau)$ is defined by
	\begin{align}
		\label{7.39}
		E_{1}(\tau)&=\frac{1}{2}\varepsilon^{2b-2a}(e^{\phi_{-}}\partial^{\alpha}\widetilde{\phi}_{y},
		\partial^{\alpha}\widetilde{\phi}_{y})+
		\frac{1}{2}(e^{\phi_{-}}\partial^{\alpha}\widetilde{\phi},
		\rho'_{\mathrm{e}}(\bar{\phi})\partial^{\alpha}\widetilde{\phi})\notag\\
		&\qquad-\frac{1}{2}(e^{\phi_{-}}\partial^{\alpha}\widetilde{\phi},
		\partial^{\alpha}\widetilde{\phi}\int^{\phi}_{\bar{\phi}}\rho''_{\mathrm{e}}(\varrho)\,d\varrho ).
	\end{align}
\end{lemma}
\begin{proof}
	Recall the definition $I_4$ given in \eqref{4.65}. If $\alpha_{1}=\alpha$, we write
	\begin{equation}
		\label{7.33}
		I_{4}=(\partial^{\alpha}\phi_{y}\partial_{v_{1}}M,e^{\phi}\partial^\alpha M[\frac{1}{\mu}-\frac{1}{M}])
		+(\partial^{\alpha}\phi_{y}\partial_{v_{1}}M,e^{\phi}\partial^\alpha M\frac{1}{M})
		:=I^{1}_{4}+I^{2}_{4}.
	\end{equation}
To estimate $I^1_4$, by \eqref{7.12}, \eqref{4.12A}, Lemma \ref{lem7.2}, \eqref{3.21} and \eqref{3.22} , we get
	\begin{eqnarray}
		\label{6.54B}
		|I^{1}_{4}|\leq&& C(\eta_{0}+k^{\frac{1}{12}}\varepsilon^{\frac{3}{5}-\frac{2}{5}a})
		\|\partial^{\alpha}\phi_{y}\|\|\frac{\partial^\alpha M}{\sqrt{\mu}}\|
		\notag\\
		\leq&& C(\eta_{0}+k^{\frac{1}{12}}\varepsilon^{\frac{3}{5}-\frac{2}{5}a})
		(\varepsilon^{1-a}\|\partial^{\alpha}\widetilde{\phi}_{y}\|^{2}
		+\varepsilon^{a-1}\|\partial^{\alpha}(\widetilde{\rho},\widetilde{u},\widetilde{\theta})\|^{2})
		\notag\\
		&&+Ck^{\frac{1}{12}}\varepsilon^{\frac{3}{5}a-\frac{2}{5}}\mathcal{D}_{2,l,q}(\tau)
		+C\varepsilon^{a-1}\varepsilon^{\frac{7}{5}+\frac{1}{15}a}(\delta+\varepsilon^{a}\tau)^{-\frac{4}{3}}.
	\end{eqnarray}
The term $I^{2}_{4}$ is very complicated. We observe from \eqref{1.7} and \eqref{1.5} that
	\begin{eqnarray}
		\label{6.54A}
		I^{2}_{4}&&=-(\partial^{\alpha}\phi_{y}\frac{(v_{1}-u_{1})}{R\theta}, e^{\phi}\partial^\alpha M)
		\notag\\
		&&=-(\frac{1}{R\theta}e^{\phi}\partial^{\alpha}\phi_{y},2\rho_yu_{1y})
		-(\frac{1}{R\theta}e^{\phi}\partial^{\alpha}\phi_{y},\rho\partial^{\alpha}u_{1}).
	\end{eqnarray}
	The first term of \eqref{6.54A} can be dominated easily by
	\begin{eqnarray}
		\label{7.34b}
		&&|(\frac{1}{R\theta}e^{\phi}\partial^{\alpha}\phi_{y},2\rho_yu_{1y})|
		\notag\\
		&&\leq \eta\varepsilon^{1-a}\|\partial^{\alpha}\widetilde{\phi}_{y}\|^{2}
		+C_{\eta}k^{\frac{1}{12}}\varepsilon^{\frac{3}{5}a-\frac{2}{5}}\mathcal{D}_{2,l,q}(\tau)\notag\\
		&&\quad+C_{\eta}\varepsilon^{a-1}\varepsilon^{\frac{7}{5}+\frac{1}{15}a}(\delta+\varepsilon^{a}\tau)^{-\frac{4}{3}}.
	\end{eqnarray}

	However the second term of \eqref{6.54A} is subtle since it is a linear term related to the highest order electric field. To bound it,
	we  write
	\begin{multline}
		\label{7.34}
		-(\frac{1}{R\theta}e^{\phi}\partial^{\alpha}\phi_{y},\rho\partial^{\alpha}u_{1})
		=-([\frac{1}{R\theta}e^{\phi}-e^{\phi}]\partial^{\alpha}\phi_{y},\rho\partial^{\alpha}u_{1})
		\\
		-([e^{\phi}-e^{\phi_{-}}]\partial^{\alpha}\phi_{y},\rho\partial^{\alpha}u_{1})
		-(e^{\phi_{-}}\partial^{\alpha}\phi_{y},\rho\partial^{\alpha}u_{1}).
	\end{multline}
	Thanks to $R=\frac{2}{3}$ and \eqref{3.24}, then
	$$
	|\frac{1}{R\theta}-1|=|(\frac{3}{2}-\theta)\frac{1}{\theta}|
	\leq C(\eta_{0}+k^{\frac{1}{12}}\varepsilon^{\frac{3}{5}-\frac{2}{5}a}).
	$$
	We thereupon deduce from this, \eqref{4.45a}, Lemma \ref{lem7.2} and \eqref{3.21} that
	\begin{eqnarray}
		\label{7.37A}
		&&|([\frac{1}{R\theta}e^{\phi}-e^{\phi}]\partial^{\alpha}\phi_{y},\rho\partial^{\alpha}u_{1})|
		\leq  C(\eta_{0}+k^{\frac{1}{12}}\varepsilon^{\frac{3}{5}-\frac{2}{5}a})\|\partial^{\alpha}\phi_{y}\|\|\partial^{\alpha}u_{1}\|
		\notag\\
		&&\leq  C(\eta_{0}+k^{\frac{1}{12}}\varepsilon^{\frac{3}{5}-\frac{2}{5}a})(\varepsilon^{1-a}\|\partial^{\alpha}\widetilde{\phi}_{y}\|^{2}
		+\varepsilon^{a-1}\|\partial^{\alpha}\widetilde{u}_{1}\|^{2})\notag\\
		&&\quad+C\varepsilon^{a-1}\varepsilon^{\frac{7}{5}+\frac{1}{15}a}(\delta+\varepsilon^{a}\tau)^{-\frac{4}{3}}.
	\end{eqnarray}
	The difference of $\bar{\phi}$ and $\phi_{-}$ is small in terms of \eqref{2.2}, \eqref{1.27} and the assumptions $(\mathcal{A})$,
	namely
	\begin{equation*}
		|\bar{\phi}-\phi_{-}|=|\rho^{-1}_{\mathrm{e}}(\bar{\rho})-\rho^{-1}_{\mathrm{e}}(\rho_{-})|
		\leq C|\bar{\rho}-\rho_{-}|\leq C(|\bar{\rho}-1|+|1-\rho_{-}|)\leq C\eta_{0},
	\end{equation*}
	which further implies that
	$$
	|e^{\phi}-e^{\phi_{-}}|\leq C(|\widetilde{\phi}|+|\bar{\phi}-\phi_{-}|)
	\leq C(\eta_{0}+k^{\frac{1}{12}}\varepsilon^{\frac{3}{5}-\frac{2}{5}a}).
	$$
	This and a similar argument as \eqref{7.37A} together leads us to
	\begin{eqnarray}
		\label{7.38A}
		&&|([e^{\phi}-e^{\phi_{-}}]\partial^{\alpha}\phi_{y},\rho\partial^{\alpha}u_{1})|\notag\\
		&&\leq  C(\eta_{0}+k^{\frac{1}{12}}\varepsilon^{\frac{3}{5}-\frac{2}{5}a})(\varepsilon^{1-a}\|\partial^{\alpha}\widetilde{\phi}_{y}\|^{2}
		+\varepsilon^{a-1}\|\partial^{\alpha}\widetilde{u}_{1}\|^{2})
		\notag\\
		&&+C_{\eta}\varepsilon^{a-1}\varepsilon^{\frac{7}{5}+\frac{1}{15}a}(\delta+\varepsilon^{a}\tau)^{-\frac{4}{3}}.
	\end{eqnarray}
For	the last term of \eqref{7.34}, it is equivalent to
	\begin{eqnarray}
		\label{7.35a}
		-(e^{\phi_{-}}\partial^{\alpha}\bar{\phi}_{y},\rho\partial^{\alpha}u_{1})
		-(e^{\phi_{-}}\partial^{\alpha}\widetilde{\phi}_{y},\rho\partial^{\alpha}\bar{u}_{1})
		-(e^{\phi_{-}}\partial^{\alpha}\widetilde{\phi}_{y},\rho\partial^{\alpha}\widetilde{u}_{1}).
	\end{eqnarray}

	For the first and second terms of \eqref{7.35a}, we have from \eqref{4.45a}, Lemma \ref{lem7.2}, \eqref{3.21} and \eqref{3.22} that
	\begin{multline}
		\label{6.60B}
		|(e^{\phi_{-}}\partial^{\alpha}\bar{\phi}_{y},\rho\partial^{\alpha}u_{1})|
		\leq C\|\partial^{\alpha}\bar{\phi}_{y}\|(\|\partial^{\alpha}\widetilde{u}_{1}\|+\|\partial^{\alpha}\bar{u}_{1}\|)
		\\
		\leq \eta\varepsilon^{1-a}\|\partial^{\alpha}\widetilde{u}_{1}\|^2
		+C_\eta\varepsilon^{a-1}\varepsilon^{\frac{7}{5}+\frac{1}{15}a}(\delta+\varepsilon^{a}\tau)^{-\frac{4}{3}},
	\end{multline}
	and
	\begin{equation}
		\label{6.61B}
		|(e^{\phi_{-}}\partial^{\alpha}\widetilde{\phi}_{y},\rho\partial^{\alpha}\bar{u}_{1})|
		\leq\eta\varepsilon^{1-a}\|\partial^{\alpha}\widetilde{\phi}_y\|^{2}
		+C_{\eta}\varepsilon^{a-1}\varepsilon^{\frac{7}{5}+\frac{1}{15}a}(\delta+\varepsilon^{a}\tau)^{-\frac{4}{3}}.
	\end{equation}

For	the last term of \eqref{7.35a}, we observe from the integration by parts that
	\begin{multline}
		\label{7.35aa}
		-(e^{\phi_{-}}\partial^{\alpha}\widetilde{\phi}_{y},\rho\partial^{\alpha}\widetilde{u}_{1})\\
		=
		(e^{\phi_{-}}\partial^{\alpha}\widetilde{\phi},\partial^{\alpha}[\rho\widetilde{u}_{1}]_{y})
		+\sum_{1\leq\alpha_{1}\leq\alpha}C^{\alpha_{1}}_{\alpha}
		(e^{\phi_{-}}\partial^{\alpha}\widetilde{\phi}_{y},\partial^{\alpha_{1}}\rho\partial^{\alpha-\alpha_{1}}\widetilde{u}_{1}).
	\end{multline}
From the first equation of \eqref{3.10}, it is straightforward to check that
	\begin{equation}
		\label{7.35}
		(e^{\phi_{-}}\partial^{\alpha}\widetilde{\phi},\partial^{\alpha}[\rho\widetilde{u}_{1}]_{y})
		=-(e^{\phi_{-}}\partial^{\alpha}\widetilde{\phi},\partial^{\alpha}\widetilde{\rho}_{\tau})
		-(e^{\phi_{-}}\partial^{\alpha}\widetilde{\phi},\partial^{\alpha}(\bar{u}_{1}\widetilde{\rho})_{y}).
	\end{equation}

To estimate the first term of \eqref{7.35}, we have from the last equation of \eqref{3.10} that
	\begin{equation}
		\label{7.36}
		-(e^{\phi_{-}}\partial^{\alpha}\widetilde{\phi},\partial^{\alpha}\widetilde{\rho}_{\tau})
		=(e^{\phi_{-}}\partial^{\alpha}\widetilde{\phi},
		\varepsilon^{2b-2a}\partial^{\alpha}\phi_{yy\tau}
		+\partial^{\alpha}[\rho_{\mathrm{e}}(\bar{\phi})-\rho_{\mathrm{e}}(\phi)]_{\tau}).
	\end{equation}
Using the integration by parts and performing the similar calculations as \eqref{4.61a}, we get
	\begin{eqnarray}
		\label{7.41a}
		&&\varepsilon^{2b-2a}(e^{\phi_{-}}\partial^{\alpha}\widetilde{\phi},
		\partial^{\alpha}\phi_{yy\tau})\notag\\
		&&=-\varepsilon^{2b-2a}(e^{\phi_{-}}\partial^{\alpha}\widetilde{\phi}_y,
		\partial^{\alpha}\widetilde{\phi}_{y\tau})+\varepsilon^{2b-2a}(e^{\phi_{-}}\partial^{\alpha}\widetilde{\phi},
		\partial^{\alpha}\bar{\phi}_{yy\tau})
		\notag\\
		&&\leq -\frac{1}{2}\varepsilon^{2b-2a}\frac{d}{d\tau}(e^{\phi_{-}}\partial^{\alpha}\widetilde{\phi}_{y},
		\partial^{\alpha}\widetilde{\phi}_{y})+Ck^{\frac{1}{12}}\varepsilon^{\frac{3}{5}a-\frac{2}{5}}\mathcal{D}_{2,l,q}(\tau)
		\notag\\
		&&\quad +C\varepsilon^{\frac{7}{5}+\frac{1}{15}a}(\delta+\varepsilon^{a}\tau)^{-\frac{4}{3}}.
	\end{eqnarray}
	In addition, it holds from \eqref{4.34} that
	\begin{align}
		\label{7.37}
		&(e^{\phi_{-}}\partial^{\alpha}\widetilde{\phi},
		\partial^{\alpha}[\rho_{\mathrm{e}}(\bar{\phi})-\rho_{\mathrm{e}}(\phi)]_{\tau})
		\notag\\
		=&-(e^{\phi_{-}}\partial^{\alpha}\widetilde{\phi},
		\partial^{\alpha}[\rho'_{\mathrm{e}}(\bar{\phi})\widetilde{\phi}]_{\tau})
		+(e^{\phi_{-}}\partial^{\alpha}\widetilde{\phi},\partial^{\alpha}J_{7\tau})
		\notag\\
		\leq& -\frac{1}{2}\frac{d}{d\tau}(e^{\phi_{-}}\partial^{\alpha}\widetilde{\phi},
		\rho'_{\mathrm{e}}(\bar{\phi})\partial^{\alpha}\widetilde{\phi})+\frac{1}{2}\frac{d}{d\tau}(e^{\phi_{-}}\partial^{\alpha}\widetilde{\phi},
		\partial^{\alpha}\widetilde{\phi}\int^{\phi}_{\bar{\phi}}\rho''_{\mathrm{e}}(\varrho)\,d\varrho )
		\notag\\
		&+C\eta\varepsilon^{a-1}\|\partial^{\alpha}\widetilde{\phi}\|^{2}
		+C_{\eta}k^{\frac{1}{12}}\varepsilon^{\frac{3}{5}a-\frac{2}{5}}\mathcal{D}_{2,l,q}(\tau)\notag\\
		&+C_{\eta}\varepsilon^{a-1}\varepsilon^{\frac{7}{5}+\frac{1}{15}a}(\delta+\varepsilon^{a}\tau)^{-\frac{4}{3}},
	\end{align}
	where in the last inequality the following two estimates have been used:
	\begin{eqnarray*}
		&&-(e^{\phi_{-}}\partial^{\alpha}\widetilde{\phi},\partial^{\alpha}[\rho'_{\mathrm{e}}(\bar{\phi})\widetilde{\phi}]_{\tau})
		\notag\\
		=&&-\frac{1}{2}\frac{d}{d\tau}(e^{\phi_{-}}\partial^{\alpha}\widetilde{\phi},
		\rho'_{\mathrm{e}}(\bar{\phi})\partial^{\alpha}\widetilde{\phi})
		+\frac{1}{2}(e^{\phi_{-}}\partial^{\alpha}\widetilde{\phi},[\rho'_{\mathrm{e}}(\bar{\phi})]_{\tau}\partial^{\alpha}\widetilde{\phi})
		\notag\\
		&&-\sum_{1\leq\alpha_{1}\leq\alpha}C^{\alpha_{1}}_{\alpha}(e^{\phi_{-}}\partial^{\alpha}\widetilde{\phi},
		\partial^{\alpha_{1}}[\rho'_{\mathrm{e}}(\bar{\phi})]\partial^{\alpha-\alpha_{1}}\widetilde{\phi}_{\tau})
		-(e^{\phi_{-}}\partial^{\alpha}\widetilde{\phi},\partial^{\alpha}[\rho''_{\mathrm{e}}(\bar{\phi})\bar{\phi}_{\tau}\widetilde{\phi}])
		\notag\\
		\leq&& -\frac{1}{2}\frac{d}{d\tau}(e^{\phi_{-}}\partial^{\alpha}\widetilde{\phi},
		\rho'_{\mathrm{e}}(\bar{\phi})\partial^{\alpha}\widetilde{\phi})
		+Ck^{\frac{1}{12}}\varepsilon^{\frac{3}{5}a-\frac{2}{5}}\mathcal{D}_{2,l,q}(\tau)\notag\\
		&&\quad+C\varepsilon^{a-1}\varepsilon^{\frac{7}{5}+\frac{1}{15}a}(\delta+\varepsilon^{a}\tau)^{-\frac{4}{3}},
	\end{eqnarray*}
	and
	\begin{eqnarray*}
		(e^{\phi_{-}}\partial^{\alpha}\widetilde{\phi},\partial^{\alpha}J_{7\tau})
		=&&(e^{\phi_{-}}\partial^{\alpha}\widetilde{\phi},\partial^{\alpha}\big[(\widetilde{\phi}_{\tau}
		+\bar{\phi}_{\tau})\int^{\phi}_{\bar{\phi}}\rho''_{\mathrm{e}}(\varrho)\,d\varrho
		-\bar{\phi}_{\tau}\widetilde{\phi}\rho''_{\mathrm{e}}(\bar{\phi})\big])
		\notag\\
		\leq&& \frac{1}{2}\frac{d}{d\tau}(e^{\phi_{-}}\partial^{\alpha}\widetilde{\phi},
		\partial^{\alpha}\widetilde{\phi}\int^{\phi}_{\bar{\phi}}\rho''_{\mathrm{e}}(\varrho)d\varrho )
		+\eta\varepsilon^{a-1}\|\partial^{\alpha}\widetilde{\phi}\|^{2}
		\notag\\
		&&+C_{\eta}k^{\frac{1}{12}}\varepsilon^{\frac{3}{5}a-\frac{2}{5}}\mathcal{D}_{2,l,q}(\tau)
		+C_{\eta}\varepsilon^{a-1}\varepsilon^{\frac{7}{5}+\frac{1}{15}a}(\delta+\varepsilon^{a}\tau)^{-\frac{4}{3}}.
	\end{eqnarray*}
	Recalling  $E_{1}(\tau)$ given by \eqref{7.39}, we plug  \eqref{7.37} and \eqref{7.41a} into \eqref{7.36} to get
	\begin{multline}
		\label{7.39a}
		-\big(e^{\phi_{-}}\partial^{\alpha}\widetilde{\phi},\partial^{\alpha}\widetilde{\rho}_{\tau}\big)
		\leq-\frac{d}{d\tau}E_{1}(\tau)
		+C\eta\varepsilon^{a-1}\|\partial^{\alpha}\widetilde{\phi}\|^{2}
		\\
		+C_{\eta}k^{\frac{1}{12}}\varepsilon^{\frac{3}{5}a-\frac{2}{5}}\mathcal{D}_{2,l,q}(\tau)
		+C_{\eta}\varepsilon^{a-1}\varepsilon^{\frac{7}{5}+\frac{1}{15}a}(\delta+\varepsilon^{a}\tau)^{-\frac{4}{3}}.
	\end{multline}
	The second term on the right hand side of \eqref{7.35} can be controlled by
	\begin{eqnarray*}
		&&(e^{\phi_{-}}\partial^{\alpha}\widetilde{\phi}_{y},\partial^{\alpha}(\bar{u}_{1}\widetilde{\rho}))\\
		&&=(e^{\phi_{-}}\partial^{\alpha}\widetilde{\phi}_{y},\bar{u}_{1}\partial^{\alpha}\widetilde{\rho})
		+\sum_{1\leq\alpha_{1}\leq\alpha}C^{\alpha_{1}}_{\alpha}
		(e^{\phi_{-}}\partial^{\alpha}\widetilde{\phi}_{y},\partial^{\alpha_{1}}\bar{u}_{1}\partial^{\alpha-\alpha_{1}}\widetilde{\rho})
		\notag\\
		&&\leq  C\eta\varepsilon^{1-a}\|\partial^{\alpha}\widetilde{\phi}_{y}\|^{2}
		+C_{\eta}k^{\frac{1}{12}}\varepsilon^{\frac{3}{5}a-\frac{2}{5}}\mathcal{D}_{2,l,q}(\tau)
		+C_{\eta}\varepsilon^{a-1}\varepsilon^{\frac{7}{5}+\frac{1}{15}a}(\delta+\varepsilon^{a}\tau)^{-\frac{4}{3}},
	\end{eqnarray*}
	where we have used the last equation of \eqref{3.10} and the similar calculations as \eqref{4.61a} to get
	\begin{eqnarray*}
		&&(e^{\phi_{-}}\partial^{\alpha}\widetilde{\phi}_{y},\bar{u}_{1}\partial^{\alpha}\widetilde{\rho})\\
		&&=\frac{1}{2}(e^{\phi_{-}}\partial^{\alpha}\widetilde{\phi}_{y},
		\bar{u}_{1y}\varepsilon^{2b-2a}\partial^{\alpha}\widetilde{\phi}_{y})\\
		&&\quad-(e^{\phi_{-}}\partial^{\alpha}\widetilde{\phi}_{y},
		\bar{u}_{1}\varepsilon^{2b-2a}\partial^{\alpha}\bar{\phi}_{yy}
		+\partial^{\alpha}[\rho_{\mathrm{e}}(\bar{\phi})-\rho_{\mathrm{e}}(\phi)])
		\notag\\
		&&\leq  \eta\varepsilon^{1-a}\|\partial^{\alpha}\widetilde{\phi}_{y}\|^{2}
		+C_{\eta}k^{\frac{1}{12}}\varepsilon^{\frac{3}{5}a-\frac{2}{5}}\mathcal{D}_{2,l,q}(\tau)
		+C_{\eta}\varepsilon^{a-1}\varepsilon^{\frac{7}{5}+\frac{1}{15}a}(\delta+\varepsilon^{a}\tau)^{-\frac{4}{3}}.
	\end{eqnarray*}
Substituting this and  \eqref{7.39a} into \eqref{7.35}, we get
	\begin{multline}
	\label{7.45a}
	(e^{\phi_{-}}\partial^{\alpha}\widetilde{\phi},\partial^{\alpha}[\rho\widetilde{u}_{1}]_{y})
		\leq-\frac{d}{d\tau}E_{1}(\tau)
	+C\eta\varepsilon^{a-1}\|\partial^{\alpha}\widetilde{\phi}\|^{2}+C\eta\varepsilon^{1-a}\|\partial^{\alpha}\widetilde{\phi}_{y}\|^{2}
	\\
	+C_{\eta}k^{\frac{1}{12}}\varepsilon^{\frac{3}{5}a-\frac{2}{5}}\mathcal{D}_{2,l,q}(\tau)
	+C_{\eta}\varepsilon^{a-1}\varepsilon^{\frac{7}{5}+\frac{1}{15}a}(\delta+\varepsilon^{a}\tau)^{-\frac{4}{3}}.	
\end{multline}
This completes the proof of the first term of \eqref{7.35aa}. The second term of \eqref{7.35aa} can be done similarly. 
Consequently, we have from \eqref{7.45a} and \eqref{7.35aa} that
	\begin{eqnarray*}
		-(e^{\phi_{-}}\partial^{\alpha}\widetilde{\phi}_{y},\rho\partial^{\alpha}\widetilde{u}_{1})
		\leq&&
		-\frac{d}{d\tau}E_{1}(\tau)+C\eta\varepsilon^{a-1}\|\partial^{\alpha}\widetilde{\phi}\|^{2}
		+C\eta\varepsilon^{1-a}\|\partial^{\alpha}\widetilde{\phi}_{y}\|^{2}
		\notag\\
		&&+C_{\eta}k^{\frac{1}{12}}\varepsilon^{\frac{3}{5}a-\frac{2}{5}}\mathcal{D}_{2,l,q}(\tau)
		+C_{\eta}\varepsilon^{a-1}\varepsilon^{\frac{7}{5}+\frac{1}{15}a}(\delta+\varepsilon^{a}\tau)^{-\frac{4}{3}},
	\end{eqnarray*}
which together with \eqref{6.60B}, \eqref{6.61B} and \eqref{7.35a} gives rise to
	\begin{eqnarray*}
		&&-(e^{\phi_{-}}\partial^{\alpha}\phi_{y},\rho\partial^{\alpha}u_{1})\\
		&&\leq
		-\frac{d}{d\tau}E_{1}(\tau)+C\eta\varepsilon^{a-1}\|\partial^{\alpha}\widetilde{\phi}\|^{2}
		+C\eta\varepsilon^{1-a}(\|\partial^{\alpha}\widetilde{\phi}_{y}\|^{2}+\|\partial^{\alpha}\widetilde{u}_{1}\|^2)
		\notag\\
		&&\quad+C_{\eta}k^{\frac{1}{12}}\varepsilon^{\frac{3}{5}a-\frac{2}{5}}\mathcal{D}_{2,l,q}(\tau)
		+C_{\eta}\varepsilon^{a-1}\varepsilon^{\frac{7}{5}+\frac{1}{15}a}(\delta+\varepsilon^{a}\tau)^{-\frac{4}{3}}.
	\end{eqnarray*}
This completes the proof of	the last term of \eqref{7.34}. Plugging this, \eqref{7.37A} and \eqref{7.38A} into \eqref{7.34}, then further combining it with \eqref{7.34b} and \eqref{6.54A}, we deduce that
	\begin{eqnarray*}
		I^{2}_{4}\leq&&-\frac{d}{d\tau}E_{1}(\tau)
		+C(\eta+\eta_{0}+k^{\frac{1}{12}}\varepsilon^{\frac{3}{5}-\frac{2}{5}a})
		\big\{\varepsilon^{1-a}\|\partial^{\alpha}\widetilde{\phi}_{y}\|^{2}\\
		&&\qquad\qquad\qquad\qquad\qquad\qquad\qquad\qquad+\varepsilon^{a-1}(\|\partial^{\alpha}(\widetilde{\rho},\widetilde{u},\widetilde{\theta})\|^{2}+\|\partial^{\alpha}\widetilde{\phi}\|^{2})\big\}
		\notag\\
		&&+C_{\eta}k^{\frac{1}{12}}\varepsilon^{\frac{3}{5}a-\frac{2}{5}}\mathcal{D}_{2,l,q}(\tau)
		+C_{\eta}\varepsilon^{a-1}\varepsilon^{\frac{7}{5}+\frac{1}{15}a}(\delta+\varepsilon^{a}\tau)^{-\frac{4}{3}},
	\end{eqnarray*}
which together with	 \eqref{6.54B} and \eqref{7.33} yields that
	\begin{align}
		\label{6.70A}
		I_{4} &\leq-\frac{d}{d\tau}E_{1}(\tau)+C(\eta+\eta_{0}+k^{\frac{1}{12}}\varepsilon^{\frac{3}{5}-\frac{2}{5}a})
		\big\{\varepsilon^{1-a}\|\partial^{\alpha}\widetilde{\phi}_{y}\|^{2}\notag\\
		&\qquad\qquad\qquad\qquad\qquad\qquad\qquad\qquad+\varepsilon^{a-1}(\|\partial^{\alpha}(\widetilde{\rho},\widetilde{u},\widetilde{\theta})\|^{2}+\|\partial^{\alpha}\widetilde{\phi}\|^{2})\big\}
		\notag\\
		&\quad+C_{\eta}k^{\frac{1}{12}}\varepsilon^{\frac{3}{5}a-\frac{2}{5}}\mathcal{D}_{2,l,q}(\tau)
		+C_{\eta}\varepsilon^{a-1}\varepsilon^{\frac{7}{5}+\frac{1}{15}a}(\delta+\varepsilon^{a}\tau)^{-\frac{4}{3}}.
	\end{align}
This completes the proof of the case $\alpha_{1}=\alpha=2$.

For the case $\alpha_{1}=1$, $I_{4}$ can also be controlled by the right terms of \eqref{6.70A}.
Then the desired estimate \eqref{4.66} follows. We thus finish the proof of Lemma \ref{lem6.7}.
\end{proof}

Combining those estimates in Lemma \ref{lem.6.6} and Lemma \ref{lem.6.5}, we are able to obtain all the derivative estimates for both the fluid and non-fluid parts.
\begin{lemma}\label{lem4.3}
	It holds that
	\begin{align}
		\label{4.77}
		&\sum_{|\alpha|=1}\big\{(\|\partial^{\alpha}(\widetilde{\rho},\widetilde{u},\widetilde{\theta})(\tau)\|^{2}
		+\|\partial^{\alpha}f(\tau)\|^{2})+(\|\partial^{\alpha}\widetilde{\phi}(\tau)\|^{2}
		+\varepsilon^{2b-2a}\|\partial^{\alpha}\widetilde{\phi}_{y}(\tau)\|^{2})\big\}
		\notag\\
		&+\varepsilon^{2-2a}\sum_{|\alpha|=2}
		\big\{(\|\partial^{\alpha}(\widetilde{\rho},\widetilde{u},\widetilde{\theta})(\tau)\|^{2}
		+\|\partial^{\alpha}f(\tau)\|^{2})\notag\\
		&\qquad\qquad\qquad\qquad+
		(\|\partial^{\alpha}\widetilde{\phi}(\tau)\|^{2}+\varepsilon^{2b-2a}\|\partial^{\alpha}\widetilde{\phi}_{y}(\tau)\|^{2})\big\}
		\notag\\
		&+\varepsilon^{a-1}\sum_{|\alpha|=1}\int^{\tau}_{0}\|\partial^{\alpha}f(s)\|_{\sigma}^{2}\,ds\notag\\
		&+\varepsilon^{1-a}\sum_{|\alpha|=2}\int^{\tau}_{0}\big\{\|\partial^{\alpha}(\widetilde{\rho},\widetilde{u},
		\widetilde{\theta})(s)\|^{2}+\|\partial^{\alpha}f(s)\|_{\sigma}^{2}\big\}\,ds
		\notag\\
		&+\varepsilon^{1-a}\sum_{|\alpha|=2}\int^{\tau}_{0}\big\{\|\partial^{\alpha}\widetilde{\phi}(s)\|^{2}
		+\varepsilon^{2b-2a}\|\partial^{\alpha}\widetilde{\phi}_{y}(s)\|^{2}\big\}\,ds
		\notag\\
		&\leq Ck^{\frac{1}{3}}\varepsilon^{\frac{6}{5}-\frac{4}{5}a}
		+C(\eta_{0}+k^{\frac{1}{12}}\varepsilon^{\frac{3}{5}-\frac{2}{5}a}+k^{\frac{1}{12}}\varepsilon^{\frac{3}{5}a-\frac{2}{5}})\int^{\tau}_{0}\mathcal{D}_{2,l,q}(s)\,ds\notag\\
		&\quad +C\int^{\tau}_{0}q_{3}(s)\mathcal{H}_{2,l,q}(s)\,ds.
	\end{align}
\end{lemma}
\begin{proof}
	By a suitable linear combination of \eqref{4.61} and \eqref{4.76}, we can obtain  the expected estimate \eqref{4.77}
	by choosing $\eta$, $\eta_{0}$ and $\varepsilon$ small enough and requiring
	\begin{equation}
		\label{4.94}
		\varepsilon^{2b-2a}\|\partial^{\alpha}\widetilde{\phi}_{y}\|^{2}\geq
		\varepsilon^{2(1-a)}\|\partial^{\alpha}\widetilde{\phi}_{y}\|^{2}\quad \Leftrightarrow \quad b\leq1.
	\end{equation}
 Then Lemma \ref{lem4.3} is proved.
\end{proof}

Finally, combining those estimates in Lemma \ref{lem4.3} and Lemma \ref{lem4.2}, we  conclude the energy estimate on solutions without any weight.

\begin{lemma}\label{lem4.1}
It holds that	
	\begin{eqnarray}
		\label{4.1b}
		&&\mathcal{E}_{2}(\tau)+\int^{\tau}_{0}\|\sqrt{\bar{u}_{1y}}(\widetilde{\rho},\widetilde{u}_{1},\widetilde{\theta})(s)\|^{2}\,ds+\int^{\tau}_{0}\mathcal{D}_{2}(s)\,ds
		\notag\\
		&&\leq Ck^{\frac{1}{3}}\varepsilon^{\frac{6}{5}-\frac{4}{5}a}
		+C(\eta_{0}+k^{\frac{1}{12}}\varepsilon^{\frac{3}{5}-\frac{2}{5}a}
		+k^{\frac{1}{12}}\varepsilon^{\frac{3}{5}a-\frac{2}{5}})\int^{\tau}_{0}\mathcal{D}_{2,l,q}(s)\,ds\notag\\
		&&\quad+C\int^{\tau}_{0}q_{3}(s)\mathcal{H}_{2,l,q}(s)\,ds.
	\end{eqnarray}
	Here $\mathcal{E}_{2}(\tau)$ and  $\mathcal{D}_{2}(\tau)$ are given by
	\begin{eqnarray*}
		\mathcal{E}_{2}(\tau):=&&\sum_{|\alpha|\leq1}(\|\partial^{\alpha}(\widetilde{\rho},\widetilde{u},\widetilde{\theta})(\tau)\|^{2}
		+\|\partial^{\alpha}f(\tau)\|^{2})
		\notag\\
		&&\quad+\varepsilon^{2-2a}\sum_{|\alpha|=2}(\|\partial^{\alpha}(\widetilde{\rho},\widetilde{u},\widetilde{\theta})(\tau)\|^{2}
		+\|\partial^{\alpha}f(\tau)\|^{2})+E(\tau),
	\end{eqnarray*}
	and
	\begin{eqnarray}
		\label{3.21a}
		\mathcal{D}_{2}(\tau):&&=\varepsilon^{1-a}\sum_{1\leq|\alpha|\leq 2}(\|\partial^{\alpha}(\widetilde{\rho},\widetilde{u},\widetilde{\theta})(\tau)\|^{2}
		+\|\partial^{\alpha}\widetilde{\phi}(\tau)\|^{2}
		+\varepsilon^{2b-2a}\|\partial^{\alpha}\widetilde{\phi}_{y}(\tau)\|^{2})
		\notag\\
		&&\hspace{1cm}+\varepsilon^{a-1}\sum_{|\alpha|\leq1}\|\partial^\alpha f(\tau)\|^2_{\sigma}
		+\varepsilon^{1-a}\sum_{|\alpha|=2}\|\partial^\alpha f(\tau)\|^2_{\sigma}.
	\end{eqnarray}
\end{lemma}
\begin{proof}
Adding \eqref{4.44}$\times \kappa_8$ together with \eqref{4.77}, then further letting $C\kappa_8<1/4$ and requiring 
	\begin{equation}
		\label{5.15a}
		\varepsilon^{a-1}\int^{\tau}_{0}\|\partial^{\alpha}f(s)\|^{2}_{\sigma}\,ds\geq \varepsilon^{1-a}\int^{\tau}_{0}\|\partial^{\alpha}f(s)\|^{2}_{\sigma}\,ds
		\quad\Leftrightarrow\quad  a\leq 1,
	\end{equation}
we can obtain the desired estimate \eqref{4.1b}. This completes the proof of Lemma \ref{lem4.1}.
\end{proof}

\section{Weighted energy estimates}\label{sec.7}
In this section we consider some estimates with the weight function $w(\alpha,\beta)$ in \eqref{3.12} in order to close the a priori estimate.
The weight functions are acted on  the microscopic component $f$ for the equation \eqref{3.7}.

\subsection{Low order weighted estimates}
We follow the approach in \cite{Guo-JAMS} to carry out the weighted estimates for $f$.
We start from weighted estimates on zeroth and first order.
\begin{lemma}\label{lem7.1A}
	Under the conditions listed in Lemma \ref{lem.5.1A}, it holds that
	\begin{eqnarray}
		\label{5.3}
		&&\sum_{|\alpha|\leq 1}\big\{\frac{1}{2}\frac{d}{d\tau}\|e^{\frac{\phi}{2}}\partial^\alpha f\|_w^2+
		q_{2}q_{3}(\tau)\|\langle v\rangle^{\frac{1}{2}}e^{\frac{\phi}{2}}\partial^\alpha f\|^2_w
		+c\varepsilon^{a-1}\|\partial^\alpha f\|^2_{\sigma,w}\big\}
		\notag\\
		&&\leq C\varepsilon^{a-1}\sum_{|\alpha|\leq 1}\|\partial^\alpha f\|^2_{\sigma}
		+C\varepsilon^{1-a}\sum_{|\alpha|\leq 1}(\|\partial^\alpha f_y\|^2_{\sigma}
		+\|(\partial^\alpha\widetilde{u}_y,\partial^\alpha\widetilde{\theta}_y)\|^2)
		\notag\\
		&&\quad+C(\eta_{0}+k^{\frac{1}{12}}\varepsilon^{\frac{3}{5}-\frac{2}{5}a})\mathcal{D}_{2,l,q}(\tau)\notag\\
		&&\quad+C\varepsilon^{\frac{7}{5}+\frac{1}{15}a}(\delta+\varepsilon^{a}\tau)^{-\frac{4}{3}}
		+Cq_3(\tau)\mathcal{H}_{2,l,q}(\tau).
	\end{eqnarray}	
\end{lemma}
\begin{proof}
	After acting  $\partial^\alpha$ with $|\alpha|\leq 1$ to equation \eqref{3.7} and further multiplying it by $e^{\phi}w^2(\alpha,0)\partial^\alpha f$,
	the direct energy estimate gives the identity
	\begin{align}
		\label{5.1}
		&(\partial^\alpha f_{\tau},e^{\phi}w^2(\alpha,0)\partial^\alpha f)
		+(v_1\partial^\alpha f_{y}+\frac{v_{1}}{2}\partial^{\alpha}(\phi_{y} f),e^{\phi}w^2(\alpha,0)\partial^\alpha f)
		\notag\\
		&-(\partial^{\alpha}(\phi_{y} \partial_{v_{1}}f),e^{\phi}w^2(\alpha,0)\partial^\alpha f)
		-\varepsilon^{a-1}(\mathcal{L}\partial^\alpha f,e^{\phi}w^2(\alpha,0)\partial^\alpha f)
		\notag\\
		=&\varepsilon^{a-1}(\partial^\alpha\Gamma(f,\frac{M-\mu}{\sqrt{\mu}})+\partial^\alpha\Gamma(\frac{M-\mu}{\sqrt{\mu}},f)
		+\partial^\alpha\Gamma(\frac{G}{\sqrt{\mu}},\frac{G}{\sqrt{\mu}}),e^{\phi}w^2(\alpha,0)\partial^\alpha f)
		\notag\\
		&+(\partial^\alpha \mathbb{K},e^{\phi}w^2(\alpha,0)\partial^\alpha f)
		\notag\\
		&+(\frac{\partial^{\alpha}(\phi_{y}\partial_{v_{1}}\overline{G})}{\sqrt{\mu}}-\frac{\partial^\alpha P_1( v_1\overline{G}_{y})}{\sqrt{\mu}}-\frac{ \partial^\alpha\overline{G}_{\tau}}{\sqrt{\mu}}, e^{\phi}w^2(\alpha,0)\partial^\alpha f),
	\end{align}
where for brevity we have denoted
\begin{equation}
\label{def.adK}
\mathbb{K}=\frac{ P_0( v_1\sqrt{\mu}f_{y})}{\sqrt{\mu}}-\frac{1}{\sqrt{\mu}}
		 P_1\{v_1M(\frac{|v-u|^2 \widetilde{\theta}_y}{2R\theta^2}
		+\frac{(v-u)\cdot  \widetilde {u}_{y}}{R\theta})\}.
\end{equation}
	We  estimate each term for \eqref{5.1}. First note that
	\begin{eqnarray}
		\label{5.2a}
		[w^{2}(\alpha,\beta)]_\tau =-2q_{2}q_{3}(\tau)\langle v\rangle w^{2}(\alpha,\beta),
	\end{eqnarray}
	due to  \eqref{3.12} and \eqref{3.13}. Then we deduce that
	\begin{eqnarray}
		\label{7.4c}
		&&(\partial^\alpha f_{\tau},e^{\phi}w^2(\alpha,0)\partial^\alpha f)
		\notag\\
		=&&\frac{1}{2}\frac{d}{d\tau}(\partial^\alpha f,e^{\phi}w^2(\alpha,0)\partial^\alpha f)
		-\frac{1}{2}(\partial^\alpha f,e^{\phi}[w^2(\alpha,0)]_{\tau}\partial^\alpha f)\notag\\
		&&-\frac{1}{2}(\partial^\alpha f,e^{\phi}\phi_{\tau}w^2(\alpha,0)\partial^\alpha f)
		\notag\\
		\geq&&\frac{1}{2}\frac{d}{d\tau}\|e^{\frac{\phi}{2}}\partial^\alpha f\|_w^2
		+q_{2}q_{3}(\tau)\|\langle v\rangle^{\frac{1}{2}}e^{\frac{\phi}{2}}\partial^\alpha f\|^2_w\notag\\
		&& -C\eta\varepsilon^{a-1}\|\partial^\alpha f\|_{\sigma,w}^2-C_{\eta}q_3(\tau)\mathcal{H}_{2,l,q}(\tau).
	\end{eqnarray}
	Here we have used \eqref{4.45a}, \eqref{3.16}, \eqref{3.14} and \eqref{4.60} such that 
	\begin{eqnarray*}
		&&|(\partial^\alpha f,e^{\phi}\phi_{\tau}w^2(\alpha,0)\partial^\alpha f)|
		\notag\\
		&&\leq \eta\varepsilon^{a-1}\|\langle v\rangle^{-\frac{1}{2}}w(\alpha,0)\partial^\alpha f\|^2
		+C_{\eta}\varepsilon^{1-a}\|\phi_{\tau}\|^2_{L_y^{\infty}}\|\langle v\rangle^{\frac{1}{2}}w(\alpha,0)\partial^\alpha f\|^2
		\notag\\
		&&\leq C\eta\varepsilon^{a-1}\|\partial^\alpha f\|_{\sigma,w}^2
		+C_{\eta}q_3(\tau)\|\langle v\rangle^{\frac{1}{2}}\partial^\alpha f\|_{w}^2
		\notag\\
		&&\leq C\eta\varepsilon^{a-1}\|\partial^\alpha f\|_{\sigma,w}^2+C_{\eta}q_3(\tau)\mathcal{H}_{2,l,q}(\tau).
	\end{eqnarray*}
	For the second term on the left hand side of \eqref{5.1}, it is equivalent to
	\begin{equation*}
		(v_1\partial^\alpha f_{y}+\frac{v_{1}}{2}\phi_{y} \partial^{\alpha}f,e^{\phi}w^2(\alpha,0)\partial^\alpha f)
		+(\frac{v_{1}}{2}\partial^{\alpha}\phi_{y} f,e^{\phi}w^2(\alpha,0)\partial^\alpha f).
	\end{equation*}
	Here if $|\alpha|=0$, the second term vanishes, and if $|\alpha|=1$, this term exists, which can be dominated by
	\begin{eqnarray}
		\label{7.4b}	
		&&C\|\partial^{\alpha}\phi_{y}\|_{L_{y}^{\infty}}\|\langle v\rangle^{-\frac{1}{2}}\langle v\rangle^{2}w(\alpha,0)f\|
		\|\langle v\rangle^{-\frac{1}{2}}w(\alpha,0)\partial^\alpha f\|
		\notag\\
		&&\leq \eta\varepsilon^{a-1}\|\partial^\alpha f\|^2_{\sigma,w}+
		C_\eta\varepsilon^{1-a}\|\partial^{\alpha}\phi_{y}\|^2_{L_{y}^{\infty}}\|\langle v\rangle^{-\frac{1}{2}}w(0,0)f\|^2
		\notag\\	
		&&\leq\eta\varepsilon^{a-1}\|\partial^\alpha f\|^2_{\sigma,w}
		+C_{\eta}q_3(\tau)\mathcal{H}_{2,l,q}(\tau),
	\end{eqnarray}
	where in the second inequality we have used $\langle v\rangle^{2}w(\alpha,0)= w(0,0)$ with $|\alpha|=1$. Observing that
	\begin{equation}
		\label{7.4A}
		(v_1\partial^\alpha f_{y}+\frac{v_{1}}{2}\phi_{y} \partial^{\alpha}f,e^{\phi}w^2(\alpha,0)\partial^\alpha f)
		=(v_{1}[e^{\frac{\phi}{2}}\partial^\alpha f]_{y},w^2(\alpha,0)e^{\frac{\phi}{2}}\partial^\alpha  f)=0,
	\end{equation}
	then it holds from the above estimates that
	\begin{equation}
		\label{7.7b}
		|(v_1\partial^\alpha f_{y}+\frac{v_{1}}{2}\partial^{\alpha}(\phi_{y} f),e^{\phi}w^2(\alpha,0)\partial^\alpha f)|
		\leq \eta\varepsilon^{a-1}\|\partial^\alpha f\|^2_{\sigma,w}
		+C_{\eta}q_3(\tau)\mathcal{H}_{2,l,q}(\tau).
	\end{equation}
	To bound the third term on the left hand side of \eqref{5.1}, we frist write
	\begin{multline}
		\label{5.2}
		-(\partial^{\alpha}(\phi_{y} \partial_{v_{1}}f),e^{\phi}w^2(\alpha,0)\partial^\alpha f)\\
		=-(\phi_{y}\partial_{v_1}\partial^\alpha f,e^{\phi}w^2(\alpha,0)\partial^\alpha f)
		-(\partial^\alpha\phi_{y}\partial_{v_1}f,e^{\phi}w^2(\alpha,0)\partial^\alpha f).
	\end{multline}
	Here if $|\alpha|=0$, the last term of \eqref{5.2} vanishes, and if $|\alpha|=1$, this term exists and it has the same bound as \eqref{7.4b}.
	After integration by parts, the first term of \eqref{5.2} can be dominated by
	\begin{eqnarray*}
		&&\frac{1}{2}|(\phi_{y}\partial^\alpha f,e^{\phi}\partial_{v_1}[w^{2}(\alpha,0)]\partial^\alpha f)|
		\notag\\
		\leq&& C\|\phi_{y}\|_{L^\infty_y}\|\langle v\rangle^{\frac{1}{2}}w(\alpha,0)\partial^{\alpha}f\|
		\| \langle v\rangle^{-\frac{1}{2}}w(\alpha,0)\partial^\alpha f\|
		\notag\\
		\leq&& \eta\varepsilon^{a-1}\|\partial^\alpha f\|^2_{\sigma,w}
		+C_{\eta}q_3(\tau)\mathcal{H}_{2,l,q}(\tau).
	\end{eqnarray*}
	Here we have used \eqref{3.12} and \eqref{3.23} to get
	\begin{eqnarray}
		\label{5.4a}
		|\partial_{v_{1}}w(\alpha,0)|=|2(l-|\alpha|)\frac{v_{1}}{\langle v\rangle^{2}}w(\alpha,0)+
		q(\tau)\frac{v_{1}}{\langle v\rangle}w(\alpha,0)|\leq
		Cw(\alpha,0).
	\end{eqnarray}
	So, from the above estimates, we get
	\begin{equation}
		\label{7.10c}
		|(\partial^{\alpha}(\phi_{y} \partial_{v_{1}}f),e^{\phi}w^2(\alpha,0)\partial^\alpha f)|
		\leq C\eta\varepsilon^{a-1}\|\partial^\alpha f\|^2_{\sigma,w}
		+C_{\eta}q_3(\tau)\mathcal{H}_{2,l,q}(\tau).
	\end{equation}
	For the last term on the left hand side of \eqref{5.1}, we derive from \eqref{7.6} and \eqref{4.45a} that
	$$
	-\varepsilon^{a-1}(\mathcal{L}\partial^\alpha f,e^{\phi}w^2(\alpha,0)\partial^\alpha f)
	\geq c\varepsilon^{a-1}\|\partial^\alpha f\|^2_{\sigma,w}-C\varepsilon^{a-1}\|\partial^\alpha f\|^2_{\sigma}.
	$$
	
	Next we are going to bound the terms on the right hand side of \eqref{5.1}.
	In view of  \eqref{7.10} and \eqref{7.19}, one has
	\begin{eqnarray*}
		&&\varepsilon^{a-1}|(\partial^\alpha\Gamma(f,\frac{M-\mu}{\sqrt{\mu}})+\partial^\alpha\Gamma(\frac{M-\mu}{\sqrt{\mu}},f)
		+\partial^\alpha\Gamma(\frac{G}{\sqrt{\mu}},\frac{G}{\sqrt{\mu}}),e^{\phi}w^2(\alpha,0)\partial^\alpha f)|
		\notag\\
		&&\leq C\eta\varepsilon^{a-1}\|\partial^\alpha f\|_{\sigma,w}^2
		+C_{\eta}(\eta_{0}+k^{\frac{1}{12}}\varepsilon^{\frac{3}{5}-\frac{2}{5}a})\mathcal{D}_{2,l,q}(\tau)
		+C_{\eta}\varepsilon^{\frac{7}{5}+\frac{1}{15}a}(\delta+\varepsilon^{a}\tau)^{-\frac{4}{3}}.
	\end{eqnarray*}
Recall \eqref{def.adK}. In a similar way as \eqref{6.30b}, we can verify that
	\begin{multline*}
		|(\partial^\alpha \mathbb{K},e^{\phi}w^2(\alpha,0)\partial^\alpha f)|\leq \eta\varepsilon^{a-1}\|\partial^\alpha f\|^2_{\sigma,w}\\
		+
		C_\eta\varepsilon^{1-a}(\| \partial^\alpha f_y\|_{\sigma}^2+\|(\partial^\alpha\widetilde{u}_y,\partial^\alpha\widetilde{\theta}_y)\|^2)
		+C_{\eta}k^{\frac{1}{12}}\varepsilon^{\frac{3}{5}-\frac{2}{5}a}\mathcal{D}_{2,l,q}(\tau),
	\end{multline*}
	while for the rest terms containing $\overline{G}$ in \eqref{5.1}, we can deduce that
	\begin{eqnarray*}
		&&|(\frac{\partial^{\alpha}(\phi_{y}\partial_{v_{1}}\overline{G})}{\sqrt{\mu}}-\frac{\partial^\alpha P_1( v_1\overline{G}_{y})}{\sqrt{\mu}}-\frac{ \partial^\alpha\overline{G}_{\tau}}{\sqrt{\mu}}, e^{\phi}w^2(\alpha,0)\partial^\alpha f)|
		\notag\\ 
		&&\leq \eta\varepsilon^{a-1}\|\partial^\alpha f\|^2_{\sigma,w}
		+C_{\eta}k^{\frac{1}{12}}\varepsilon^{\frac{3}{5}-\frac{2}{5}a}\mathcal{D}_{2,l,q}(\tau)
		+C_{\eta}\varepsilon^{\frac{7}{5}+\frac{1}{15}a}(\delta+\varepsilon^{a}\tau)^{-\frac{4}{3}}.
	\end{eqnarray*}
	Therefore, \eqref{5.3} follows from all the above estimates by summing over $|\alpha|\leq 1$ and choosing
	$\eta>0$ sufficiently small. This completes the proof of Lemma \ref{lem7.1A}.
\end{proof}
\subsection{Second-order weighted estimates}
In what follows we give the weighted estimates on second order.
\begin{lemma}\label{lem7.2A}
	Under the conditions listed in Lemma \ref{lem.5.1A}, it holds that
	\begin{eqnarray}
		\label{5.8}
		&&\varepsilon^{2(1-a)}
		\sum_{|\alpha|=2}\big\{\frac{1}{2}\frac{d}{d\tau}\|e^{\frac{\phi}{2}}\frac{\partial^\alpha F}{\sqrt{\mu}}\|_w^2+\frac{1}{2}q_{2}q_{3}(\tau)\|\langle v\rangle^{\frac{1}{2}}e^{\frac{\phi}{2}}\partial^\alpha f\|_{w}^2+c\varepsilon^{a-1}\| \partial^\alpha f\|_{\sigma,w}^2\big\}
		\notag\\
		&&\leq C\varepsilon^{1-a}\sum_{|\alpha|=2}(\| \partial^\alpha f\|_{\sigma}^2+\|\partial^\alpha(\widetilde{\rho},\widetilde{u},\widetilde{\theta})\|^2)
		+C\varepsilon^{3(1-a)}\sum_{|\alpha|=2}\|\partial^{\alpha}\widetilde{\phi}_{y}\|^{2}
		\notag\\
		&& \quad+Ck^{\frac{1}{12}}\varepsilon^{\frac{3}{5}-\frac{2}{5}a}\mathcal{D}_{2,l,q}(\tau)
		+ C\varepsilon^{\frac{7}{5}+\frac{1}{15}a}(\delta+\varepsilon^{a}\tau)^{-\frac{4}{3}}\notag\\
		&&\quad+ Cq_{3}(\tau)\mathcal{H}_{2,l,q}(\tau).
	\end{eqnarray}	
\end{lemma}
\begin{proof}
	We take $\partial^{\alpha}$ with $|\alpha|=2$ of \eqref{4.62}
	and take the inner product of the resulting equation with $e^{\phi}w^2(\alpha,0)\frac{\partial^\alpha F}{\sqrt{\mu}}$ over $\mathbb{R}\times{\mathbb R}^3$ to get
	\begin{eqnarray}
		\label{5.4}
		&&((\frac{\partial^\alpha F}{\sqrt{\mu}})_\tau,e^{\phi}w^2(\alpha,0)\frac{\partial^\alpha F}{\sqrt{\mu}})
		+(v_1(\frac{\partial^\alpha F}{\sqrt{\mu}})_{y}+\frac{v_{1}}{2}\phi_{y}\frac{\partial^{\alpha}F}{\sqrt{\mu}},e^{\phi}w^2(\alpha,0)\frac{\partial^\alpha F}{\sqrt{\mu}})
		\notag\\
		&&-(\phi_{y}\partial_{v_{1}}(\frac{\partial^{\alpha}F}{\sqrt{\mu}}),e^{\phi}w^2(\alpha,0)\frac{\partial^\alpha F}{\sqrt{\mu}})\notag\\
		&&
		+\sum_{1\leq\alpha_{1}\leq\alpha}C^{\alpha_{1}}_{\alpha}(\frac{v_{1}}{2}\partial^{\alpha_{1}}\phi_{y}\frac{
			\partial^{\alpha-\alpha_{1}}F}{\sqrt{\mu}}-\partial^{\alpha_{1}}\phi_{y}\partial_{v_{1}}(\frac{\partial^{\alpha-\alpha_{1}}F}{\sqrt{\mu}}),
		e^{\phi}w^2(\alpha,0)\frac{\partial^\alpha F}{\sqrt{\mu}})
		\notag\\
		&&=\varepsilon^{a-1}(\mathcal{L} \partial^\alpha f,e^{\phi}w^2(\alpha,0)\frac{\partial^\alpha F}{\sqrt{\mu}})+\varepsilon^{a-1}(\partial^\alpha\Gamma(f,\frac{M-\mu}{\sqrt{\mu}})\notag\\
		&&\quad+\partial^\alpha\Gamma(\frac{M-\mu}{\sqrt{\mu}},f),e^{\phi}w^2(\alpha,0)\frac{\partial^\alpha F}{\sqrt{\mu}})
		\notag\\
		&&\quad+\varepsilon^{a-1}(\partial^\alpha\Gamma(\frac{G}{\sqrt{\mu}},\frac{G}{\sqrt{\mu}}),e^{\phi}w^2(\alpha,0)\frac{\partial^\alpha F}{\sqrt{\mu}})\notag\\
		&&\quad+\varepsilon^{a-1}(\frac{\partial^{\alpha}L_{M}\overline{G}}{\sqrt{\mu}},e^{\phi}w^2(\alpha,0)\frac{\partial^\alpha F}{\sqrt{\mu}}).
	\end{eqnarray}
	We  estimate \eqref{5.4} term by term. Similar to \eqref{4.68}, for $|\alpha|=2$, we can claim that
	\begin{eqnarray}
		\label{5.7}
		\|\frac{\partial^\alpha F}{\sqrt{\mu}}\|_{\sigma,w}^2
		\leq&& C(\|\partial^\alpha f\|_{\sigma,w}^2+\|\partial^\alpha(\widetilde{\rho},\widetilde{u},\widetilde{\theta})\|^2)
		\notag\\
		&&+Ck^{\frac{1}{12}}\varepsilon^{\frac{3}{5}-\frac{2}{5}a}\mathcal{D}_{2,l,q}(\tau)
		+C\varepsilon^{\frac{7}{5}+\frac{1}{15}a}(\delta+\varepsilon^{a}\tau)^{-\frac{4}{3}}.
	\end{eqnarray}
	From $F=M+\overline{G}+\sqrt{\mu}f$ and \eqref{3.23}, it is direct to verify that
	\begin{eqnarray*}
		&&q_{2}q_{3}(\tau)\|e^{\frac{\phi}{2}}\langle v\rangle^{\frac{1}{2}}\frac{\partial^\alpha F}{\sqrt{\mu}}\|_{w}^2\\
		&&\geq \frac{1}{2}q_{2}q_{3}(\tau)\|e^{\frac{\phi}{2}}\langle v\rangle^{\frac{1}{2}}\partial^\alpha f\|_{w}^2
		-Cq_{2}q_{3}(\tau)\|\langle v\rangle^{\frac{1}{2}}\frac{\partial^\alpha (M +\overline{G})}{\sqrt{\mu}}\|_{w}^2
		\notag\\
		&&\geq \frac{1}{2}q_{2}q_{3}(\tau)\|e^{\frac{\phi}{2}}\langle v\rangle^{\frac{1}{2}}\partial^\alpha f\|_{w}^2
		-C\varepsilon^{a-1}\|\partial^\alpha(\widetilde{\rho},\widetilde{u},\widetilde{\theta})\|^2
		\notag\\
		&&\hspace{1cm}-Ck^{\frac{1}{12}}\varepsilon^{\frac{3}{5}a-\frac{2}{5}}\mathcal{D}_{2,l,q}(\tau)
		-C\varepsilon^{a-1}\varepsilon^{\frac{7}{5}+\frac{1}{15}a}(\delta+\varepsilon^{a}\tau)^{-\frac{4}{3}}.
	\end{eqnarray*}
	In a similar way as \eqref{4.64}, we can verify that
	\begin{eqnarray*}
		&&|(\frac{\partial^\alpha F}{\sqrt{\mu}},e^{\phi}\phi_{\tau}w^2(\alpha,0)\frac{\partial^\alpha F}{\sqrt{\mu}})|\\
		&&\leq C\eta\varepsilon^{a-1}(\|\partial^{\alpha}f\|_{\sigma,w}^{2}
		+\|\partial^{\alpha}(\widetilde{\rho},\widetilde{u},\widetilde{\theta})\|^{2})
		+C_{\eta}k^{\frac{1}{12}}\varepsilon^{\frac{3}{5}a-\frac{2}{5}}\mathcal{D}_{2,l,q}(\tau)
		\notag\\
		&&\quad +C_{\eta}\varepsilon^{a-1}\varepsilon^{\frac{7}{5}+\frac{1}{15}a}(\delta+\varepsilon^{a}\tau)^{-\frac{4}{3}}
		+C_{\eta}q_3(\tau)\varepsilon^{2a-2}\mathcal{H}_{2,l,q}(\tau).
	\end{eqnarray*}
	With the above two estimates and \eqref{5.2a} in hand, we get
	\begin{eqnarray*}
		&&((\frac{\partial^\alpha F}{\sqrt{\mu}})_\tau,e^{\phi}w^2(\alpha,0)\frac{\partial^\alpha F}{\sqrt{\mu}})
		\notag\\
		=&&\frac{1}{2}\frac{d}{d\tau}\|e^{\frac{\phi}{2}}\frac{\partial^\alpha F}{\sqrt{\mu}}\|_w^2
		+q_{2}q_{3}(\tau)\|\langle v\rangle^{\frac{1}{2}}e^{\frac{\phi}{2}}\frac{\partial^\alpha F}{\sqrt{\mu}}\|_{w}^2
		-\frac{1}{2}(\frac{\partial^\alpha F}{\sqrt{\mu}},e^{\phi}\phi_{\tau}w^2(\alpha,0)\frac{\partial^\alpha F}{\sqrt{\mu}}).
		\notag\\
		\geq&& \frac{1}{2}\frac{d}{d\tau}\|e^{\frac{\phi}{2}}\frac{\partial^\alpha F}{\sqrt{\mu}}\|_w^2
		+\frac{1}{2}q_{2}q_{3}(\tau)\|e^{\frac{\phi}{2}}\langle v\rangle^{\frac{1}{2}}\partial^\alpha f\|_{w}^2
		- C\eta\varepsilon^{a-1}\|\partial^{\alpha}f\|_{\sigma,w}^{2}\notag\\
		&&-C_{\eta}\varepsilon^{a-1}\|\partial^\alpha(\widetilde{\rho},\widetilde{u},\widetilde{\theta})\|^2-C_{\eta}k^{\frac{1}{12}}\varepsilon^{\frac{3}{5}a-\frac{2}{5}}\mathcal{D}_{2,l,q}(\tau)\notag\\
		&&-C_{\eta}\varepsilon^{a-1}\varepsilon^{\frac{7}{5}+\frac{1}{15}a}(\delta+\varepsilon^{a}\tau)^{-\frac{4}{3}}
		-C_{\eta}q_3(\tau)\varepsilon^{2a-2}\mathcal{H}_{2,l,q}(\tau).
	\end{eqnarray*}
	The second term on the left hand side of \eqref{5.4} vanishes by employing the same way as \eqref{7.4A}, while for
	the third term on the left hand side of \eqref{5.4}, we can use integration by parts, \eqref{5.4a} and the similar calculations as \eqref{4.64} to get
	\begin{eqnarray*}
		&&|(\phi_{y}\partial_{v_{1}}(\frac{\partial^{\alpha}F}{\sqrt{\mu}}),e^{\phi}w^2(\alpha,0)\frac{\partial^\alpha F}{\sqrt{\mu}})|
		=|(w(\alpha,0)\phi_{y}\frac{\partial^{\alpha}F}{\sqrt{\mu}},e^{\phi}\partial_{v_{1}}w(\alpha,0)\frac{\partial^\alpha F}{\sqrt{\mu}})|
		\notag\\
		&&\hspace{0.5cm}\leq  C\eta\varepsilon^{a-1}(\|\partial^{\alpha}f\|_{\sigma,w}^{2}
		+\|\partial^{\alpha}(\widetilde{\rho},\widetilde{u},\widetilde{\theta})\|^{2})+
		C_{\eta}k^{\frac{1}{12}}\varepsilon^{\frac{3}{5}a-\frac{2}{5}}\mathcal{D}_{2,l,q}(\tau)
		\notag\\
		&&\hspace{1cm}+C_{\eta}\varepsilon^{a-1}\varepsilon^{\frac{7}{5}+\frac{1}{15}a}(\delta+\varepsilon^{a}\tau)^{-\frac{4}{3}}
		+C_{\eta}q_3(\tau)\varepsilon^{2a-2}\mathcal{H}_{2,l,q}(\tau).
	\end{eqnarray*}
	The last term on the left hand side of \eqref{5.4} can be treated in the similar way as \eqref{4.65}.
	For $1\leq|\alpha_{1}|\leq|\alpha|=2$, we have from $F=M+\overline{G}+\sqrt{\mu}f$ that
	\begin{eqnarray}
		\label{5.7a}
		&&-(\frac{v_{1}}{2}\partial^{\alpha_{1}}\phi_{y}\frac{
			\partial^{\alpha-\alpha_{1}}F}{\sqrt{\mu}}-\partial^{\alpha_{1}}\phi_{y}\partial_{v_{1}}(\frac{\partial^{\alpha-\alpha_{1}}F}{\sqrt{\mu}}),
		e^{\phi}w^2(\alpha,0)\frac{\partial^\alpha F}{\sqrt{\mu}})\notag\\
		&&=(\partial^{\alpha_{1}}\phi_{y}\frac{\partial_{v_{1}}\partial^{\alpha-\alpha_{1}}F}{\sqrt{\mu}}
		,e^{\phi}w^2(\alpha,0)\frac{\partial^\alpha F}{\sqrt{\mu}})
		\notag\\
		&&=(\partial^{\alpha_{1}}\phi_{y}\frac{\partial_{v_{1}}\partial^{\alpha-\alpha_{1}}(M+\overline{G})}{\sqrt{\mu}},
		e^{\phi}w^2(\alpha,0)\frac{\partial^\alpha F}{\sqrt{\mu}})\notag\\
		&&\quad+(\partial^{\alpha_{1}}\phi_{y}\frac{\partial_{v_{1}}\partial^{\alpha-\alpha_{1}}(\sqrt{\mu}f)}{\sqrt{\mu}},
		e^{\phi}w^2(\alpha,0)\frac{\partial^\alpha(M+\overline{G})}{\sqrt{\mu}})
		\notag\\
		&&\quad+(\partial^{\alpha_{1}}\phi_{y}\frac{\partial_{v_{1}}\partial^{\alpha-\alpha_{1}}(\sqrt{\mu}f)}{\sqrt{\mu}},
		e^{\phi}w^2(\alpha,0)\frac{\partial^\alpha(\sqrt{\mu}f)}{\sqrt{\mu}})\notag\\
		&&:=J^{1}_{9}+J^{2}_{9}+J^{3}_{9}.
	\end{eqnarray}
	We need to bound the terms from $J^{1}_{9}$ to $J^{3}_{9}$ in \eqref{5.7a}. Firstly, $J^{1}_{9}$ can be estimated as follows.
	When $\alpha_{1}=\alpha$, it is bounded by
	\begin{eqnarray*}
		|J^{1}_{9}|\leq&& \eta\varepsilon^{a-1}\|\langle v\rangle^{-\frac{1}{2}}w(\alpha,0)\frac{\partial^\alpha F}{\sqrt{\mu}}\|^2
		+C_{\eta}\varepsilon^{1-a}\|\langle v\rangle^{\frac{1}{2}}w(\alpha,0)\partial^{\alpha}\phi_{y}
		\frac{\partial_{v_{1}}(M+\overline{G})}{\sqrt{\mu}}\|^{2}
		\notag\\
		\leq&& C\eta\varepsilon^{a-1}(\|\partial^\alpha f\|_{\sigma,w}^2+\|\partial^\alpha(\widetilde{\rho},\widetilde{u},\widetilde{\theta})\|^2)
		+C_{\eta}\varepsilon^{1-a}\|\partial^{\alpha}\widetilde{\phi}_{y}\|^{2}
		\notag\\
		&&+C_{\eta}k^{\frac{1}{12}}\varepsilon^{\frac{3}{5}a-\frac{2}{5}}\mathcal{D}_{2,l,q}(\tau)
		+C_{\eta}\varepsilon^{a-1}\varepsilon^{\frac{7}{5}+\frac{1}{15}a}(\delta+\varepsilon^{a}\tau)^{-\frac{4}{3}},
	\end{eqnarray*}
	where we have used \eqref{3.16}, \eqref{5.7}, Lemma \ref{lem7.2}, \eqref{3.21} and \eqref{3.22}. When $\alpha_{1}\neq\alpha$, it has the same bound as the above.
	The term $J^{2}_{9}$ can be handled in the similar manner as $J^{1}_{9}$, while for $J^{3}_{9}$, we follow the similar approach as \eqref{4.91a} to get
	\begin{eqnarray*}
		|J^{3}_{9}|\leq C\eta\varepsilon^{a-1}\| \partial^\alpha f\|^{2}_{\sigma,w}+C_{\eta}q_{3}(\tau)\mathcal{H}_{2,l,q}(\tau).
	\end{eqnarray*}
	Therefore, combining the estimates $J^{1}_{9}$ with $J^{2}_{9}$ and $J^{3}_{9}$, we get
	\begin{eqnarray*}
		&&\sum_{1\leq\alpha_{1}\leq\alpha}C^{\alpha_{1}}_{\alpha}|(\frac{v_{1}}{2}\partial^{\alpha_{1}}\phi_{y}\frac{
			\partial^{\alpha-\alpha_{1}}F}{\sqrt{\mu}}-\partial^{\alpha_{1}}\phi_{y}\partial_{v_{1}}(\frac{\partial^{\alpha-\alpha_{1}}F}{\sqrt{\mu}}),
		e^{\phi}w^2(\alpha,0)\frac{\partial^\alpha F}{\sqrt{\mu}})|
		\notag\\
		&&\leq C\eta\varepsilon^{a-1}(\|\partial^\alpha f\|_{\sigma,w}^2+\|\partial^\alpha(\widetilde{\rho},\widetilde{u},\widetilde{\theta})\|^2)
		+C_{\eta}\varepsilon^{1-a}\|\partial^{\alpha}\widetilde{\phi}_{y}\|^{2}
		\notag\\
		&&\quad+C_{\eta}k^{\frac{1}{12}}\varepsilon^{\frac{3}{5}a-\frac{2}{5}}\mathcal{D}_{2,l,q}(\tau)
		+C_{\eta}\varepsilon^{a-1}\varepsilon^{\frac{7}{5}+\frac{1}{15}a}(\delta+\varepsilon^{a}\tau)^{-\frac{4}{3}}
		+C_{\eta}q_{3}(\tau)\mathcal{H}_{2,l,q}(\tau).
	\end{eqnarray*}
	
	Next we move to bound the terms on the right hand side of \eqref{5.4}. We can separate the first term on the right hand side of \eqref{5.4}
	into two parts by using $F=M+\overline{G}+\sqrt{\mu}f$, namely
	$$
	-\varepsilon^{a-1}(\mathcal{L} \partial^\alpha f, e^{\phi} w^2(\alpha,0)\partial^\alpha f)
	\geq c\varepsilon^{a-1}\|\partial^\alpha f\|_{\sigma,w}^2-C\varepsilon^{a-1}\|\partial^\alpha f\|_{\sigma }^2,
	$$
	and
	\begin{eqnarray*}
		&&\varepsilon^{a-1}|(\mathcal{L} \partial^\alpha f, e^{\phi}w^2(\alpha,0)\frac{ \partial^\alpha(M+\overline{G}) }{\sqrt{\mu}})|\\
		&&\leq C\varepsilon^{a-1}\|\partial^\alpha f\|_\sigma\|w^2(\alpha,0)\frac{ \partial^\alpha(M+\overline{G})}{\sqrt{\mu}}\|_\sigma
		\notag\\
		&&\leq C \varepsilon^{a-1}(\|\partial^\alpha f\|_\sigma^2+\|\partial^\alpha(\widetilde{\rho},\widetilde{u},\widetilde{\theta})\|^2)
		+Ck^{\frac{1}{12}}\varepsilon^{\frac{3}{5}a-\frac{2}{5}}\mathcal{D}_{2,l,q}(\tau)\\
&&\quad		+C\varepsilon^{a-1}\varepsilon^{\frac{7}{5}+\frac{1}{15}a}(\delta+\varepsilon^{a}\tau)^{-\frac{4}{3}},
	\end{eqnarray*}
	in terms of \eqref{7.6} and the fact that $\mathcal{L}g=\Gamma(g,\sqrt{\mu})+\Gamma(\sqrt{\mu}, g)$.
	So, we have
	\begin{eqnarray*}
		&&\varepsilon^{a-1}(\mathcal{L} \partial^\alpha f, e^{\phi}w^2(\alpha,0)\frac{\partial^\alpha F}{\sqrt{\mu}})\\\
		\leq&&
		-c\varepsilon^{a-1}\|\partial^\alpha f\|_{\sigma,w}^2
		+C\varepsilon^{a-1}(\|\partial^\alpha f\|_{\sigma}^2+\|\partial^\alpha(\widetilde{\rho},\widetilde{u},\widetilde{\theta})\|^2)
		\notag\\
		&&+Ck^{\frac{1}{12}}\varepsilon^{\frac{3}{5}a-\frac{2}{5}}\mathcal{D}_{2,l,q}(\tau)
		+C\varepsilon^{a-1}\varepsilon^{\frac{7}{5}+\frac{1}{15}a}(\delta+\varepsilon^{a}\tau)^{-\frac{4}{3}}.
	\end{eqnarray*}
	The second term  on the right hand side of \eqref{5.4} can be handled in the similar way as \eqref{4.85}, which can be bounded by
	\begin{eqnarray}
		\label{7.13b}
		&&C(\eta+\eta_{0}+k^{\frac{1}{12}}\varepsilon^{\frac{3}{5}-\frac{2}{5}a})\varepsilon^{a-1}(\|\partial^{\alpha}f\|^{2}_{\sigma,w}	+\|\partial^{\alpha}(\widetilde{\rho},\widetilde{u},\widetilde{\theta})\|^{2})
		\notag\\
		&&\hspace{1cm}+C_{\eta}k^{\frac{1}{12}}\varepsilon^{\frac{3}{5}a-\frac{2}{5}}\mathcal{D}_{2,l,q}(\tau)
		+C_{\eta}\varepsilon^{a-1}\varepsilon^{\frac{7}{5}+\frac{1}{15}a}(\delta+\varepsilon^{a}\tau)^{-\frac{4}{3}}.
	\end{eqnarray}
	For the third term on the right hand side of \eqref{5.4}, using \eqref{4.34A} and \eqref{5.7}, we deduce that
	\begin{eqnarray*}
		&&\varepsilon^{a-1}|(\partial^\alpha\Gamma(\frac{G}{\sqrt{\mu}},\frac{G}{\sqrt{\mu}}),e^{\phi}w^2(\alpha,0)\frac{\partial^\alpha F}{\sqrt{\mu}})|
		\notag\\
		&&\leq C\eta\varepsilon^{a-1}(\|\partial^{\alpha}f\|^{2}_{\sigma,w}
		+\|\partial^{\alpha}(\widetilde{\rho},\widetilde{u},\widetilde{\theta})\|^{2})\\
&&\quad		+C_{\eta}k^{\frac{1}{12}}\varepsilon^{\frac{3}{5}a-\frac{2}{5}}\mathcal{D}_{2,l,q}(\tau)
		+C_{\eta}\varepsilon^{a-1}\varepsilon^{\frac{7}{5}+\frac{1}{15}a}(\delta+\varepsilon^{a}\tau)^{-\frac{4}{3}},
	\end{eqnarray*}
	while for the last term on the right hand side of \eqref{5.4}, we can follow the similar approach as  \eqref{4.97b} to get the same bound as \eqref{7.13b}.
	
	Substituting all the above estimates into \eqref{5.4} and choosing $\eta>0$ small enough, then
	multiplying the resulting equation by $\varepsilon^{2(1-a)}$, the desired estimate \eqref{5.8} follows.
	We thus complete the proof of Lemma \ref{lem7.2A}.
\end{proof}
Combining those estimates in Lemma \ref{lem7.1A} and Lemma\ref{lem7.2A}, we are able to obtain the following result.
\begin{lemma}\label{lem5.2}
	It holds that
	\begin{eqnarray}
		\label{5.9}
		&&\sum_{|\alpha|\leq 1}\|\partial^\alpha f(\tau)\|_w^2
		+\varepsilon^{2(1-a)}\sum_{|\alpha|=2}\|\partial^\alpha f(\tau)\|_w^2
		\notag\\
		&&+c\varepsilon^{a-1}\int^{\tau}_{0}\big\{\sum_{|\alpha|\leq 1}\|\partial^\alpha f(s)\|^2_{\sigma,w}+\varepsilon^{2(1-a)}\sum_{|\alpha|=2}\|\partial^\alpha f(s)\|^2_{\sigma,w}\big\}\,ds
		\notag\\
		&&+q_{2}\int^{\tau}_{0}q_{3}(s)\big\{\sum_{|\alpha|\leq 1}\|\langle v\rangle^{\frac{1}{2}}\partial^\alpha f(s)\|^2_w
		+\varepsilon^{2(1-a)}\sum_{|\alpha|=2}\|\langle v\rangle^{\frac{1}{2}}\partial^\alpha f(s)\|^2_w\big\}\,ds
		\notag\\
		\leq&& C\varepsilon^{2(1-a)}\sum_{|\alpha|=2}\|\partial^\alpha(\widetilde{\rho},\widetilde{u},\widetilde{\theta})(\tau)\|^2
		+Ck^{\frac{1}{3}}\varepsilon^{\frac{6}{5}-\frac{4}{5}a}+
		C\int^{\tau}_{0}\mathcal{D}_{2}(s)\,ds
		\notag\\
		&&+C(\eta_{0}+k^{\frac{1}{12}}\varepsilon^{\frac{3}{5}-\frac{2}{5}a})\int^{\tau}_{0}\mathcal{D}_{2,l,q}(s)\,ds
		+C\int^{\tau}_{0}q_{3}(s)\mathcal{H}_{2,l,q}(s)\,ds.
	\end{eqnarray}
\end{lemma}
\begin{proof}
	Adding \eqref{5.8} to \eqref{5.3} and integrating the resulting equation with respect to $\tau$ and using
	the following two estimates, namely
	\begin{multline*}
		\varepsilon^{2(1-a)}\sum_{|\alpha|=2}\|e^{\frac{\phi}{2}}\frac{\partial^\alpha F(\tau)}{\sqrt{\mu}}\|_w^2\\
		\geq c\varepsilon^{2(1-a)}\sum_{|\alpha|=2}(\|\partial^{\alpha}f\|_{w}^{2}
		-\|\partial^{\alpha}(\widetilde{\rho},\widetilde{u},\widetilde{\theta})\|^{2})-Ck^{\frac{1}{3}}\varepsilon^{\frac{6}{5}-\frac{4}{5}a},
	\end{multline*}
	and
	\begin{equation*}
		\varepsilon^{2(1-a)}\sum_{|\alpha|=2}\|e^{\frac{\phi}{2}}\frac{\partial^\alpha F(0)}{\sqrt{\mu}}\|_w^2
		\leq Ck^{\frac{1}{3}}\varepsilon^{\frac{6}{5}-\frac{4}{5}a},
	\end{equation*}
	then the  expected estimates \eqref{5.9} follows by combining this, \eqref{3.20}, \eqref{4.94} and \eqref{3.21a}. 
	This hence completes the proof of Lemma \ref{lem5.2}.
\end{proof}
\subsection{Weighted mixed derivative estimates}
Next we give the weighted estimates on mixed derivatives.
\begin{lemma}\label{lem5.3}
	It holds that
	\begin{align}
		\label{5.13}
		&\sum_{|\alpha|+|\beta|\leq 2,|\beta|\geq1}\big\{\|\partial^\alpha_\beta f(\tau)\|_w^2
		+q_{2}\int^{\tau}_{0}q_{3}(s)\|\langle v\rangle^{\frac{1}{2}}\partial^\alpha_\beta f(s)\|^2_w\,ds\notag\\	&\qquad\qquad\qquad\qquad+c\varepsilon^{a-1}\int^{\tau}_{0}\|\partial^\alpha_\beta f(s)\|^2_{\sigma,w}\,ds \big\}
		\notag\\
		\leq& Ck^{\frac{1}{3}}\varepsilon^{\frac{6}{5}-\frac{4}{5}a}
		+C\varepsilon^{1-a}\int^{\tau}_{0}\sum_{ 1\leq|\alpha|\leq 2}\|\partial^\alpha(\widetilde {u},\widetilde{\theta})(s)\|^2\,ds
		\nonumber\\
		&+C\varepsilon^{a-1}\int^{\tau}_{0}\{\sum_{ |\alpha|\leq 1}\|\partial^\alpha f(s)\|_{\sigma,w}^2+\varepsilon^{2(1-a)}\sum_{ |\alpha|=2}\|\partial^\alpha f(s)\|_{\sigma,w}^2\}\,ds
		\notag\\
		&+C(\eta_{0}+k^{\frac{1}{12}}\varepsilon^{\frac{3}{5}-\frac{2}{5}a})\int^{\tau}_{0}\mathcal{D}_{2,l,q}(s)\,ds+
		C\int^{\tau}_{0}q_{3}(s)\mathcal{H}_{2,l,q}(s)\,ds.
	\end{align}
\end{lemma}
\begin{proof}
	Let $|\alpha|+|\beta|\leq2$ and $|\beta|\geq1$, then $|\alpha|\leq1$.
	Applying $\partial^\alpha_\beta$ to equation \eqref{3.7} and further multiplying it $e^{\phi}w^2(\alpha,\beta)\partial_{\beta}^\alpha f$, we have
	\begin{align}
		\label{5.10}
		&(\partial^\alpha_\beta f_\tau,e^{\phi}w^2(\alpha,\beta)\partial_{\beta}^\alpha f)
		+(v_1\partial^\alpha_\beta f_{y}+\frac{v_{1}}{2}\partial^\alpha_\beta[\phi_{y} f],e^{\phi}w^2(\alpha,\beta)\partial_{\beta}^\alpha f)
		\notag\\
		&+(C_\beta^{\beta-e_1}\partial^\alpha_{\beta-e_1}f_{y},e^{\phi}w^2(\alpha,\beta)\partial_{\beta}^\alpha f)
		+(\frac{1}{2}C_\beta^{\beta-e_1}\partial^\alpha_{\beta-e_1}(\phi_{y} f),e^{\phi}w^2(\alpha,\beta)\partial_{\beta}^\alpha f)
		\notag\\
		&-(\partial^\alpha_\beta[\phi_{y}\partial_{v_{1}}f],e^{\phi}w^2(\alpha,\beta)\partial_{\beta}^\alpha f)
		-\varepsilon^{a-1}(\partial^\alpha_\beta\mathcal{L}f,e^{\phi}w^2(\alpha,\beta)\partial_{\beta}^\alpha f)
		\notag\\
		=&\varepsilon^{a-1}(\partial^\alpha_\beta\Gamma(f,\frac{M-\mu}{\sqrt{\mu}})
		+\partial^\alpha_\beta\Gamma(\frac{M-\mu}{\sqrt{\mu}},f)
		+\partial^\alpha_\beta\Gamma(\frac{G}{\sqrt{\mu}},\frac{G}{\sqrt{\mu}}),e^{\phi}w^2(\alpha,\beta)\partial_{\beta}^\alpha f)
		\notag\\
		&+(\partial^\alpha_\beta \mathbb{K},e^{\phi}w^2(\alpha,\beta)\partial_{\beta}^\alpha f)
		\notag\\
		&+(\partial^\alpha_\beta[\frac{\phi_{y}\partial_{v_{1}}\overline{G}}{\sqrt{\mu}}]-\partial^\alpha_\beta[\frac{P_1( v_1\overline{G}_{y})}{\sqrt{\mu}}]-\partial^\alpha_\beta[\frac{\overline{G}_{\tau}}{\sqrt{\mu}}]
		,e^{\phi}w^2(\alpha,\beta)\partial_{\beta}^\alpha f).
	\end{align}
	Here $e_1=(1,0,0)$ and we recall \eqref{def.adK}  for $\mathbb{K}$.
	
	We need to compute each term for \eqref{5.10}. The first term on the left hand side of \eqref{5.10} can be treated in the similar way as \eqref{7.4c}, which can be dominated by
	\begin{eqnarray*}
		&&(\partial^\alpha_\beta f_\tau,e^{\phi}w^2(\alpha,\beta)\partial^\alpha_\beta f)
		\notag\\
		&&=\frac{1}{2}\frac{d}{d\tau}\|e^{\frac{\phi}{2}}\partial^\alpha_\beta f\|_w^2
		-\frac{1}{2}(\partial^\alpha_\beta f,e^{\phi}[w^2(\alpha,\beta)]_{\tau}\partial^\alpha_\beta f)-\frac{1}{2}(\partial^\alpha_\beta f,[e^{\phi}]_{\tau}w^2(\alpha,\beta)\partial^\alpha_\beta f)
		\notag\\
		&&\geq\frac{1}{2}\frac{d}{d\tau}\|e^{\frac{\phi}{2}}\partial^\alpha_\beta f\|_w^2
		+q_{2}q_{3}(\tau)\|\langle v\rangle^{\frac{1}{2}}e^{\frac{\phi}{2}}\partial^\alpha_\beta f\|^2_w\\
		&&\quad -C\eta\varepsilon^{a-1}\|\partial^\alpha_\beta f\|^2_{\sigma,w}-C_{\eta}q_{3}(\tau)\mathcal{H}_{2,l,q}(\tau).
	\end{eqnarray*}
	Following the similar arguments as \eqref{7.7b}, the second term on the left hand side of \eqref{5.10} can be bounded by
	\begin{eqnarray*}
		&&(v_1\partial^\alpha_\beta f_{y}+\frac{v_{1}}{2}\phi_{y}\partial^{\alpha}_\beta f,e^{\phi}w^2(\alpha,\beta)\partial_{\beta}^\alpha f)+(\frac{v_{1}}{2}\partial^{\alpha}\phi_{y}\partial_\beta f
		,e^{\phi}w^2(\alpha,\beta)\partial^\alpha_\beta f)
		\notag\\
		&&\leq C\eta\varepsilon^{a-1}\|\partial^\alpha_\beta f\|^2_{\sigma,w}+C_{\eta}q_{3}(\tau)\mathcal{H}_{2,l,q}(\tau).
	\end{eqnarray*}
	For the third term on the left hand side of \eqref{5.10}, it is straightforward to get
	\begin{eqnarray*}
		&&|(\partial^\alpha_{\beta-e_1}f_y,e^{\phi}w^2(\alpha,\beta)\partial^\alpha_\beta f)|
		\notag\\
		&&\leq \eta\varepsilon^{1-a}\|\langle v\rangle^{-\frac{1}{2}}w(\alpha,\beta)\partial^\alpha_{\beta-e_1}f_y\|^2
		+C_\eta\varepsilon^{a-1}\|\langle v\rangle^2\langle v\rangle^{-\frac{3}{2}}w(\alpha,\beta) \partial_{e_1}\partial^\alpha_{\beta-e_1}f\|^2
		\notag\\
		&&\leq C\eta\varepsilon^{1-a}\|w(\alpha+1,\beta-e_1)\partial^\alpha_{\beta-e_1}f_y\|^2_\sigma
		+C_\eta\varepsilon^{a-1}\|\langle v\rangle^2w(\alpha,\beta) \partial^\alpha_{\beta-e_1}f\|^2_\sigma
		\notag\\
		&&\leq C\eta\varepsilon^{1-a}\|\partial^\alpha_{\beta-e_1}f_y\|^2_{\sigma,w(\alpha+1,\beta-e_1)}
		+C_\eta\varepsilon^{a-1}\| \partial^\alpha_{\beta-e_1}f\|^2_{\sigma,w(\alpha,\beta-e_1)}.
	\end{eqnarray*}
	Here we have used the facts that \eqref{4.45a}, \eqref{3.16}, $\partial^\alpha_\beta=\partial_{e_1}\partial^\alpha_{\beta-e_1}$ and 
	$$
	\langle v\rangle^2\langle v\rangle^{2(l-|\alpha|-|\beta|)}=\langle v\rangle^{2(l-|\alpha|-|\beta-e_1|)}.
	$$
	For the fourth term on the left hand side of \eqref{5.10}, we use \eqref{4.45a}, \eqref{3.16}, \eqref{3.14} and \eqref{4.60} to get
	\begin{eqnarray*}
		&&|(\partial^\alpha[\phi_{y}\partial_{\beta-e_1}f],e^{\phi}w^2(\alpha,\beta)\partial^\alpha_\beta f)|
		\notag\\
		&&\leq C\sum_{\alpha_{1}\leq \alpha}\|\partial^{\alpha_{1}}\phi_{y}\|_{L_y^\infty}\|\langle v\rangle^{\frac{1}{2}}w(\alpha,\beta)\partial_{\beta-e_1}\partial^{\alpha-\alpha_{1}}f\|
		\|\langle v\rangle^{-\frac{1}{2}}w(\alpha,\beta)\partial^\alpha_\beta f\|
		\notag\\
		&&\leq \eta\varepsilon^{a-1}\|\langle v\rangle^{-\frac{1}{2}}w(\alpha,\beta)\partial^\alpha_\beta f\|^2+C_\eta q_{3}(\tau)\|\langle v\rangle^{\frac{1}{2}}w(\alpha,\beta)\partial_{\beta-e_1}\partial^{\alpha-\alpha_{1}}f\|^2
		\notag\\
		&&\leq C\eta\varepsilon^{a-1}\|\partial^\alpha_\beta f\|^2_{\sigma,w}+C_{\eta}q_{3}(\tau)\mathcal{H}_{2,l,q}(\tau).
	\end{eqnarray*}
	Using the similar arguments as \eqref{7.10c}, the fifth term on the left hand side of \eqref{5.10}  can be dominated by
	\begin{equation*}
		C\eta\varepsilon^{a-1}\|\partial_{\beta}^\alpha f\|_{\sigma,w}^{2}+C_{\eta}q_{3}(\tau)\mathcal{H}_{2,l,q}(\tau).
	\end{equation*}
	For the last term on the left hand side of \eqref{5.10}, we deduce from \eqref{7.5} and \eqref{4.45a} that
	\begin{multline*}
		-\varepsilon^{a-1}(\partial^\alpha_\beta\mathcal{L} f,e^{\phi}w^2(\alpha,\beta)\partial^\alpha_\beta f)\\
		\geq \varepsilon^{a-1}\big(c\|\partial^\alpha_\beta f\|^2_{\sigma,w}-\eta\sum_{|\beta_1|=|\beta|}\|\partial^\alpha_{\beta_1} f\|_{\sigma,w}^2
		-C_\eta\sum_{|\beta_1|<|\beta|}\|\partial^\alpha_{\beta_1}f\|_{\sigma,w}^2\big).
	\end{multline*}
	
	Now we go to bound the terms on the right hand side of \eqref{5.10}.
	The first term on the right hand side of \eqref{5.10} can be directly obtained from \eqref{7.9} and \eqref{7.18}.
	In view of \eqref{1.10}, Lemma \ref{lem7.2}, \eqref{3.21}, \eqref{3.22}, \eqref{7.25aa}, \eqref{7.24}, \eqref{7.25} and \eqref{5.26b}, we can claim that the other terms on the right hand side of \eqref{5.10} are bounded by
	\begin{eqnarray*}
		&&C\eta\varepsilon^{a-1}\|\partial_{\beta}^\alpha f\|_{\sigma,w}^{2}
		+C_\eta\varepsilon^{1-a}(\|\partial^\alpha(\widetilde {u}_{y},\widetilde{\theta}_y)\|^2
		+\|\partial^\alpha f_y\|_\sigma^2)
		\notag\\
		&&\hspace{1cm}+C_{\eta}k^{\frac{1}{12}}\varepsilon^{\frac{3}{5}-\frac{2}{5}a}\mathcal{D}_{2,l,q}(\tau)
		+C_{\eta}\varepsilon^{\frac{7}{5}+\frac{1}{15}a}(\delta+\varepsilon^{a}\tau)^{-\frac{4}{3}}.
	\end{eqnarray*}
	For $|\alpha|+|\beta|\leq2$ with $|\beta|\geq1$ and any $\eta>0$, we conclude form the above estimates and \eqref{5.10} that
	\begin{eqnarray}
		\label{5.12}
		\frac{1}{2}\frac{d}{d\tau}&&\|e^{\frac{\phi}{2}}\partial^\alpha_\beta f\|_w^2
		+q_{2}q_{3}(\tau)\|\langle v\rangle^{\frac{1}{2}}e^{\frac{\phi}{2}}\partial^\alpha_\beta f\|^2_w
		+c\varepsilon^{a-1}\|\partial^\alpha_\beta f\|^2_{\sigma,w}
		\notag\\
		\leq&& C\eta\varepsilon^{a-1}\sum_{|\beta_1|=|\beta|}\|\partial^\alpha_{\beta_1}f\|_{\sigma,w}^2+ C\eta\varepsilon^{1-a}\| \partial^\alpha_{\beta-e_1}f_{y}\|^2_{\sigma,w}
		\notag\\
		&&+C_\eta\varepsilon^{a-1}\sum_{|\beta_1|<|\beta|}\|\partial^\alpha_{\beta_1}f\|_{\sigma,w}^2
		+C_\eta\varepsilon^{1-a}(\|\partial^\alpha(\widetilde {u}_{y},\widetilde{\theta}_y)\|^2+\|\partial^\alpha f_y\|_{\sigma}^2)
		\notag\\
		&&+C_{\eta}(\eta_{0}+k^{\frac{1}{12}}\varepsilon^{\frac{3}{5}-\frac{2}{5}a})\mathcal{D}_{2,l,q}(\tau)\notag\\
		&&+C_{\eta}\varepsilon^{\frac{7}{5}+\frac{1}{15}a}(\delta+\varepsilon^{a}\tau)^{-\frac{4}{3}}+
		C_{\eta}q_{3}(\tau)\mathcal{H}_{2,l,q}(\tau).
	\end{eqnarray}
	Since the coefficient of the first term in the third line is large, we can use the induction on $|\beta|$ in \eqref{5.12} to cancel this term
	and then take $\eta>0$ small enough. For this, by the suitable linear combinations of \eqref{5.12} for all the cases that $|\alpha|+|\beta|\leq2$ with $|\beta|\geq1$,
	then integrating the resulting equation with respect to $\tau$ and using \eqref{3.20} as well as \eqref{3.21}, we prove that \eqref{5.13} holds.
	This completes the proof of Lemma  \ref{lem5.3}.
\end{proof}
Based on Lemma \ref{lem5.2} and Lemma \ref{lem5.3}, we are able to obtain all the weighted energy estimates.
\begin{lemma}
	It holds that
	\begin{eqnarray}
		\label{5.14}
		&&\sum_{|\alpha|\leq 1}\|\partial^\alpha f(\tau)\|^{2}_{w}
		+\varepsilon^{2(1-a)}\sum_{|\alpha|=2}\|\partial^\alpha f(\tau)\|^{2}_{w}
		+\sum_{|\alpha|+|\beta|\leq 2,|\beta|\geq1}\|\partial^\alpha_\beta f(\tau)\|^{2}_{w}
		\notag\\
		&&\hspace{0.5cm}
		+c\varepsilon^{a-1}\int^{\tau}_{0}\big\{\sum_{|\alpha|\leq 1}\|\partial^\alpha f(s)\|^2_{\sigma,w}
		+\varepsilon^{2(1-a)}\sum_{|\alpha|=2}\|\partial^\alpha f(s)\|^2_{\sigma,w}\big\}\,ds
		\notag\\
		&&\hspace{0.5cm}+c \varepsilon^{a-1}\int^{\tau}_{0}\sum_{|\alpha|
			+|\beta|\leq 2,|\beta|\geq1}\|\partial^\alpha_\beta f(s)\|^2_{\sigma,w}\,ds
		+q_{2}\int^{\tau}_{0}q_{3}(s)\mathcal{H}_{2,l,q}(s)\,ds
		\notag\\
		&&\leq C\varepsilon^{2(1-a)}\sum_{|\alpha|=2}\|\partial^\alpha(\widetilde{\rho},\widetilde{u},\widetilde{\theta})(\tau)\|^2
		+Ck^{\frac{1}{3}}\varepsilon^{\frac{6}{5}-\frac{4}{5}a}+C\int^{\tau}_{0}\mathcal{D}_{2}(s)\,ds
		\notag\\
		&&\hspace{0.5cm}
		+C(\eta_{0}+k^{\frac{1}{12}}\varepsilon^{\frac{3}{5}-\frac{2}{5}a})\int^{\tau}_{0}\mathcal{D}_{2,l,q}(s)\,ds
		+C\int^{\tau}_{0}q_{3}(s)\mathcal{H}_{2,l,q}(s)\,ds.
	\end{eqnarray}
\end{lemma}

\section{Global existence and convergence rate}\label{sec.8}
In this section, we  give the proof of the main result Theorem \ref{thm2.1}  stated in Section \ref{sec.2.1} using the a priori estimates obtained in previous Sections \ref{sec.5}-\ref{sec.7}. In the meantime, we also conclude the proof of Theorem \ref{thm.rs} which is a rough version of Theorem \ref{thm2.1}. 

\medskip
\noindent{\it Proof of Theorem \ref{thm2.1} and Theorem \ref{thm.rs}.}
By a suitable linear combination of \eqref{4.1b} and \eqref{5.14}, there exists a constant $C_{3}>0$ such that
\begin{eqnarray}
	\label{6.1}
	&&\mathcal{E}_{2,l,q}(\tau)+\int^{\tau}_{0}\|\sqrt{\bar{u}_{1y}}(\widetilde{\rho},\widetilde{u}_{1},\widetilde{\theta})(s)\|^{2}\,ds\notag\\
	&&\quad+q_{2}\int^{\tau}_{0}q_{3}(s)\mathcal{H}_{2,l,q}(s)\,ds
	+\int^{\tau}_{0}\mathcal{D}_{2,l,q}(s)\,ds
	\notag\\
	&&\leq C_{3}k^{\frac{1}{3}}\varepsilon^{\frac{6}{5}-\frac{4}{5}a}
	+C_{3}(\eta_{0}+k^{\frac{1}{12}}\varepsilon^{\frac{3}{5}-\frac{2}{5}a}+k^{\frac{1}{12}}\varepsilon^{\frac{3}{5}a-\frac{2}{5}})\int^{\tau}_{0}
	\mathcal{D}_{2,l,q}(s)\,ds\notag\\
	&&\quad+C_{3}\int^{\tau}_{0}q_{3}(s)\mathcal{H}_{2,l,q}(s)\,ds.
\end{eqnarray}
Here $\mathcal{E}_{2,l,q}(\tau)$, $\mathcal{D}_{2,l,q}(\tau)$, $q_{3}(\tau)$ and $\mathcal{H}_{2,l,q}(\tau)$ are given by \eqref{3.17}, \eqref{3.18}, \eqref{3.14} and \eqref{4.60}, respectively. At the moment,  one has to require that the second term on the right hand side of \eqref{6.1} should be absorbed  by the left hand side. Therefore, this leads us to impose
\begin{equation}
	\label{6.2}
	\frac{3}{5}-\frac{2}{5}a\geq 0 \quad \mbox{and} \quad \frac{3}{5}a-\frac{2}{5}\geq0,
	\quad \mbox{that is} \quad \frac{2}{3}\leq a\leq\frac{3}{2}.
\end{equation}
We should point out that the restriction  condition \eqref{6.2} has been used in previous Sections \ref{sec.4}-\ref{sec.7}.
In the meantime we also have used \eqref{4.25a},  \eqref{4.68b}, \eqref{5.15a} and \eqref{6.2}, that is
\begin{equation}
	\label{6.3b}
	\max\{4-5b,\frac{2}{3},\frac{21-10b}{24}\}\leq a\leq1.
\end{equation}
Due to \eqref{6.3b} and \eqref{4.94}, we need to require that
\begin{equation}
	\label{6.4b}
	4-5b\leq 1 \quad \mbox{and}\quad  b\leq1, \quad \mbox{that is}\quad \frac{3}{5}\leq b\leq1.
\end{equation}
That is the reason why we let $b\in[\frac{3}{5},1]$ in \eqref{1.4B} which we start with.
Hence, by \eqref{6.3b} and \eqref{6.4b}, we need to require \eqref{3.2a}
for the choice of the parameter $a$ in the scaling transformation \eqref{3.1} at the beginning.

Choosing $q_{2}=k^{-\frac{1}{12}}$ and letting $\eta_0$ and $k$ be small enough such that
$C_{3}< \frac{1}{4}q_{2}$ and $C_{3}(\eta_{0}+2k^{\frac{1}{12}})<\frac{1}{2}$, we thus have from \eqref{6.1} that
\begin{eqnarray}
	\label{6.3}
	&&\mathcal{E}_{2,l,q}(\tau)+\int^{\tau}_{0}\|\sqrt{\bar{u}_{1y}}(\widetilde{\rho},\widetilde{u}_{1},\widetilde{\theta})(s)\|^{2}\,ds\notag\\
	&&\quad+\frac{1}{2}q_{2}\int^{\tau}_{0}q_{3}(s)\mathcal{H}_{2,l,q}(s)\,ds
	+\frac{1}{2}\int^{\tau}_{0}\mathcal{D}_{2,l,q}(s)\,ds
	\notag\\
	&&\leq C_{3}k^{\frac{1}{3}}\varepsilon^{\frac{6}{5}-\frac{4}{5}a}<
	\frac{1}{2}k^{\frac{1}{6}}\varepsilon^{\frac{6}{5}-\frac{4}{5}a}.
\end{eqnarray}
Then \eqref{6.3} implies that 
\begin{equation}
	\label{6.4}
	\sup_{0\leq\tau\leq\tau_{1}}\mathcal{E}_{2,l,q}(\tau)+\int^{\tau_{1}}_{0}\mathcal{D}_{2,l,q}(\tau)\,d\tau
	<k^{\frac{1}{6}}\varepsilon^{\frac{6}{5}-\frac{4}{5}a},
\end{equation}
for $\tau_{1}\in(0,+\infty)$, which is strictly stronger than \eqref{3.22}. Thus the a priori assumption \eqref{3.22} can be closed by \eqref{6.4}.
The existence and uniqueness of the local solutions to the VPL system near a global Maxwellian
were well known in the torus or the whole space, cf. \cite{Guo-JAMS,Strain-Zhu}.
By a straightforward modification of the argument there, one can obtain the local existence of the solutions to the VPL system \eqref{prefeq} with \eqref{prefeqid1} and \eqref{prefeqid2}
under the assumptions in Theorem \ref{thm2.1}.
For brevity, we omit details of the proof. By the uniform a priori estimates \eqref{3.22} and the local existence of the solution,
the standard continuity argument gives the existence and uniqueness of global solutions
to the VPL system \eqref{prefeq} with \eqref{prefeqid1} and \eqref{prefeqid2}. 
Moreover, the desired estimate \eqref{2.5} holds true. This then completes the proof of Theorem \ref{thm2.1}.

Let's be back to the original Cauchy problem \eqref{1.1} in variables $(t,x)$ with \eqref{1.4B}, \eqref{1.1id}, \eqref{1.4a} and  \eqref{1.5a}  and continue to justify the uniform convergence rate in $\varepsilon$ as in \eqref{thm.rate} which consequently implies \eqref{thm.con}. In fact, by using the embedding inequality,  \eqref{3.17} and \eqref{6.4}, we obtain
\begin{equation}
	\label{6.5}
	\sup_{0\leq\tau\leq+\infty}\{\|(\widetilde{\rho},\widetilde{u},\widetilde{\theta},\widetilde{\phi})(\tau)\|_{L_{y}^{\infty}}
	+\|f(\tau)\|_{L^{\infty}_{y}L^{2}_{v}}\}\leq Ck^{\frac{1}{12}}\varepsilon^{\frac{3}{5}-\frac{2}{5}a}.
\end{equation}
From \eqref{4.27A} and \eqref{3.21}, it is straightforward to check that
\begin{equation}
	\label{6.6}
	\sup_{0\leq\tau\leq+\infty}\|\frac{\overline{G}(\tau)}{\sqrt{\mu}}\|_{L_{y}^{\infty}L_{v}^{2}}
	\leq Ck^{\frac{1}{12}}\varepsilon^{\frac{3}{5}-\frac{2}{5}a}.
\end{equation}
Recalling that $F=M+\overline{G}+\sqrt{\mu}f$, we thereupon conclude from \eqref{3.24}, \eqref{6.5} and \eqref{6.6} that
\begin{eqnarray}
	\label{6.7}
	&&\sup_{0\leq\tau\leq+\infty}\|(\frac{F-M_{[\bar{\rho},\bar{u},\bar{\theta}]}}{\sqrt{\mu}})(\tau)\|_{L_{x}^{\infty}L_{v}^{2}}
	\notag\\
	&&\leq \sup_{0\leq\tau\leq+\infty}\big\{\|(\frac{M-M_{[\bar{\rho},\bar{u},\bar{\theta}]}}{\sqrt{\mu}})(\tau)\|_{L_{y}^{\infty}L_{v}^{2}}
	+\|f(\tau)\|_{L_{y}^{\infty}L_{v}^{2}}+\|\frac{\overline{G}(\tau)}{\sqrt{\mu}}\|_{L_{y}^{\infty}L_{v}^{2}}
	\big\}
	\notag\\
	&&\leq Ck^{\frac{1}{12}}\varepsilon^{\frac{3}{5}-\frac{2}{5}a}.
\end{eqnarray}
In addition, by Lemma \ref{lem7.1} and \eqref{3.21}, we have for $t>0$ that
\begin{equation}
	\label{6.8}
	\|(\bar{\rho},\bar{u},\bar{\theta},\bar{\phi})(t,x)-(\rho^{R},u^{R},\theta^{R},\phi^{R})(\frac{x}{t})\|_{L^{\infty}_{x}}\leq
	C \frac{1}{k}t^{-1}\varepsilon^{\frac{3}{5}-\frac{2}{5}a}\{\ln(1+t)+|\ln \varepsilon|\}.
\end{equation}
Hence, for any given constant $\ell>0$ and all $t\in[\ell,+\infty)$, by using \eqref{6.7} and \eqref{6.8},
there exists a constant $C_{\ell,k}>0$, independent of $\varepsilon$,   such that
\begin{eqnarray*}
	&&\|\frac{F(t,x,v)-M_{[\rho^{R},u^{R},\theta^{R}](x/t)}(v)}{\sqrt{\mu}}\|_{L_{x}^{\infty}L_{v}^{2}}\\
	&&\leq \|\frac{F-M_{[\bar{\rho},\bar{u},\bar{\theta}]}}{\sqrt{\mu}}\|_{L_{x}^{\infty}L_{v}^{2}}
	+\|\frac{M_{[\bar{\rho},\bar{u},\bar{\theta}]}-M_{[\rho^{R},u^{R},\theta^{R}]}}{\sqrt{\mu}}\|_{L_{x}^{\infty}L_{v}^{2}}
	\notag\\
	&&\leq C_{\ell,k}\varepsilon^{\frac{3}{5}-\frac{2}{5}a}|\ln\varepsilon|,
\end{eqnarray*}
and
\begin{equation*}
	\|\phi(t,x)-\phi^{R}(\frac{x}{t})\|_{L^{\infty}_{x}}
	\leq C_{\ell,k}\varepsilon^{\frac{3}{5}-\frac{2}{5}a}|\ln\varepsilon|.
\end{equation*}
This gives \eqref{thm.rate} and hence
ends the proof of Theorem \ref{thm.rs}.\qed

\section{Appendix}\label{sec.9}

In this appendix, we give the details of deriving the estimate \eqref{4.7} for completeness.

\medskip
\noindent{\it Proof of \eqref{4.7}:}
First note that $\theta=\frac{3}{2}\frac{1}{2\pi e}\rho^{\frac{2}{3}}\exp(S)$ due to \eqref{4.1}, then one gets
$$
\theta_{\rho}:=\partial_{\rho}\theta=\frac{1}{2\pi e}\rho^{-\frac{1}{3}}\exp(S)=\frac{2}{3}\frac{\theta}{\rho},\quad \theta_{S}=\frac{3}{2}\frac{1}{2\pi e}\rho^{\frac{2}{3}}\exp(S)=\theta.
$$
By employing this and \eqref{4.4}, we have from an elementary computation that
\begin{eqnarray*}
	\eta_{\bar{\rho}}(\tau,y)&&=-\frac{\rho}{\bar{\rho}}\bar{\theta}(S-\bar{S})-\frac{5}{3\bar{\rho}}\bar{\theta}(\rho-\bar{\rho}),
	\quad \eta_{\bar{u}}(\tau,y)=-\frac{3}{2}\rho(u-\bar{u}),
	\notag\\
	\eta_{\bar{S}}(\tau,y)&&=-\frac{3}{2}\rho\bar{\theta}(S-\bar{S})-\bar{\theta}(\rho-\bar{\rho}),
	\quad \bar{S}_{\tau}=-\frac{2}{3}\frac{\bar{\rho}_{\tau}}{\bar{\rho}}+\frac{\bar{\theta}_{\tau}}{\bar{\theta}}.
\end{eqnarray*}
Using the scaling transformation \eqref{3.1} and the fact that $\bar{u}_{2}=\bar{u}_{3}=0$, we deduce from \eqref{1.28} that
\begin{eqnarray*}
	\bar{\rho}_{\tau}&=&-(\bar{\rho}\bar{u}_{1y}+\bar{\rho}_{y}\bar{u}_{1}),\\ 
	\bar{u}_{1\tau}&=&-\frac{2}{3}
	\frac{\bar{\theta}}{\bar{\rho}}\bar{\rho}_{y}-\frac{2}{3}\bar{\theta}_{y}-\bar{u}_{1}\bar{u}_{1y}-\bar{\phi}_{y},\\ \bar{\theta}_{\tau}&=&-\bar{u}_{1}\bar{\theta}_{y}
	-\frac{2}{3}\bar{\theta}\bar{u}_{1y}.
\end{eqnarray*}
With the help of the above estimates, we have from the straightforward computations that
\begin{eqnarray*}
	&&\nabla_{[\bar{\rho},\bar{u},\bar{S}]}\eta(\tau,y)\cdot(\bar{\rho},\bar{u},\bar{S})_{\tau}\\
	&&=\eta_{\bar{\rho}}(\tau,y)\bar{\rho}_{\tau}
	+\eta_{\bar{u}}(\tau,y)\bar{u}_{\tau}+\eta_{\bar{S}}(\tau,y)\bar{S}_{\tau}
	\notag\\
	&&=\frac{\bar{\theta}}{\bar{\rho}}(\bar{\rho}-\rho)\bar{\rho}_{\tau}-\frac{3}{2}\rho(u-\bar{u})\bar{u}_{\tau}
	-\{\frac{3}{2}\rho(S-\bar{S})+(\rho-\bar{\rho})\}\bar{\theta}_{\tau}
	\notag\\
	&&=\frac{5}{3}(\rho-\bar{\rho})\bar{\theta}\bar{u}_{1y}+\frac{\bar{\theta}}{\bar{\rho}}\rho u_{1}\bar{\rho}_{y}
	+\rho u_{1}\bar{\theta}_{y}-\bar{u}_{1}\bar{\theta}\bar{\rho}_{y}+\frac{3}{2}\rho\bar{u}_{1}(u-\bar{u})\bar{u}_{y}
	\notag\\
	&&\quad-\bar{\rho}\bar{u}_{1}\bar{\theta}_{y}
	-\frac{3}{2}\rho(\bar{S}-S)(\bar{u}_{1}\bar{\theta}_{y}+\frac{2}{3}\bar{\theta}\bar{u}_{1y})
	+\frac{3}{2}\rho(u_{1}-\bar{u}_{1})\bar{\phi}_{y}.
\end{eqnarray*}
Similarly,
\begin{eqnarray*}
	&&\nabla_{[\bar{\rho},\bar{u},\bar{S}]}q(\tau,y)\cdot(\bar{\rho},\bar{u},\bar{S})_{y}\\
	&&=q_{\bar{\rho}}(\tau,y)\bar{\rho}_{y}
	+q_{\bar{u}}(\tau,y)\bar{u}_{y}+q_{\bar{S}}(\tau,y)\bar{S}_{y}
	\notag\\
	&&=\bar{\rho}\bar{\theta}\bar{u}_{1y}-\rho\theta\bar{u}_{1y}
	-\frac{\bar{\theta}}{\bar{\rho}}\rho u_{1}\bar{\rho}_{y}
	-\rho u_{1}\bar{\theta}_{y}+\bar{u}_{1}\bar{\theta}\bar{\rho}_{y}
	\notag\\
	&&\quad-\frac{3}{2}\rho u_{1}(u-\bar{u})\bar{u}_{y}+\bar{\rho}\bar{u}_{1}\bar{\theta}_{y}
	+\frac{3}{2}\rho u_{1}(\bar{S}-S)\bar{\theta}_{y}.
\end{eqnarray*}
Combining the above two estimates, it holds that
\begin{eqnarray*}
	I_{13}:=&&-\nabla_{[\bar{\rho},\bar{u},\bar{S}]}\eta(\tau,y)\cdot(\bar{\rho},\bar{u},\bar{S})_{\tau}
	-\nabla_{[\bar{\rho},\bar{u},\bar{S}]}q(\tau,y)\cdot(\bar{\rho},\bar{u},\bar{S})_{y}\\
	&&+\frac{3}{2}\rho(u_{1}-\bar{u}_{1})\bar{\phi}_{y}
	\notag\\
	=&&\frac{3}{2}\rho\bar{u}_{1y}(u_{1}-\bar{u}_{1})^{2}+\frac{2}{3}\rho\bar{\theta}\bar{u}_{1y}\Psi(\frac{\bar{\rho}}{\rho})
	+\rho\bar{\theta}\bar{u}_{1y}\Psi(\frac{\theta}{\bar{\theta}})\\
	&&+\frac{3}{2}\rho\bar{\theta}_{y}(u_{1}-\bar{u}_{1})(\frac{2}{3}\ln\frac{\bar{\rho}}{\rho}+\ln\frac{\theta}{\bar{\theta}})
	\notag\\
	=&&(\frac{3}{2}\rho-\mathbf{a})\bar{u}_{1y}(u_{1}-\bar{u}_{1})^{2}+\frac{2}{3}\rho\bar{\theta}\bar{u}_{1y}\Psi(\frac{\bar{\rho}}{\rho})
	+\rho\bar{\theta}\bar{u}_{1y}\Psi(\frac{\theta}{\bar{\theta}})\notag\\
	&&+\mathbf{a}\bar{u}_{1y}(u_{1}-\bar{u}_{1})^{2}
	+\frac{3}{2}\rho\bar{\theta}_{y}(u_{1}-\bar{u}_{1})(\frac{2}{3}\ln\frac{\bar{\rho}}{\rho}+\ln\frac{\theta}{\bar{\theta}})\\
	&&+\mathbf{b}(\frac{2}{3}\ln\frac{\bar{\rho}}{\rho}+\ln\frac{\theta}{\bar{\theta}})^{2}-
	\mathbf{b}(\frac{2}{3}\ln\frac{\bar{\rho}}{\rho}+\ln\frac{\theta}{\bar{\theta}})^{2},
\end{eqnarray*}
where $\mathbf{a}>0$ and $\mathbf{b}>0$ be determined later and $\Psi(s)=s-\ln s-1$.
To bound $I_{13}$, we let  $4\bar{u}_{1y}\mathbf{a}\mathbf{b}=(\frac{3}{2}\rho\bar{\theta}_{y})^{2}$ and $0<\mathbf{a}<\frac{3}{2}\rho$
to get
\begin{eqnarray*}
	I_{13}
	\geq&&(\frac{3}{2}\rho-\mathbf{a})\bar{u}_{1y}\widetilde{u}_{1}^{2}+\frac{2}{3}\rho\bar{\theta}\bar{u}_{1y}\Psi(\frac{\bar{\rho}}{\rho})
	+\rho\bar{\theta}\bar{u}_{1y}\Psi(\frac{\theta}{\bar{\theta}})
	-\mathbf{b}(\frac{2}{3}\ln\frac{\bar{\rho}}{\rho}+\ln\frac{\theta}{\bar{\theta}})^{2}
	\notag\\
	\geq&& c_{3}\bar{u}_{1y}(\widetilde{\rho}^{2}+\widetilde{u}_{1}^{2}+\widetilde{\theta}^{2}),
\end{eqnarray*}
for some constant $c_{3}>0$, where we have claimed that
\begin{equation}
	\label{7.32}
	\frac{2}{3}\rho\bar{\theta}\bar{u}_{1y}\Psi(\frac{\bar{\rho}}{\rho})
	+\rho\bar{\theta}\bar{u}_{1y}\Psi(\frac{\theta}{\bar{\theta}})
	-\mathbf{b}(\frac{2}{3}\ln\frac{\bar{\rho}}{\rho}+\ln\frac{\theta}{\bar{\theta}})^{2}\geq c\bar{u}_{1y}\{(\frac{\bar{\rho}}{\rho}-1)^{2}+(\frac{\theta}{\bar{\theta}}-1)^{2}\}.
\end{equation}
In what follows, we prove the inequality \eqref{7.32}. Letting $\mathbf{b}=\frac{1}{5}\rho\bar{\theta}\bar{u}_{1y}>0$ and using
$$
\bar{\theta}_{y}=\frac{1}{2\pi e}e^{S_{*}}\bar{\rho}^{-\frac{1}{3}}\frac{1}{\sqrt{\frac{5}{3}k\bar{\rho}^{-\frac{4}{3}}e^{S_{*}}
		+\bar{\rho}^{-1}\big(\frac{d}{d\rho}\rho^{-1}_{\mathrm{e}}(\bar{\rho})\big)}}
\bar{u}_{1y}
$$
in Lemma \ref{lem7.1} as well as  $4\bar{u}_{1y}\mathbf{a}\mathbf{b}=(\frac{3}{2}\rho\bar{\theta}_{y})^{2}$, one has
$$
\mathbf{a}=\frac{45\rho}{16\bar{\theta}}\{\frac{1}{2\pi e}e^{S_{*}}\bar{\rho}^{-\frac{1}{3}}
\frac{1}{\sqrt{\frac{5}{3}\frac{1}{2\pi e}\bar{\rho}^{-\frac{4}{3}}e^{S_{*}}
		+\bar{\rho}^{-1}\big(\frac{d}{d\rho}\rho^{-1}_{\mathrm{e}}(\bar{\rho})\big)}}\}^{2}<\frac{3}{2}\rho,
$$
where we used $\frac{1}{2\pi e}\exp(S_{*})=\frac{2}{3}\rho_{-}^{-2/3}\theta_{-}$, $\bar{\rho}^{-1}\big(\frac{d}{d\rho}\rho^{-1}_{\mathrm{e}}(\bar{\rho})\big)
=\bar{\rho}^{-1}(\rho'_{\mathrm{e}}(\bar{\phi}))^{-1}>0$ and \eqref{2.2}.
We define the following function
$$
f(x_{1},x_{2})=\frac{2}{3}\Psi(x_{1})+\Psi(x_{2})-\frac{1}{5}(\frac{2}{3}\ln x_{1}+\ln x_{2})^{2}.
$$
It is easy to check that
\begin{eqnarray*}
	&&f(1,1)=f_{x_{1}}(1,1)=f_{x_{2}}(1,1)=0,\quad f_{x_{1}x_{1}}(1,1)=\frac{22}{45},
	\notag\\
	&&f_{x_{1}x_{2}}(1,1)=f_{x_{2}x_{1}}(1,1)=-\frac{4}{15},\quad f_{x_{2}x_{2}}(1,1)=\frac{3}{5}.
\end{eqnarray*}
Note that the Hessian matrix $\det\nabla^{2}f(1,1)=\frac{2}{9}$ by the above facts,
and hence $f(x_{1},x_{2})$ is convex function near the point $(1,1)$. We thereupon conclude that
$$
f(x_{1},x_{2})\geq c[(x_{1}-1)^{2}+(x_{2}-1)^{2}].
$$
Note that $(\frac{\bar{\rho}}{\rho},\frac{\theta}{\bar{\theta}})$ is close enough  to the state $(1,1)$ in terms of \eqref{3.22}.
Taking $x_{1}=\frac{\bar{\rho}}{\rho}$ and $x_{2}=\frac{\theta}{\bar{\theta}}$ in the above inequality,
then we can obtain the expected estimate \eqref{7.32}, and hence completes the proof of \eqref{4.7}.\qed

\medskip

\noindent {\bf Acknowledgment:}\, The research of Renjun Duan was partially supported by the General Research Fund (Project No.~14301719) from RGC of Hong Kong and a Direct Grant from CUHK. The research of Hongjun Yu was supported by the GDUPS 2017 and the NNSFC Grant 11371151. Dongcheng Yang would like to thank Department of Mathematics, CUHK  for hosting his visit in the period 2020-2023.

\medskip

\noindent{\bf Conflict of Interest:} The authors declare that they have no conflict of interest.


\normalsize

\end{document}